\documentclass[reqno,10pt]{amsart}

\usepackage{amsmath,mathrsfs}
\usepackage{amssymb, amsmath}
\usepackage{graphicx}
\usepackage{enumerate,color,xcolor}
\usepackage[colorlinks,
citecolor=green
]{hyperref}

\usepackage[normalem]{ulem}

\newtheorem{defi}{Definition}[section]
\newtheorem{thm}{Theorem}[section]
\newtheorem{lem}{Lemma}[section]
\newtheorem{rmk}{Remark}[section]
\newtheorem{cor}{Corollary}[section]
\newtheorem{prop}{Proposition}[section]

\numberwithin{equation}{section}

\newcommand{\beq}{\begin{equation}}
\newcommand{\eeq}{\end{equation}}
\newcommand{\ben}{\begin{eqnarray}}
\newcommand{\een}{\end{eqnarray}}
\newcommand{\beno}{\begin{eqnarray*}}
\newcommand{\eeno}{\end{eqnarray*}}
\let\f=\frac

\newcommand{\be}{\begin{equation} \label}
	\newcommand{\ee}{\end{equation}}
\newcommand{\bea}{\begin{eqnarray}\label}
	\newcommand{\eea}{\end{eqnarray}}
\newcommand{\bas}{\begin{eqnarray*}}
	\newcommand{\eas}{\end{eqnarray*}}
\newcommand{\bit}{\begin{itemize}}
	\newcommand{\eit}{\end{itemize}}

\newcommand{\N}{{\mathbb N}}
\newcommand{\Z}{{\mathbb Z}}
\newcommand{\R}{{\mathbb R}}
\newcommand{\T}{{\mathbb T}}

\newcommand{\pa}{\partial}

\newcommand{\supp}{{\rm supp} \, }

\newcommand{\tm}{T_{max}}

\newcommand{\ba}{\begin{aligned}}
\newcommand{\ea}{\end{aligned}}
 \def\na{\nabla}
 \newcommand{\lr}[1]{\langle #1 \rangle}
\def\eqdefa{\buildrel\hbox{\footnotesize def}\over =}
\let\pa=\partial
\let\g=\gamma

\let\Ga=\Gamma

\let\lam=\lambda

\let\f=\frac

\let\om=\omega

\let\D=\Delta

\let\ka=\kappa
\def\mP{\mathbb{P}}
\def\a{\mathfrak{a}}
\def\c{\mathfrak{m}}
\def\mF{\mathbb{F}}
\def\b{\mathfrak{b}}

\def\cF{{\mathcal F}}

\def\fM{{\mathfrak M}}

\def\cP{{\mathcal P}}

\def\pa{\partial}

\def\ta{\widetilde{a}}

\def\virgp{\raise 2pt\hbox{,}}
\def\cdotpv{\raise 2pt\hbox{;}}

\def\eqdefa{\buildrel\hbox{\footnotesize def}\over =}

\def\si{\sigma}
\def\C{\mathop{\mathbb C\kern 0pt}\nolimits}
\def\DD{\mathop{\mathbb D\kern 0pt}\nolimits}
\def\EE{\mathop{{\mathbb E \kern 0pt}}\nolimits}
\def\K{\mathop{\mathbb K\kern 0pt}\nolimits}
\def\N{\mathop{\mathbb N\kern 0pt}\nolimits}
\def\Q{\mathop{\mathbb Q\kern 0pt}\nolimits}
\def\R{\mathop{\mathbb R\kern 0pt}\nolimits}
\def\SS{\mathop{\mathbb S\kern 0pt}\nolimits}

\def \mh{\mathbf{h}}
\def \mj{\mathbf{j}}
\def \mA{\mathbf{A}}
\def \mB{\mathbf{B}}
\def \mH{\mathbf{H}}
\def \mn{\mathbf{n}}

\def\F{{\mathfrak F }}
\def\U{{\mathfrak U }}
\def\tP{{\tilde{\mathcal{P}} }}
\def\tF{{\tilde{\mathfrak F} }}
\def\tm{{\tilde{m} }}
\def\tphi{{\tilde{\phi} }}

\def\<{\langle}
\def\>{\rangle}

\def\si{\sigma}
\def\th{\theta}
\def\ga{\gamma}
\def\al{\alpha}
\def\be{\beta}
\def\de{\delta}
\def\vphi{\varphi}
\def\gs{\gtrsim}
\def\ls{\lesssim}
\def\S{\mathbb{S}}
\def\vep{\varepsilon}
\begin{document}

\title[Regularity estimates for the non-cutoff Boltzmann equation]{Regularity estimates for the non-cutoff soft potential Boltzmann equation with typical rough and slowly  decaying  data}

\author{Ling-Bing He}
\address[Ling-Bing He]{Department of Mathematical Sciences, Tsinghua University, Beijng, 100084, P.R. China.}
\email{hlb@tsinghua.edu.cn}

\author{Jie Ji}
\address[Jie Ji]{Beijing International Center for Mathematical Research, Beijng, 100080, P.R. China.}
\email{jij22@pku.edu.cn}

\begin{abstract} In the present work, we investigate estimates of regularity for weak solutions to the non-cutoff Boltzmann equation with soft potentials. We restrict our focus to the so-called "typically rough and slowly decaying data", which is constructed to satisfy typical properties: low regularity and having exact polynomial decay in high velocity regimes. By exploring the degenerate and non-local properties of the collision operator, we demonstrate that (i) such data induce only finite smoothing effects for weak solutions in Sobolev spaces; (ii) this finite smoothing property implies that the Leibniz rule does not hold for high derivatives of the collision operator (even in the weak sense). Moreover, we can also prove that the average of the solution or the average of the collision operator on a special domain in $\mathbb{R}^3_v$ will induce discontinuity in the $x$ variable. These facts present major obstacles to proving the conjecture that solutions to the equation will instantly become infinitely smooth for both spatial and velocity variables at any positive time if the initial data has only polynomial decay in high velocity regimes. 
\end{abstract}

\maketitle

\setcounter{tocdepth}{1}
\tableofcontents

%\keywords{Landau equation, Landau operator, degenerate diffusion, Coulomb interaction}

%\subjclass[2010]{35B65, 35K67, 45G05, 76P05, 82C40, 82D10}

%\setcounter{tocdepth}{1}
%\tableofcontents
\section{Introduction} The Boltzmann equation is a fundamental nonlinear evolution model from statistical mechanics, given by:
\begin{equation}\label{Boltzmann}
\partial_t f + v \cdot \nabla_x f = Q(f,f),
\end{equation}
where $f(t,x,v) \geq 0$ is a distributional function representing colliding particles that move with velocity $v \in \R^3$, at time $t>0$ and position $x \in \T^3$.
If the macroscopic quantities related to the solution $f$ satisfy:
\begin{equation}\label{hyas}
0<m_0 \leq \rho(t,x) \leq M_0, \quad E(t,x)\leq E_0, \quad H(t,x)\leq H_0,
\end{equation}
where:
\begin{equation}\label{Macroquantiy}
\rho(t,x) := \int_{\R^3} f(t,x,v) dv, \quad E(t,x) := \int_{\R^3} f(t,x,v) |v|^2 dv, \quad H(t,x) := \int_{\R^3} f \ln f(t,x,v) dv,
\end{equation}
then, by \cite{ADVW, AMUXY1, FB, BD, GS, HE}, the leading part of the equation will behave like a fractional Kolmogorov equation:
\begin{equation}\label{fkolmeq}
\partial_t f + v \cdot \nabla_x f + C_f (-\Delta_v)^s (\langle v \rangle^{\gamma} f) = 0.
\end{equation}
Here, $\gamma \in (-3,0)$ and $s \in (0,1)$. Due to this diffusive property, the equation attracts a lot of attention in relation to regularity issues. There exist two strategies to obtain global regularity which rely heavily on coercivity estimates for the collision operator and the averaging lemma (or the hypocoercivity property):

\noindent\underline{(i).} {\it The energy method in Sobolev spaces:} %The $L^2$ energy method relies heavily on coercivity estimates for the collision operator and the averaging lemma (or the hypocoercivity property) due to the free transport operator $\partial_t + v \cdot \nabla_x$, which is used to transfer the regularity from the $v$ variable to the $x$ variable. 
See \cite{AMUXY1, CH1, CH2, LD, DW, MY} for $H^\infty$-smoothing effects. Further Gevrey or Gelfand-Shilov smoothing properties can also be obtained. See \cite{BHRV, LMPX} and related references for details.

\noindent\underline{(ii).} {\it De Giorgi and Schauder argument in Hölder spaces:} In a series of works \cite{IMS1, IS2, IS1, IS3}, the authors proved global regularity for the non-cutoff Boltzmann equation in the case of $\gamma + 2s \geq 0$, conditional on macroscopic bounds. When $\gamma < 0$, the initial data are imposed to have rapid decay with any polynomial rate. To prove the desired result, the authors established $L^\infty$ estimates, weak Harnack inequality, Schauder estimates, and pointwise decay estimates for kinetic integro-differential equations. We mention that these arguments can be used to prove well-posedness as well (see \cite{HST1, HST2, HW}).
 \smallskip

 In the present work, we are curious about the regularity of the weak solutions to the nonlinear Boltzmann equation \eqref{Boltzmann} even though the global existence is still an open problem. In particular, we shall explore the impact and the associated consequences of degenerate factor $\lr{v}^\gamma$(see \eqref{fkolmeq}) on the smoothing property of the weak solution for \eqref{Boltzmann}.  Before going further, we begin with the basic assumptions on the collision kernel and the introduction of the weak solution and the re-normalized solution to \eqref{Boltzmann}.

\subsection{Collision kernel, weak solution and re-normalized solution} In this subsection, we will list the basic assumptions on the collision kernel and explain why we choose the weak solution as a target rather than the re-normalized solution.

 \subsubsection{Collision kernel}The Boltzmann collision operator $Q$ is a bilinear operator which acts only on the velocity variable $v$, that is
\[
Q(g,f)(v)=\int_{\R^3}\int_{\mathbb{S}^2}B(v-v_*,\si)(g'_*f'-g_*f)d\si dv_*.
\]
 
\noindent Some explanations on the collision operator are in order. 

\noindent$(i)$. We use the standard shorthand $f=f(v),g_*=g(v_*),f'=f(v'),g'_*=g(v'_*)$, where $v',v'_*$ are given by
	\ben\label{sigmare}
	v'=\frac{v+v_*}{2}+\frac{|v-v_*|}{2}\sigma,~~v_*'=\frac{v+v_*}{2}-\frac{|v-v_*|}{2}\sigma,~\si\in\mathbb{S}^2.
	\een
	This representation follows the parametrization of set of solutions of the physical law of elastic collision:
	\beno
	v+v_*=v'+v'_*,  \quad
	|v|^2+|v_*|^2=|v'|^2+|v'_*|^2.
	\eeno

\noindent$(ii)$. The nonnegative function $B(v-v_*,\si)$ in the collision operator is called the Boltzmann collision kernel. It is always assumed to depend only on $|v-v_*|$ and the deviation angle $\th$ through $\cos\th:=\frac{v-v_*}{|v-v_*|}\cdot\si$.
  In the present work,  our {\bf basic assumptions on the kernel $B$}  can be concluded as follows:
	\begin{itemize}
		\item[$\mathbf{(A1)}$] The Boltzmann kernel $B$ takes the product form: $B(v-v_*,\si)=|v-v_*|^\ga b(\frac{v-v_*}{|v-v_*|}\cdot\si)$, where   $b$ is a nonnegative function.
		
		\item[$\mathbf{(A2)}$] The angular function $b(t)$ is not locally integrable and it satisfies
		\[
	C\theta^{-1-2s}\leq \sin\theta b(\cos\theta)\leq C^{-1}\theta^{-1-2s},~\mbox{with}~0<s<1,~C>0.
		\]
		
		\item[$\mathbf{(A3)}$]
		The parameter $\ga$ and $s$ satisfy the condition $-2s-1<\gamma <0$.

		\item[$\mathbf{(A4)}$]  Without lose of generality, we may assume that $B(v-v_*,\si)$ is supported in the set $0\leq \th\leq \pi/2$, i.e.$\frac{v-v_*}{|v-v_*|}\cdot\si\geq0$, for otherwise $B$ can be replaced by its symmetrized form:
		\beno
		\overline{B}(v-v_*,\si)=|v-v_*|^\gamma\big(b(\frac{v-v_*}{|v-v_*|}\cdot\si)+b(\frac{v-v_*}{|v-v_*|}\cdot(-\si))\big) \mathbf{1}_{\frac{v-v_*}{|v-v_*|}\cdot\si\ge0},
		\eeno
		where $\mathbf{1}_A$ is the characteristic function of the set $A$.
	\end{itemize}
 
\begin{rmk} For inverse repulsive potential, it holds that $\gamma = \frac {p-5} {p-1}$ and $s = \frac 1 {p-1}$ with $p > 2$. It is easy to check that $\gamma + 4s = 1$ which implies that  assumption $\mathbf{(A3)}$  is satisfied for the full range of the inverse power law model with soft potentials($\ga<0$). In addition, the case $\gamma > 0$,  $\gamma = 0$ correspond to so-called hard and Maxwellian potentials respectively.
\end{rmk}

\subsubsection{Weak solution}\label{Secweak}  Let us give the definition of the weak solution to \eqref{Boltzmann} as follows:
\begin{defi}[Weak solutions to \eqref{Boltzmann}]\label{Defiweak}
	Let $f_0$ be a non-negative function satisfying 
	\beno \int_{\T^3\times\R^3} f_0(1+|v|^2)dvdx<\infty, H(f_0):= \int_{\T^3\times\R^3} f_0\log(f_0)dvdx<\infty. \eeno  Under assumptions $\mathbf{(A1-A4)}$, we say that a non-negative measureable function $f(t,x,v)\in L^\infty([0,\infty),L^1_x(L^1_2\cap L\log L))$ is a weak solution of Eq.\eqref{Boltzmann} if $f$ satisfies the following statements:
	
	(i). The mass, momentum and energy is conserved, i.e. for $t\ge0$,
	\beno
	&&\int_{\T^3\times\R^3} f(t,x)dxdv=\int_{\T^3\times\R^3}f_0(x,v)dvdx;\quad\int_{\T^3\times\R^3}f(t,x,v)vdvdx=\int_{\T^3\times\R^3}f_0(x,v)vdvdx;\\
	&&\int_{\T^3\times\R^3} f(t,x)|v|^2dx=\int_{\T^3\times\R^3}f_0(x,v)|v|^2dvdx.
	\eeno
	Moreover for any $T>0$ and $D(f):=\f14\int_{\T^3\times\R^3}\int_{\S^2}B(v-v_*,\sigma)(f'_*f'-f_*f)\log\f{f'_*f'}{f_*f}d\si dvdv_*dx\ge0$, the entropy inequality holds
	\beno
	 H(f(T))+\int_0^T D(f(t,x,\cdot))dt\leq  H(f_0).
	\eeno
	
	(ii). For any $\phi\in C^1_tC^2_{x,v}$ with compact support in $[0,\infty)\times\T^3\times\R^3$ and $T\ge0$, it holds that
	\ben\label{weakformB}
	&&\int_{\T^3\times\R^3}f(T,x,v)\phi(T,x,v)dvdx-\int_{\T^3\times\R^3}f_0(x,v)\phi(0,x,v)dvdx\\
	&=&\notag\int_0^T\int_{\T^3\times\R^3}f(t,x,v)\pa_t\phi(t,x,v)dxdvdt+\int_0^T\int_{\T^3\times\R^3}f(t,x,v)v\cdot\na_x\phi(t,x,v)dxdvdt\\
	&&\notag+\lim_{\vep\rightarrow0}\int_0^T\int_{\T^3\times\R^3}Q_\vep(f,f)(t,x,v)\phi(t,x,v)dvdxdt, 
	\een 
	where $Q_\vep(g,f):=\int_{\R^3\times\S^2}B^\vep(v-v_*,\sigma)(g'_*f'-g_*f)d\si dv_*$ with $B^\vep(v-v_*,\sigma):=B(v-v_*,\sigma)\mathbf{1}_{\f{v-v_*}{|v-v_*|}\cdot\sigma\leq \cos\vep}$.
\end{defi}

Two comments are in order:

\smallskip

 \noindent{(i).} The weak form \eqref{weakformB} for the nonlinear equation follows the cutoff approximation (see \cite{HJZ}). It is especially necessary for $s\ge 1/2$ in order to make sense of the collision operator.  
\smallskip

  \noindent{(ii).} We emphasize again that the existence of a weak solution to \eqref{Boltzmann} is still an open question. The primary obstacle lies in the lack of strong compactness from the equation. To address this, DiPerna-Lions proposed the so-called "renormalized solution" under the Grad cutoff assumptions for the kernel in \cite{DL}. Alexandre-Villani extended it to the non-cutoff case in \cite{AV1} (see the definition below).
 
 \subsubsection{Re-normalized solution} We have the following definition: 
\begin{defi}[Re-normalized solutions to \eqref{Boltzmann}]\label{Defirenorw}
	Suppose the kernel $B$ verifying $\mathbf{(A1)}-\mathbf{(A4)}$. Let $f_0$ be an initial datum satisfying 
	$$\int_{\R^3_x\times\R^3_v}f_0(x,v)(1+|v|^2+|x|^2+\log f_0(x,v))dxdv<+\infty.$$
	We  say that a nonnegative function $f\in C(\R^+;D'(\R_x^3\times \R^3_v))\cap L^\infty(\R^+;L^1(1+|v|^2+|x|^2)dxdv)$ with the initial data $f_0$
	is a re-normalized solution of the Boltzmann equation with a defect measure if for all nonlinearity $\be\in C^2(\R^+,\R^+)$ satisfying
$$\be(0)=0,\quad0<\be'(f)\leq \f C{1+f}.\quad \be''(f)<0,$$
	the following inequality holds in the sense of distributions:
\ben\label{renomsol}\f{\pa \be(f)}{\pa t}+v\cdot \na_x\be(f)\geq\be'(f)Q(f,f). \een
\end{defi}

We focus our investigation on the weak solution rather than the re-normalized solution due to the following observations:

$\bullet$ The work of \cite{HST1} provides   the local existence of a solution in $L^\infty([0,T]; L^\infty_x L^\infty_q)$. This availability of a solution in a rough space makes it reasonable to explore the regularity problem specifically for the weak solution.

$\bullet$ In the study conducted by \cite{AM}, the authors derived a local smoothing estimate for $\beta(f)$, but not for the solution $f$ itself. Specifically, they demonstrated that $f/(1+f)$ belongs to the local function space $W^{\delta,p}_{\text{loc}}(dt,dx,dv)$, where $\delta=\delta(s)<s$ and $p<2/3$. Transferring such a smoothing estimate directly to the solution itself proves to be a challenging task. Hence, it further supports our decision to concentrate on the regularity problem solely for the weak solution.

\subsection{Basic notation and function space}  We  list  notations and function spaces in the below.
\subsubsection{Notations}
$(i)$  We use the notation $a\ls b(a\gs b)$ to indicate that there is a uniform constant $C$, which may be different on different lines, such that $a\leq Cb(a\geq Cb)$. We use the notation $a\sim b$ whenever $a\ls b$ and $b\ls a$.

$(ii)$ We denote $C_{a_1,a_2,\cdots,a_n}$(or $C(a_1,a_2,\cdots,a_n)$) by a constant depending on parameters $a_1,a_2,\cdots,a_n$. Moreover,  parameter $\varepsilon$ is used  to represent different positive numbers much less than 1 and determined in different cases.

$(iii)$ We write $a\pm$ to indicate $a\pm\varepsilon$, where $\varepsilon>0$ is sufficiently small.  We set  $a^+:=\max\{0,a\}$, $a^-:=-\min\{a,0\}$ and use $[a]$ to denote the maximum integer which does not exceed $a$.

$(iv)$ $\mathbf{1}_\Omega$ is the characteristic function of the set $\Omega$. We use $(f, g)_{L^2_v}$ and $(f, g)_{L^2_{x, v}}$ to denote the inner product of $f, g$ in $v$ variable and in the  $x, v$ variables, respectively.  Sometimes, we use $(f,g)$ to denote the  inner product for short without ambiguity.

$(v)$ We use $\mathcal{F}$ to denote the Fourier transform.  Subscripts will be used to specify the variables. For example, $\mathcal{F}_x$ is used to denote the Fourier transform w.r.t. $x$ variable. Correspondingly the dual variables of $t,x$ and $v$ are denoted by $\omega,m$ and $\xi$ respectively.

$(vi)$ Suppose $A$ and $B$ are two operators, then the commutator $[A,B]$ between $A$ and $B$ is defined by $[A,B]=AB-BA$. 

%$(vii)$ Let $f:\T^3\rightarrow \bar{\R}=\R\cup\{+\infty\}$. We say that $f$ is generalized continuous if (1). $f|_{f<+\infty}$ is continuous; (2). if $f(x_0)=+\infty$, then $\forall N>0$, there exists $\de>0$ such that $f(x)>N,\forall\, |x-x_0|<\de$.

\subsubsection{Function spaces}
We  give several definitions to spaces involving different variables.
\smallskip

 $(1)$ \textit{Function spaces in $v$ variable.} Let $f=f(v)$ and $\<v\>:=(1+|v|^2)^{1/2}$. For $m,l\in\R $, we define the weighted Sobolev space $H_l^m$ by
$H^m_l:=\Big\{f(v)|\|f\|_{ H^m_l}=\|\<D\>^m\<\cdot\>^lf\|_{L^2}<+\infty\Big\}.$
Here $\<D\>^m$ is a pseudo-differential operator with the symbol $\<\xi\>^m$, i.e. \beno
\<D\>^mf(v):= \frac{1}{(2\pi)^3}\int_{\R^3}\int_{\R^3}e^{i(v-u)\xi}\<\xi\>^mf(u)dud\xi.
\eeno

\noindent The weighted $L^p$ spaces($1\leq p\leq \infty$) are defined by
$L^p_l:=\{f(v)|\|f\|_{L^p_l}=\|\<\cdot\>^lf\|_{L^p}<+\infty\}$.

 $(2)$ \textit{Function spaces in $v$ and $x$ variables.} Let $f=f(x,v)$. The differential operator on $x$, $\<D_x\>^a$(or $|D_x|^a$) with $a>0$,   is defined by 
$
  \<D_x\>^af:=\sum_{q\in\Z^3}\<q\>^a\mathcal{F}_x(f)(q)e^{2\pi i q\cdot x}
$( or $|D_x|^a f:=\sum_{q\in\Z^3} |q|^a \mathcal{F}_x(f)(q)e^{2\pi i q\cdot x}$).
  For $n\geq0$, $m,l\in \R$, 
  \beno
  H^n_xH^m_l:=\Big\{f(x,v)|\|f\|^2_{H^n_xH^m_l}=\sum_{q\in\Z^3}\<q\>^{2n}\|\mathcal{F}_x(f)(q)\|^2_{H^m_l}<+\infty \Big\}.
  \eeno
For simplicity, we set $\|f\|_{H^n_xH^m_l}:=\|f\|_{H^n_xL^2_v}$ when $m=l=0$ and $\|f\|_{H^n_xH^m_l}:=\|f\|_{L^2_xH^m_l}$ when $n=0$.
 \smallskip

  $(3)$ \textit{Function spaces in $t,x,v$ variables.} Let $f=f(t,x,v)$ and $X$ be a function space in $x,v$ variables. Then $L^p([0,T],X)$ and $ L^\infty([0,T],X)$ are defined by
 \[ L^p([0,T],X):=\bigg\{f(t,x,v)\big|\|f\|^p_{ L^p([0,T],X)}=\int_0^T\|f(t)\|^p_{X}dt<+\infty\bigg\},\quad 1\leq p<\infty,\]
  \[L^\infty([0,T],X):=\Big\{f(t,x,v)|\|f\|_{ L^\infty([0,T],X)}=\mathrm{esssup}_{t\in [0,T]}\|f(t)\|_{X}<+\infty\Big\}. \]
Sometimes, we also write $L^p([0,T],\T_x^3\times \R^3_v)$ as $L^p([0,T]\times \T^3_x\times \R^3_v)$.

(4)  \textit{Besov spaces}. Let $n\in\R^+,1\leq p,q< \infty$ and $\F_j$ be defined in \eqref{DefFj}. We set
$$B^n_{p,q}(\R^3):=\Big\{f(v)|\|f\|_{B^n_{p,q}(\R^3)}=\Big(\sum_{j=-1}^\infty 2^{njq}\|\F_jf\|^q_{L^p(\R^3)}\Big)^{1/q}<+\infty\Big\},$$
$$
B^n_{p,\infty}(\R^3):=\Big\{f(v)|\|f\|_{B^n_{p,\infty}(\R^3)}=\mathrm{sup}_{j\geq-1} 2^{nj}\|\F_jf\|_{L^p(\R^3)}<+\infty\Big\}.
$$
For $f=f(t,x,v)$, we further introduce the space $\widetilde{L^p}([0,T]\times\T^3,B^n_{p,q}(\R^3))$ simply induced by the norm
\beno
\|f\|_{\widetilde{L^p}([0,T]\times\T^3,B^n_{p,q}(\R^3))}:=\Big(\sum_{j=-1}^\infty 2^{njq}\|\F_jf\|^q_{L^p([0,T]\times\T^3\times\R^3)}\Big)^{1/q}.
\eeno

\subsection{Purpose and related topics}
The main purpose of this work is to establish several quantitative estimates on weak solutions  using a constructive argument due to the degenerate property and the non-local property of the collision operator. To be more precise,  
\smallskip

\noindent$\bullet$ We construct so-called {\it typical rough and slowly decaying data} to satisfy the typical properties: low regularity and having exact polynomial decay in high velocity regimes. It is designed to capture the degenerate property of the dissipation in \eqref{fkolmeq}(recalling that $\gamma<0$) via the standard $L^2$ energy method.

\smallskip

\noindent$\bullet$  We show that the weak solution with such data possesses only finite smoothing regularity in Sobolev spaces for any positive time. As a result, this indicates that after taking certain derivatives, the solution has polynomial growth  in the large velocity regime for any positive time. In other words, $H^\infty$ smoothing property does not hold for general weak solutions.

\smallskip
\noindent$\bullet$ We prove that the finite smoothing property will further induce local properties for any positive time: $(i).$ that the Leibniz rule does not hold for high order derivatives of the collision operator. $(ii).$ Discontinuity in the $x$ variable will take place either   for the average of the solution  or for the average of the collision operator on some certain domain in $\R^3_v$.
\smallskip

Our analysis in particular sheds light on proving
  following conjecture(see also \cite{IS4}):
 \smallskip

 \noindent{\bf Conjecture 1.1} {\it Suppose that the macroscopic quantities defined in \eqref{Macroquantiy} satisfy \eqref{hyas}. Then the solution $f=f(t,x,v)$ to \eqref{Boltzmann} with soft potentials will  immediately become $C^\infty$ smooth in both   $x$ and $v$ variables, even if the initial data has only polynomial decay in large velocity regimes.} 
\smallskip 

Two comments are in order:

 \noindent$(1).$ We address that {\it Conjecture 1.1} can be proven by an additional assumption that the initial data $f_0$ rapidly decays, similar to the Schwartz function, in the $v$ variable. The strategies used in previous works relies on a bootstrap argument, where at each step, we gain regularity at the expense of some decay power of $f$ in the large velocity regime. This means that  the strategies mentioned in the above will both fail after finite steps. Thus, {\it Conjecture 1.1} is highly non-trivial.
\smallskip

\noindent$(2).$ The above results  suggest  that we cannot improve regularity via the bootstrap argument, especially through taking derivatives of the equation. At the same time, we have to take care of the potential discontinuity in $x$ variable for the average of the solution or the average of collision operator.

 \subsection{Toy model, typical rough and slowly  decaying data and basic $L^2$ energy method} To illustrate the main purpose, we begin with  the introduction of so-called {\it typical rough and slowly  decaying  data}  and the explanation of why it works well in capturing the finite smoothing effect for the nonlinear equation. For simplicity, we start with a toy model.

\subsubsection{Informal analysis of a toy model}\label{SecinftoyM} We choose the fractional Kolmogorov equation as our toy model: 
\ben\label{toym}
\pa_t f+(-\Delta_v)^s(\<v\>^{\ga}f)=0.
\een
We consider the Cauchy problem with the initial data $f_0\in (L^2_{\ell}\backslash   L^2_{\ell+\vep}) \cap (L^2\backslash H^\vep)$ with $\vep\ll1$. We are curious about the role of the factor $\<v\>^{\ga}$($\gamma<0$) in the regularity problem since it   degenerates at infinity for the dissipation. To accomplish this, we will conduct an informal analysis using dyadic decomposition.

\smallskip

\noindent\underline{\it Step 1.}  Roughly speaking, using the dyadic decomposition in phase space, we may regard the solution $f$ as the sum of the sequence $\{f_k\}_{k\ge-1}$, where $f_k$ solves the equation 
\ben\label{toymcPk}
\pa_tf_k+2^{k\ga}(-\Delta)^s f_k=0; \quad f|_{t=0}=\cP_kf_0.
\een
Here $\cP_k$ is a localized operator in phase space defined in \eqref{Defcpj}. We expect that \eqref{toymcPk} plays the same role as  localizing   the solution of \eqref{toym} to the region  $|v|\sim 2^k$.   
\smallskip

\noindent\underline{\it Step 2.} Suppose that  $\cP_kf\sim f_k$ and $G(v):=\mathcal{F}^{-1}(e^{-|\cdot|^{2s}})(v)$. Then we have
\beno
\cP_kf(t,v)\sim f_k(t,v)=\f1{(2^{\ga k}t)^{3/(2s)}}\int_{\R^3}G\Big(\frac{v-u}{(2^{\ga k}t)^{1/(2s)}}\Big)(\cP_kf_0)(u)du,
\eeno which implies that
\beno f\sim \sum_{k\ge-1} f_k=(2^{\ga k}t)^{-\frac{3}{2s}}\sum_{k\ge-1}  G\big((2^{\ga k}t)^{-\frac{1}{2s}}\cdot\big)*(\cP_kf_0). \eeno

\noindent\underline{\it Step 3.} Assuming that $f_0$ is rough, in this case, by taking the $\alpha$-th order derivative of $f$, we obtain that
\beno \pa^\alpha f&\sim& \sum_{k\ge-1} \pa^\alpha f_k=\sum_{k\ge-1} (2^{\ga k}t)^{-\frac{3+|\alpha|}{2s}}(\pa^\alpha G)\big((2^{\ga k}t)^{-\frac{1}{2s}}\cdot\big)*(\cP_kf_0)\\
&=&\sum_{k\ge-1}  t^{-\frac{|\alpha|}{2s}}2^{(-\ga\frac{|\alpha|}{2s}-\ell)k}\big[ t^{-\frac{3}{2s}}2^{-\frac{3}{2s}\ga k}(\pa^\alpha G)\big((2^{\ga k}t)^{-\frac{1}{2s}}\cdot\big)\big]*(2^{k\ell}\cP_kf_0).\eeno 
In particular, the index $(-\gamma\frac{|\alpha|}{2s}-\ell)$ reflects the growth rate of $\partial^\alpha f$ in the $v$ variable, i.e.,
\ben\label{PointwToyM}|(\pa^\alpha f)(t,v)|\sim C(t,f_0)\<v\>^{-\ga\frac{|\alpha|}{2s}-\ell}.\een

\smallskip

 The challenging problem now is to give a rigorous proof to \eqref{PointwToyM}. Since we lack of the fundamental solution to \eqref{toym}, obtaining the pointwise estimate in \eqref{PointwToyM} seems difficult. Instead, we consider \eqref{PointwToyM} in the $L^2$ framework for any $t>0$. It suffices to show:

 $\bullet$ If $\al\in\Z_+^3$, then 
\ben\label{smoothingToyM} \|\<\cdot\>^{\frac{\ga}{2s}|\alpha|+\ell}\pa_v^\alpha f(t)\|^2_{L^2}\sim C(t,f_0).\een

$\bullet$ If $|\al|>2s\ell/(-\ga)$, then
  \ben\label{smoothingToyM1} \|\pa^\al_v f(t)\|_{L^2}^2=+\infty.\een

\subsubsection{Typical rough and slowly decaying  data} To show  (\ref{smoothingToyM}-\ref{smoothingToyM1}), we   adapt the initial data to be   suitable for the energy method.  This motives us to introduce  so-called  {\it `` typical rough and slowly decaying  data"} which are  characterized by dyadic decomposition  in both frequency and phase spaces.  

\begin{defi}\label{RDSL2} Let $ a\in\R^+$ and $\mathbf{M}:=\cup_{j=1}^\infty\mathbf{M}_j:=\cup_{j=1}^\infty\{(m_1,m_2,m_3)| m_1=\pm[2^{(1-\vep)j}],m_2=m_3=0\}$ with $0<\vep\ll1$. Then the set of  {\it ``typical rough and slowly decaying   data''} can be defined by 
\ben\label{RDs} 
\mathcal{R}_{sd}(\ell;\vep,a)&:=&\bigg\{0\leq f\in  L^\infty_xL^2_\ell\bigg|\, \sum_{j=1}^{\infty} \sum_{l>a j} 2^{(2\ell a+\vep) j}\sum_{m\in\mathbf{M}_j}\|\F_j\cP_l\mathcal{F}_x(f)(m,\cdot)\|^2_{L^2_v}=+\infty\bigg\},
\een
where  $\F_j$ and $\cP_l$ are localized operators in frequency space and in phase space respectively(see \eqref{DefFj} and \eqref{Defcpj}). The set of   \, {\it ``regular in $x$ but slowly decaying   data''} is defined by:
\beno
\mathcal{H}^n_{sd}(\ell;\vep,a)&:=&\bigg\{0\leq f\in  H^n_xL^2_\ell\bigg|\, \sum_{j=1}^{\infty} \sum_{l>a j} 2^{(2\ell a+2n(1-\vep)+\vep) j}\sum_{m\in\mathbf{M}_j}\|\F_j\cP_l\mathcal{F}_x(f)(m,v)\|^2_{L^2_v}=+\infty\bigg\}.
\eeno
\end{defi}

Several explanations are in order:
\smallskip

\noindent\underline{(i).} For simplicity, let us assume that $f$ is independent of $x$. The ``infinity  condition'' in \eqref{RDs} is reduced to \ben\label{homgrds}\sum_{j=-1}^{\infty} \sum_{l>a j} 2^{(2\ell a+\vep) j}\|\F_j\cP_lf\|^2_{L^2_v}=+\infty.\een
 Thanks to the profile of the weighted Sobolev spaces, for $f=f(v)\in H^m_\ell$, we have 
\[\|f\|_{H^m_\ell}^2\sim \sum_{l,j=-1}^\infty 2^{2\ell l+2mj} \|\F_j\cP_lf\|^2_{L^2_v}.\] The summation region  $\{(j,l)|l>a j\}$ in \eqref{homgrds} indicates  that localization in phase space prevails over  localization in the frequency space. As a consequence,  \eqref{homgrds} implies that
\beno &&\|f\|_{L^2_{\ell+(2a)^{-1}\vep}}^2\sim \sum_{l,j=-1}^\infty  2^{2(\ell+(2a)^{-1}\vep)l}\|\F_j\cP_lf\|^2_{L^2_v}\ge \sum_{j=-1}^{\infty} \sum_{l>a j} 2^{2(\ell+(2a)^{-1}\vep)aj}\|\F_j\cP_lf\|^2_{L^2_v}=+\infty; \\
&& \|f\|_{H^{\vep/2}_{\ell}}^2\sim \sum_{l,j=-1}^\infty 2^{2\ell l}2^{\vep j} \|\F_j\cP_lf\|^2_{L^2_v}\ge\sum_{l>a j} 2^{(2\ell a+ \vep)j}\|\F_j\cP_lf\|^2_{L^2_v}=+\infty.
\eeno
 These coincide with our expectation on the data: roughness and slowly  decay  in the large velocity regime.
\smallskip

\noindent\underline{(ii).}   The parameter $a$ in \eqref{RDs} should be well-chosen according to the degenerate property from the dissipation. In fact, it reflects some kind of the competition between the phase and the frequency for the dissipation. For instance, for the toy model \eqref{toym}, we will set  $a:=(2s)/|\gamma|$. The reason stems from the point of view in the frequency space : by Fourier transform, $(-\triangle)^s \lr{v}^{\gamma}  \F_j\cP_l$ is comparable to $2^{2sj}2^{\gamma l}\F_j\cP_l$.  Then the  dissipation vanishes when $l=(2s)j/|\gamma|$. One may also observe it in the phase space. For instance,  estimate \eqref{PointwToyM} indicates that the critical case takes place when $|\alpha|=(2s)\ell/|\gamma|$.  
\smallskip

\noindent\underline{(iii).}  Due to the technical restrictions, we have no idea how to handle the initial data in the space $(L^\infty_xL^2_{\ell}\backslash L^\infty_xL^2_{\ell+\vep})\cap (L^\infty_xL^2_v\backslash L^\infty_xH^\vep_v)$ with $\vep\ll1$, which is probably   the biggest function space for rough and slowly decaying data.  We mention that the essential boundedness w.r.t. the $x$ variable is used to handle the quadratic terms in the collision operator. It seems to be the minimal requirement to close the energy estimates.
 \smallskip

Before ending the subsection, based on Lemma \ref{contra}, we   prove that the sets $\mathcal{R}_{sd}(\ell;\vep,a)$ and  $\mathcal{H}^n_{sd}(\ell;\vep,a)$ are not empty. This might have independent interest.

\begin{prop}\label{NemptyRsd}  $\mathcal{R}_{sd}(\ell;\vep,a)$ and $\mathcal{H}^n_{sd}(\ell;\vep,a)$(defined in Definition \ref{RDSL2}) are nonempty.
\end{prop}
\begin{proof}  Thanks to Lemma \ref{contra}, there exists a function $h=h(v)\in L^2_\ell$ such  that $$\sum_{j=1}^{\infty} \sum_{l>a j} 2^{(2\ell a+\vep/2) j}\|\F_j\cP_lh\|^2_{L^2_v}=+\infty.$$   Let $f(x,v):=g(x)h(v)$.  If $g:=100+\sum\limits_{m\in\mathbf{M}} (\ln|m|)^{-2}e^{-2\pi im\cdot x}$, then $0<g \in L^\infty_x$.  This yields that $f^+\in\mathcal{R}_{sd}(\ell;\vep,a)$ or  $f^-\in\mathcal{R}_{sd}(\ell;\vep,a)$. 
Similarly, if $g:=100+\sum\limits_{m\in\mathbf{M}} |m|^{-n}(\ln|m|)^{-2}e^{-2\pi im\cdot x}$, then  $f^+\in \mathcal{H}^n_{sd}(\ell;\vep,a)$ or  $f^-\in \mathcal{H}^n_{sd}(\ell;\vep,a)$.
\end{proof}

\subsubsection{$L^2$ energy method}\label{IDENTRDS} We will give a rigorous proof  of  (\ref{smoothingToyM}-\ref{smoothingToyM1}) for the toy model \eqref{toym} with {\it typical rough and slowly decaying data} via the basic $L^2$ energy method. The idea can be adapted for the nonlinear equation \eqref{Boltzmann}.
\smallskip

 \underline{(1).} We first prove (\ref{smoothingToyM}). By applying localized operators $\F_j\cP_k$(defined in Subsection \ref{DDP}) to \eqref{toym}, we obtain 
 \ben\label{ajk}
 \pa_t\F_j\cP_k f+\F_j\cP_k(-\Delta_v)^s\<v\>^{\ga}f=0.
 \een
 Thanks to estimate on the commutator(see Lemma \ref{refc}), we have,  for $n,l\in\R$,
\beno
\f d{dt}2^{2nj}2^{2lk}\|\F_j\cP_kf\|^2_{L^2}+2^{2nj}2^{2lk}\|\F_j\cP_kf\|^2_{H^s_{\ga/2}}\le C_N2^{2nj}2^{2lk}(2^{(2s-1)j}2^{(\ga-1)k}\|\mF_j\mP_kf\|^2_{L^2}+2^{-jN}2^{-kN}\|f\|^2_{L^2}).
\eeno
 Summing up w.r.t. $j,k\geq-1$,   due to Lemma \ref{lemma1.4}, we have  
\beno
\f d{dt}\|f\|^2_{H^n_{l}}+\|f\|^2_{H^{n+s}_{l+\ga/2}}\leq  \|f\|^2_{L^2}. 
\eeno
From this, we can derive that  $\|f(t)\|_{L^2_\ell}\le C(T,\|f_0\|_{L^2_\ell})$ for any $t\in [0,T]$. 
Using interpolation inequality $\|f\|_{H^n_{l}}\ls \|f\|^{n/(n+s)}_{H^{n+s}_{l+\ga/2}}\|f\|^{s/(n+s)}_{L^2_\ell}$ with $l:=\f{\ga}{2s}n+\ell$ and Lemma \ref{le1.6}, we can easily obtain that $\|f(t)\|_{H^n_{l}}<+\infty,~\forall t>0$. This proves (\ref{smoothingToyM})
and the $C^\infty$ smoothing effect for any $t>0$.

\smallskip

 \underline{(2).} Next we  focus on   (\ref{smoothingToyM1}). Again by \eqref{ajk} and Lemma \ref{refc},
 we  obtain
\beno
\f d{dt}\|\F_j\cP_kf\|^2_{L^2}\geq-2^{2sj}2^{\ga k}\|\F_j\cP_kf\|^2_{L^2}- C_N(2^{(2s-1)j}2^{(\ga-1)k}\|\mF_j\mP_kf\|^2_{L^2}+2^{-jN}2^{-kN}\|f\|_{H^{-N}}).
\eeno
Let $k>\f{2s+\de}{-\ga} j$ with $0<\de\ll1$. By multiplying $2^{2\ell\f{2s+\de}{-\ga} j+\de j}$ on both sides and summing up w.r.t. $j\geq-1$, we can derive that
\beno
\f{d}{dt}\sum_{j=-1}^\infty\sum_{k>\f{2s+\de}{-\ga} j}2^{2\ell\f{2s+\de}{-\ga} j+\de j}\|\F_j\cP_kf\|^2_{L^2}\geq -\sum_{j=-1}^\infty\sum_{k>\f{2s+\de}{-\ga} j}2^{2\ell\f{2s+\de}{-\ga} j}\|\F_j\cP_kf\|^2_{L^2}-\|f\|^2_{L^2_\ell}\geq-C\|f\|^2_{L^2_\ell}.
\eeno
This inspires us to set that $f_0$ is $x$-independent and $f_0\in \mathcal{R}_{sd}(\ell;\delta,\f{2s+\de}{-\ga})$, i.e.,
\ben\label{moti}
\sum_{j=-1}^\infty\sum_{k>\f{2s+\de}{-\ga} j}2^{2\ell\f{2s+\de}{-\ga} j+\de j}\|\F_j\cP_kf_0\|^2_{L^2}=+\infty.
\een
Thus, for any $t>0$ and $n>\f{2s}{-\ga}\ell$, we have  
\ben\label{asd}
\|f(t)\|^2_{H^n}\gs\sum_{j=-1}^\infty\sum_{k>\f{2s+\de}{-\ga} j}2^{2\ell\f{2s+\de}{-\ga} j+\de j}\|\F_j\cP_kf(t)\|^2_{L^2}=+\infty.
\een
This  proves the desired result.

\subsection{Main results} We are now in a position to state our main results regarding the weak  solutions to the  nonlinear equations with  initial data $f_0\in \mathcal{R}_{sd}(\ell;\vep,\a+\de)$.
  We begin with our first result, which concerns the regularity estimate for the weak solution to the nonlinear Boltzmann equation when the initial datum is in the set  $\mathcal{R}_{sd}(\ell;\vep,\a+\de)$. Here, we make full use of the degenerate property of the dissipation.
 
\begin{thm}[{Regularity estimates for weak solution(I)}]\label{RSWBol1}
	Suppose $\ga+2s<0,\ell\ge3, \a:=\f {2s}{|\ga+2s|}$ and $\de:=\f{4\vep}{|\ga+2s|}$ with $0<\vep\ll1$. Let $f(t,x,v)$ be a weak solution to \eqref{Boltzmann} with $f_0\in \mathcal{R}_{sd}(\ell;\vep,\a+\de)$. Then for $0<T<\infty$, one of the following assertions must hold:
	\begin{itemize} 
\item[(i)]  $\|f\|_{L^\infty([0,T]; L^\infty_x L^2_\ell)}=+\infty$;
\item[(ii)]  $\bullet$  If $\ga>-\f32$, then for any $p\in(1,2)$, $\<v\>^{\ga/2}f\in \widetilde{L^p}([0,T]\times\T^3,B^\c_{p,\infty})$ with $\c:=\f{s(2-\f2p)(\f2p-1)}{\f2p-1+11(1-\f23\ga)(2-\f2p)}$. Moreover, $\<D_x\>^{\tilde{\c}/(1+2s+\tilde{\c})}\<v\>^{-(\ga+2s+\f32)}f\in L^p([\tau_1,\tau_2]\times \T^3\times\R^3)$ for  $\tilde{\c}<\c$ and $0<\tau_1<\tau_2<T$.\\
 $\bullet$ If $\mathbf{n},\mathbf{m}\in \R $ satisfying 
    $\mathbf{n}(1-\vep)+\mathbf{m}> \ell \a +4\ell\vep/|\ga+2s|+\vep/2$, then for any $t\in[0,T]$,
    \ben\label{FSE} \|f(t)\|_{H^{\mathbf{n}}_xH^{\mathbf{m}}_v}=+\infty. \een
In particular, for any $\mathbf{n}+\mathbf{m}>\ell\a$, there exist $\vep>0$ and $f_0\in \mathcal{R}_{sd}(\ell;\vep,\a+\de)$ such that $\|f(t)\|_{H^{\mathbf{n}}_xH^{\mathbf{m}}_v}=+\infty$ for all $t\in[0,T]$.\end{itemize} 
 \end{thm}

We now investigate the consequences of \eqref{FSE} that lead to the second theorem on the quantitative estimates of the local properties of the solution. These estimates are mainly due to the non-local properties of the collision operator.

\begin{thm}[{Regularity estimates for weak solution(II)}]\label{RSWBol2}
	Suppose $\ga+2s<0,\ell\ge3, \a:=\f {2s}{|\ga+2s|}$ and $\de:=\f{4\vep}{|\ga+2s|}$ with $0<\vep\ll1$. Let $f(t,x,v)$ be a weak solution to \eqref{Boltzmann} with $f_0\in \mathcal{R}_{sd}(\ell;\vep,\a+\de)$ satisfying that $f\in L^\infty([0,T]; L^\infty_x L^2_\ell)$. Then 
 for any $\bar{t}\in(0,T)$, there exist $\mathfrak{N}_1(\bar{t}),\mathfrak{N}_2(\bar{t})\in \Z_+^3$ such that
 \beno
 &&|\mathfrak{N}_1(\bar{t})|\in [1,2+\big(\a(\ell-(\ga+2s))+\f{4\vep}{|\ga+2s|}(\ell-(\ga+2s))+\f\vep 2(1-(\ga+2s)/\ell)+2\big)/(1-\vep)];\\
 &&|\mathfrak{N}_2(\bar{t})|\in [1,4+\a (\ell-(\ga+2s))+(\f{4\vep}{|\ga+2s|}+\f\vep{2\ell})(\ell-(\ga+2s))-2\vep],
 \eeno
and it holds that
    \ben
    &&\label{BlowupxN1L1}\|\pa^{\mathfrak{N}_1}_xf(\bar{t})\|_{L^1_xL^1_{\ga+2s}}=+\infty\quad\mbox{and}\quad  \|\pa^{\mathfrak{N}}_xf(\bar{t})\|_{L^1_xL^1_{\ga+2s}}<\infty,\quad\forall~ |\mathfrak{N}|<|\mathfrak{N}_1|;\\
    &&\label{BlowupvN2L1}\|\pa^{\mathfrak{N}_2}_vf(\bar{t})\|_{L^1_xL^1_{\ga+2s}}=+\infty\quad\mbox{and}\quad\|\pa^{\mathfrak{N}}_vf(\bar{t})\|_{L^1_xL^1_{\ga+2s}}<\infty,\quad\forall~ |\mathfrak{N}|<|\mathfrak{N}_2|.
    \een
   Here $\|\pa^{\mathfrak{N}_1}_xf(\bar{t})\|_{L^1_xL^1_{\ga+2s}}=+\infty$( $\|\pa^{\mathfrak{N}_2}_vf(\bar{t})\|_{L^1_xL^1_{\ga+2s}}=+\infty$) means  either the weak derivative $\pa^{\mathfrak{N}_1}_xf(\bar{t})$( $\pa^{\mathfrak{N}_2}_vf(\bar{t})$)  does not exist or the weak derivative $\pa^{\mathfrak{N}_1}_xf(\bar{t})$(  $\pa^{\mathfrak{N}_2}_vf(\bar{t})$) is not in weighted $L^1$ space. In what follows, we only consider the case  that the weak derivative $\pa^{\mathfrak{N}_1}_xf(\bar{t})$(  $\pa^{\mathfrak{N}_2}_vf(\bar{t})$) exists  but  is   not in weighted $L^1$ space.
   This leads to the failure of Leibniz rule for collision operator: 

(1). Suppose that $|\mathfrak{N}_1|\geq5$ and the density function $\rho$ defined in \eqref{Macroquantiy} at time $\bar{t}$ has the lower bound:
  \ben\label{bartlowbrho}\inf_{x\in\T^3}\rho(\bar{t},x)\ge c>0. \een  There exists functions   
   $\chi\in C^\infty_c(\R^3_v)$ and  $\varrho\in C^\infty(\T^3_x)$ such that the following formula \textbf{does not hold} %in the sense that function in right-hand parentheses not in $L^1$ 
  \ben\label{FLebnizx}
	  \big\langle Q(f,f)(\bar{t}), \chi (-1)^{|\mathfrak{N}_1|}\pa^{\mathfrak{N}_1}_x\varrho  \big\rangle_{x,v}= \big\langle \sum_{\al+\be=\mathfrak{N}_1}Q(\pa^\al_xf,\pa^\be_xf)(\bar{t}) , \varrho \chi \big\rangle_{x,v}.
	  \een

(2). Suppose that $|\mathfrak{N}_2|\geq5$ and $f$ is a spatially homogeneous solution to \eqref{Boltzmann}.
    There exists    a function 
   $\chi\in C^\infty_c(\R^3_v)$  such that the following formula \textbf{does not hold}
  \ben \big\langle Q(f,f)(\bar{t}),  (-1)^{|\mathfrak{N}_2|}\pa^{\mathfrak{N}_2}_v\chi \big\rangle_{v}= \big\langle \sum_{\al+\be=\mathfrak{N}_2}Q(\pa^\al_vf,\pa^\be_vf)(\bar{t}),  \chi\big\rangle_{v}.
    \label{FLebnizv}\een
\end{thm}
\smallskip

  Comments on Theorem \ref{RSWBol1} and Theorem \ref{RSWBol2}  are in order. 
\smallskip

 \noindent$\bullet$ {\it Comment on finite smoothing effect.}   
\smallskip

\noindent\underline{(i).}  Firstly, the finite smoothing estimates are derived {\it without} imposing any lower bound condition on the density function. The proof combines the $L^2$ energy method  introduced in Section \ref{IDENTRDS} and the interpolation method from \cite{AM}.  
The assumption that $f\in L^\infty([0,T]; L^\infty_x L^2_\ell)$ is utilized to establish the estimates for the quadratic terms arising from the collision operator. However, if we replace this assumption with $f\in L^\infty([0,T]; L^\infty_x L^\infty_\ell)$, then the restriction $\gamma>-3/2$ can be eliminated, resulting in improved regularity in Sobolev spaces. 
\smallskip

 \noindent\underline{(ii).}  The smoothing effect in Sobolev spaces suggests that the non-cutoff equation differs significantly from the cutoff equation. However, the smoothing effect for the $x$ and $v$ variables is extremely limited, especially when $s\ll1$, as indicated by \eqref{FSE}.
\smallskip

 \noindent\underline{(iii).}  We have not been able to obtain the pointwise estimate given in \eqref{PointwToyM}.  Instead, we can demonstrate that for  high derivative   $\pa^\al$ and any positive time $t$, there exists a sub-sequence $\{(x_k,v_k)\}_{k\in\N}$ such that 
 \ben\label{behavinfity}\lim_{k\rightarrow\infty} | (\pa^\alpha_x f)(t,x_k,v_k)|=+\infty.\een
  This can be shown by using the facts that  $\|f(t)\|_{L^2_xL^2_3}<+\infty$ and $\|f(t)\|_{H^n_xL^2_v}=+\infty$, and by applying interpolation to get  \[\|f(t)\|_{H^{2n}_xL^2_{-3}}\gs\|f(t)\|^2_{H^n_xL^2_v}\|f(t)\|^{-1}_{L^2_xL^2_3}=+\infty.\]
 Finally, if $f(t)$ is in $C^{2n}_{x,v}$, then \eqref{behavinfity} can be easily derived using a contradiction argument.

 \smallskip

 \noindent$\bullet$ {\it Comment on initial data in $\mathcal{R}_{sd}(\ell;\vep,\a+\de)$ with $\a=\f {2s}{|\ga+2s|}$.} 
\smallskip

 \noindent\underline{(i).}  As previously mentioned, the main part of the nonlinear Boltzmann equation bears resemblance to the fractional Kolmogorov-type equation \eqref{fkolmeq}.
This led us to choose initial data from $\mathcal{R}_{sd}(\ell;\vep,a+\de)$, where   $a=\f {2s}{|\ga|}$(see explanation (ii) after Definition \ref{RDSL2}). However,   this line of thought is incorrect due to the anisotropic structure within the collision operator $Q$. Specifically, if we let $\mu:=e^{-|v|^2/2}/(2\pi)^{3/2}$, then it holds that
\ben\label{Qbehavior} -Q(\mu, f)(v)\sim (-\triangle_v)^s(\lr{v}^\gamma f)+(-\triangle_{\S^2})^s(\lr{v}^\gamma f)+(v\cdot \na_v)^s(\lr{v}^\gamma f)+L.O.T. . \een   
 Technically, we will utilize the fact that $\|(-\triangle_{\S^2})^{s/2}\lr{\cdot}^{\gamma/2}f\|_{L^2}\lesssim \|f\|_{H^s_{s+\gamma/2}}$to demonstrate that the fractional Laplace-Beltrami operator is bounded in weighted Sobolev spaces. This  implies that
\[ \|\F_j\cP_k(-\triangle_{\S^2})^{s/2}\lr{\cdot}^{\gamma/2}f\|_{L^2}^2\lesssim 2^{(\gamma+2s)k}2^{2sj}\|\F_j\cP_k f\|_{L^2}^2+L.O.T..\] 
Following the argument   in Section \ref{IDENTRDS}, we must change the index $\a$ from $\f {2s}{|\ga|}$ to $\f {2s}{|\ga+2s|}$ for the nonlinear equation \eqref{Boltzmann}.  
\smallskip

\noindent\underline{(ii).} The anisotropic structure is also the only reason why Theorem \ref{RSWBol1} and \ref{RSWBol2} are restricted to the case $\ga+2s<0$. 
While for the radially symmetric homogeneous equation (which implies that the solution satisfies $f=f(|v|)$), \eqref{Qbehavior} turns to be 
\beno -Q(\mu, f)(r)\sim (-\pa_r^2)^s(r^\gamma f) +(r\pa_r)^s(r^\gamma f)+L.O.T. . \eeno   
It implies that we have to  set $\a=\f{2s}{|\gamma+s|}$. Furthermore, the restriction $\gamma+2s<0$ in the theorem can be relaxed to $\gamma+s<0$ due to the improved upper bounds for the collision operator in \cite{HE}.  Roughly speaking, if $h=h(|v|), f=f(|v|)$, then
\[ |(Q(\mu,h),f)_v|\le  (\|h\|_{H^s_{\gamma/2}}+\|h\|_{H^{s/2}_{(\gamma+s)/2}})(\|f\|_{H^s_{\gamma/2}}+\|f\|_{H^{s/2}_{(\gamma+s)/2}}).\]
  Once again, following  the argument used in Section \ref{IDENTRDS}, this leads to the change of the   index to $\f {2s}{|\ga+s|}$.
\smallskip

\noindent$\bullet$ {\it Comment on the consequences of non-local property of the collision operator.}

\noindent\underline{(i).}  As we demonstrated in the informal analysis of the toy model \eqref{toym}, one additional derivative on the solution leads to polynomial growth in the $v$ variable. We expect that the same holds for the nonlinear equation. Assuming that the solution to the nonlinear equation satisfies the same pointwise estimate as \eqref{PointwToyM} for   \eqref{toym} (i.e., $(\pa^\alpha f)(v)\sim \lr{v}^q$), we can consider the following scenario for the collision operator: 
 \ben\label{scenarioQ} (Q(\pa^\alpha f,f), \psi)_v= \int_{\R^6\times\SS^2}  (\pa^\alpha f)_* Bf(\psi'-\psi)d\sigma dv_*dv\sim \int_{\R^6\times\SS^2}  \lr{v_*}^qBf(\psi'-\psi)d\sigma dv_*dv\sim  \infty.\een   This scenario reveals two important properties of the collision operator: the degenerate property, which induces growth in the $v$ variable at infinity, and the non-local property, which makes it impossible to take a derivative of the operator.
 
\noindent\underline{(ii).} To implement the scenario \eqref{scenarioQ} for the operator, we split our strategy into two steps:

\underline{(ii.1)} We first reduce the infinity result \eqref{FSE} from $L^2$ space to weighted $L^1$ space(see (\ref{BlowupxN1L1}-\ref{BlowupvN2L1})). This step allows us to replace the growth property $(\pa^\alpha f)(v)\sim \lr{v}^q$ with a weighted $L^1$   norm. It ensures that \eqref{scenarioQ} still holds even though we cannot obtain the lower bound of the solution.

\underline{(ii.2)} The key step is to construct the domain $\mathcal{A}\subset\R^6\times\SS^2$ such that 
 $\int_{\mathcal{A}}  (\pa^\alpha f)_* Bf(\psi'-\psi)d\sigma dv_*dv\sim \infty$.
The construction of $\mathcal{A}$ is essential and   depends on whether we take the derivative with respect to the $x$ or $v$ variable. If we take the derivative with respect to $x$, we focus on the loss term of the collision operator (i.e., $Q^-$), while if we take the derivative with respect to $v$, we concentrate on the gain term of the collision operator (i.e., $Q^+$).
 
\smallskip

\noindent\underline{(iii).} %Let us provide further explanation on the local properties described in Theorem \ref{RSWBol2}.
%\smallskip  %\underline{(iii.1)}% Result $(i)$ in 
 Theorem \ref{RSWBol2} indicates that after taking certain derivative(which can be given quantitatively),  the Leibniz rule does not hold  for  the derivatives of collision operator in weak sense. This result can be expressed in a more precise way. If $\psi$ is defined in \eqref{7.1}, then  for any smooth functions $\chi=\chi(v),\varrho=\varrho(x)$,  it holds that 
	  \ben 
 &&( Q(f,f)(\bar{t},\cdot,\cdot),(-1)^{|\mathfrak{N}_1|}\pa^{\mathfrak{N}_1}_x\varrho\chi)_{L^2_{x,v}}
	  =\lim_{n\rightarrow \infty} (\mathcal{Q}^{(n)}_{\mathfrak{N}_1}, \chi\varrho)_{L^2_{x,v}}=\lim_{n\rightarrow \infty} \big((\mathcal{Q}^{(n)}_{\mathfrak{N}_1},\chi)_{L^2_v}, \varrho\big)_{L^2_{x}},   \label{QC}\\
   \notag  &&\mbox{where}\quad \mathcal{Q}^{(n)}_{\mathfrak{N}_1}(\bar{t},x,v):=\sum_{\al+\be=\mathfrak{N}_1} \mathcal{Q}^{(n)}_{\mathfrak{N}_1,\al,\be}(\bar{t},x,v):=\sum_{\al+\be=\mathfrak{N}_1}
   \int_{\R^3\times\S^2}\psi\big((|v|^2+|v_*|^2)/n\big)B(|v-v_*|,\si)\\
	  &&\times \big(\underbrace{\pa^\al_xf(\bar{t},x,v'_*)\pa^{\be}_xf(\bar{t},x,v')}_{\mbox{positive part}}-\underbrace{\pa^\al_xf(\bar{t},x,v_*)\pa^{\be}_xf(\bar{t},x,v)}_{\mbox{negative part}}\big)d\si dv_*. \label{QnN1}\een
	  %&&  \mbox{}\quad (Q(f,f)(\bar{t},\cdot),(-1)^{|\mathfrak{N}_2|}\pa^{\mathfrak{N}_2}_v\chi)_{L^2_{v}}
    %=\lim_{n\rightarrow \infty} \sum_{\al+\be=\mathfrak{N}_2}\int_{\R^3\times\R^3\times\S^2}\psi((|v|^2+|v_*|^2)/n)\nonumber\\
    %\notag&&\times B(|v-v_*|,\si)(\pa^\al_vf(\bar{t},x,v'_*)\pa^{\be}_vf(\bar{t},x,v')-\pa^\al_vf(\bar{t},x,v_*)\pa^{\be}_vf(\bar{t},x,v))\chi(v)\varrho(x)d\si dv_*dvdx.\nonumber
    %\een
We note that the right-hand side of \eqref{QC} only holds in the improper sense, which suggests that the weak derivative $\pa^{\mathfrak{N}_1}_x$ on the collision operator might not exist. Both properties will prevent us from improving the regularity by a bootstrap argument, particularly through taking derivatives on the nonlinear equation. 
\smallskip

\noindent\underline{(iv).}  Theorem \ref{RSWBol2} and  Proposition \ref{GainSR}(see Section \ref{substrong}) demonstrate that the multi-indices $\mathfrak{N}_1(\bar{t})$ and $\mathfrak{N}_2(\bar{t})$ are linearly dependent on $\ell$. Thus the assumption that $|\mathfrak{N}_1(\bar{t})|\ge 5$ is  reasonable if $\ell$ is sufficiently large.  This observation also suggests that the maximal regularity index in Sobolev space is almost the same as that in the continuous space.

\smallskip
\noindent\underline{(v).} As mentioned earlier, the similar results in Theorem \ref{RSWBol2} can be extended to radial solutions of the homogeneous equation. In this case, the restriction $\gamma+2s<0$ can be relaxed to $\gamma+s<0$. This implies that the global regularity result obtained in \cite{IS1} is still open for $\gamma\in[-2s,0)$, particularly when the initial data is rough and decays slowly at large velocities. 

\subsection{Two applications}\label{substrong}  As  direct consequences of Theorem \ref{RSWBol1} and Theorem \ref{RSWBol2}, we obtain several applications for the  solutions to the nonlinear equation \eqref{Boltzmann}. The first application is to show that discontinuity will occur either for   the average of the solution or for the average of the collision operator on some certain domain in $\R^3_v$.

\begin{prop}[{Potential discontinuity in $x$ variable}]\label{coro1} We say that $f:\T^3\rightarrow \bar{\R}=\R\cup\{+\infty\}$ is generalized continuous if (1).$f\not\equiv \infty$; (2). $f|_{f<+\infty}$ is continuous; (3). if $f(x_0)=+\infty$, then $\forall N>0$, there exists $\de>0$ such that $f(x)>N,\forall\, |x-x_0|<\de$.
	%\underline{(ii).}
%	(Di)
 Under the condition of Theorem \ref{RSWBol2} and suppose that for any $x\in\T^3$, $f(\bar{t},x,v)\not\equiv 0$. Then one of the following discontinuities takes place:
	\smallskip
	
	\noindent\underline{Case 1}:  $(\pa^{\mathfrak{N}_1}_xf)(\bar{t}{},x,v)$ or  $\big(\pa^{\mathfrak{N}_1}_x(\pa_t +v\cdot \na_x)\big) f(\bar{t},x,v)$ 
	is not continuous.
	\smallskip
	
	\noindent\underline{Case 2}:  $\int_{\R^3}|\pa^{\mathfrak{N}_1}_xf(\bar{t},x,v)|\<v\>^{\ga+2s}dv$ is not generalized continuous in $\T^3$.
	\smallskip
	
	\noindent\underline{Case 3}: Assume that $(\pa^{\mathfrak{N}_1}_xf)(\bar{t}),\big(\pa^{\mathfrak{N}_1}_x(\pa_t +v\cdot \na_x)\big) f(\bar{t})\in C(\T^3\times\R^3)$ and $\int_{\R^3}|\pa^{\mathfrak{N}_1}_xf(\bar{t},x,v)|\<v\>^{\ga+2s}dv$ is generalized continuous.
	
	Thanks to \eqref{BlowupxN1L1}, there exists $\bar{x}\in\T^3$ such that 
	\ben\label{finfty}
	\int_{\R^3} |(\pa^{\mathfrak{N}_1}_xf)(\bar{t},\bar{x},v)|\<v\>^{\ga+2s}dv=+\infty.
	\een
	Then one of the following discontinuities takes place:
	\begin{itemize}
		\item[(3.1)]  There exists an open subset $S\subset\R^3$ such that 
		$\int_{S}\pa^{\mathfrak{N}_1}_xf(\bar{t},x,v)\<v\>^{\ga+2s}dv$ is not generalized continuous at $\bar{x}$;
		\item[(3.2)]  There exists an open subset $S\subset\R^3$ such that  
		$\int_{S}\pa^{\al}_xf(\bar{t},x,v)\<v\>^{\ga+2s}dv$ is not continuous at $\bar{x}$  for some $\al\in \Z^3_+$ verifying $|\al|< |\mathfrak{N}_1|$;
		\item[(3.3)] There exists an open subset $\mathcal{A}\subset\S^2\times \R^3\times\R^3$ such that
		\ben\label{partofQ}
		\pa^{\mathfrak{N}_1}_x\bigg(\int_{\mathcal{A}} B(|v-v_*|,\si)f(\bar{t}, x,v_*)f(\bar{t},x,v)(\chi(v')-\chi(v))d\si dv_* dv\bigg)
		\een is not continuous at $\bar{x}$.
	\end{itemize}
\end{prop}
%\underline{(iii.2)}  
\begin{rmk}
Let us explain the results described in {\it Case 3} more clearly.  Suppose that {\it Case 1} and {\it Case 2} do not hold, then by \eqref{QC}, we have
\[(\pa^{\mathfrak{N}_1}_xQ(f,f),\chi\varrho)_{L^2_{x,v}}
=\lim_{n\rightarrow \infty} (\mathcal{Q}^{(n)}_{\mathfrak{N}_1}, \chi\varrho)_{L^2_{x,v}}=\big((\pa^{\mathfrak{N}_1}_xQ(f,f), \chi)_{L^2_v},\varrho\big)_{L^2_{x}}=\lim_{n\rightarrow \infty} \big((\mathcal{Q}^{(n)}_{\mathfrak{N}_1},\chi)_{L^2_v}, \varrho\big)_{L^2_{x}}.\]
Thus,  the discontinuities described in  {\it Case 3} suggest that   the  ``  negative part '' of $\mathcal{Q}^{(n)}_{\mathfrak{N}_1}$ or $(\mathcal{Q}^{(n)}_{\mathfrak{N}_1},\chi)_{L^2_v}$   contains the potential discontinuity in   $x$ variable.  
\end{rmk}
%\noindent$\bullet$ {\it Comment on the discontinuity in $v$ variable and the conditional results on {\it Conjecture 1.1} and {\it  1.2}.} 

%\underline{(i).}  Unfortunately, even for the spatially homogeneous equation of \eqref{Boltzmann}, we cannot prove {\it Conjecture 1.1} and {\it Conjecture 1.2}. The only available result is that the Leibniz rule does not hold for high derivatives of the collision operator, as stated in \eqref{FLebnizv}.

%\underline{(ii).} We cannot apply the {\it constructive argument} used for the $x$ variable to obtain discontinuity in the $v$ variable. The main obstacle is that the integration by parts in the $v$ variable will induce a "boundary term" that is difficult to handle.

% \underline{(iii).} The assumpiton in result $(iii)$ is equivalent to that $\lim\limits_{|v|\rightarrow+\infty} |(\pa^{\mathfrak{N}_1}_xf)(\bar{t},x,v)|\ge c>0$, which stems from \eqref{PointwToyM} for the toy model \eqref{toym}.  

Our second application is to inquire whether our quantitative estimates in Theorems \ref{RSWBol1} and \ref{RSWBol2} are sharp or not. To answer this question, we focus on the gain of Sobolev regularity from the initial data with regularity in  $x$ variable but slow decay in  $v$ variable.

 \begin{prop}[{Sobolev regularity for strong solution}]\label{GainSR} Let  $\ell\geq 14,T>0,-1-2s<\ga<0$. 
 	Suppose that $f$ is a strong solution to the Boltzmann equation satisfying \eqref{hyas} and $f\in L^\infty([0,T],H^2_xL^2_\ell)\cap L^2([0,T],H^2_xH^s_{\ell+\gamma/2})$. Let $0<\tau_1<\tau_2<T$, $\mathbf{a}:=\max\{-\ga/2,2s\},\mathbf{b}:=\f s{1+2s}$ and $\mathcal{K}:=\max\{-\f{\ga}{2s},2(2s-1)^+,(2s+1)\mathbf{a}/s\}$.
 	\begin{itemize} \item (Gain of regularity for spatial variable) If $0\le m \le (\ell-14)/\mathbf{a}$, then
 	\ben\label{GainRx}
 	f\in L^\infty([\tau_1,\tau_2],H^{2+m\mathbf{b}}_xL^2_{\ell-m\mathbf{a}})\cap L^2([\tau_1,\tau_2],H^{2+m\mathbf{b} }_xH^s_{\ell-m\mathbf{a}+\gamma/2}).
 	\een
 \item (Gain of regularity for velocity variable) If $\th\in[0,1]$, then
 	\ben\label{GainRv}
 	f\in L^\infty([\tau_1,\tau_2],L^{2}_xH^{{(\ell-14)(1-\th)}/{\mathcal{K}}}_{\ell\th}).
 	\een
 	\end{itemize}
 \end{prop}

As a direct consequence of the above results, we have the following corollary:
\begin{cor}\label{globaldecay} 
Suppose that $\ga+2s<0$, $\ell\ge 3$ and $f=f(t,x,v)$ is a global classical solution for the Boltzmann equation  \eqref{Boltzmann} with the kernel $B$ verifying \big($\mathbf{(A1)}$-$\mathbf{(A4)}$\big) and  initial datum $f_0\in \mathcal{H}^2_{sd}(\ell;\vep,\a+\de)$ with $\de=\f{4\vep}{|\ga+2s|}$ and $0<\vep\ll1$. Assume that \eqref{hyas} hold. Then the following statements hold:
\begin{itemize}
\item Suppose that at $\bar{t}>0$, $f(\bar{t})$ has the maximal smoothing estimates: $f(\bar{t})\in H^{\mathbf{n}}_xL^2_v\cap L^2_xH^{\mathbf{m}}_v$, then $\mathbf{n}\in [2+(\ell-14) \mathbf{b}/\mathbf{a},(\ell \a+4\ell \vep/|\ga+2s|+\vep/2)/(1-\vep)+2)$ and $\mathbf{m}\in[(\ell-14)/\mathcal{K},\ell\a+4\ell \vep/|\ga+2s|+\vep/2+2(1-\vep))$;
\item For all $t\in \R^+$, $f(t)$ is $x$-dependent(i.e., $\|\na_x f(t)\|_{L^2}\neq 0$);
\item For  all $t\in \R^+$, $\|f(t)-\mu\|_{L^2}\neq0$.
\end{itemize}
\end{cor} 

 Comments on Proposition \ref{GainSR} and Corollary \ref{globaldecay}  are in order:
\smallskip

\noindent\underline{(i).} The condition $f\in L^\infty([0,T],H^2_xL^2_\ell)\cap L^2([0,T],H^2_xH^s_{\ell+\gamma/2})$ is motivated by the result in \cite{CHJ}, where the authors established the well-posedness of \eqref{Boltzmann} if $f_0\in H^2_xL^2_\ell$ with $\ell\ge 32$. Therefore, the set of solutions that satisfy the conditions in Proposition \ref{GainSR} and Corollary \ref{globaldecay} is nonempty.

\smallskip

\noindent\underline{(ii).} The key innovation of Proposition \ref{GainSR} lies in (\ref{GainRx}-\ref{GainRv}), which indicate that the gain of Sobolev regularity depends linearly on the parameter $\ell$ for both $x$ and $v$ variables. These estimates are comparable to the estimate \eqref{smoothingToyM} for the toy model. Specifically, the index of the gain of regularity in the $v$ variable for the nonlinear equation is $(\ell-14)(1-\th)/\mathcal{K}$, which is nearly identical to the index $\ell(1-\th)/(-\ga/2s)$ in \eqref{smoothingToyM}. The fact that $\mathcal{K}\neq -\ga/2s$ is due to the anisotropic structure of the collision operator, as stated in \eqref{Qbehavior}.
\smallskip

\noindent\underline{(iii).} The results in Corollary \ref{globaldecay} follow directly from Theorem \ref{RSWBol1} and  Proposition \ref{GainSR}. As an immediate consequence, we obtain two interesting assertions. The first result shows that an inhomogeneous solution with slowly decaying data cannot become a homogeneous solution for any positive time. The second result concerns the "rigidity" property of the equation, which states that any bounded solution to the Boltzmann equation cannot reach the associated equilibrium in a finite time unless it coincides with the equilibrium initially. We refer the reader to \cite{lumohot} for a more general result on the spatially homogeneous cut-off equation with hard potentials. Here, we only prove this result for slowly decaying data.

\subsection{ $C^\infty$-regularity for Linear Boltzmann equation} \label{linBE}We finally consider the linear Boltzmann equation which reads:
\ben\label{LHB}
\pa_th+v\cdot\na_xh=-Lh:=(Q(\mu,h)+Q(h,\mu));\quad h|_{t=0}=h_0,
\een
where $\mu=e^{-|v|^2/2}/(2\pi)^{3/2}$. 
\smallskip

We first note that the linear Boltzmann equation possesses not only the properties of the toy model \eqref{toym} but also those of the nonlinear equation. Specifically,
\underline{(i).} Since the Gaussian function $\mu$ has exponential decay in both frequency and phase spaces, it guarantees that after taking certain derivatives, the solution has at most polynomial growth in the $v$ variable, as in \eqref{PointwToyM} for the toy model \eqref{toym}. \underline{(ii).} Due to the dissipation degeneracy induced by $Q(\mu,h)$ and the non-local property of the term $Q(h,\mu)$, if $h_0\in \mathcal{R}_{sd}(\ell;\vep,\a+\de)$ with $\a=\f{2s}{|\gamma+2s|}$ and $\de=\f{4\vep}{|\ga+2s|}$, the similar results in Theorems \ref{RSWBol1} and \ref{RSWBol2} still hold for \eqref{LHB}
\smallskip

In this scenario, we are curious about how to enhance the regularity estimates, particularly in cases where the Leibniz rule fails for high-order derivatives of $Lh$ with respect to both $x$ and $v$ variables. We believe the answer to this question will shed light on the nonlinear equation, especially in understanding the structure of the collision operator. 
 The regularity estimates for the linear equation can be concluded as follows:
 
\begin{thm}\label{CtvLBE} Let  $h$ be  a solution to the linear  Boltzmann equation \eqref{LHB}.

	(i).  Suppose that $-1-2s<\ga<0$. Let $h\in L^\infty([0,T],H^\infty_x(L^2_3\cap L^1_2))$. Then for any $t\in(0,T],n\in\N$, $h(t)\in H^\infty_xH^n_{-\ell}$ with $\ell>\mathbf{1}_{s\geq1/2}(3-\ga/2)(2n+2s-1)+\mathbf{1}_{s<1/2}(3-\ga/2)n/s$. In other word, for any positive time $t>0$, $h(t)\in C^\infty_{x,v}$. 
\smallskip

(ii). Suppose that $\ga+2s\leq0$ and $h\in L^\infty([0,T],H_x^2(L^2_3\cap L^1_2))$. If $\omega(t,v):=\int_{\T^3}h(t,x,v)dx$, which is the spatially homogeneous solution to \eqref{LHB}  with $w|_{t=0}=\int_{\T^3}h_0(x,v)dx$, then for any positive time $\tau>0$ and $m\in\N$, $\<D_t\>^{m/2}\omega\in L^2((\tau,T),H^{ms/2}_{(3-\ell(m
	))/2})$ with $\ell(m)>\mathbf{1}_{2s\ge1}(3-\ga/2)(6ms+2s-1)+\mathbf{1}_{2s<1}3m(3-\ga/2)$. In particular, $\omega\in C^\infty_{t,v}$, for all $t>0$.
\end{thm}

Several comments are in order:
\smallskip

\noindent\underline{(1).}  Following the work \cite{CHJ}, it is not difficult to prove that   $h\in L^\infty([0,T],H^\infty_x(L^2_3\cap L^1_2))$ if $h_0\in H^\infty_x(L^2_3\cap L^1_2)$.

\noindent\underline{(2).}   Result $(i)$ demonstrates that if we impose $H^\infty$ regularity on the spatial variable, we can obtain a $C^\infty$ regularizing effect for any positive time, but only with respect to the $x$ and $v$ variables. We have no information on the regularizing effect for the $t$ variable. This is similar to the incompressible Navier-Stokes equations, where it is uncertain whether the $C^\infty$ smoothing effect holds for the $t$ variable due to the non-local pressure term.  

\smallskip

\noindent\underline{(3).}  Result $(ii)$ reveals that the average of the solution with respect to the spatial variable does have a $C^\infty$ regularizing effect in the $t$ and $v$ variables for any positive time.
\smallskip

\noindent\underline{(4).} Results $(i)$ and $(ii)$ indicate that proving the $C^\infty$ smoothing property for the $x$ variable is still hindered by the scenario described in \eqref{scenarioQ}, even for the linear equation.
\smallskip

\noindent\underline{(5).} Motivated by the argument in Section \ref{IDENTRDS} for \eqref{toym}, our main idea relies on the following two observations: on one hand, we use localized techniques (both in phase and frequency spaces) to avoid taking derivatives of the equation; on the other hand, thanks to \eqref{PointwToyM} for the toy model \eqref{toym}, we can improve the regularity using the $L^2$ energy method in Sobolev spaces with negative weights or in negative Sobolev spaces.

 \subsection{Organization of the paper}
In Section 2, we   introduce some preliminaries on the hypo-elliptic estimates, dyadic decomposition, and their applications to the collision operator. Section 3 is devoted to the proof of Proposition \ref{GainSR}. Section 4 and Section 5 are dedicated to proving  the finite smoothing effect in Sobolev spaces and  the local properties. In Section 6, we consider the infinity smoothing effect for the linear Boltzmann equation. We provide all the necessary background information in the Appendix.

\section{Hypo-elliptic estimates, dyadic decomposition and their applications}
This section is dedicated to the preliminaries on hypo-elliptic estimates for the transport equation and dyadic decomposition in both frequency and phase spaces. Based on these, we obtain technical lemmas on the collision operator, which will be frequently used later on.

\subsection{Hypo-elliptic estimates for transport equation} We have
 
\begin{lem}\label{hypo} Let  $g,h\in L^p(\R_t\times \T^3_x\times\R^3_v)$ and $D^\beta_v f\in L^p(\R_t\times \T^3_x\times\R^3_v)$ with $s\in[0,2],\beta\in[0,1]$.
Assume that $f\in L^p(\R_t\times \T^3_x\times\R^3_v)$ with $1<p<\infty$ satisfies 
\ben\label{Dsg}
\pa_t f+v\cdot\nabla_x f=\<D_v\>^s g+h.
\een
Then $\<D_x\>^\al f\in L^{p}(\R_t\times \T^3_x\times\R^3_v)$ with $\alpha:=\frac{\beta}{s+1+\beta}$, and
\beno
\|\<D_x\>^\alpha f\|_{L^p_{t,x,v}}\leq C_{s,\be}\big(\|\<D_v\>^\beta f\|_{L^p_{t,x,v}}+\|g\|_{L^p_{t,x,v}}+\|h\|_{L^p_{t,x,v}}\big).
\eeno
\end{lem}

For the sake of completeness, we first prove a key lemma. 
\begin{lem}\label{t3r3}
	Let $M(m,\xi)\in C^\infty(\R^3\times\R^3)$ be a  function verifying 
 $|\pa^{\nu_1}_m\pa^{\nu_2}_\xi M(m,\xi)|\leq C_{\nu_1,\nu_2}\<m\>^{-|\nu_1|}\<\xi\>^{-|\nu_2|}$ for all $\nu_1,\nu_2\in\Z^3_+$.	 
Define $S_M$   by $\mathcal{F}_{x,v}(S_Mf)(m,\xi):=M(m,\xi)\mathcal{F}_{x,v}f(m,\xi)$, $(m,\xi)\in\Z^3\times\R^3$, then $S_M$ is a bounded operator on $L^p(\T^3\times\R^3)$ for any $1<p<\infty$.
\end{lem}
\begin{proof}  Let $T_M$ be a Fourier multiplier  on $\R^3\times\R^3$ associated with the symbol $M(m,\xi)$. Then  $T_M$ is  bounded  on $L^p(\R^3\times\R^3)$ for any $1<p<\infty$.

Suppose that $P$ and $Q$ are trigonometric polynomials on $\T^3$ and let $g(v),h(v)\in C^\infty_c(\R^3)$ and   $L_\vep(v)=e^{-\pi\vep|v|^2}$ with $v\in\R^3$ and $\vep>0$. Then the following identity is valid whenever $\al,\be>0$ and $\al+\be=1$:
\ben\label{iden}
\lim_{\vep\rightarrow 0}\vep^{\f32}\int_{\R^3\times\R^3} T_M(PL_{\vep \al}g)(x,v)\overline{Q(x)L_{\vep\be}(x)h(v)}dvdx=\int_{\T^3\times\R^3}S_M(Pg)(x,v)\overline{Q(x)h(v)}dxdv.
\een
Indeed, it suffices to prove the required assertion for $P(x)=2^{2\pi in\cdot x}$ and $Q(x)=e^{2\pi ik\cdot x}$ with $n,k\in\Z^3$, since the general case follows by linear combination. On one hand, by Parseval's equality, we have 
\beno
\int_{\T^3\times\R^3}S_M(Ph)(x,v)\overline{Q(x)h(v)}dxdv&=&\sum_{m\in\Z^3}\int_{\R^3}M(m,\xi)\mathcal{F}_x(P)(m)\overline{\mathcal{F}_x(Q)(m)}\mathcal{F}_v(g)(\xi)\overline{\mathcal{F}_v(h)(\xi)}d\xi\\
&=&\left \{
\begin{array}{lr}
	\int_{\R^3}M(n,\xi)\mathcal{F}_v(g)(\xi)\overline{\mathcal{F}_v(h)(\xi)}d\xi,                    & n=k,\\	
	0,                    & n\neq k.		
\end{array}	
\right.
\eeno
On the other hand, we obtain that
\ben\label{431}
&&\notag\vep^{\f32}\int_{\R^3\times\R^3} T_M(PL_{\vep \al}g)(x,v)\overline{Q(x)L_{\vep\be}(x)h(v)}dvdx=\vep^{\f32}\int_{\R^3\times\R^3} M(m,\xi)\mathcal{F}_{x,v}(PL_{\vep\al}g)(m,\xi)\\
&&\notag\times\overline{\mathcal{F}_{x,v}(QL_{\vep\be}h)(m,\xi)}dmd\xi
=\vep^{\f32}\int_{\R^3\times\R^3}M(m,\xi)(\vep\al)^{-\f32}e^{-\pi\f{|m-n|^2}{\vep\al}}(\vep\be)^{-\f32}e^{-\pi\f{|m-k|^2}{\vep\be}}   \mathcal{F}_v(g)(\xi)\overline{ \mathcal{F}_v(h)}(\xi)dmd\xi\\
&&=(\vep\al\be)^{-\f32}\int_{\R^3\times\R^3}M(m,\xi)e^{-\pi\f{|m-n|^2}{\vep\al}}e^{-\pi\f{|m-k|^2}{\vep\be}}   \mathcal{F}_v(g)(\xi)\overline{ \mathcal{F}_v(h)}(\xi)dmd\xi.
\een

\noindent$\bullet$ If $n=k$, by $\al+\be=1$, we have
\beno
\lim_{\vep\rightarrow0}\,(\vep\al\be )^{-\f32}\int_{\R^3}\left(\int_{\R^3}M(m,\xi)  \mathcal{F}_v(g)(\xi)\overline{ \mathcal{F}_v(h)}(\xi)d\xi\right)e^{-\pi \f{|m-n|^2}{\vep\al\be}}dm=\int_{\R^3}M(n,\xi)\mathcal{F}_v(g)(\xi)\overline{\mathcal{F}_v(h)(\xi)}d\xi.
\eeno

\noindent$\bullet$ If $n\neq k$, then $|n-k|>1$. Since either $|m-n|>1/2$ or $|m-k|>1/2$, \eqref{431} can be controlled by  $
\|M\|_{L^\infty}(\al^{-\f32}e^{-\f{\pi}{4\vep\al}}+\be^{-\f32}e^{-\f{\pi}{4\vep\be}})\int_{\R^3}|\mathcal{F}_v(g)\mathcal{F}_v(h)|(\xi)d\xi$.
  This ends the proof of \eqref{iden}.

Next we prove that $S_M$ is bounded on $L^p(\T^3\times\R^3)$ by duality. For $\{P_j\}_{j=1}^{N_1},\{Q_k\}_{k=1}^{N_2}$ trigonometric polynomials and $\{g_j\}_{j=1}^{N_1},\{h_k\}_{k=1}^{N_2}\subset C^\infty_c(\R^3)$,  using \eqref{iden}, for $1/p+1/q=1$, we have  
\beno
 &&\Big|\int_{\T^3\times\R^3}S_M(\sum_{j=1}^{N_1}P_jg_j)(x,v)\overline{\sum_{k=1}^{N_2}Q_k(x)h_k(v)}dxdv\Big|= \Big|\lim_{\vep\rightarrow 0}\vep^{\f32}\int_{\R^3\times\R^3} T_M(\sum_{j=1}^{N_1}P_jL_{\vep /p}g_j)(x,v)\\&&\times\overline{\sum_{k=1}^{N_2}Q_k(x)L_{\vep/q}(x)h_k(v)}dvdx\Big| 
 \le\|T_M\|_{L^p\rightarrow L^p}\overline{\lim\limits_{\vep\rightarrow 0}}\left(\vep^{\f32}\int_{\R^3\times\R^3}e^{-\vep\pi|x|^2}|\sum_{j=1}^{N_1}P_j(x)g_j(v)|^pdxdv\right)^{\f1p}\\&&\times\left(\vep^{\f32}\int_{\R^3\times\R^3}e^{-\vep\pi|x|^2}|\sum_{k=1}^{N_2}Q_k(x)h_k(v)|^qdxdv\right)^{\f1q} 
 \le\|T_M\|_{L^p\rightarrow L^p} \|\sum_{j=1}^{N_1}P_jg_j\|_{L^p}\|\sum_{k=1}^{N_2}Q_kh_k\|_{L^q}, 
\eeno
where we assume that for all continuous 1-periodic functions $f$ on $\R^3$, it holds that
\ben\label{4313}
\lim_{\vep\rightarrow 0}\vep^{\f32} \int_{\R^3}f(x)e^{-\vep\pi|x|^2}dx=\int_{\T^3}f(x)dx.
\een
 By the density argument, we get that
$\|S_M\|_{L^p(\T^3\times\R^3)\rightarrow L^p(\T^3\times\R^3)}\leq \|T_M\|_{L^p(\R^3\times\R^3)\rightarrow L^p(\R^3\times\R^3)}.$
We only need to prove \eqref{4313}. Observe that the left hand side of \eqref{4313} can be written as 
\beno
&&\vep^{\f32}\sum_{k\in\Z^3}\int_{\T^3} f(x-k)e^{-\vep\pi |x-k|^2}dx=\int_{\T^3} f(x)\vep^{\f32}\sum_{k\in\Z^3} e^{-\vep\pi |x-k|^2}dx\\
&=&\int_{\T^3} f(x)\sum_{k\in\Z^3} e^{-\pi|k|^2/\vep}e^{2\pi i x\cdot k}dx
=\int_{\T^3}f(x)dx+A_\vep,
\eeno
where we use the Poisson summation formula
$\sum_{m\in\Z^3} (\mathcal{F}_xf)(m)e^{2\pi im\cdot x}=\sum_{k\in\Z^3} f(x+k)
 $ in the second equality. Thanks to the fact that
$|A_\vep|\leq \|f\|_{L^\infty}\sum_{|k|\geq1} e^{-\pi|k|^2/\vep}$, we complete the proof of the lemma.
\end{proof}

\begin{proof}[Proof of Lemma \ref{hypo}] 
We prove it in the spirit of \cite{FB}. Since $\lr{D_v}^sg+h=\lr{D_v}^s(g+\<D_v\>^{-s}h)$, we may assume that  $h\equiv0$. Let $\varrho_1\in C_c^\infty$ be a smooth function in $v$ variable verifying that 
$\int_{\R^3}\varrho_1=1$ and $\varrho_{\varepsilon}(v)=\f1{\varepsilon^3}\varrho_1(\f v \varepsilon).$
 
Let $\varepsilon=\varepsilon(m):=\varepsilon_0{\<m\>^{-\f1{s+1+\beta}}}$, where $\varepsilon_0>0$ will be chosen in the later, and  $m$ be the the dual variable of $x$.  We have the following decomposition:
\ben\label{decf}
f=\chi_1*_xf+(\delta-\chi_1)*_x(f-Pf)+(\delta-\chi_1)*_xPf.
\een
Here $\chi_1*_x$, $(\delta-\chi_1)*_x$ and the operator $P$ are defined by
\ben\label{chi_k}  &&
\chi_1*_x f:=\sum_{|m|\leq 1} (\mathcal{F}_x{f})(m)e^{2\pi im\cdot x},\quad(\delta-\chi_1)*_x f:=\sum_{|m|> 1} (\mathcal{F}_x{f})(m)e^{2\pi im\cdot x}\\\notag&&
 \mbox{and}\quad
\mathcal{F}_{t,x}(Pf):=\varrho_{\varepsilon(m)}*_v \mathcal{F}_{t,x}(f).
\een
By definition, we also have
$\mathcal{F}_{t,x,v}(Pf)(\omega,m,\xi)=\mathcal{F}_v(\varrho_1)(\varepsilon(m)\xi)
\mathcal{F}_{t,x,v}f(\omega,m,\xi).$

\underline{\it Step 1:} It is not difficult  to verify that for any $\nu_1,\nu_2\in\Z_+^3$ there exists a constant $C_{\nu_1,\nu_2}$ independent of $\varepsilon_0$ such that if $M(m,\xi):=\f{1-\mathcal{F}_v(\varrho_1)(\varepsilon(m)\xi)}{(\varepsilon(m)  |\xi|)^\beta},  |m|>1$, then \ben\label{P2} 
|\pa^{\nu_1}_m\pa^{\nu_2}_\xi M(m,\xi)|\leq C_{\nu_1,\nu_2}|m|^{-|\nu_1|}\langle\xi\rangle ^{-|\nu_2|}.
\een  
 Thanks to Lemma \ref{t3r3},   $M$ is 
a bounded multiplier on $L^p(\T^3\times\R^3_v)$ for any $1<p<\infty$.  Recalling that $\alpha=\frac{\beta}{s+1+\beta}$, we get that for $|m|>1$,
 $\mathcal{F}_{x,v}\big(|D_x|^{\al}(\delta-\chi_k)*_x(f-Pf) \big)(m,\xi)=\vep_0^{\be}M(m,\xi)\f{|m|^\alpha}{\lr{m}^\alpha}\f{|\xi|^\beta}{\lr{\xi}^\beta}\lr{\xi}^\beta$,    which implies that 
$\||D_x|^{\al}(\delta-\chi_1)*_x(f-Pf)\|^2_{L^p_{t,x,v}}\leq  C_{s,\be}\varepsilon_0^{\beta}\|\<D_v\>^\be f\|_{L^p_{t,x,v}}.$
 This enables to estimate \eqref{decf} as follows:
$\||D_x|^{\al}f\|_{L^p_{t,x,v}}\leq C_{s,\be}\|f\|_{L^p_{t,x,v}}+C_{s,\be}\varepsilon_0^{\beta}
\|\<D_v\>^{\be} f\|_{L^p_{t,x,v}}+\||D_x|^{\al}(\delta-\chi_1)*_x(Pf)\|_{L^p_{t,x,v}}.$

\underline{\it Step 2:} From \eqref{Dsg}(recalling that $h\equiv0$), we have 
$i(\omega+v\cdot m)(\mathcal{F}_{t,x}f)(\omega,m,v) =\<D_v\>^s (\mathcal{F}_{t,x}g)(\omega, m,v).$
By introducing the  parameter $\lambda(m):=\lambda_0\langle m\rangle^{1-\f 1{s+1+\beta}}$  
with $\lambda_0>0$, we derive that  
\beno
(\mathcal{F}_{t,x}f)(\omega, m,v)=\f{\lambda (m)}{\lambda(m)+i(\omega+v\cdot m)}(\mathcal{F}_{t,x}f)(\omega, m,v)+\f{1}{\lambda(m)+i(\omega+v\cdot m)} \<D_v\>^s (\mathcal{F}_{t,x}g)(\omega, m,v),
\eeno
which yields that 
\ben\label{hatPRWf}
&&\mathcal{F}_{t,x}(Pf)(\omega,m,v)=\mathcal{F}_{t,x}(Rf)(\omega,m,v)+\mathcal{F}_{t,x}(Wf)(\omega,m,v):=\int \f{\lambda (m)}{\lambda(m)+i(\omega+\eta \cdot m)}\\\notag&&\times\varrho_{\varepsilon(m)}(v-\eta )\mathcal{F}_{t,x}f(\omega,m,\eta )d\eta+\int \f{1}{\lambda(m)+i(\omega+v\cdot m)} \varrho_{\varepsilon(m)}(v-\eta )\<D_\eta\>^s \mathcal{F}_{t,x}g(\omega,m,\eta)d\eta.\een
To handle the second term in right-hand side of \eqref{hatPRWf}, we introduce a generalized operator: if $\ta\in\Z^3_+$,
\ben\label{Waf}
 \mathcal{F}_{t,x}(W_{\ta} g)(\omega,m,v):=\int \f{1}{\lambda(m)+i(\omega+\eta \cdot m)} \varrho_{\varepsilon(m)}(v-\eta )\pa_\eta^{\ta} \mathcal{F}_{t,x}g(\omega,m,\eta)d\eta.
\een
 
$\bullet$ {\it Estimate of $Rf$.} We introduce the translation operator $T$ and the operator $\mathcal{M}$ defined by
\ben\label{tsf1}
(Tf)(t,x,v):=f(t,x+vt,v),\quad
\mathcal{F}_{t,x}(\mathcal{M}f)(\omega,m,v):=\f{\lambda (m)}{\lambda(m)+i \omega}\mathcal{F}_{t,x}f(\omega,m,v).
\een
Then we have $Rf=(PT^{-1}\mathcal{M}T)f$. Moreover, the multiplier $\phi(\om,m):=\f1{1+i\om/\lambda(m)}$, which is associated to the operator $\mathcal{M}$,   
 satisfies that
$|\pa^{\nu_1}_\om\pa_m^{\nu_2}\phi(\om,m)|\leq C_{\nu_1,\nu_2} \<\om\>^{-|\nu_1|}\<m\>^{-|\nu_2|}$
for some $C_{\nu_1,\nu_2}$, which is independent of $\lambda_0$. Thus  $\mathcal{M}$ is   bounded   on $L^p(\R\times\T^3)$ for any $1<p<\infty$.  By a direct estimate similar to \eqref{P2}, $P$ is also bounded on $L^p_{t,x,v}$.
 Since $T$ is an isometry of $L^p$, we conclude that  $R$ is also bounded, i.e., there exists a constant $C_{s,\be}$ depending neither on $\vep_0$ nor $\lambda_0$ such that
$\|Rf\|_{L^p_{t,x,v}}\leq C_{s,\be}\|f\|_{L^p_{t,x,v}}.$

To bound $|D_x|^\al(Rf-\chi_1*_x(Rf))$, we first notice that
\beno
|\varrho_{\varepsilon(m)}*_v\mathcal{F}_{t,x}(f)(\omega,m,v)|\leq \|\mathcal{F}_{t,x}(f)(\omega,m,\cdot)|\varrho_{\varepsilon(m)}(v-\cdot)|^{1/2}\|_{L^2(\R^3)}\times \left(\int_{\R^3}\f{|\varrho_{\varepsilon(m)}(v-\eta)|}{|1+i(\omega+\eta\cdot m)/\lambda(m)|^2}d\eta \right)^{1/2}.
\eeno
In order to estimate the last integral, by the decomposition that $\eta=\tilde{\eta}\f m{|m|}+\eta'$ with $\tilde{\eta}=\eta\cdot \f m{|m|}$ and $\eta'\cdot m=0$, as well as the fact  that $|\varrho_{\varepsilon(m)}(v)|\leq C_\beta \varepsilon(m)^{-3}1_{|v|\lesssim\varepsilon(m)}$, we obtain that
\beno
\int_{\R^3}\f{|\varrho_{\varepsilon(m)}(v-\eta)|}{|1+i(\omega +\eta\cdot m)|^2}d\eta \leq C_{s,\be} \int_{\R}\f{\f1\varepsilon 1_{|\f{v\cdot m}{|m|}-\tilde{\eta}|<\varepsilon}}{|1+i(\omega+|m|\tilde{\eta})/\lambda|^2}d\tilde{\eta}\leq C_{s,\be} \f {\lambda(m)}{\varepsilon(m) |m|}=C_{s,\be}\left(\f{\lambda_0}{\varepsilon_0}\right),
\eeno
which implies that 
\ben\label{RF2}
\|Rf-\chi_1*_x(Rf)\|_{L^2_{t,x,v}}\leq C_{s,\be} \left(\f{\lambda_0}{\varepsilon_0}\right)^{1/2}\|f\|_{L^2_{t,x,v}}.
\een

For given $p$, we choose either $1<p_1<p$ if $p<2$, or $p<p_1<\infty$ if $p>2$. Since $R$ is bounded on $L^p_{t,x,v}$, we have $\|Rf-\chi_1*_x(Rf)\|_{L^{p_1}}\leq C\|f\|_{L^{p_1}}$. From this together with   \eqref{RF2},  we conclude that
\ben\label{defiiota}
\|Rf-\chi_1*_x(Rf)\|_{L^{p}}\leq C\left(\f{\lambda_0}{\vep_0}\right)^\iota\|f\|_{L^{p}},\quad \iota:=\f12\f{1/p-1/p_1}{1/2-1/p_1}>0.
\een
Thus we get that
$\||D_x|^\al(Rf-\chi_1*_x(Rf))\|_{L^{p}}\leq C\left(\f{\lambda_0}{\vep_0}\right)^\iota\||D_x|^\al f\|_{L^{p}}.$

$\bullet$ {\it Estimate of $W_{\ta} g$}. We only consider the cases that $|\ta|=2~\mbox{and}~|\ta|=0$.  
 For $|\ta|=2$, after integration by parts, we are led to 
 $\mathcal{F}_{t,x}(W_{\ta} g)(\omega,m,v)=\int \f{m^{\ta_1}}{\lambda(m)+i(\omega+v\cdot m)} \partial^{\ta_2}\varrho_{\varepsilon(m)}(v-\eta )\mathcal{F}_{t,x}g(\omega,m,\eta)d\eta$
 with $\ta_1+\ta_2=\ta$. We can write  $\mathcal{F}_{t,x}(W_{\ta} g)=\f{m^{\ta_1}}{\lambda(m)^{|\ta_1|+1}\varepsilon (m)^{|\ta_2|}}\mathcal{F}_{t,x}(\underline{W}g)$
 with 
 \beno
 \mathcal{F}_{t,x}(\underline{W}g)(\omega,m,v)=\int \f1{[1+i\f{\omega+\eta\cdot m}{\lambda(m)}]^{|\ta_1|+1}}\f1{\varepsilon (m)^{3}}
 (\partial^{\ta_2}\varrho_1)\left(\f{v-\eta}{\varepsilon(m)}\right) \mathcal{F}_{t,x}g(\omega,m,\eta)d\eta.
 \eeno

Observe that $\underline{W}g$ enjoys the similar structure as $Rg$, we   obtain that $(\delta-\chi_1)*_x\underline{W}$ is bounded on $L^p_{t,x,v}$ with the constant $C_{s,\be}(\lambda_0/\varepsilon_0)^{\iota}$.  Since the multiplier
$\f{m^{\ta_1}}{\lambda(m)^{|\ta_1|+1}\varepsilon (m)^{|\ta_2|}}\<m\>^{1-\f3{s+1+\be}}=\f1{\lambda_0^{1+|\ta_1|}\vep_0^{|\ta_2|}}$
is bounded on $L^p(\R^3)$, we deduce that if $|\ta|=2$,
\ben\label{al2}
\|\<D_x\>^{1-\f3{s+1+\be}}(\de-\chi_1)*_x(W_{\ta} g)\|_{L^p_{t,x,v}}\leq \f{C_{s,\be}(\lambda_0/\varepsilon_0)^{\iota}}{\lambda_0^{1+|\ta_1|}\vep_0^{|\ta_2|}}\|g\|_{L^p_{t,x,v}}.
\een
Similar argument can be used to get that for $|\ta|=0$, it holds that
\ben\label{al0}
\|\<D_x\>^{1-\f1{s+1+\be}}(\de-\chi_1)*_x(W_{\ta} g)\|_{L^p_{t,x,v}}\leq \f{C_{s,\be}(\lambda_0/\varepsilon_0)^{\iota}}{\lambda_0}\|\<D_v\>^sg\|_{L^p_{t,x,v}}.
\een

Now we are in a position to control $Wg$ defined in \eqref{hatPRWf}. On one hand, by \eqref{al0}, we have  
\ben\label{g11}
\|\<D_x\>^{1-\f1{s+1+\be}}(\de-\chi_1)*_x(Wg)\|_{L^p_{t,x,v}}\leq \f{C_{s,\be}(\lambda_0/\varepsilon_0)^{\iota}}{\lambda_0}\|\<D_v\>^sg\|_{L^p_{t,x,v}}.
\een
 On the other hand, due to the fact  that $s\in[0,2]$ and  $\<D_v\>^sg=\<D_v\>^{2}\<D_v\>^{s-2}g=\sum_{|\al|=0,2}\pa_v^\al \<D_v\>^{s-2}g$,   \eqref{al2} and  \eqref{al0} yield that
 \ben\label{g22}
 \|\<D_x\>^{1-\f3{s+1+\be}}(\de-\chi_1)*_x(W g)\|_{L^p_{t,x,v}}\leq C_{s,\be,\lambda_0,\vep_0}\| \<D_v\>^{s-2}g\|_{L^p_{t,x,v}}.
 \een
Notice that if $\th=\f s2$ then $(s-2)\th+s(1-\th)=0,(1-\f3{s+1+\be})\th+(1-\f 1{s+1+\be})(1-\th)=\f\be{1+s+\be}=\al$. The interpolation between \eqref{g11} and \eqref{g22} implies that 
$\|\lr{D_x}^{\al}(\delta-\chi_1)*_x(Wg)\|_{L^p_{t,x,v}}\leq  C_{s,\be,\lambda_0,\vep_0}\|g\|_{L^p_{t,x,v}}.$

We conclude that
\beno
\||D_x|^{\al}f\|_{L^p_{t,x,v}}&\leq&  C_{s,\be}\big(\|f\|_{L^p}+\varepsilon_0^{\beta}
\|\lr{D_v}^\beta f\|_{L^p}+ \left(\lambda_0/\varepsilon_0\right)^{\iota}\||D_x|^{\al}f\|_{L^p_{t,x,v}}+C_{s,\be,\lambda_0,\vep_0}\|g\|_{L^p_{t,x,v}}\big).
\eeno
Choose $\lambda_0/\varepsilon_0$ sufficiently small and then we have
$\|\<D_x\>^\al f\|_{L^p_{t,x,v}}\leq C_{s,\be}\big(\|\<D_v\>^\be f\|_{L^p_{t,x,v}}+\|g\|_{L^p_{t,x,v}}\big).$
This ends the proof of the lemma.
\end{proof}

As a result, we obtain the following corollary.

%, where the second result keep the time interval unchanged with the sacrifice of regularity.

\begin{cor}\label{corhy} Let $\be\in[0,1]$, $1<p<\infty$, $s\in[0,2]$, $\al:=\f \be{1+s+\be}$ and $\<D_v\>^\be f\in L^p([T_1,T_2]\times \T^3_x\times\R^3_v)$. Suppose that $f$ is a solution to 
\beno
  \pa_t f+v\cdot\nabla_x f=\<D_v\>^s g+h.
  \eeno
	Then
	for any $[\tau_1,\tau_2]\subset (T_1,T_2)$, it holds that
	\beno
	\|\lr{D_x}^\alpha f\|_{L^p([\tau_1,\tau_2]\times \T^3_x\times\R^3_v)}\leq C_{s,\be,\tau_1,\tau_2}\big( \|\<D_v\>^\beta f\|_{L^p([T_1,T_2]\times \T^3_x\times\R^3_v)}+\|g\|_{L^p([T_1,T_2]\times \T^3_x\times\R^3_v)}+\|h\|_{L^p([T_1,T_2]\times \T^3_x\times\R^3_v)}\big).
	\eeno
 
\end{cor}
\begin{proof}
	Let the function $\vphi(t)\in C_c^\infty(T_1,T_2)$ verify that $0\le\vphi\le1$ and $\vphi(t)=1$ for $t\in[\tau_1,\tau_2] $. Then 
	\beno
	\pa_t (\vphi f)+v\cdot\nabla_x (\vphi f)=\<D_v\>^s (\vphi g)+\vphi'(t)f+\vphi(t)h.
	\eeno
	  Lemma \ref{hypo} implies that
	\beno
	\|\vphi\<D_x\>^\alpha f\|_{L^p([T_1,T_2]\times\T^3\times\R^3)}&\leq& C_{s}\big(\|\vphi\<D_v\>^\beta f\|_{L^p([T_1,T_2]\times\T^3\times\R^3)}+\|\vphi g\|_{L^p([T_1,T_2]\times\T^3\times\R^3)}\\
	&&+\|\vphi' f+\vphi h\|_{L^p([T_1,T_2]\times\T^3\times\R^3)}\big),
	\eeno
which yields the desired result.
\end{proof}

\subsection{Dyadic decomposition}\label{DDP} This subsection is  dedicated to the basic knowledge on the dyadic decomposition which plays the essential role in our analysis. Let $B_{\frac{4}{3}}:=\{\xi\in\R^3||\xi|\leq\frac{4}{3}\}$ and $C:=\{\xi\in\R^3||\frac{3}{4}\leq|\xi|\leq\frac{8}{3}\}$. Then one may introduce two radial functions $\psi\in C_0^\infty(B_{\frac{4}{3}})$ and $\vphi\in C_0^\infty(C)$ which satisfy
\ben
\label{7.1}\psi,\vphi\geq0,~~&and&~~\psi(\xi)+\sum_{j\geq0}\vphi(2^{-j}\xi)=1,~\xi\in\R^3,\\
\notag|j-k|\geq2&\Rightarrow& \rm{Supp}~\vphi(2^{-j}\cdot)\cap \rm{Supp}~\vphi(2^{-k}\cdot)=\emptyset,\\
\notag j\geq1&\Rightarrow& \rm{Supp}~\psi\cap \rm{Supp}~\vphi(2^{-j}\cdot)=\emptyset.
\een

$\underline{(1).}$ We first introduce the dyadic decomposition in the phase space. The dyadic operator in the phase space $\cP_j$ can be defined as
\ben\label{Defcpj}
\cP_{-1}f(x):=\psi(x)f(x),~\cP_jf(x):=\vphi(2^{-j}x)f(x),~j\geq0.
\een
Let $\tP_lf(x)=\sum\limits_{|k-l|\leq N_0}\cP_kf(x)$ and $\U_jf(x)=\sum\limits_{k\leq j}\cP_kf(x)$ where $N_0\geq2$ will be chosen in the   later which verifies $\cP_j\cP_k=0$ if $|j-k|\geq N_0$. For any smooth function $f$, we have $f=\cP_{-1}f+\sum\limits_{j\geq0}\cP_jf.$

$\underline{(2).}$ Next we introduce the dyadic decomposition in the frequency space. We denote $\tm:=\cF^{-1}\psi$ and $\tphi:=\mathcal{F}^{-1}\vphi$ where they are the inverse Fourier transform of $\vphi$ and $\psi$. If we set $\tphi_j(x)=2^{3j}\tphi(2^jx)$, then the dyadic operator in the frequency space $\F_j$ can be defined as follows
\ben\label{DefFj}
	\F_{-1}f(x):=\int_{\R^3}\tm(x-y)f(y)dy,~\F_jf(x):=\int_{\R^3}\tphi_j(x-y)f(y)dy,~j\geq0.
\een
Let $\tF_jf(x)=\sum_{|k-j|\leq 3N_0}\F_kf(x)$ and $S_jf(x)=\sum_{k\leq j}\F_kf$. Then for any $f\in \mathscr{S}'$, it holds $f=\F_{-1}f+\sum_{j\geq0}\F_jf.$

\bigskip

We recall the definitions of symbol $S^m_{1,0}$ and pseudo-differential operator:
\begin{defi}\label{de2.1}
	A smooth function $a(v,\xi)$ is said to be a symbol of type  $S^m_{1,0}$ if $a(v,\xi)$ verifies for any multi-indices $\alpha$ and $\beta$,
	 $|(\partial_\xi^\alpha\pa_v^\beta a)(v,\xi)|\leq C_{\alpha,\beta}\<\xi\>^{m-|\al|},$
 where $C_{\alpha,\beta}$ is a constant depending only on $\alpha$ and $\beta$. $a(x,D)$ is called a pseudo-differential operator with the symbol $a(x,\xi)$ if it is defined by
	\beno
	(a(x,D)f)(x):=\frac{1}{(2\pi)^3}\int_{\R^3}\int_{\R^3}e^{i(x-y)\xi}a(x,\xi)f(y)dyd\xi.
	\eeno
\end{defi}

\begin{defi}\label{Fj} Let  $\alpha=(\alpha_1,\alpha_2,\alpha_3,),|\alpha|:=\alpha_1+\al_2+\al_3$ and $\vphi_\alpha:=(\frac{1}{i}\pa_{x_1})^{\alpha_1}(\frac{1}{i}\pa_{x_2})^{\alpha_2}(\frac{1}{i}\pa_{x_3})^{\alpha_3}\vphi$. To simplify the presentation of the estimates for the commutator $[\cP_k, \F_j]$(thanks to   Lemma \ref{le1.2}),  we introduce  $\cP_{j,\alpha}$, $\F_{j,\alpha}$ and $\hat{\F}_{j,\al}$ defined by
	\beno
	&&\cP_{-1,\alpha}f:=\psi_{\alpha}f,\quad\quad\quad\quad \cP_{j,\alpha}f:=\vphi_\alpha(2^{-j}\cdot)f,\quad j\geq0;\\
	&&\F_{-1,\alpha}f:=\psi_{\alpha}(D)f,\quad\quad\quad \F_{j,\alpha}:=\vphi_\alpha(2^{-j}D)f,\quad j\geq0;\\
	&&\hat{\F}_{-1,\alpha}f:=(\psi^2)_{\alpha}(D)f,\quad\quad \hat{\F}_{j,\alpha}:=(\vphi^2)_\alpha(2^{-j}D)f,\quad j\geq0.
	\eeno
	Similar to $\tP_j$   and $\tF_j$, we can also  introduce
	$$\tP_{l,\alpha}:=\sum_{|k-l|<N_0}\cP_{k,\alpha},\quad \U_{j,\alpha}:=\sum_{k\leq j}\cP_{k,\alpha}, \quad \tF_{j,\alpha}:=\sum_{|k-j|<3N_0}\F_{k,\alpha}.$$
	
	To unify  notations $\tF_j, \tF_{j,\alpha}$, $\hat{\F}_{j,\alpha}$ and $\tP_l, \tP_{l,\alpha}$, we introduce  localized operators $\mF_j$   and $\mP_j$:
	\smallskip
	
	\noindent {\rm (i).} The support of the Fourier transform of $\mF_jf$ and the support of $\mP_jf$  will be localized in the annulus $\{|\cdot|\sim 2^j\},j\geq0$ or in the sphere $\{|\cdot|\ls 1\},j=-1$;\smallskip
	
	\noindent{\rm (ii).} It holds that for   fixed $N\in\N$, $\|\tF_j f\|_{L^2}+\sum\limits_{|\alpha|\le N} (\|\tF_{j,\alpha}f\|_{L^2}+\|\hat{\F}_{j,\alpha}f\|_{L^2})\le C_N\|\mF_jf\|_{L^2}$ and $\|\tP_j f\|_{L^2}+\sum\limits_{|\alpha|\le N} \|\tP_{j,\alpha}f\|_{L^2}\le C_N\|\mP_jf\|_{L^2}$.
\end{defi}

\subsubsection{Application(I)} The first application is on the collision operator to get some technical lemmas.
\smallskip

\underline{(1).} \textit{Dyadic decomposition of the operator in the  phase space.} We first use the dyadic decompositions to reduce the commutator to the annulus in the phase space.
We
set \ben\label{DefPhi} \Phi_k^\gamma(v):=
\left\{\begin{aligned} & |v|^\gamma \varphi(2^{-k}|v|), \quad\mbox{if}\quad k\ge0;\\
	& |v|^\gamma \psi( |v|),\quad\mbox{if}\quad k=-1.\end{aligned}\right.\een
Then we derive that 
$(Q(g, h), f )_{L^2_v}=\sum_{k=-1}^\infty ( Q_k(g, h), f )_{L^2_v}=\sum_{k=-1}^\infty\sum_{j=-1}^\infty ( Q_k(\mathcal{P}_jg, h), f )_{L^2_v}, $
where
\ben\label{deQl}
 Q_{k}(g, h):=\iint_{\sigma\in \SS^2,v_*\in \R^3} \Phi_k^\gamma(|v-v_*|)b(\cos\theta) (g'_*h'-g_*h)d\sigma dv_*.\een

It is not difficult to check that there exists a integer $N_0\in \N$ such that(see also  $(2.1)$ in \cite{HE})
\ben\label{ubdecom} ( Q(g,h), f)_{L^2_v} &=&\sum_{k\ge N_0-1}( Q_k(\U_{k-N_0} g, \tilde{\mathcal{P}}_kh), \tilde{\mathcal{P}}_kf )_{L^2_v} +
\sum_{j\ge k+N_0}( Q_k(\mathcal{P}_{j} g, \tilde{\mathcal{P}}_jh), \tilde{\mathcal{P}}_jf )_{L^2_v}\notag\\&&\quad+\sum_{|j-k|\le N_0}( Q_k( \mathcal{P}_{j} g, \U_{k+N_0}h), \U_{k+N_0}f )_{L^2_v}.  \een

\underline{(2).} \textit{Dyadic decomposition of the operator in the  frequency space.} By Bobylev's equality we have
\ben\label{bobylev}&& \qquad(\mathcal{F}\big( Q_k(g, h)\big), \mathcal{F}f )_{L^2_\xi}\\&&=\iint_{\sigma\in \SS^2, \eta,\xi\in \R^3} b(\f{\xi}{|\xi|}\cdot \sigma)\big[ \mathcal{F}(\Phi_k^\gamma ) (\eta-\xi^{-})-\mathcal{F}(\Phi_k^\gamma)(\eta)\big](\mathcal{F}g)(\eta)(\mathcal{F}h)(\xi-\eta)\overline{(\mathcal{F}f)}(\xi)d\sigma d\eta d\xi,\nonumber \een
where $\mathcal{F}f$ denotes the Fourier transform of $f$ and $\xi^{\pm}:= \frac{\xi\pm|\xi|\si}{2}$. Then we have 
 \[(Q_k(g,h),f)_{L^2_v} 
 = \sum_{l\leq p-N_0}\fM^1_{k,p,l}(g,h,f)+\sum_{l\geq-1}\fM^2_{k,l}(g,h,f)+\sum_{p\geq-1}\fM^3_{k,p}(g,h,f)+\sum_{m<p-N_0}\fM^4_{k,p,m}(g,h,f),\]
where
 $\fM^1_{k,p,l}(g,h,f):=\iint_{\si\in\S^2,v_*,v\in\R^3}(\tF_p\Phi_k^\ga)(|v-v_*|)b(\cos\th)(\F_pg)_*(\F_lh)[(\tF_pf)'-\tF_pf]d\si dv_*dv,$ 
	$\fM^2_{k,l}(g,h,f):=\iint_{\si\in\S^2,v_*,v\in\R^3}\Phi_k^\ga(|v-v_*|)b(\cos\th)(S_{l-N_0} g)_*(\F_lh)[(\tF_lf)'-\tF_lh]d\si dv_*dv,$ 
	$\fM^3_{k,p}(g,h,f):=\iint_{\si\in\S^2,v_*,v\in\R^3}\Phi_k^\ga(|v-v_*|)b(\cos\th)(\F_pg)_*(\tF_ph)[(\tF_pf)'-\tF_pf]d\si dv_*dv,$  and
	$\fM^4_{k,p,m}(g,h,f):=\iint_{\si\in\S^2,v_*,v\in\R^3}(\tF_p\Phi_k^\ga)(|v-v_*|)b(\cos\th)\\ \times(\F_pg)_*(\tF_ph)[(\F_mf)'-\F_mf]d\si dv_*dv.$
 \smallskip

Moreover, we have the following estimates:
	\begin{lem}\label{lemma1.7}
	(i) If $l\leq p-N_0$, then for $k\ge 0$ and $N \ge 0$,
	\beno
	|\fM^1_{k,p,l}|&\ls& 2^{k(\ga+\frac{5}{2}-N)}(2^{-p(N-2s)}2^{2s(l-p)}+2^{-(N-\frac{5}{2})p}2^{\frac{3}{2}(l-p)})\|\Phi_0^\ga\|_{H^{N+2}}\|\vphi\|_{W_N^{2,\infty}}\|\F_pg\|_{L^1}\|\F_lh\|_{L^2}\|\tF_pf\|_{L^2}.
	\eeno
	Moreover,
$  |\fM^1_{-1,p,l}|\ls 2^{2sl}\|\F_pg\|_{L^2}\|\F_lh\|_{L^2}\|\tF_pf\|_{L^2}.$
 
	\noindent(ii) If $k\geq0$,
	$|\fM_{k,l}^2|\ls2^{(\ga+2s)k}2^{2sl}\|S_{l-N_0}g\|_{L^1}\|\F_lh\|_{L^2}\|\tF_lf\|_{L^2}$,  
	$|\fM^3_{k,p}|\ls2^{(\ga+2s)k}2^{2s p}\|\F_pg\|_{L^1}\|\tF_ph\|_{L^2}\|\tF_pf\|_{L^2}$.
 Moreover, 
	$|\fM_{-1,l}^2|\ls2^{2sl}\|S_{l-N_0}g\|_{L^2}\|\F_lh\|_{L^2}\|\tF_lf\|_{L^2},\,
    |\fM^3_{-1,p}|\ls2^{2s p}\|\F_pg\|_{L^2}\|\tF_ph\|_{L^2}\|\tF_pf\|_{L^2}.$

	\noindent(iii) If $m<p-N_0$, then for $k\geq0,$
 $|\fM^4_{k,p,m}|\ls2^{2s(m-p)}2^{(\ga+\frac{3}{2}-N)k}2^{-p(N-\frac{5}{2})}\|\Phi_0^\ga\|_{H^{N+2}}\|\vphi\|_{W_N^{2,\infty}}\|\F_pg\|_{L^1}\|\tF_ph\|_{L^2}\\\times\|\F_mf\|_{L^2}.$
 Moreover, $|\fM^4_{-1,p,m}|\ls 2^{2ms}\|\F_pg\|_{L^2}\|\tF_ph\|_{L^2}\|\F_mf\|_{L^2}.$
	 
	\noindent(iv) Let $a, b\in[0,2s]$ with $a+b=2s$, then
	$|(Q_{-1}(g, h),f)|\ls (\|g\|_{L^1}+\|g\|_{L^2})\|h\|_{H^a}\|f\|_{H^b}$. Moreover, for any $k\geq0$, $|(Q_k(g,h),f)|\ls 2^{(\ga+2s)k}\|g\|_{L^1}\|h\|_{H^a}\|f\|_{H^b}$.
\end{lem}
\begin{proof}
	The above estimates comes from Lemma 2.1, Lemma 2.2 and Lemma 2.3 in \cite{HE}.
\end{proof}

With the aid of Lemma \ref{lemma1.7} and decomposition \eqref{ubdecom}, we can    derive the following lemma which will be used in the proof of Theorem \ref{CtvLBE}(ii):
\begin{lem}\label{forCtv} For $\ga+2s>-1$, then for any $n>0$, we have  
	\beno
	|(Q(\mu,h),f)_{L^2}|\leq C_n\|h\|_{H^{-n+2s}_{3+\ga+2s}}\|f\|_{H^{n}_{-3}};\quad
	|(Q(h,\mu),f)_{L^2}|\leq C_n\|h\|_{H^{-n+2s}_3}\|f\|_{H^{n}_{-3}}.
	\eeno
\end{lem}
\begin{proof}
We begin with the first inequality. By   Lemma \ref{lemma1.7}, we have  
	\beno
	&&|(Q_k(\mu,h),f)_{L^2}|\leq \sum_{l\leq p-N_0}2^{(\ga+2s)k}2^{2sl}\|\F_p\mu\|_{L^2_2}\|\F_lh\|_{L^2}\|\tF_pf\|_{L^2}+\sum_{l\geq-1}2^{(\ga+2s)k}2^{2sl}\|S_{l-N_0}\mu\|_{L^2_2}\|\F_lh\|_{L^2}\\
	&&\times\|\tF_lf\|_{L^2}
	+\sum_{p\geq-1}2^{(\ga+2s)k}2^{2s p}\|\F_p\mu\|_{L^2_2}\|\tF_ph\|_{L^2}\|\tF_pf\|_{L^2}+\sum_{m<p-N_0}2^{(\ga+2s)k}2^{2ms}\|\F_p\mu\|_{L^2_2}\|\tF_ph\|_{L^2}\|\F_mf\|_{L^2},
	\eeno
From this together with \eqref{ubdecom}, we can obtain that
\beno
&&|Q(\mu,h),f)|\le \sum_{k\geq N_0-1}\Big(\sum_{l\leq p-N_0}2^{(\ga+2s)k}2^{2sl}\|\F_p\U_{k-N_0}\mu\|_{L^2_2}\|\F_l\tP_kh\|_{L^2}\|\tF_p\tP_kf\|_{L^2}+\sum_{l\geq-1}2^{(\ga+2s)k}2^{2sl}\\
&&\times\|S_{l-N_0}\U_{k-N_0}\mu\|_{L^2_2}\|\F_l\tP_kh\|_{L^2}\|\tF_l\tP_kf\|_{L^2}
+\sum_{p\geq-1}2^{(\ga+2s)k}2^{2s p}\|\F_p\U_{k-N_0}\mu\|_{L^2_2}\|\tF_p\tP_kh\|_{L^2}\|\tF_p\tP_kf\|_{L^2}\\
&&+\sum_{m<p-N_0}2^{(\ga+2s)k}2^{2ms}\|\F_p\U_{k-N_0}\mu\|_{L^2_2}\|\tF_p\tP_kh\|_{L^2}\|\F_m\tP_kf\|_{L^2}\Big)+\sum_{j\geq k+N_0}\Big(\sum_{l\leq p-N_0}2^{(\ga+2s)k}2^{2sl}\|\F_p\cP_j\mu\|_{L^2_2}\\
&&\times\|\F_l\tP_jh\|_{L^2}\|\tF_p\tP_jf\|_{L^2}+\sum_{l\geq-1}2^{(\ga+2s)k}2^{2sl}\|S_{l-N_0}\cP_j\mu\|_{L^2_2}\|\F_l\tP_jh\|_{L^2}\|\tF_l\tP_jf\|_{L^2}
+\sum_{p\geq-1}2^{(\ga+2s)k}2^{2s p}\\
&&\times\|\F_p\cP_j\mu\|_{L^2_2}\|\tF_p\tP_jh\|_{L^2}\|\tF_p\tP_jf\|_{L^2}+\sum_{m<p-N_0}2^{(\ga+2s)k}2^{2ms}\|\F_p\cP_j\mu\|_{L^2_2}\|\tF_p\tP_jh\|_{L^2}\|\F_m\tP_jf\|_{L^2}\Big)\\
&&  +\sum_{|j-k|\leq N_0}\Big(\sum_{l\leq p-N_0}2^{(\ga+2s)k}2^{2sl}\|\F_p\cP_j\mu\|_{L^2_2}\|\F_l\U_{k+N_0}h\|_{L^2}\|\tF_p\U_{k+N_0}f\|_{L^2}+\sum_{l\geq-1}2^{(\ga+2s)k}2^{2sl}\|S_{l-N_0}\cP_j\mu\|_{L^2_2}\\
&&\times \|\F_l\U_{k+N_0}h\|_{L^2}\|\tF_l\U_{k+N_0}f\|_{L^2}
+\sum_{p\geq-1}2^{(\ga+2s)k}2^{2s p}\|\F_p\cP_j\mu\|_{L^2_2}\|\tF_p\U_{k+N_0}h\|_{L^2}\|\tF_p\U_{k+N_0}f\|_{L^2}\\
&& +\sum_{m<p-N_0}2^{(\ga+2s)k}2^{2ms}\|\F_p\cP_j\mu\|_{L^2_2}\|\tF_p\U_{k+N_0}h\|_{L^2}\|\F_m\U_{k+N_0}f\|_{L^2}\Big).
\eeno
Thanks to Lemma \ref{lemma1.4}, by Cauchy inequality,  we can obtain that
$|Q(\mu,h),f| \ls C_n\|h\|_{H^{-n+2s}_{3+\ga+2s}}\|f\|_{H^n_{-3}}.$
Following the same argument in the above, we can also get that $|(Q(h,\mu),f)_{L^2}|\leq C_n\|h\|_{H^{-n+2s}_3}\|f\|_{H^{n}_{-3}}$. This ends the proof of the lemma.
\end{proof}

 \subsubsection{Application(II)} Another important application is to show   the following lemma:
\begin{lem}\label{contra}{}
	Let $\a\in\R^+$ and $\ell>0$, then for any $\de>0$, there exists a function $f\in  L^2_\ell$ such that
	\ben\label{con1}
	\sum_{j=1}^\infty\sum\limits_{l>\a j}2^{2(\ell+\de)\a j}\|\F_j\cP_{l}f\|^2_{L^2}=+\infty.
	\een
\end{lem}
\begin{proof}
	We prove it by contradiction. Let $\delta\ll1$ and assume that $\eqref{con1}$ does not hold for any $f\in L^2_\ell$. Using Lemma \ref{le1.2} and Lemma \ref{lemma1.4}, we can deduce that for $j\ge1$, 
  \ben\label{contra1}
  && \notag \big|2^{2(\ell+\de) \a j}\sum_{l>\a j} \|\cP_l\F_jf\|^2_{L^2}- 2^{2(\ell+\de) \a j}\sum_{l>\a j} \|\F_j\cP_lf\|^2_{L^2}\big|\le \sum_{l>\a j}2^{2(\ell+\de) \a j}\big(2^{-2l}2^{-2j}\|\mF_j\mP_lf\|^2_{L^2}\\
  &&\qquad\qquad\qquad\qquad+C_N2^{-2Nl}2^{-2Nj}\|r_{N}f\|^2_{L^2}\big)\leq C\|f\|^2_{L^2_\ell}.
  \een  Thus for all $f\in L^2_\ell$, we have
	\ben\label{BA}2^{2(\ell +\de)\a j}\sum\limits_{l>\a j}\|\F_j\cP_{l}f\|^2_{L^2}+2^{2(\ell +\de)\a j}\sum\limits_{l>\a j}\|\cP_{l}\F_jf\|^2_{L^2}<+\infty,\quad \forall j\geq1.\een

  Set $A_{j}:=2^{(\ell+\de ) \a j}\sum\limits_{l>\a j}\cP_{l}\F_j$, then $\{A_{j}\}_{j\geq0}$ is a family of bounded linear operators from $L^2_\ell$ to $L^2$ since 
$\|A_jf\|^2_{L^2}\sim 2^{2(\ell+\de)\a j}\sum_{l>\a j} \|\cP_{l}\F_jf\|^2_{L^2}\le C_j\|f\|_{L^2_\ell}^2.$ From this together with \eqref{BA},  uniform boundedness principle implies that there exists a universal constant $C$ such that $\sup\limits_{j\geq1}\|A_{j}f\|_{L^2}\leq C\|f\|_{L^2_\ell}$. In other words, by \eqref{contra1}, for all $f\in L^2_\ell$,
 \ben\label{BA1} 2^{2(\ell+\de) \a j}\sum\limits_{l>\a j} \|\cP_{l}\F_jf\|^2_{L^2}+2^{2(\ell+\de) \a j}\sum\limits_{l>\a j} \|\F_j\cP_{l}f\|^2_{L^2}\le C\|f\|^2_{L^2_\ell},\quad \forall j\geq1.\een

	 	  Let $\tphi=\mathcal{F}^{-1}\vphi$(see the definition in \eqref{7.1}) and $\vep:=\de \a/(\a+1)$. We first claim that there exists a  function $\mathrm{f}$ such that  $\mathrm{f}\in L^2 $ such that $g=\tphi*\mathrm{f}\in  L^2\backslash L^2_{{\vep}}$. We prove it by contradiction.  Assume that $\|\tphi* \mathrm{f}\|_{L^2_{\vep}}<\infty$ for any $\mathrm{f}\in L^2$. Then uniform boundedness principle can be applied to operators $\{\mathbf{1}_{|\cdot|<i}\tphi*\}_{i\geq1}$ to get that $\|\tphi* \mathrm{f}\|_{L^2_{\vep}}\ls \|\mathrm{f}\|_{L^2}$ for any $\mathrm{f}\in L^2$. Let $\mathrm{f}\in L^2$ verify $\tphi*\mathrm{f}\not\equiv 0$ and $\{T_j\mathrm{f}\}_{j\geq1}:=\{\mathrm{f}(\cdot-v_j)\}_{j\geq1}$ with $|v_j|=j$. We deduce that 
  \beno
  \| \tphi*T_j\mathrm{f}\|_{L^2_\vep}\ls \|T_j\mathrm{f}\|_{L^2}=\|\mathrm{f}\|_{L^2}.
  \eeno
  Observing  that  
  $|v_j|^{\vep}\|\tphi*\mathrm{f}\|_{L^2}=|v_j|^{\vep}\|\tphi*T_j\mathrm{f}\|_{L^2}\leq C(\||\cdot-v_j|^{\vep}(\tphi*T_j\mathrm{f})\|_{L^2}+\||\cdot|^{\vep}(\tphi*T_j\mathrm{f}) \|_{L^2})$ and $\||\cdot-v_j|^{\vep}(\tphi*T_j\mathrm{f})\|_{L^2}=\||\cdot|^{\vep}(\tphi*\mathrm{f})\|_{L^2}$, we get that
 $|v_j|^{\vep}\|\tphi*\mathrm{f}\|_{L^2}\ls \|\mathrm{f}\|_{L^2},$
  which   contradicts with the facts that $|v_j|=j$ and $\tphi*\mathrm{f}\not\equiv 0$. This ends the proof of the claim.

   Let $\mathrm{f}$ be the function in the claim. Then it is not difficult to check that $f:=\mathrm{f}\lr{v}^\ell$ satisfies
  \ben\label{assu}
  f\in  L^2_\ell  \quad\mathrm{and}\quad g:=\tphi*f\in  L^2_\ell\backslash L^2_{{\ell+\vep}},\quad\mathrm{with}\quad\vep=\de \a/(\a+1).
  \een
	 Next, suppose that $f$ satisfies \eqref{assu}. Let $f_j(v):=2^{\ell j}2^{\f32 j}f(2^jv)$ and $\tilde{f}_j:=f_j(1-\psi)$ with $\psi$ defined in \eqref{7.1}. Then 
 \beno
 \|\tilde{f}_j\|^2_{L^2_\ell}=2^{2\ell j}2^{3j}\int_{\R^3} |f(2^jv)|^2\<v\>^{2\ell}(1-\psi(v))^2dv\sim 2^{2\ell j}2^{3j}\int_{\R^3} |f(2^jv)|^2|v|^{2\ell}(1-\psi(v))^2dv\ls \|f\|^2_{L^2_\ell}
 \eeno
 for any $j\geq1$. Notice that for $l\geq \a j$ with $j\geq1$, it holds that  $\cP_l \tilde{f}_j=\cP_lf_j$. From \eqref{BA1}, we have 
 \beno
 2^{2(\ell+\de)\a j}\sum_{l>\a j}\|\F_j\cP_l f_j\|_{L^2}^2=2^{2(\ell+\de)\a j}\sum_{l>\a j}\|\F_j\cP_l \tilde{f}_j\|_{L^2}^2\leq C\|\tilde{f}_j\|^2_{L^2_\ell}\leq C\|f\|_{L^2_\ell}^2.
 \eeno
From this together with \eqref{contra1}, we have $2^{2(\ell+\de)\a j}\sum_{l>\a j}\|\cP_l\F_j f_j\|_{L^2}^2\leq C\|f\|_{L^2_\ell}^2$. 

By \eqref{assu}, we notice that $\F_jf_j(v)=2^{\ell j}2^{3j+\f32j}\int_{\R^3}\tphi(2^j(v-u))f(2^ju)du=2^{\ell j}2^{\f32j}g(2^jv)$. Then we have
	\ben\label{contra2}
	\sum_{l>\a j}\|\cP_l\F_jf_{j}\|^2_{L^2}\geq2^{2\ell j}\int_{|v|\geq \f83\cdot2^{\a j+1}}2^{3j}|g(2^jv)|^2dv= 2^{2\ell j}\int_{|v|\geq \f{16}3\cdot 2^{(\a+1)j} }|g(v)|^2dv.
	\een 
By \eqref{contra1} and \eqref{contra2}, we are led to that
	\beno
	&&\|f\|^2_{L^2_\ell}\gs\sum_{j=1}^\infty 2^{-\de \a j}\|f\|^2_{L^2_\ell}\gs\sum\limits_{j=1}^\infty\sum\limits_{l>\a j}2^{2(\ell+\de/2)\a j}\|\cP_{l}\F_jf_j\|^2_{L^2}\\
	&=&\sum\limits_{j=1}^\infty2^{2(\ell+\de/2)\a j}2^{2\ell j}\int_{|v|\geq \f{16}3\cdot 2^{(\a+1) j} }|g(v)|^2dv\geq \sum\limits_{j=1}^\infty2^{2(\ell+\de/2)\a j}2^{2\ell j}\int_{ \f{16}3\cdot 2^{(\a +1) j}\leq|v|\leq \f{16}3 2^{(\a+1)(j+1)} }|g(v)|^2dv\\
	&\gs& \sum\limits_{j=1}^\infty\int_{ \f{16}3\cdot (2^{ j)^{2\a}}\leq|v|\leq \f{16}3 2^{2\a}(2^{j})^{2\a}}|v|^{2(\ell+\vep)}|g(v)|^2dv\gs
	\|g\|^2_{L^2_{\ell+\vep}}-\|g\mathbf{1}_{|\cdot|\leq4^{\a +1}}\|^2_{L^2},
	\eeno
	where $\vep=\de \a/(1+\a)$. This contradicts with \eqref{assu} and then completes the proof of this lemma.
	\end{proof}

\section{Proof of Proposition \ref{GainSR}: gain of Sobolev regularity}
This section is devoted to the proof of  Proposition \ref{GainSR} for the gain of regularity in Sobolev space. We will make full use of the coercivity of the collision operator  and the hypo-elliptic estimates for the transport equation. 

\subsection{Gain of finite Sobolev regularity in spatial variable $x$} We  want to prove \eqref{GainRx}. The key point relies on the   hypo-elliptic property of the transport equation. 
 
\begin{proof}[Proof of   \eqref{GainRx} in Proposition \ref{GainSR}]
We prove it by inductive argument. Recall that $\mathbf{a}:=\max\{-\ga/2,2s\}$ and $\mathbf{b}:=\f s{1+2s}$.
Let $\{\tau^{(m)}_i\}_{m\in\N}\in \R,i=1,2$ satisfy $0=\tau_1^{(0)}<\cdots<\tau^{(m)}_1<\tau^{(m+1)}_1<\cdots<\tau^{(m+1)}_2<\tau^{(m)}_2<\cdots<\tau_2^{(0)}=T$. Then it suffices to show that  for $0\le m\le (\ell-14)/\mathbf{a}$,
 \ben\label{Hythesis}(t-\tau^{(m)}_1)^{m/2}f\in L^\infty([\tau_1^{(m)},\tau_2^{(m)}],H^{2+m\mathbf{b}}_xL^2_{\ell-m\mathbf{a}})\cap L^2([\tau_1^{(m)},\tau_2^{(m)}],H^{2+m\mathbf{b} }_xH^s_{\ell-m\mathbf{a}+\ga/2}).\een
The validity of \eqref{Hythesis} for $m=0$ is obvious. Suppose that it holds for $m$. We shall show   \eqref{Hythesis} holds for $m+1$.

\noindent\underline{\it Step 1:} For $t\in\left[\f{\tau^{(m)}_1+\tau^{(m+1)}_1}2,\f{\tau^{(m)}_2+\tau^{(m+1)}_2}2\right]$, we first observe that
\beno
&&\pa_t \big(\<v\>^{\ell-(m+1)\mathbf{a}}(t-\tau^{(m)}_1)^{m/2}\<D_x\>^{2+m\mathbf{b} }f\big)+v\cdot \nabla_x\big(\<v\>^{\ell-(m+1)\mathbf{a}}(t-\tau^{(m)}_1)^{m/2}\<D_x\>^{2+m\mathbf{b} }f\big)\\ &=&\langle D_v\rangle^s  (t-\tau^{(m)}_1)^{m/2}\<v\>^{\ell-(m+1)\mathbf{a}}\langle D_v\rangle^{-s}\<D_x\>^{2+m\mathbf{b} }Q(f,f)+\f m2(t-\tau^{(m)}_1)^{m/2-1}\<v\>^{\ell-(m+1)\mathbf{a}}\<D_x\>^{2+m\mathbf{b} }f.
\eeno
Then by Corollary \ref{corhy}, we have  
\beno
&&\|(t-\tau^{(m)}_1)^{m/2}\<v\>^{\ell-(m+1)\mathbf{a}}\<D_x\>^{(m+1)\mathbf{b}}f\|_{L^2([\tau^{(m+1)}_1,\tau^{(m+1)}_2],H^2_xL^2_v)}\\
&\leq&  \|\<v\>^{\ell-(m+1)\mathbf{a}}(t-\tau^{(m)}_1)^{m/2}\<D_x\>^{m\mathbf{b}}f\|_{L^2([\tau^{(m)}_1,\tau^{(m)}_2],H^2_xH^s_v)}\\
&&+\f m2\|(t-\tau^{(m)}_1)^{m/2-1}\<v\>^{\ell-(m+1)\mathbf{a}}\<D_x\>^{2+m\mathbf{b} }f\|_{L^2\left(\left[\f{\tau^{(m)}_1+\tau^{(m+1)}_1}2,\f{\tau^{(m)}_2+\tau^{(m+1)}_2}2\right],L^2_{x,v}\right)}\\
&&+\|\<v\>^{\ell-(m+1)\mathbf{a}}(t-\tau^{(m)}_1)^{m/2}\<D_v\>^{-s}\<D_x\>^{2+m\mathbf{b}}Q(f,f)\|_{L^2\left(\left[\f{\tau^{(m)}_1+\tau^{(m+1)}_1}2,\f{\tau^{(m)}_2+\tau^{(m+1)}_2}2\right],L^2_{x,v}\right)}:=\mathbf{I}_1+\mathbf{I}_2+\mathbf{I}_3.
\eeno

  Since $\mathbf{a}\geq -\ga/2$, by \eqref{Hythesis}, we have
$\mathbf{I}_1+\mathbf{I}_2\leq \|(t-\tau^{(m)}_1)^{m/2}f\|_{L^2([\tau^{(m)}_1,\tau^{(m)}_2],H^{2+m\mathbf{b}}_xH^s_{\ell-m\mathbf{a}+\ga/2})}\le C_{\ell,m,\mathbf{b},\tau}.
 $
 Next we   want to estimate $\mathbf{I}_{3}$ by duality.  We  have the decomposition:
\beno
&&(\<v\>^{\ell-(m+1)\mathbf{a}}\<D_x\>^{2+m\mathbf{b}}Q(f,f),\<D_v\>^{-s}g)_{L^2_{x,v}}=(\<D_x\>^{2+m\mathbf{b}}Q(f,\<v\>^{\ell-(m+1)\mathbf{a}}f),\<D_v\>^{-s}g)_{L^2_{x,v}}\\
&&+(\<v\>^{\ell-(m+1)\mathbf{a}}\<D_x\>^{2+m\mathbf{b}}Q(f,f)-\<D_x\>^{2+m\mathbf{b}}Q(f,\<v\>^{\ell-(m+1)\mathbf{a}}f),\<D_v\>^{-s}g)_{L^2_{x,v}}:=\mathbf{\tilde{I}}_{3,1}+\mathbf{\tilde{I}}_{3,2}.
\eeno
For $\mathbf{\tilde{I}}_{3,1}$, by Lemma \ref{upperQ} and the fact that $\<q\>^{2+m\mathbf{b}}\leq C_{m,\mathbf{b}}(\<p\>^{2+m\mathbf{b}}+\<q-p\>^{2+m\mathbf{b}})$, we have 
\[|\mathbf{\tilde{I}}_{3,1}|\leq \sum_{p,q\in \Z^3}\<q\>^{2+m\mathbf{b}}\|\mathcal{F}_x(f)(p)\|_{L^2_{14}}\|\<v\>^{\ell-(m+1)\mathbf{a}}\mathcal{F}_x(f)(q-p)\|_{H^s_{\gamma+2s}}\|\mathcal{F}_x(g)(q)\|_{L^2_v} \]\[
 \leq   C_{m,\mathbf{b}}\big( \|f\|_{H^{2+m\mathbf{b}}_xL^2_{\ell-m\mathbf{a}}}\|f\|_{H^2_xH^s_{\ell+\ga/2}}+\|f\|_{H^{2+m\mathbf{b}}_xH^s_{\ell-m\mathbf{a}+\ga/2}}\big)\|g\|_{L^2_{x,v}},\]
where we use the facts that $14\leq \ell-m\mathbf{a}$ and $-\mathbf{a}+\ga+2s\leq\ga/2$.
For $\mathbf{\tilde{I}}_{3,2}$, due to Lemma \ref{VQ}, we have  
\beno
|\mathbf{\tilde{I}}_{3,2}| &\leq& C_{\ell,m,\mathbf{b}}\sum_{p,q\in\Z^3}(\<p\>^{2+m\mathbf{b}}+\<q-p\>^{2+m\mathbf{b}})\Big(\|\mathcal{F}_x(f)(p)\|_{L^2_{14}}\|\<v\>^{l-(m+1)\mathbf{a}}\mathcal{F}_x(f)(q-p)\|_{H^s_{\gamma/2}}\\
&&+\|\<v\>^{l-(m+1)\mathbf{a}}\mathcal{F}_x(f)(p)\|_{H^s_{\ga/2}}\|\mathcal{F}_x(f)(q-p)\|_{L^2_{14}}\Big)\|\mathcal{F}_x(g)(q)\|_{L^2_v}\\
&\leq&C_{\ell,m,\mathbf{b}} \big(\|f\|_{H^{2+m\mathbf{b}}_xL^2_{14}}\|f\|_{H^2_xH^s_{l-(m+1)\mathbf{a}+\ga/2}}+\|f\|_{H^2_{x}L^2_{14}}\|f\|_{H^{2+m\mathbf{b}}_xH^s_{l-(m+1)\mathbf{a}+\ga/2}}\big)\|g\|_{L^2_{x,v}}.
\eeno
Finally by duality, these imply that
\[|\mathbf{I}_3|\leq C_{\ell,m,\mathbf{b},\tau}\big(\|f\|_{L^\infty([\tau^{(m)}_1,\tau^{(m)}_2],H^{2+m\mathbf{b}}_xL^2_{\ell-m\mathbf{a}})}\|(t-\tau^{(m)}_1)^{m/2}f\|_{L^2([\tau^{(m)}_1,\tau^{(m)}_2],H^{2+m\mathbf{b}}_xH^s_{\ell-m\mathbf{a}+\gamma/2})}\big).\] 
Due to inductive assumption(see \eqref{Hythesis}), we conclude that 
\ben\label{inm}
\|(t-\tau^{(m)}_1)^{m/2}f\|_{L^2([\tau^{(m+1)}_1,\tau^{(m+1)}_2],H^{2+(m+1)\mathbf{b}}_xL^2_{\ell-(m+1)\mathbf{a}})}\le C_{\ell,m,\mathbf{b},\tau}.
\een

\noindent\underline{\it Step 2:}  
By  energy estimate, it is easy to see that
\ben\label{ddt}
  &&\f d{dt}\|(t-\tau^{(m+1)}_1)^{(m+1)/2}f\|^2_{H^{2+(m+1)\mathbf{b}}_xL^2_{\ell-(m+1)\mathbf{a}}}
=(m+1)\|(t-\tau^{(m+1)}_1)^{m/2}f\|^2_{H^{2+(m+1)\mathbf{b}}_{x}L^2_{\ell-(m+1)\mathbf{a}}}\een\[
  +(\<v\>^{\ell-(m+1)\mathbf{a}}\<D_x\>^{2+(m+1)\mathbf{b}}Q(f,f),(t-\tau^{(m+1)}_1)^{m+1}\<v\>^{\ell-(m+1)\mathbf{a}}\<D_x\>^{2+(m+1)\mathbf{b}}f)_{L^2_{x,v}}:=\mathbf{J}_1+\mathbf{J}_2, \]
where $\mathbf{J}_2=\mathbf{J}_{2,1}+\mathbf{J}_{2,2}+\mathbf{J}_{2,3}$ with
 $\mathbf{J}_{2,1}:=(t-\tau^{(m+1)}_1)^{m+1}(Q(f,\<v\>^{\ell-(m+1)\mathbf{a}} \<D_x\>^{2+(m+1)\mathbf{b}}f), \<v\>^{\ell-(m+1)\mathbf{a}}\\\<D_x\>^{2+(m+1)\mathbf{b}}f)_{L^2_{x,v}},$
 $\mathbf{J}_{2,2}:=(t-\tau^{(m+1)}_1)^{m+1}\sum_{p,q\in \Z^3}\<q\>^{2+(m+1)\mathbf{b}}\big(\<q\>^{2+(m+1)\mathbf{b}}-\<q-p\>^{2+(m+1)\mathbf{b} }\big)\\\times \big(Q\big(\mathcal{F}_x(f)(p),\<v\>^{\ell-(m+1)\mathbf{a}}\mathcal{F}_x(f)(q-p)\big), 
 \<v\>^{\ell-(m+1)\mathbf{a}}\mathcal{F}_x(f)(q)\big)_{L^2_v},$
 $\mathbf{J}_{2,3}:=(t-\tau^{(m+1)}_1)^{m+1}\sum\limits_{p,q\in \Z^3}\<q\>^{2(2+(m+1)\mathbf{b})}\\ \times \big(\<v\>^{\ell-(m+1)\mathbf{a}}Q(\mathcal{F}_x(f)(p),\mathcal{F}_x(f)(q-p))-Q(\mathcal{F}_x(f)(p),\<v\>^{\ell-(m+1)\mathbf{a}}\mathcal{F}_x(f)(q-p)), \<v\>^{\ell-(m+1)\mathbf{a}}\mathcal{F}_x(f)(q)\big)_{L^2_v}.$ 
 
By \eqref{inm}, $\mathbf{J}_{1}$ can be bounded by
\[
\int_{\tau^{(m+1)}_1}^{\tau_2^{(m+1)}} |\mathbf{J}_1|dt\leq (m+1)\|(t-\tau^{(m)}_1)^{m/2}\<v\>^{\ell-(m+1)\mathbf{a}}\<D_x\>^{(m+1)\mathbf{b}}f\|^2_{L^2([\tau^{(m+1)}_1,\tau^{(m+1)}_2],H^2_xL^2_v)}<C_{\ell,m,\mathbf{b},\tau}.\] 
Next we turn to the estimate of $\mathbf{J}_{2}$. Firstly Lemma \ref{coer} implies that 
\[\mathbf{J}_{2,1}\ls -\|(t-\tau^{(m+1)}_1)^{(m+1)/2}f\|^2_{H^{2+(m+1)\mathbf{b}}_xH^s_{\ell-(m+1)\mathbf{a}+\ga/2}}+\|(t-\tau^{(m+1)}_1)^{(m+1)/2}f\|^2_{H^{2+(m+1)\mathbf{b}}_xL^2_{\ell-(m+1)\mathbf{a}+\ga/2}}.\]
Since $|\<q\>^{2+(m+1)\mathbf{b}}-\<q-p\>^{2+(m+1)\mathbf{b} }|\leq C\<p\>^{2+(m+1)\mathbf{b} }+C_{m,\mathbf{b}}\<p\>\<q-p\>^{2+(m+1)\mathbf{b} -1}$, we deduce from Lemma \ref{upperQ} and Lemma \ref{pq}  that 
\[|\mathbf{J}_{2,2}| \leq C_{\ell,m,\mathbf{b},\tau}\Big(\|(t-\tau^{(m+1)}_1)^{(m+1)/2}f\|_{H^{2+(m+1)\mathbf{b}}_xL^2_{14}}\|f\|_{H^2_xH^s_{l-(m+1)\mathbf{a}+\ga/2+2s}} \|(t-\tau^{(m+1)}_1)^{\f{m+1}2}f\|_{H^{2+(m+1)\mathbf{b}}_xH^s_{\ell-(m+1)\mathbf{a}+\ga/2}}\]\[+\|(t-\tau_1^{(m+1)})^{1/2}f\|_{H^{2+\mathbf{b}}_xL^2_{14}}\|(t-\tau_1^{(m+1)})^{m/2}f\|_{H^{2+m\mathbf{b}}_xH^s_{l-(m+1)\mathbf{a}+\ga/2+2s}} \|(t-\tau^{(m+1)}_1)^{\f{m+1}2}f\|_{H^{2+(m+1)\mathbf{b}}_xH^s_{\ell-(m+1)\mathbf{a}+\ga/2}}\Big).
\]By Lemma  \ref{VQ}, we  get that
\beno
|\mathbf{J}_{2,3}| 
&\leq& C_{\ell,m,\mathbf{b},\tau}\Big(\|f\|_{H^2_xL^2_{14}}\|(t-\tau^{(m+1)}_1)^{\f{m+1}2}f\|_{H^{2+(m+1)\mathbf{b}}_xH^s_{\ell-(m+1)\mathbf{a}+\ga/2}}\|(t-\tau^{(m+1)}_1)^{\f{m+1}2}f\|_{H^{2+(m+1)\mathbf{b}}_xH^\varrho_{\ell-(m+1)\mathbf{a}+\ga/2}}\\
&&+\|f\|_{H^2_xH^s_{l-(m+1)\mathbf{a}+\ga/2}}\|(t-\tau^{(m+1)}_1)^{\f{m+1}2}f\|_{H^{2+(m+1)\mathbf{b}}_xL^2_{14}} \|(t-\tau^{(m+1)}_1)^{\f{m+1}2}f\|_{H^{2+(m+1)\mathbf{b}}_xH^\varrho_{\ell-(m+1)\mathbf{a}+\ga/2}}\Big),
\eeno
where $0<\varrho<s$. Thanks to the  estimates in the above, using interpolation, we can obtain that
\beno
|\mathbf{J}_2|&\ls& -\|(t-\tau^{(m+1)}_1)^{(m+1)/2}f\|^2_{H^{2+(m+1)\mathbf{b}}_xH^s_{\ell-(m+1)\mathbf{a}+\ga/2}}+C_{\ell,m,\mathbf{b},\tau}\Big(\|(t-\tau^{(m+1)}_1)^{(m+1)/2}f\|^2_{H^{2+(m+1)\mathbf{b}}_xL^2_{\ell-(m+1)\mathbf{a}+\ga/2}}\\
&&+\|(t-\tau^{(m+1)}_1)^{(m+1)/2}f\|^2_{H^{2+(m+1)\mathbf{b}}_xL^2_{\ell-(m+1)\mathbf{a}+\ga/2}}\big(1+\|f\|^2_{H^{2+m\mathbf{b}}_xH^s_{\ell-m\mathbf{a}+\ga/2}}+\|f\|^2_{H^2_xH^s_{\ell+\ga/2}}\big)\Big).
\eeno
%Thus we can obtain that
%\beno
%&&t^{m+1}\sum_{p\in \Z^3}\sum_{q\in \Z^3}\<q\>^{2+(m+1)\alpha}(\<q\>^{2+(m+1)\alpha}-\<q-p\>^{2+(m+1)\alpha }))\int_{\R^3_v}(Q(\mathcal{F}_x(f)(p),\mathcal{F}_x(f)(q-p)),\mathcal{F}_x(f)(q))dv\\
%&\leq & \|t^{(m+1)/2}f\|_{H^{2+(m+1)\alpha}_xL^2_2}\|f\|_{H^2_xH^s_{\gamma/2+2s}}\|t^{(m+1)/2}f\|_{H^{2+(m+1)\alpha}_xH^s_{\gamma/2}}+\sum_{k=1}^{m}\|t^{k/2}f\|_{H^{2+k\alpha}_xL^2_2}\|t^{m+1-k}f\|_{H^{2+(m+1-k)\alpha}_xH^s_{\gamma/2+2s}}\|t^{(m+1)/2}f\|_{H^{2+(m+1)\alpha}_xH^s_{\gamma/2}}
%\eeno
 
Now plugging the estimates of $\mathbf{J}_1$ and $\mathbf{J}_2$ into \eqref{ddt} and integrating w.r.t. time variable from $\tau_1^{(m+1)}$ to $t(\leq \tau_2^{(m+1)})$,   we are led to that
\beno
&&\|(t-\tau^{(m+1)}_1)^{(m+1)/2}f\|^2_{H^{2+(m+1)\mathbf{b}}_xL^2_{\ell-(m+1)\mathbf{a}}}+\int_{\tau_1^{(m+1)}}^t\|(t-\tau^{(m+1)}_1)^{(m+1)/2}f\|^2_{H^{2+(m+1)\mathbf{b}}_xH^s_{\ell-(m+1)\mathbf{a}+\ga/2}}dt\\
&\leq&C_{\ell,m,\mathbf{b},\tau}\int_{\tau_1^{(m+1)}}^t (1+\|f\|^2_{H^2_xH^s_{\ell+\ga/2}}+\|f\|^2_{H^{2+m\mathbf{b}}_xH^s_{\ell-m\mathbf{a}+\ga/2}}) \|(t-\tau^{(m+1)}_1)^{(m+1)/2}f\|^2_{H^{2+(m+1)\mathbf{b}}_xL^2_{\ell-(m+1)\mathbf{a}+\ga/2}} dt.
\eeno
We conclude   \eqref{Hythesis} with $m:=m+1$  by Gronwall inequality. This ends the proof of  \eqref{GainRx}.
\end{proof}

\subsection{Gain of finite Sobolev regularity in velocity variable $v$} The main strategy rests on the  two types of the decomposition which are performed  in both phase and frequency spaces.

\begin{proof}[Proof of \eqref{GainRv} in Proposition \ref{GainSR}] Applying  $\F_j$ to the Boltzmann equation \eqref{Boltzmann},   we have
\ben\label{JJ1}
\notag\f d{dt}\sum_{j=-1}^\infty2^{2jn}\|\F_jf\|^2_{L^2_{x,v}}&=&\sum_{j=-1}^\infty2^{2jn}(Q(f,\F_jf),\F_jf)_{L^2_{x,v}}+\sum_{j=-1}^\infty2^{2jn}(\F_jQ(f,f)-Q(f,\F_jf),\F_jf)_{L^2_{x,v}}\\
&&+\sum_{j=-1}^\infty2^{2jn}([\F_j,v]\cdot \na_x\F_jf ,\F_jf)_{L^2_{x,v}}:=\mathbf{L}_1+\mathbf{L}_2+\mathbf{L}_3,
\een
where $[\F_j,v]=\F_j v-v\F_j$ is the commutator between $\F_j$ and $v$. We only focus on the case that
$2s>1$.
\smallskip

\noindent\underline{\it Estimate of $\mathbf{L}_1$.} Thanks to coercivity estimate (see Lemma \ref{coer}), we have  
\beno
\mathbf{L}_1\ls -\sum_{j=-1}^\infty 2^{2jn}\|\F_j f\|^2_{L^2_xH^s_{\ga/2}}+\sum_{j=-1}^\infty 2^{2jn}\|\F_j f\|^2_{L^2_xL^2_{\ga/2}}\sim -C_n\|f\|^2_{L^2_xH^{n+s}_{\ga/2}}+C_n\|f\|^2_{L^2_xH^{n}_{\ga/2}} .
\eeno
We remark that the last equivalence stems from Lemma \ref{lemma1.4}.

\noindent\underline{\it Estimate of $\mathbf{L}_2$.} We split it into two parts: $\mathbf{L}_2=\mathbf{L}_{2,1}+\mathbf{L}_{2,2}$ where 
$\mathbf{L}_{2,1}:=\sum_{j=-1}^\infty 2^{2jn} (\F_jQ_{-1}(f,f)-Q_{-1}(f,\F_jf),\F_jf)_{L^2_{x,v}}$ and 
$\mathbf{L}_{2,2}:=\sum_{l=0}^\infty (\F_jQ_l(f,f)-Q_l(f,\F_jf),\F_jf)_{L^2_{x,v}}.
 $
If $c_i,d_i,i=1,2,3$ and $\om_i,i=1,\cdots,6$ are defined in Lemma \ref{F_jQ}\eqref{Q_1}, then for any $N\in\N$, we have
\beno
&&|\mathbf{L}_{2,1}|\ls C_N\sum_{j=-1}^\infty 2^{2jn}\Big(\|f\|_{H^2_xL^2_{2-\om_1-\om_2}}(\|\mF_jf\|_{L^2_xH^{c_1}_{\om_1}}+2^{-jN}\|f\|_{L^2_xH^{-N}_{-N}})(\|\mF_jf\|_{L^2_xH^{d_1}_{\om_2}}+2^{-jN}\|f\|_{L^2_xH^{-N}_{-N}})\\
&&+\sum_{p>j+3N_0}C_{N}2^{-(3-2s)(p-j)}\|f\|_{H^2_xL^2_{-\omega_3-\omega_4}}(\|\mF_p f\|_{L^2_xH^{c_2}_{\omega_3}}+2^{-pN}\|h\|_{L^2_xH_{-N}^{-N}})(\|\mF_j f\|_{L^2_xH^{d_2}_{\omega_4}}+2^{-jN}\|f\|_{L^2_xH_{-N}^{-N}})\\
&&+\|f\|_{H^2_xL^2_{-\omega_5-\omega_6}}(\|\mF_jf\|_{L^2_xH^{c_3}_{\omega_5}}+2^{-jN}\|f\|_{L^2_xH_{-N}^{-N}})(\|\mF_jf\|_{L^2_xH^{d_3}_{\omega_6}}+2^{-jN}\|f\|_{L^2_xH_{-N}^{-N}})\Big).
\eeno
Choose $N=2n+1$, then  Lemma \ref{lemma1.4} yields that
$|\mathbf{L}_{2,1}|\ls C_n \|f\|_{H^2_xL^2_{2+(\ga+2s-1)^+}}\|f\|_{L^2_xH^{n+s}_{\ga/2}}\|f\|_{L^2_xH^{n+s-1/2}_{\ga/2+2s-1}}.$
Similarly, by Lemma \ref{F_jQ}\eqref{Q_2}, we can derive that
$|\mathbf{L}_{2,2}|\ls C_n\|f\|_{H^2_xL^2_{5}}\|f\|_{L^2_xH^{n+s}_{\ga/2}}\|f\|_{L^2_xH^{n+s-1}_{\ga/2+2s-1}}.$
 
Using the fact that $f\in L^\infty([0,T],H^2_xL^2_5)$, we obtain that
\[|\mathbf{L}_2|\ls C_n\|f\|_{H^2_xL^2_{5}}\|f\|_{L^2_xH^{n+s}_{\ga/2}}\|f\|_{L^2_xH^{n+s-1/2}_{\ga/2+2s-1}}
\leq\varepsilon \|f\|^2_{L^2_xH^{n+s}_{\ga/2}}+C_{n,\varepsilon}\|f\|^2_{L^2_xH^{n+s-1/2}_{\ga+2s-1}}.
 \]

\noindent\underline{\it Estimate of $\mathbf{L}_3$.} Due to Lemma \ref{le1.2}, we first have  
$[\F_j,v]=2^{-j}\mF_j+2^{-jN}r_N(v,D)$
where $r_N\in S^{-N}_{1,0}$ verifies that $\|r_N(\cdot,D)f\|_{L^2_{x,v}}\leq \|f\|_{L^2_xH^{-N}_v}$. This implies that
\beno
|\mathbf{L}_3|
&\leq& \|f\|_{H^1_xH^{n-1}_v} \|f\|_{L^2_xH^{n}_v}\leq \|f\|^{1/n}_{H^n_xL^2_v}\|f\|^{(n-1)/n}_{L^2_xH^n_v}\|f\|_{L^2_xH^{n}_v}\ls C_n(\|f\|^{2}_{H^n_xL^2_v}+\|f\|^2_{L^2_xH^{n}_v}).
\eeno

Plugging the above estimates into \eqref{JJ1}, we   derive that
\ben\label{df}
&&\f d{dt}\|f\|^2_{L^2_{x}H^n_v}+C_n\|f\|^2_{L^2_xH^{n+s}_{\ga/2}}\leq C_n(\|f\|^2_{L^2_xH^n_{v}}+\|f\|^{2}_{H^n_xL^2_v}+\|f\|^2_{L^2_xH^{n+s-1/2}_{\ga/2+2s-1}}).
\een
Observing that
$\|f\|_{L^2_xH^n_{v}}\ls C_{n} \|f\|^{\f n{n+s}}_{L^2_xH^{n+s}_{\ga/2}}\|f\|^{\f s{n+s}}_{L^2_xL^{2}_{k_1}}$ and $ \|f\|_{L^2_xH^{n+s-1/2}_{\ga/2+2s-1}}\ls C_n\|f\|^{\f {n+s-1/2}{n+s}}_{L^2_xH^{n+s}_{\ga/2}}\|f\|^{\f{1/2}{n+s}}_{L^2_xL^{2}_{k_2}}$
with $k_1:=-\f{n\ga}{2s}$ and $k_2:=2(n+s)(2s-1)+\ga/2$, 
by \eqref{df}, we are led to that
\beno
\f d{dt}\|f\|^2_{L^2_{x}H^n_v}+\|f\|^{2(1+\f sn)}_{L^2_xH^{n}_{v}}\leq \|f\|^2_{L^2_xL^2_{k_1}}+\|f\|^2_{L^2_xL^2_{k_2}}+\|f\|^{2}_{H^n_xL^2_v}.
\eeno
To employee \eqref{GainRx},  we  impose that $k_1,k_2\leq \ell,2+m\mathbf{b}\geq n,\ell-m\mathbf{a}\geq 14$, which is equivalent to  $\ell\geq\max\big\{-\f{n\ga}{2s},2(n+s)(2s-1)+\ga/2,(n-2)(2s+1)\mathbf{a}/s+14\big\}$. These imply that $\sup_{t\in [\tau_1,\tau_2]}(\|f(t)\|_{L^2_xL^2_{k_1}}+\|f(t)\|_{L^2_xL^2_{k_2}}+\|f(t)\|_{H^n_xL^2_v})<\infty$.
By Lemma \ref{le1.6}, we get that $f\in L^\infty([\tau_1,\tau_2],L^2_xH^n_v)$  for any $0<\tau_1<\tau_2<T$. 

To apply the above assertion to our case,  let $n:=(\ell-14)/\mathcal{K}$ with
$\mathcal{K}:=\max\{-\f{\ga}{2s},2(2s-1),(2s+1)\mathbf{a}/s\}$,  then
for fixed $\ell\geq 14$,  we conclude that if $f\in L^\infty([0,T],H^2_xL^2_\ell)$, then $f\in L^\infty([\tau_1,\tau_2],L^2_xH^{n}_v),\quad\forall 0<\tau_1<\tau_2<T$.
Furthermore, by interpolation, we can deduce that for  $\th\in[0,1]$, $f\in L^\infty([\tau_1,\tau_2],L^{2}_xH^{{(\ell-14)(1-\th)}/{\mathcal{K}}}_{\ell\th})$.  It ends the proof of the case $2s>1$.
The case  $2s\leq 1$ can be handled similarly if we choose  $\mathcal{K}:=\max\{-\f{\ga}{2s},(2s+1)\mathbf{a}/s\}$. We complete the proof of  \eqref{GainRv}.
\end{proof}

\section{Proof of Theorem \ref{RSWBol1} and Corollary \ref{globaldecay}: finite smoothing effect} In this section, we will prove  finite smoothing effect in Sobolev spaces for the weak solution to the nonlinear equation with {\it typical rough and slowly decaying data}.

\subsection{Gain of regularity in Besov spaces in $v$ variable} We begin with   some auxiliary lemmas.

\begin{lem}\label{CM11}
	Let $f$ be a nonnegative function verifying that $\|f\|_{L^2_3}\leq M$. Then 
	\beno
	\mathcal{F}_v(f)(0)-|\mathcal{F}_v(f)(\xi)|\geq C_M\|f\|^{11}_{L^1}(\mathbf{1}_{|\xi|\leq 1}|\xi|^2+\mathbf{1}_{|\xi|> 1}).
	\eeno
\end{lem}
\begin{proof} Inspired by \cite{ADVW},   for given $\xi\in\R^3$, there exists $\th\in[0,1)$ such that 
	\beno
	\mathcal{F}_v(f)(0)-|\mathcal{F}_v(f)(\xi)|&=&\mathcal{F}_v(f)(0)-\mathcal{F}_v(f)(\xi)e^{-2\pi i \th}
	=\int_{\R^3}f(v)(1-\cos(2\pi(v\cdot\xi+\th)))dv\\
	&=&2\int_{\R^3}f(v)\sin^2(\pi(v\cdot\xi+\th))dv
	\geq 2\sin^2(\pi\vep)\int_{\R^3}f(v)\mathbf{1}_{\{|v\cdot\xi+\th-p|\geq\vep|\forall p\in\Z\}}dv.
	\eeno
Using the fact that $\sin(\pi\vep)\geq 2\vep$ with $\vep\in(0,1/2)$, we deduce that for $R>0$,
\ben\label{f0-fxi}
&&\mathcal{F}_v(f)(0)-|\mathcal{F}_v(f)(\xi)|\geq 8\vep^2\Big(\|f\|_{L^1}-\int_{\R^3}f(v)\mathbf{1}_{\cup_{p\in\Z}|v\cdot\xi+\th-p|<\vep}dv\Big)\\
&&\geq C\vep^2\Big(\|f\|_{L^1}-\f{\|f\|_{L^2_3}}{R}-\int_{|v|<R}f(v)\mathbf{1}_{\cup_{p\in\Z}|v\cdot\xi+\th-p|<\vep}dv\Big)\geq C\vep^2\Big(\|f\|_{L^1}-\f{\|f\|_{L^2_3}}{R}-\|f\|_{L^2}\notag\\
&&\times\Big|B_R\cap \cup_{p\in\Z}\Big\{\big|v\cdot\f{\xi}{|\xi|}+\f{\th-p}{|\xi|}\big|<\f{\vep}{|\xi|}\Big\} \Big|\geq C\vep^2\Big(\|f\|_{L^1}-\f{\|f\|_{L^2_3}}{R}-\|f\|_{L^2_3}\f{\vep^{1/2}}{|\xi|^{1/2}}R(R|\xi|+1)^{1/2}\Big).\notag
\een
 If $R:=2\|f\|_{L^2_3}/\|f\|_{L^1}$, then 
$\mathcal{F}_v(f)(0)-|\mathcal{F}_v(f)(\xi)|\geq C\vep^2\Big(\|f\|_{L^1}-\f{\vep^{1/2}}{|\xi|^{1/2}}\f{\|f\|^2_{L^2_3}}{\|f\|_{L^1}}\Big(\f{\|f\|_{L^2_3}}{\|f\|_{L^1}}|\xi|+1\Big)^{1/2}\Big).
 $
 If $|\xi|\leq 1$,  we choose 
$\vep:=\f{\|f\|_{L^1}^5}{4\|f\|^4_{L^2_3}(\|f\|_{L^2_3}+\|f\|_{L^1})}|\xi|$, which
 yields that
$\mathcal{F}_v(f)(0)-|\mathcal{F}_v(f)(\xi)|\geq C\|f\|_{L^1}\vep^2 \geq C_M\|f\|^{11}_{L^1}|\xi|^2.$
If $|\xi|>1$, we choose $\vep:=\f{\|f\|_{L^1}^5}{4\|f\|^4_{L^2_3}(\|f\|_{L^2_3}|\xi|+\|f\|_{L^1})}|\xi|,$
which implies that
$\mathcal{F}_v(f)(0)-|\mathcal{F}_v(f)(\xi)|\geq C\|f\|_{L^1}\vep^2 \geq C_M\|f\|_{L^1}^{11}$.
We get the desired result from these two estimates and then end the proof of this lemma.
\end{proof}

\begin{cor}\label{CM112}
	Let $f$ be a nonnegative function verifying that $\|f\|_{L^2_3}\leq M$. Then 
	\ben\label{lobo}
	\int_{\T^3\times\R^6}b(\f{v-v_*}{|v-v_*|}\cdot\si)f_*(f-f')^2d\si dv_*dv+\|f\|_{L^2}^2\geq C_M\|f\|^{11}_{L^1}\|f\|^2_{H^s}.
	\een
\end{cor}
\begin{proof} Again by \cite{ADVW}, if $\xi^{\pm}:=\f{\xi\pm|\xi|\si}2$, then  
\beno
&&\int_{\T^3\times\R^6}b(\f{v-v_*}{|v-v_*|}\cdot\si)f_*(f-f')^2d\si dv_*dv 
\geq \int_{\S^2\times\R^3} b(\f{\xi}{|\xi|}\cdot\si) (\mathcal{F}_v(f)(0)-|\mathcal{F}_v(f)(\xi^-)|)(|\mathcal{F}_v(f)(\xi)|^2\\
&&  +|\mathcal{F}_v(f)(\xi^+)|^2)d\si d\xi\ge \int_{\R^3}|\mathcal{F}_v(f)(\xi)|^2\int_{\S^2}b(\f{\xi}{|\xi|}\cdot\si)(\mathcal{F}_v(f)(0)-|\mathcal{F}_v(f)(\xi^-)|)d\si d\xi.
\eeno
 We conclude the desired result by Lemma \ref{CM11} and the fact that $|\xi^-|=|\xi|\sin \f\th 2$.
\end{proof}

\begin{lem}\label{L1lemma}
	Suppose that $\gamma\in(-3/2,0),s\in (0,1)$. Let $f\in L^1_2\cap L^2_3$ be a nonnegative function. Then
	\beno
	-(Q(f,f),f)_{L^2_v}\geq  C_f\|f\|^{11(1-\f23\ga)}_{L^1}\|f\|^2_{H^s_{\ga/2}}-C(1+\|f\|_{L^2_3})\|f\|^2_{L^2_{\gamma/2}},
	\eeno
where  $C_f=C(\|f\|_{L^2_3})$ depends on the upper bound of $\|f\|_{L^2_3}$.
\end{lem}
\begin{proof}
We take the following decomposition: $(-Q(f,f),f)_{L^2_v}=\mathcal{L}+\mathcal{E}^\ga(f)$ where 
 $\mathcal{L}:=-\f12 \int_{\R^6\times\S^2}B(|v-v_*|,\si)f_*(f'^2-f^2)d\si dv_*dv$ and 
 $\mathcal{E}^\ga(f):=\f12 \int_{\R^6\times\S^2}B(|v-v_*|,\si)f_*(f'-f)^2d\si dv_*dv$.
 
 By change of variables, since  $\ga>-\f32$, we have 
\[  
|\mathcal{L}|= |\S^1|\Big|\int_{\R^6}\int_0^{\pi/2}\sin \th\Big(\f1{\cos^3\f\th 2}B(\f{|v-v_*|}{\cos\f\th 2},\cos\th)-B(|v-v-*|,\cos\th)\Big)f_*f^2d\th dv_*dv\Big| \leq   C\|f\|_{L^2_3}\|f\|_{L^2_{\gamma/2}}^2.
\]

   Let $F=f\<v\>^{\ga/2}$, then we  have  
$\mathcal{E}^\ga(f)=\f12\int_{\S^2\times\R^6}b(\cos\th)|v-v_*|^\ga f_*(\<v'\>^{\ga/2}F'-\<v\>^{-\ga/2}F)^2d\si dv_*dv.$
Since $(A-B)^2\geq \f{A^2}2-B^2$, one has 
\beno
&&2\mathcal{E}^\ga(f)\geq \f12 \int_{\S^2\times\R^6}|v-v_*|^\ga b(\cos\th) f_*\<v'\>^{-\ga}(F'-F)^2d\si dv_* dv\\
&&-\int_{\S^2\times\R^6}|v-v_*|^\ga b(\cos\th) f_* F^2(\<v'\>^{-\ga/2}-\<v\>^{-\ga/2})^2d\si dv_*dv:=\mathcal{R}_1+\mathcal{R}_2.
\eeno
Observing that $|v-v_*|^\ga\<v'\>^{\ga}\<v_*\>^{\ga}\gs 1$, by Corollary \ref{CM112}, we have
$\mathcal{R}_1+\|F\|_{L^2}^2\gs \int_{\R^3}b(\cos\th)(f_*\<v_*\>^\ga)(F'-F)^2d\si dv_*dv\gs C_f\|f\<\cdot\>^{\ga}\|^{11}_{L^1}\|F\|^2_{H^s},$
where $C_f$ depends on $\|f\<\cdot\>^{\ga}\|_{L^2_3}$. Since 
\beno
\int_{\R^3}f(v)\<v\>^{\ga}dv&\geq& R^{\ga}\int_{\<v\>\leq R} f(v)dv=R^{\ga}\Big(\|f\|_{L^1}-\int_{\<v\>>R}f(v)dv\Big)\\
&\gs& R^\ga \Big(\|f\|_{L^1}-\|f\|_{L^2_3}\big(\int_{\<v\>>R}\<v\>^{-6}dv\big)^{1/2}\Big)\geq R^\ga \Big(\|f\|_{L^1}-\|f\|_{L^2_3}R^{-3/2}\Big),
\eeno
 we have
$\mathcal{R}_1\gs C_f\|f\|^{11(1-\f23 \ga)}_{L^1}\|F\|^2_{H^s}$ if $R:=\big(2\|f\|_{L^2_3}/\|f\|_{L^1}\big)^{2/3}$.
Next we turn to the estimate of $\mathcal{R}_2$. One may have
$|\mathcal{R}_2|\ls \int_{\S^2\times\R^6} b(\cos\th)|v-v_*|^\ga f_* F^2|v-v'|^2\<v(\kappa)\>^{-\ga-2}d\si dv_*dv$
with $v(\kappa)=v+\kappa(v-v')$. Since $-\ga-2<0$ and $|v-v'|=|v-v_*|\sin \f \th 2$,
we deduce that
$|\mathcal{R}_2|\ls\int_{\R^3} f_* F^2\<v_*\>^{\ga+2}dv_*dv\leq \|f\|_{L^1_2}\|F\|^2_{L^2}.$

We conclude that for $-\f32<\ga<0$,
$\mathcal{E}^\ga(f)\geq C_f \|f\|^{11(1-\f23\ga)}_{L^1}\|f\|^2_{H^s_{\ga/2}}-\|f\|_{L^1_2}\|f\|^2_{L^2_{\ga/2}}.$
From this together with the estimate of $|\mathcal{L}|$, we get the desired result.
\end{proof}

\smallskip

We are now ready to prove the  regularity in Besov spaces. 
\begin{proof}[Proof of Theorem \ref{RSWBol1}: Gain of Besov regularity] For $0<\de<1$,
We  introduce  the following macroscopic cutoff
$I_\de:=\{(t,x)\in [0,T]\times\T^3|\int_{\R^3}f(t,x,v)dv\geq\de\}$.
By the standard energy method, we can deduce that
\ben\label{fT}
&&\|f(T)\|^2_{L^2_{x,v}}-\|f_0\|^2_{L^2_{x,v}}=\int_0^T(Q(f,f),f)_{L^2_{x,v}}dt\\
\notag&=&\int_0^T \int_{\T^3}\mathbf{1}_{I_\de}(t,x)(Q(f,f),f)_{L_v^2}dxdt+\int_0^T \int_{\T^3}\mathbf{1}_{I^c_\de}(t,x)(Q(f,f),f)_{L_v^2}dxdt:=\mathcal{P}.
\een

From Lemma \ref{L1lemma}, the condition that $f\in L^\infty ([0,T],L^\infty_xL^2_\ell)$ and the definition of $I_\de$, we have 
$\mathcal{P}\le -C_f\de^{11(1-\f23\ga)}\|\mathbf{1}_{I_\de}\<D_v\>^s\<v\>^{\ga/2}f\|^2_{L^2([0,T]\times\T^3\times\R^3)}+C(\|f\|_{L^\infty([0,T];L^\infty_xL^2_\ell)}),$
which implies that 
\ben\label{Ide}
\de^{11(1-\f23\ga)}\mathbf{1}_{I_\de}\<D_v\>^s\<v\>^{\ga/2}f\in L^2([0,T]\times\T^3\times\R^3).
\een
	
Next, we focus on the term $\mathbf{1}_{I^c_\de} f$. For $(x,t)\in I^c_\de$, we have $\int_{\R^3}f(t,x,v)dv<\de$. Thus using the condition that $f\in L^\infty ([0,T],L^\infty_xL^2_\ell)$, for any $1< p<2$,  we get that
 $\|\mathbf{1}_{I^c_\de}f\|_{L^p}\le \|\mathbf{1}_{I^c_\de}f\|_{L^1}^{\f{2}{p}-1}\|\mathbf{1}_{I^c_\de}f\|_{L^2}^{2-\f{2}{p}}
   \leq C(T,\|f\|_{L^\infty([0,T];L^\infty_xL^2_\ell)})\de^{\f{2}{p}-1}$,
which yields that
\ben\label{Idec}
\de^{1-\f2p}\mathbf{1}_{I^c_\de}\<v\>^{\ga/2} f\in L^p([0,T]\times\T^3\times\R^3).
\een
	
Now we can obtain Besov regularity   from \eqref{Ide} and \eqref{Idec}. We first observe that
\beno
2^{js}\de^{11(1-\f23\ga)}\|\mathbf{1}_{I_\de}\F_j\<v\>^{\ga/2}f\|_{L^2([0,T]\times\T^3\times\R^3)}&\leq& C\|
\de^{11(1-\f23\ga)}\mathbf{1}_{I_\de}\<D_v\>^s\<v\>^{\ga/2}f\|_{L^2([0,T]\times\T^3\times\R^3)};\\
\|\mathbf{1}_{I_\de}\F_j\<v\>^{\ga/2}f\|_{L^1([0,T]\times\T^3\times\R^3)}&\leq& C\|
\mathbf{1}_{I_\de}\<v\>^{\ga/2}f\|_{L^1([0,T]\times\T^3\times\R^3)},
\eeno
where $C$ is a universal constant. By interpolation, we get that for any $1<p<2$,
\beno
2^{js(2-\f2p)}\de^{11(1-\f23\ga)(2-\f2p)}\|\mathbf{1}_{I_\de}\F_j\<v\>^{\ga/2}f\|_{L^p([0,T]\times\T^3\times\R^3)}\leq C.
\eeno
Since  $\de^{11(1-\f23\ga)(2-\f2p)}\mathbf{1}_{I_\de}\<v\>^{\ga/2}f\in \widetilde{L^p}([0,T]\times\T^3,B^{s(2-\f2p)}_{p,\infty})$ and $\de^{1-\f2p}\mathbf{1}_{I^c_\de}\<v\>^{\ga/2}f\in \widetilde{L^p}([0,T]\times\T^3,B^0_{p,\infty})$,  we choose  $\de^{\f2p-1+11(1-\f23\ga)(2-\f2p)}=2^{-js(2-\f2p)}$ and $\c:=\f{s(2-\f2p)(\f2p-1)}{\f2p-1+11(1-\f23\ga)(2-\f2p)}$ to derive that
\beno
2^{j\c}\|\F_j\<v\>^{\ga/2}f\|_{L^p([0,T]\times\T^3\times\R^3)}&\leq& 2^{j\c}\|\mathbf{1}_{I_\de}\F_j\<v\>^{\ga/2}f\|_{L^p([0,T]\times\T^3\times\R^3)}+2^{j\c}\|\mathbf{1}_{I^c_\de}\F_j\<v\>^{\ga/2}f\|_{L^p([0,T]\times\T^3\times\R^3)}\\
&\leq& C2^{j\c}(2^{-js(2-\f2p)}\de^{-11(1-\f23\ga)(2-\f2p)}+\de^{\f2p-1})<+\infty.
\eeno
It yields that $\<v\>^{\ga/2}f\in \widetilde{L^p}([0,T]\times\T^3,B^\c_{p,\infty})$. 
For any $\tilde{\c}<\c$,  by using the Bernstein inequality, one has
\beno
\|\<D_v\>^{\tilde{\c}}\<v\>^{\ga/2}f\|_{L^p}\leq \sum_{j=-1}^\infty \|\<D_v\>^{\tilde{\c}}\F_j\<v\>^{\ga/2}f\|_{L^p}\ls \sum_{j=-1}^\infty 2^{(\tilde{\c}-\c)j}2^{j\c}\|\F_j\<v\>^{\ga/2}f\|_{L^p}\ls \sum_{j=-1}^\infty 2^{(\tilde{\c}-\c)j}<+\infty,
\eeno
 In particular, it implies that $\<D_v\>^{\tilde{\c}}\<v\>^{-(\ga+2s+\f32)}f\in L^p([0,T]\times\T^3\times\R^3)$.

If $g:=\<D_v\>^{-2s}\<v\>^{-(\ga+2s+\f32)}Q(f,f)$, then
$\pa_t (\<v\>^{-(\ga+2s+\f32)}f)+v\cdot\na_x (\<v\>^{-(\ga+2s+\f32)}f)=\<D_v\>^{2s}g.$
 Since
$\|g\|_{L^p_{t,x,v}}=\sup\limits_{\|h\|_{L^q_{t,x,v}}\leq 1}(Q(f,f),\<v\>^{-(\ga+2s+\f32)}\<D_v\>^{-2s}h)$  for $\f1p+\f1q=1$,
  by Lemma \ref{upperQ}, we have
\beno &|(Q(f,f),\<v\>^{-(\ga+2s+\f32)}\<D_v\>^{-2s}h)_{t,x,v}|\leq \int_{0}^T\int_{\T^3}\|\<v\>^{-(\ga+2s+\f32)}\<D_v\>^{-2s}h(t,x,\cdot)\|_{H^{2s}_{\ga+2s}}\|f(t,x,\cdot)\|_{L^2_3}\\
&\times\|f(t,x,\cdot)\|_{L^2}dxdt\le C_T\|f\|^2_{L^\infty([0,T];L^\infty_xL^2_\ell)}\int_{0}^T\int_{\T^3}\|h(t,x,\cdot)\|_{L^q}dxdt\ls C_T\|h\|_{L^q_{t,x,v}},
\eeno
which implies that $\|g\|_{L^p_{t,x,v}}<+\infty$. Then by Corollary \ref{corhy}, we conclude that
\beno
\|\<D_x\>^{\f{\tilde{\c}}{1+2s+\tilde{\c}}}\<v\>^{-(\ga+2s+\f32)}f\|_{L^p([\tau_1,\tau_2]\times \T^3\times\R^3)}\ls \|\<D_v\>^{\tilde{\c}}\<v\>^{-(\ga+2s+\f32)}f\|_{L^p([0,T]\times \T^3\times\R^3)}+\|g\|_{L^p([0,T]\times \T^3\times\R^3)}.
\eeno
This ends the proof.
\end{proof}

\subsection{Finite smoothing effect in Sobolev spaces}  We begin with a lemma. 
  
\begin{lem}\label{lemma31} Let $g\geq0$ and  $Q_l,\mP_k,\U_k$ and $\cP_k$ be  defined in Subsection \ref{DDP}.

  (1). If $l\geq0$, then 
	\beno
	&\left|\big(\cP_kQ_l(g,h)-Q_l(g,\cP_kh),\cP_kf\big)\right|\leq 2^{(\ga/2) l}2^{(\ga/2+s)k}(\|g\|_{L^2_2}+2^{-2k}\|g\|_{L^1_2})\|\mP_k h\|_{L^2}\|\mP_kf\|_{H^s}+2^{(\ga+2s) l} \|g\|_{L^1_{2-2s}}\\
	&\times \|\mP_kh\|_{L^2}\|\mP_kf\|_{L^2}+2^{(\ga+s) l}2^{-k}\|g\|_{L^1_{1+s}}\|\mP_kh\|_{L^2}\|\mP_kf\|_{L^2}+2^{(\ga+2s) l}2^{-2k}\|g\|_{L^1_{2}}\|\mP_kh\|_{L^2}\|\mP_kf\|_{L^2}\\
	&+ 2^{(\ga+2s) l}2^{-3k}\sum_{a> k}\|\cP_ag\|_{L^2_{3-2s}}\|\U_ah\|_{L^1}\|\mP_kf\|_{L^2}.
	\eeno

	(2). If $l=-1$, then
	\beno
		&&\left|\big(\cP_kQ_{-1}(g,h)-Q_{-1}(g,\cP_kh),\cP_kf\big)\right|\leq 2^{(\ga/2+s-1)k}\|g\|_{L^2_2}\|\mP_k h\|_{L^2}\|\mP_kf\|_{H^s}+2^{-2k}\|g\|_{L^2_2}\|\mP_kh\|_{L^2}\|\mP_kf\|_{L^2}.
	\eeno
\end{lem}
\begin{proof} 
	We only provide a proof for $k\geq0$ since the case that $k=-1$ can be proved similarly. Before going further, we introduce the notation  $\zeta^\ga_g(f):=\int_{\R^6\times\S^2}b(\cos\th)|g_*| |v-v_*|^\ga(f(v')-f(v))^2d\si dvdv_*.$
	We also define a nonnegative radial function $\chi\in C^\infty$ verifying $0\leq\chi(v)\leq1,\chi(v)=1,|v|\leq 2$ and $\chi(v)=0,|v|>3$. 	Recalling that   $\vphi$ is defined in (\ref{7.1}), we perform the following decomposition:
	\beno
	&&\big(\cP_kQ_l(g,h)-Q_l(g,\cP_kh),\cP_kf\big)=\int_{\R^6\times\S^2}b(\cos\th)\Phi_l^\ga(|v-v_*|) g_* h \vphi(2^{-k}v')f(v')(\vphi(2^{-k}v')-\vphi(2^{-k}v))\\
  &&\times(1_{|v|\geq4|v_*|}+\chi(v'-v) 1_{|v|<4|v_*|})d\si dv dv_*+\int_{\R^6\times\S^2}b(\cos\th)\Phi_l^\ga(|v-v_*|) g_* h \vphi(2^{-k}v')f(v')(\vphi(2^{-k}v')\\
  && -\vphi(2^{-k}v)) 
 (1-\chi(v-v'))1_{|v|<4|v_*|}d\si dv dv_*:=\mA+\mB.
	\eeno
	In what follows, we will frequently use the following fact: if $v(\ka):=v+(1-\ka)(v'-v)$ with $\ka\in[0,1]$, then
	\ben\label{vvv}
	1_{|v|\geq 4|v_*|}|v|\sim1_{|v|\geq 4|v_*|}|v-v_*|\sim1_{|v|\geq 4|v_*|}|v(\ka)|.
	\een
	
	\noindent\underline{\it Step 1: Estimate of $\mA$.} Since
	$\vphi(2^{-k}v')-\vphi(2^{-k}v)=2^{-k}\na\vphi(2^{-k}v)\cdot(v'-v)+2^{-2k-1}\int_0^1(1-\ka)(\na^2\vphi(2^{-k}v):(v'-v)\otimes(v'-v)d\ka$, we have the further decomposition $\mA=\mA_1+\mA_2$ where
	\beno
	\mA_1&:=&\int_{\R^6\times\S^2}b(\cos\th)\Phi_l^\ga(|v-v_*|)  g_* h \vphi(2^{-k}v')f(v')(2^{-k}\na\vphi(2^{-k}v)\cdot(v'-v))\\
	&&\times(1_{|v|\geq4|v_*|}+\chi(v'-v) 1_{|v|<4|v_*|})d\si dv dv_*,\\
	\mA_2&:=&\f12\int_{\R^6\times\S^2}\int_0^1(1-\ka)b(\cos\th)\Phi_l^\ga(|v-v_*|)  g_* h \vphi(2^{-k}v')f(v')(2^{-2k}(\na^2\vphi(2^{-k}v(\ka)):\\
	&&(v'-v)\otimes(v'-v))(1_{|v|\geq4|v_*|}+\chi(v'-v) 1_{|v|<4|v_*|})d\si dv dv_*d\ka.
	\eeno
	
	\noindent\underline{\it Step 1.1: Estimate of $\mA_1$.} It is not difficult to see that
	\beno
	\mA_1&=&\int_{\R^6\times\S^2}b(\cos\th)\Phi_l^\ga(|v-v_*|)  g_* h (\vphi(2^{-k}v')f(v')-\vphi(2^{-k}v)f(v))(2^{-k}\na\vphi(2^{-k}v)\cdot(v'-v))\\
	&&\times(1_{|v|\geq4|v_*|}+\chi(v'-v) 1_{|v|<4|v_*|})d\si dv dv_*+\int_{\R^6\times\S^2}b(\cos\th)\Phi_l^\ga(|v-v_*|) g_* h \vphi(2^{-k}v)f(v)\\
	&&\times(2^{-k}\na\vphi(2^{-k}v)\cdot(v'-v))(1_{|v|\geq4|v_*|}+\chi(v'-v) 1_{|v|<4|v_*|})d\si dv dv_*:=\mA_{1,1}+\mA_{1,2}.
	\eeno
	
	\underline{\it Estimate of $\mA_{1,1}$.} Since $|v-v'|=|v-v_*|\sin(\th/2)$,  by Cauchy-Schwartz inequality, if $l\geq0$, then 
	\beno
	&&|\mA_{1,1}|\leq  2^{(\ga/2) l}(\zeta^\ga_g)^{1/2}(\vphi(2^{-k}\cdot)f)\Big(\int_{\R^6\times \S^2}b(\cos\th)g_* h^2|2^{-k}\na \vphi(2^{-k}v)\cdot (v'-v)|^2(1_{|v|\geq 4|v_*|}\\
	&&+\chi (v'-v)1_{|v|<4|v_*|})^2d\si dvdv_*\Big)^{1/2}\leq  2^{(\ga/2) l}(\zeta^\ga_g)^{1/2}(\vphi(2^{-k}\cdot)f)(\|g\|^{1/2}_{L^1}+2^{-k}\|g\|^{1/2}_{L^1_2})\|\na\vphi(2^{-k}\cdot)h\|_{L^2}. 
%	&\leq& 2^{\ga l}2^{(\ga+2s)k}\|g\|_{L^2_2}\|\vphi(2^{-k}\cdot)f\|^2_{H^{s}}+2^{\ga l}\|g\|_{L^1_2}\|\na\vphi(2^{-k}\cdot)h\|^2_{L^2}.
	\eeno
If $l=-1$, then
	$|\mA_{1,1}|\leq 2^{-k}(\zeta^\ga_g)^{1/2}(\vphi(2^{-k}\cdot)f)\|g\|^{1/2}_{L^2_2}\|\na\vphi(2^{-k}\cdot)h\|_{L^2}.$
Observing that $\zeta^\ga_g(f)\leq \|g\|_{L^2_2}\|f\|^2_{H^s_{\ga/2+s}}$(see \cite{HE}),   we   get that
$|\mA_{1,1}|\leq \mathbf{1}_{l=-1}2^{(\ga/2+s-1)k}\|g\|_{L^2_2}\|\mP_k h\|_{L^2}\|\mP_kf\|_{H^s}+\mathbf{1}_{l\geq0}2^{(\ga/2) l}2^{(\ga/2+s)k}(\|g\|_{L^2_2}+2^{-2k}\|g\|_{L^1_2})\\\times\|\mP_k h\|_{L^2}\|\mP_kf\|_{H^s}.$

	\underline{\it Estimate of $\mA_{1,2}$.} Recalling that $\chi$ is a radial function, by the fact that
	\beno
	&&\int_{\S^2} b(\f{v-v_*}{|v-v_*|}\cdot \si)(v-v')\chi(|v-v'|)d\si=\int_{\S^2} b(\f{v-v_*}{|v-v_*|}\cdot \si)\f{1-(\f{v-v_*}{|v-v_*|},\si)}2\chi(|v-v'|)d\si (v-v_*),
	\eeno
	we derive that
	$|\mA_{1,2}| \ls \int_{\R^6\times \S^2}b(\cos\th)\sin^2\th\Phi_l^\ga(|v-v_*|)|v-v_*| |g_* h \vphi(2^{-k}v)f(v) (2^{-k}\na\vphi(2^{-k}v)(1_{|v|\geq4|v_*|}+\chi(v'-v) 1_{|v|<4|v_*|})|d\si dv dv_*.$
	Similar to $\mA_{1,1}$, we can obtain that
	$|\mA_{1,2}|\leq \mathbf{1}_{l=-1}2^{-3k}\|g\|_{L^2_2}\|\mP_kh\|_{L^2}\|\mP_kf\|_{L^2}+\mathbf{1}_{l\geq0}2^{\ga l}( \|g\|_{L^1}\|\mP_kh\|_{L^2} \|\mP_kf\|_{L^2}+2^{-k}\|g\|_{L^1_1}\|\mP_kh\|_{L^2}\|\mP_kf\|_{L^2}).
   $ We conclude that
	\beno
	&&|\mA_1|\leq\mathbf{1}_{l=-1}(2^{(\ga/2+s-1)k}\|g\|_{L^2_2}\|\mP_k h\|_{L^2}\|\mP_kf\|_{H^s}+2^{-3k}\|g\|_{L^2_2}\|\mP_kh\|_{L^2}\|\mP_kf\|_{L^2})\\
	&&+\mathbf{1}_{l\geq0}(2^{(\ga/2) l}2^{(\ga/2+s)k}(\|g\|_{L^2_2}+2^{-2k}\|g\|_{L^1_2})\|\mP_k h\|_{L^2}\|\mP_kf\|_{H^s}+2^{\ga l}(\|g\|_{L^1}+2^{-k}\|g\|_{L^1_1})\|\mP_kh\|_{L^2}\|\mP_kf\|_{L^2}).
	\eeno
	
	\noindent\underline{\it Step 1.2: Estimate of $\mA_2$.} In this situation, one has $\<v\>\sim \<v'\>\sim 2^{k}$. Then we have
	
	\noindent$\bullet$ $l\geq0$. We have 
 $|\mA_2|\leq 2^{(\g+2s) l}2^{-2k}\int_{\<v\>\sim \<v'\>\sim 2^{k}} b(\cos\th)\sin^2\f\th 2|v-v_*|^{2-2s} |g_* h\vphi(2^{-k}v')f'|d\si dvdv_*\leq 2^{(\ga+2s) l}\|g\|_{L^1_{2-2s}}\|\mP_kh\|_{L^2}\|\mP_kf\|_{L^2}.$
Here we use the  regular change of variables(Lemma \ref{chv}), i.e., 
\beno
\int_{\R^3\times \S^2} b(\cos\th)|v-v_*|^{2-2s} f(v')d\si dv=\int_{\R^3\times \S^2} b(\cos\th)\f1{\cos^{3+\ga}(\th/2)}|v-v_*|^{2-2s} f(v)d\si dv.
\eeno

	\noindent$\bullet$ $l=-1$. $|\mA_2|\leq2^{-2k}\int_{\<v\>\sim\<v'\>\sim 2^k}b(\cos\th)\sin^2(\th/2)|v-v_*|^{2+\ga} g_*hf'd\si dvdv_* 
    \leq 2^{-2k}\|g\|_{L^2}\|\mP_kh\|_{L^2}\|\mP_kf\|_{L^2}.$
 
	We conclude that
	\beno
	&&|\mA|\leq \mathbf{1}_{l=-1}(2^{(\ga/2+s-1)k}\|g\|_{L^2_2}\|\mP_kf\|_{H^s}\|\mP_k h\|_{L^2}+2^{-2k}\|g\|_{L^2_2}\|\mP_kh\|_{L^2}\|\mP_kf\|_{L^2})+\mathbf{1}_{l\geq0}(2^{(\ga/2) l}2^{(\ga/2+s)k}\\
	&&\,\times(\|g\|_{L^2_2}+2^{-2k}\|g\|_{L^1_2})\|\mP_k h\|_{L^2}\|\mP_kf\|_{H^s}+ (2^{(\ga+2s) l}\|g\|_{L^1_{2-2s}}+2^{\ga l}2^{-k}\|g\|_{L^1_1})\|\mP_kh\|_{L^2}\|\mP_kf\|_{L^2}).
	\eeno
	
	\noindent\underline{\it Step 2: Estimate of $\mathbf{B}$.} By Taylor expansion, we first have
	$|\vphi(2^{-k}v')-\vphi(2^{-k}v)|\leq \sum_{i=1}^2C_i2^{-ik}|\na^i\vphi(2^{-k}v)||v'-v|^i+C_32^{-3k}|v'-v|^3$,
	which implies that
	\beno
	&&|\mB| \ls \sum_{i=1}^22^{-ik}\int_{\R^6\times\S^2}b(\cos\th)\Phi_l^\ga(|v-v_*|) |g_* h \vphi(2^{-k}v')f(v')||\na^i\vphi(2^{-k}v)||v'-v|^i (1-\chi(v-v'))\\
  &&\times1_{|v|<4|v_*|}d\si dv dv_*
	+2^{-3k}\int_{\R^6\times\S^2}b(\cos\th)\Phi_l^\ga(|v-v_*|) |g_* h \vphi(2^{-k}v')f(v')||v'-v|^3(1-\chi(v-v'))\\
  &&\times 1_{|v|<4|v_*|}d\si dv dv_*:=\mB_1+\mB_2+\mB_3.
	\eeno
	Noticing that $|v-v'|>2$ and $|v|<4|v_*|$ imply that  $\sin(\th/2)>\f 2{|v-v_*|}$ and $|v-v_*|\geq 2$, we only need to consider  $l\geq0$. Moreover for $\mB_1$ and $\mB_2$, we have $\<v\>\sim \<v'\>\sim 2^k$.
	
	\underline{\it Estimate of $\mB_{1}$ and $\mathbf{B}_{2}$.} Since $l\geq0$, we have  
	\beno
	&&|\mB_1|\leq 2^{(\ga+s) l}2^{-k}\int_{\<v\>\sim \<v'\>\sim 2^k}b(\cos\th)\mathrm{1}_{\th\geq 2/|v-v_*|}\sin(\th/2)|v-v_*|^{1-s}|g_*hf'|d\si dvdv_*\leq 2^{(\ga+s) l}2^{-k}\\
	&&\times\left(\int_{\<v\>\sim 2^k}b(\cos\th)\mathrm{1}_{\th\geq 2/|v-v_*|}\sin^{1-2s}(\th/2)|v-v_*|^{1-3s}|g_*h^2d\si dvdv_*\right)^{1/2} \bigg(\int_{\<v\>\sim 2^k}b(\cos\th)\\
	&&\times\f1{\cos^{3+\ga}(\th/2)}\sin^{1+2s}(\th/2)|v-v_*|^{1+s}|g_*f^2d\si dvdv_*\bigg)^{1/2}\le 2^{(\ga+s) l}2^{-k}\|g\|_{L^1_{1+s}}\|\mP_kh\|_{L^2}\|\mP_kf\|_{L^2}.
	\eeno Similarly, we have
$|\mB_2|\leq   2^{(\ga+2s) l}2^{-2k}\|g\|_{L^1_{2}}\|\mP_kh\|_{L^2}\|\mP_kf\|_{L^2}.$

	\underline{\it Estimate of $\mB_{3}$.}
	  We split it into three parts:  $\mB_{3,1},\mB_{3,2}$ and $\mB_{3,3}$ which correspond to three cases: $|v|\sim |v'|\sim2^{k}$, $|v|\ll2^k$ and $|v|\gg2^k$ respectively. Similar to $\mB_1$, we have
	$|\mB_{3,1}|\leq 2^{(\ga+s) l}2^{-k}\|g\|_{L^1_{1+s}}\|\mP_kh\|_{L^2}\|\mP_kf\|_{L^2}.
   $
	For $\mB_{3,2}$ and $\mB_{3,3}$, we observe that $|v'|^2\le |v|^2+|v_*|^2\ls|v_*|^2$, which implies that $|v_*|\gs2^k$. We have
	\beno
	&&|\mB_{3,2}+\mB_{3,3}|\leq 2^{-3k}2^{(\ga+2s) l}\sum_{a>k}\int_{\R^6\times \S^2}b(\cos\th)\sin^3(\th/2)|v-v_*|^{3-2s}|(\cP_ag)_*(\U_ah)\vphi(2^{-k}v')f'|d\si dvdv_*\\
	&&\leq2^{(\ga+2s) l}2^{-3k}\sum_{a> k}\left(\int_{\R^6}b(\cos\th)\sin^{6+3+\ga-1-2s}(\th/2)|v-v_*|^{6-4s}(\cP_ag)_*^2|\U_ah| dv dv_*\right)^{1/2}\bigg(\int_{\R^6\times\S^2}b(\cos\th)\\
	&&\times\sin^{1+2s}(\th/2)|\U_ah|(\vphi(2^{-k}v_*)f(v_*))^2 dv dv_*\bigg)^{1/2}\ls 2^{(\ga+2s) l}2^{-3k}\sum_{a> k}\|\cP_ag\|_{L^2_{3-2s}}\|\U_ah\|_{L^1}\|\mP_kf\|_{L^2},
	\eeno
where we use the singular change of variables(Lemma \ref{chv}).   
Then we conclude that
	$|\mB|\leq 2^{(\ga+s) l}2^{-k}\|g\|_{L^1_{1+s}}\\\times\|\mP_kh\|_{L^2}\|\mP_kf\|_{L^2}+2^{(\ga+2s) l}2^{-2k}\|g\|_{L^1_{2}}\|\mP_kh\|_{L^2}\|\mP_kf\|_{L^2} + 2^{(\ga+2s) l}2^{-3k}\sum_{a> k}\|\cP_ag\|_{L^2_{3-2s}}\|\U_ah\|_{L^1}\|\mP_kf\|_{L^2}$.
 
	Patching together the estimates of $\mA$ and $\mB$,  we complete the proof of this lemma.
\end{proof}

Our next lemma focuses on the estimates of commutator between $\F_j$ and $Q(\cdot,\cP_k\cdot)$.
\begin{lem}\label{FjQ}
	Let $\mathbf{\Gamma}=(\F_jQ(g,\cP_kh)-Q(g,\F_j\cP_kh),\F_j\cP_k f)$. Then for any $N\in\N$ and $\de\ll 1$, we have 
	\ben\label{pre}&&
 \notag	|\mathbf{\Ga}|\leq \notag C_N2^{(s+(s-1/2)^++\de)j}2^{-k}\|g\|_{L^2_3}\|\mF_j\mP_kh\|_{L^2}\|\mF_j\mP_kf\|_{L^2}+\sum_{p>j+3N_0}C_N2^{-(p-j)}2^{(s+(s-1/2)^++\de)j}2^{-k}\|g\|_{L^2_3}\\
\notag &&\times\|\mF_p\mP_kh\|_{L^2} \|\mF_j\mP_kf\|_{L^2}+C_N2^{(s+(s-1/2)^++\de)j}\|\mF_j\mP_k g\|_{L^2}\|\mP_kh\|_{L^2} \|\mF_j\mP_kf\|_{L^2}+C_N2^{(\ga+2s-1/2)k}2^{sj}\\
 &&\times\|g\|_{L^1_2}\|\mF_j\mP_kh\|_{L^2}\|\mF_j\mP_kf|\|_{L^2}+C_N2^{-kN}2^{-jN}(\|g\|_{L^1_2}+\|g\|_{L^2_3})\|h\|_{L^2}\|f\|_{L^2}.
	\een
 In particular, it yields that
	\ben\label{rou}
	\notag|\mathbf{\Ga}|&\leq& C_N2^{(\ga+2s)k}2^{2sj}(\|g\|_{L^1_2}+\|g\|_{L^2_3})\|\mP_kh\|_{L^2}\|\mF_j\mP_kf|\|_{L^2}+C_N2^{-kN}2^{-jN}(\|g\|_{L^1_2}+\|g\|_{L^2_3})\|h\|_{L^2}\|f\|_{L^2}.\\
	\een
\end{lem}
\begin{proof}
	We split $\mathbf{\Gamma}$ into two parts: $\mathbf{\Ga}_1$ and $\mathbf{\Ga}_2$, where 
$\mathbf{\Ga}_1:=(\F_jQ_{-1}(g,\cP_kh)-Q_{-1}(g,\F_j\cP_kh),\F_j\cP_k f)$ and  
$\mathbf{\Ga}_2:=\sum_{l=0}^\infty (\F_jQ_{l}(g,\cP_kh)-Q_{l}(g,\F_j\cP_kh),\F_j\cP_k f)$.
 By Lemma \ref{F_jQ}\eqref{Q_1} with $\om_i=-1/2,i=1,2,3,4$, $\om_5=\om_6=0$, $c_2=c_3=0$, we have  
\beno
&&|\mathbf{\Ga}_1| \ls  2^{(s+(s-1/2)^++\de)j}2^{-k} \|g\|_{L^2_3}\|\mF_j\mP_kh\|_{L^2}\|\mF_j\mP_kf\|_{L^2}+\sum_{p>j+3N_0} 2^{-(p-j)}2^{(s+(s-1/2)^++\de)j}2^{-k}\|g\|_{L^2_3}\\
&&\times\|\mF_p\mP_kh\|_{L^2} \|\mF_j\mP_kf\|_{L^2}+ 2^{(s+(s-1/2)^++\de)j}\|\mF_j\mP_k g\|_{L^2}\|\mP_kh\|_{L^2} \|\mF_j\mP_kf\|_{L^2}+ 2^{-kN}2^{-jN}\|g\|_{L^2_3}\|h\|_{L^2}\|f\|_{L^2},
\eeno
where we also use the facts that $(2s-1/2)^+\mathbf{1}_{2s\neq1}++(1/2+\de)\mathbf{1}_{2s=1}\leq s+(s-1/2)^++\de$ with $\de\ll1$ and  
$\|\mF_j\cP_kf\|_{H^m_p}\sim 2^{mj}2^{pk}\|\mF_j\mP_kf\|_{L^2}+C_N2^{-jN}2^{-kN}\|f\|_{H^{-N}_{-N}}$ which can be derived from Lemma \ref{le1.2}.
 
Due to Lemma \ref{F_jQ}\eqref{Q_2} with  $a+b\leq s,\om_1+\om_2\leq \ga+2s-1/2$, we have  
	\beno
	|\mathbf{\Ga}_2|&\ls& C_N\|g\|_{L^1_2}(\|\mF_j\cP_kh\|_{H^a_{\om_1}}+2^{-jN}\|\mF_j\cP_kh\|_{H^{-N}_{-N}})(\|\mF_j\cP_kf|\|_{H^b_{\om_2}}+2^{-jN}\|\cP_kf\|_{H^{-N}_{-N}})\\
	&\ls&C_N2^{(\ga+2s-1/2)k}2^{sj}\|g\|_{L^1_2}\|\mF_j\mP_kh\|_{L^2}\|\mF_j\mP_kf|\|_{L^2}+C_N2^{-kN}2^{-jN}\|g\|_{L^1_2}\|h\|_{L^2}\|f\|_{L^2}.
	\eeno
 Combining the estimates of $\mathbf{\Ga}_1$ and $\mathbf{\Ga}_2$, we can get the estiamte \eqref{pre} and then \eqref{rou} follows.
\end{proof}

Finally we arrive at the estimates of commutator between $\cP_k$ and $Q$.
\begin{lem}\label{PkQ2}
Let $\mathbf{\D}:=(\cP_kQ(g,h)-Q(g,\cP_kh),\F_j^2\cP_kf)$. If $\ga+2s<0$, then for any $N\in\N$,
	\ben\label{prees}
	\notag	|\mathbf{\D}|&\leq& C_N\Big( (2^{(\ga/2+s)k}2^{sj}+2^{-k})(\|g\|_{L^1_2}+\|g\|_{L^2_3})\|\mP_kh\|_{L^2}\|\mF_j\mP_kf\|_{L^2}+2^{-3k}\sum_{a>k}\|\cP_ag\|_{L^2_{3-2s}}\\
	&&\times \|\U_ah\|_{L^2_2}\|\mP_k\mF_jf\|_{L^2}+2^{-kN}2^{-jN}(\|g\|_{L^1_2}+\|g\|_{L^2_3})\|h\|_{L^2}\|f\|_{L^2}\Big).
	\een
For $-1<\ga+2s<2$, we have more precise estimate as follows:
\[|\mathbf{\D}|\ls C_N\big(2^k2^{sj}(\|g\|_{L^1_2}+\|g\|_{L^2_3})\|\mP_k\mF_j h\|_{L^2}\|\mP_k\mF_jf\|_{L^2}+2^{sj}\|\mP_k\mF_jg\|_{L^2}\| h\|_{L^2_2}\|\mP_k\mF_jf\|_{L^2}+2^{sj}\sum_{p>j}\|h\|_{L^2_2}\]
 \[\times\|\mP_k\mF_pg\|_{L^2}\|\mP_k\mF_jf\|_{L^2}+2^{sj}\|\U_{N_0}\mF_jg\|_{L^2}\|h\|_{L^2}\|\mP_k\mF_jf\|_{L^2} +2^{sj}\sum_{p>j}\|\U_{N_0}\mF_pg\|_{L^2}\|\U_{N_0}\mF_ph\|_{L^2}\|\mP_k\mF_jf\|_{L^2}\]
\ben\label{moes}+2^{sj}\sum_{a>k}\|\mP_ag\|_{L^2_3}\|\U_a\mF_jh\|_{L^2}\|\mP_k\mF_jf\|_{L^2}+2^{sj}\|g\|_{L^2_2}\|\U_{N_0}\mF_jh\|_{L^2}\|\mP_k\mF_jf\|_{L^2}+2^{-jN}\|g\|_{L^2_2}\|h\|_{L^2}\|f\|_{L^2}\big).
\een
\end{lem}
\begin{proof} We first prove \eqref{prees}. Due to Lemma  \ref{le1.2}, we have 
	$\mathbf{\D}=(\cP_kQ(g,h)-Q(g,\cP_kh),\mP_k\mF_jf)+(\cP_kQ(g,h)-Q(g,\cP_kh),r_N(v,D_v)f):=\mathbf{S}_1+\mathbf{S}_2.$
	We only need to focus on the estimate of $\mathbf{S}_1$   because  Lemma \ref{upperQ} and Lemma \ref{le1.2} imply that
	$|\mathbf{S}_2|\leq C_N 2^{-jN}2^{-kN}(\|g\|_{L^1_2}+\|g\|_{L^2_1})\|h\|_{L^2}\|f\|_{L^2}.$
 For  $\mathbf{S}_1$, we split it into two parts: $\mathbf{S}_{1,1}:=\sum_{l=0}^\infty(\cP_kQ_l(g,h)-Q_l(g,\cP_kh),\mP_k\mF_jf)$ and $
    \mathbf{S}_{1,2}:=(\cP_kQ_{-1}(g,h)-Q_{-1}(g,\cP_kh),\mP_k\mF_jf).$
 \smallskip

\noindent\underline{\it Estimate of $\mathbf{S}_{1,1}$.} By \eqref{ubdecom}, one has $\mathbf{S}_{1,1}=\sum_{i=1}^3\mathbf{S}_{1,1,i}$ where $\mathbf{S}_{1,1,1}=\sum_{l\geq N_0-1,l\sim k}(\cP_kQ_l(\U_{l}g,\tP_lh)-Q_l(\U_{l}g,\cP_k\tP_lh),\mP_k\tP_l\mF_jf)$, $\mathbf{S}_{1,1,2}=\sum_{m\geq l+N_0,m\sim k}(\cP_kQ_l(\cP_mg,\tP_mh)-Q_l(\cP_mg,\cP_k\tP_m h),\mP_k\cP_m\mF_jf)$ and $\mathbf{S}_{1,1,3}=\sum_{|m-l|\leq N_0,l>k-N_0}(\cP_kQ_l(\cP_mg,\U_{l+N_0}h)-Q_l(\cP_mg,\cP_k\U_{l+N_0}h)$.
	 
Using the condition that $l\sim k$, by Lemma \ref{lemma31}($l\geq0$), we derive  that
$ |\mathbf{S}_{1,1,1}| \leq2^{(\ga+2s)k}(\|g\|_{L^1_2}+\|g\|_{L^2_2})\|\mP_kh\|_{L^2}\|\mP_k\mF_jf\|_{H^s}+2^{(\ga+2s-3)k}\sum_{a>k}\|\cP_ag\|_{L^2_{3-2s}}\|\U_ah\|_{L^2_2}\|\mP_k\mF_jf\|_{L^2}.$
 Since $l\ls m\sim k$, Lemma \ref{lemma31}($l\geq0$) yields that
 $|\mathbf{S}_{1,1,2}|\le 2^{(\ga/2+s)k}(\|g\|_{L^1_2}+\|g\|_{L^2_3})\|\mP_kh\|_{L^2}\|\mP_k\mF_jf\|_{H^s}+2^{-3k}\sum_{a>k}\|\cP_ag\|_{L^2_{3-2s}}\|\U_ah\|_{L^2_2}\\\times\|\mP_k\mF_jf\|_{L^2}.$
 Since $k\ls m\sim l$,
 again by Lemma \ref{lemma31}($l\geq0$), we have 
\beno
&&|\mathbf{S}_{1,1,3}|\ls\sum_{|m-l|\leq N_0,l>k-N_0}\Big(2^{(\ga/2)l}2^{(\ga/2+s)k}(\|\mP_mg\|_{L^2_2}+2^{-2k}\|\mP_mg\|_{L^1_2})\|\mP_kh\|_{L^2}\|\mP_k\mF_jf\|_{H^s}+2^{(\ga+2s)l}\eeno\beno
&& \times\|\mP_mg\|_{L^1_{2-2s}}\|\mP_kh\|_{L^2}\|\mP_k\mF_jf\|_{L^2}+2^{(\ga+s)l-k}\|\mP_mg\|_{L^1_{1+s}}\|\mP_kh\|_{L^2}\|\mP_k\mF_jf\|_{L^2}+2^{(\ga+2s)l-2k}\|\mP_mg\|_{L^1_2}\\
&&\times\|\mP_kh\|_{L^2}\|\mP_k\mF_jf\|_{L^2}+2^{(\ga+2s)l}2^{-3k}\sum_{a>k}\|\mP_m\cP_ag\|_{L^2_{3-2s}}\|\U_a\mP_kh\|_{L^2_2}\|\mP_k\mF_jf\|_{L^2}.
\eeno
Using the fact that $\|\mP_mg\|_{L^1_{2-2s}}+\|\mP_mg\|_{L^2_2}\ls 2^{-sm}(\|\mP_mg\|_{L^1_{2}}+\|\mP_mg\|_{L^2_3})$, we further derive that
\[|\mathbf{S}_{1,1,3}|\ls 2^{(\ga/2+s)k}(\|g\|_{L^1_2}+\|g\|_{L^2_3})\|\mP_kh\|_{L^2}\|\mP_k\mF_jf\|_{H^s}+2^{-k}(\|g\|_{L^1_2}+\|g\|_{L^2_3})\|\mP_kh\|_{L^2}\|\mP_k\mF_jf\|_{L^2} +2^{-3k}\]\[\times\sum_{a>k}\|\cP_ag\|_{L^2_{3-2s}}\|\U_ah\|_{L^2_2}\\ \times\|\mP_k\mF_jf\|_{L^2}.\]
 We  conclude that
	$|\mathbf{S}_{1,1}|\ls 2^{(\ga/2+s)k}(\|g\|_{L^1_2}+\|g\|_{L^2_3})\|\mP_kh\|_{L^2}\|\mP_k\mF_jf\|_{H^s}+2^{-k}(\|g\|_{L^1_2}+\|g\|_{L^2_3})\|\mP_kh\|_{L^2}\|\mP_k\mF_jf\|_{L^2}\\
    +2^{-3k}\sum_{a>k}\|\cP_ag\|_{L^2_{3-2s}}\|\U_ah\|_{L^2_2}\|\mP_k\mF_jf\|_{L^2}.$

	\noindent\underline{\it Estimate of $\mathbf{S}_{1,2}$.}	
	Thanks to \eqref{ubdecom}, we have
	$\mathbf{S}_{1,2}=\sum_{a\geq  N_0-1,a\sim k}(\cP_kQ_{-1}(\cP_ag,\tP_ah)-(Q_{-1}(\cP_ag,\cP_k\tP_ah),\\\mP_k\tP_a\mF_jf) 
     +\sum_{b\leq N_0}(\cP_kQ_{-1}(\cP_bg,\U_{N_0-1}h)-(Q_{-1}(\cP_bg,\cP_k\U_{N_0-1}h),\mP_k\U_{N_0-1}\mF_jf).$
	Again by Lemma \ref{lemma31}($l=-1$),  the same argument used for $\mathbf{S}_{1,1}$ can be applied to have
 $|\mathbf{S}_{1,2}|\ls 2^{(\ga/2+s-1)k}\|\mP_kg\|_{L^2_2}\|\mP_k h\|_{L^2}\|\mP_k\mF_jf\|_{H^s}+2^{-3k}\|\mP_kg\|_{L^2_3}\|\mP_kh\|_{L^2}\|\mP_k\mF_jf\|_{L^2}.
   $
Patching together the estimates of $\mathbf{S}_{1,1}$ and $\mathbf{S}_{1,2}$, we get that
\beno
|\mathbf{S}_1|&\ls&2^{(\ga/2+s)k}(\|g\|_{L^1_2}+\|g\|_{L^2_3})\|\mP_kh\|_{L^2}\|\mP_k\mF_jf\|_{H^s}+2^{-3k}\sum_{a>k}\|\cP_ag\|_{L^2_{3-2s}}\|\U_ah\|_{L^2_2}\|\mP_k\mF_jf\|_{L^2}\\
&&+2^{-k}(\|g\|_{L^1_2}+\|g\|_{L^2_3})\|\mP_kh\|_{L^2}\|\mP_k\mF_jf\|_{L^2}.
\eeno
Then   \eqref{prees} follows the estimates of $\mathbf{S}_1$ and $\mathbf{S}_2$.

Next we prove \eqref{moes}. Thanks to Subsection \ref{DDP}, we have $\mathbf{\D}=\mathbf{\D}_1+\mathbf{\D}_2$, where
	\beno
	&&\mathbf{\D}_1:=\sum_{l=-1}^\infty \Big(\sum_{a<p-N_0,p\sim j}(\cP_kQ_l(\F_pg,\F_ah)-Q_l(\F_pg,\cP_k\F_ah),\tF_p\F_j^2\cP_kf)+\sum_{a\geq-1,a\sim j}(\cP_kQ_l(S_{a-N_0}g,\F_ah)\\
  &&-Q_l(S_{a-N_0}g,\cP_k\F_ah),\tF_a\F_j^2\cP_kf)+\sum_{p\geq-1,p\sim j}(\cP_kQ_l(\F_pg,\F_ph)-Q_l(\F_pg,\cP_k\F_ph),\tF_p\F_j^2\cP_kf)\\
	&&+\sum_{b<p-N_0,b\sim j}(\cP_kQ_l(\F_pg,\tF_ph)-Q_l(\F_pg,\cP_k\tF_ph),\F_b\F_j^2\cP_kf)\Big) :=\mathbf{\D}_{1,1}+\mathbf{\D}_{1,2}+\mathbf{\D}_{1,3}+\mathbf{\D}_{1,4};\\
	&&\mathbf{\D}_2:=\sum_{l=-1}^\infty \Big(\sum_{a<p-N_0}(Q_l(\F_pg,\F_ah),[\tF_p,\cP_k] \F^2_j\cP_kf)-(Q_l(\F_pg,[\F_a,\cP_k]h), \tF_p\F^2_j\cP_kf)\Big)+\sum_{l=-1}^\infty\Big(\sum_{a\geq-1}\\
	&&(Q_l(S_{a-N_0}g,\F_ah),[\tF_a,\cP_k] \F^2_j\cP_kf)-(Q_l(S_{a-N_0}g,[\F_a,\cP_k]h), \tF_a\F^2_j\cP_kf)\Big)+\sum_{l=-1}^\infty\Big(\sum_{p\geq-1}(Q_l(\F_pg,\F_ph),\\
	&&[\tF_p,\cP_k] \F^2_j\cP_kf)-(Q_l(\F_pg,[\F_p,\cP_k]h), \tF_p\F^2_j\cP_kf)\Big)+\sum_{l=-1}^\infty\Big(\sum_{b<p-N_0}(Q_l(\F_pg,\tF_ph),[\F_b,\cP_k] \F^2_j\cP_kf)\\
	&&-(Q_l(\F_pg,[\tF_p,\cP_k]h), \F_b\F^2_j\cP_kf)\Big):=\mathbf{\D}_{2,1}+\mathbf{\D}_{2,2}+\mathbf{\D}_{2,3}+\mathbf{\D}_{2,4}.
	\eeno

  In what follows, the superscript of the notation  $\mathbf{\D}^{\geq0}$ is referred to the summation of $l$ from zero to infinity while the notation $\mathbf{\D}^{-1}$ is used to focus  on the case that $l=-1$. 

	 \noindent$\bullet$\underline{\it\, Estimate of $\mathbf{\D}_{1,1}$ and $\mathbf{\D}_{1,4}$.} We only give a detailed proof for $\mathbf{\D}_{1,1}$ since they enjoy the same structure. We first notice that the same argument used for $\mathbf{S}_{1,2}$ can be applied to have
 $|\mathbf{\D}^{-1}_{1,1}|\ls C_N\big(2^{2k}\|\mP_k\mF_jg\|_{L^2}\|\mP_k h\|_{L^2}\\\times\|\mP_k\mF_jf\|_{H^s}+2^{-jN}\|g\|_{L^2_2}\|h\|_{L^2}\|f\|_{L^2}\big).$
And Lemma \ref{lemma1.7} implies that
	$|\mathbf{\D}^{\geq0}_{1,1}|\ls C_N(2^{-jN}\|\mF_jg\|_{L^2_2}\|h\|_{L^2}\|\mF_jf\|_{L^2}\\+2^{-jN}\|g\|_{L^2_2}\|h\|_{L^2}\|f\|_{L^2}\big).$ Then we are led to that
	$|\mathbf{\D}_{1,1}|\ls C_N\big(2^{2k}\|\mP_k\mF_jg\|_{L^2}\|\mP_k h\|_{L^2}\|\mP_k\mF_jf\|_{H^s}+2^{-jN}\\\times\|g\|_{L^2_2}\|h\|_{L^2}\|f\|_{L^2}\big).$
	Similarly,  
	$|\mathbf{\D}_{1,4}|\ls C_N\big(\sum\limits_{p>j}2^{2k}\|\mP_k\mF_pg\|_{L^2}\|\mP_kh\|_{L^2}\|\mP_k\mF_jf\|_{H^s}+2^{-jN}\|g\|_{L^2_2}\|h\|_{L^2}\|f\|_{L^2}\big).$

	\noindent$\bullet$\underline{\it\, Estimate of $\mathbf{\D}_{1,2}$ and $\mathbf{\D}_{1,3}$.} We have the similar decomposition: $\mathbf{\D}_{1,2}=\mathbf{\D}_{1,2}^{-1}+\mathbf{\D}^{\geq0}_{1,2}$. Since $a\sim j$,   by Lemma \ref{lemma31}(l=-1), we have 
$|\mathbf{\D}_{1,2}^{-1}|\ls  C_N\big(2^{2k}\|\mP_kg\|_{L^2}\|\mP_k\mF_j h\|_{L^2}\|\mP_k\mF_jf\|_{H^s}+2^{-jN}\|g\|_{L^2_2}\|h\|_{L^2}\|f\|_{L^2}\big).$
 The same argument used for   $\mathbf{S}_{1,1}$ can be copied to have
 $|\mathbf{\D}_{1,2}^{\geq0}|\ls C_N\big((\|g\|_{L^1_2}+\|g\|_{L^2_3})\|\mP_k \mF_jh\|_{L^2}\|\mP_k\mF_jf\|_{H^s}+\sum_{a>k}\|\cP_ag\|_{L^2_{3-2s}}\|\U_a\mF_jh\|_{L^2_2}\|\mP_k \mF_jf\|_{L^2} +2^{-jN}\|g\|_{L^2_2}\|h\|_{L^2}\|f\|_{L^2}\big).$
Then we have 
$|\mathbf{\D}_{1,2}|\ls C_N\big((\|g\|_{L^1_2}+\|g\|_{L^2_3})\|\mP_k\mF_j h\|_{L^2}\|\mP_k\mF_jf\|_{H^s}+\sum_{a>k}\|\cP_ag\|_{L^2_{3-2s}}\|\U_a\mF_jh\|_{L^2_2}\|\mP_k \mF_jf\|_{L^2} 
+2^{-jN}\|g\|_{L^2_2}\|h\|_{L^2}\|f\|_{L^2}\big). $

 Similarly,  
$|\mathbf{\D}_{1,3}|\ls C_N\big((\|g\|_{L^1_2}+\|g\|_{L^2_3})\|\mP_k\mF_j h\|_{L^2}\|\mP_k\mF_jf\|_{H^s}+\sum_{a>k}\|\cP_ag\|_{L^2_{3-2s}}\|\U_a\mF_jh\|_{L^2_2}\|\mP_k \mF_jf\|_{L^2}+2^{-jN}\|g\|_{L^2_2}\|h\|_{L^2}\|f\|_{L^2}\big). $
 We arrive at that
	\beno
	|\mathbf{\D}_1|&\ls& C_N\big((\|g\|_{L^1_2}+\|g\|_{L^2_3})\|\mP_k\mF_j h\|_{L^2}\|\mP_k\mF_jf\|_{H^s}+\|\mP_k\mF_jg\|_{L^2}\| h\|_{L^2_2}\|\mP_k\mF_jf\|_{H^s}+\sum_{p>j}\|\mP_k\mF_pg\|_{L^2}\\
	&&\times\|h\|_{L^2_2}\|\mP_k\mF_jf\|_{L^2}+\sum_{a>k}\|\cP_ag\|_{L^2_{3-2s}}\|\U_a\mF_jh\|_{L^2}\|\mP_k \mF_jf\|_{L^2}+2^{-jN}\|g\|_{L^2_2}\|h\|_{L^2}\|f\|_{L^2}\big).
	\eeno

\noindent$\bullet$\underline{\it Estimate of $\mathbf{\D}_{2,1}$ and $\mathbf{\D}_{2,4}$.} Let
 $\mathbf{\D}_{2,1}:=\mathbf{\D}^{-1}_{2,1}+\mathbf{\D}_{2,1}^{\geq0}.$
  By Lemma \ref{lemma1.7}, we have 
	 $|\mathbf{\D}^{\geq0}_{2,1}|\ls C_N\big(2^{-jN}\|g\|_{L^2_2}\|h\|_{L^2}\\\times\|\mP_k\mF_jf\|_{L^2}+2^{-jN}\|g\|_{L^2_2}\|h\|_{L^2}\|f\|_{L^2}\big).$
 For $\mathbf{\D}^{-1}_{2,1}$,   \eqref{ubdecom} implies that
	\beno
	&&\mathbf{\D}_{2,1}^{-1}=\sum_{a<p-N_0}\Big(\sum_{q\geq N_0-1}(Q_{-1}(\cP_q\tF_pg,\tP_q\tF_ah),\tP_q[\tF_p,\cP_k] \F^2_j\cP_kf) +\sum_{q<N_0+1}(Q_{-1}(\cP_q\tF_pg,\U_{N_0}\tF_ah),\\
  &&\U_{N_0}[\tF_p,\cP_k] \F^2_j\cP_kf)\Big)-\sum_{a<p-N_0}\Big(\sum_{q\geq N_0-1}(Q_{-1}(\cP_q\tF_pg,\tP_q[\tF_a,\cP_k]h),\tP_q\tF_p\F^2_j\cP_kf)\\
	&&+\sum_{q<N_0+1}(Q_{-1}(\cP_q\tF_pg,\U_{N_0}[\tF_a,\cP_k]h),\U_{N_0}\tF_p\F^2_j\cP_kf)\Big):= \mathbf{G}_1+\mathbf{G}_2+\mathbf{G}_3+\mathbf{G}_4.
	\eeno

For $\mathbf{G}_1$, observe that 
$[\tF_p,\cP_k]=2^{-k-p}\mP_k\mF_p+r_N(v,D),$
where $r_N\in S^{-N}_{1,0}$, then  $\mathbf{G}_1:=\mathbf{G}_{1,1}+\mathbf{G}_{1,2}$ where 
 $\mathbf{G}_{1,1}:= 2^{-k-p}\sum\limits_{a<p-N_0} \sum\limits_{q\geq N_0-1} |(Q_{-1}(\cP_q\tF_pg,\tP_q\tF_ah),\tP_q\mP_k\mF_p \F^2_j\cP_kf)|$ and $\mathbf{G}_{1,2}:= 2^{-k-p}\sum\limits_{a<p-N_0} \sum\limits_{q\geq N_0-1} \\|(Q_{-1}(\cP_q\tF_pg,\tP_q\tF_ah),\tP_qr_N(v,D) \F^2_j\cP_kf)|$.
By the facts that $q\sim k$ and $p\sim j$ as well as Lemma \ref{lemma1.7}(iv), we can get that
$ \mathbf{G}_{1,1}\ls 2^{-k}2^{(2s-1)j}\|\mP_k\mF_jg\|_{L^2_2}\|\mP_kh\|_{L^2}\|\mP_k\mF_jf\|_{L^2}+C_N 2^{-jN}\|g\|_{L^2_2}\|h\|_{L^2}\|f\|_{L^2}.$ 
For $\mathbf{G}_{1,2}$, by  Lemma \ref{le1.2}, we have
$\mathbf{G}_{1,2}\ls C_N2^{-jN}\|g\|_{L^2_2}\|h\|_{L^2}\|f\|_{L^2}$, which implies that
\beno
\mathbf{G}_1\ls C_N\big(2^{-k}2^{(2s-1)j}\|\mP_k\mF_jg\|_{L^2_2}\|\mP_kh\|_{L^2}\|\mP_k\mF_jf\|_{L^2}+ 2^{-jN}\|g\|_{L^2_2}\|h\|_{L^2}\|f\|_{L^2}\big).
\eeno

For $\mathbf{G}_2$,  we have $\mathbf{G}_2:=\mathbf{G}_{2,1}+\mathbf{G}_{2,2}$ where $\mathbf{G}_{2,1}:=2^{-k-p}\sum\limits_{a<p-N_0} \sum\limits_{q<N_0+1} |(Q_{-1}(\cP_q\tF_pg,\U_{N_0}\tF_ah),\U_{N_0}\mP_k\mF_p \F^2_j\\ \cP_kf)|$ and $\mathbf{G}_{2,1}:=2^{-k-p}\sum\limits_{a<p-N_0} \sum\limits_{q<N_0+1}|(Q_{-1}(\cP_q\tF_pg,\U_{N_0}\tF_ah),\U_{N_0}r_N(v,D) \F^2_j\cP_kf)|$
For $\mathbf{G}_{2,1}$, since $k\gs N_0$ and $p\sim j$,  one has
$\mathbf{G}_{2,1}\ls C_N\big( 2^{-k}2^{(2s-1)j}\|\U_{N_0}\mF_jg\|_{L^2_2}\|\U_{N_0}h\|_{L^2}\|\mP_k\mF_jf\|_{L^2} +2^{-jN}\|g\|_{L^2_2}\|h\|_{L^2}\|f\|_{L^2}\big).$
 For $\mathbf{G}_{2,2}$, similar to $\mathbf{G}_{2,1}$, we can derive that
$\mathbf{G}_{2,2}\ls C_N2^{-jN}\|g\|_{L^2_2}\|h\|_{L^2}\|f\|_{L^2}.$
It yields that
\[\mathbf{G}_2\ls C_N\big( 2^{-k}2^{(2s-1)j}\|\U_{N_0}\mF_jg\|_{L^2_2}\|\U_{N_0}h\|_{L^2}\|\mP_k\mF_jf\|_{L^2}+2^{-jN}\|g\|_{L^2_2}\|h\|_{L^2}\|f\|_{L^2}\big).\]
 Following the similar argument, we may have 
\[\mathbf{G}_3+\mathbf{G}_4\ls C_N\big( 2^{-k}2^{sj}(\|\mP_k\mF_jg\|_{L^2_2}\|\mP_kh\|_{L^2}+\|\U_{N_0}\mF_jg\|_{L^2_2}\|\U_{N_0}h\|_{L^2}) \|\mP_k\mF_jf\|_{L^2}+2^{-jN}\|g\|_{L^2_2}\|h\|_{L^2}\|f\|_{L^2}\big).\]

Combining the estimates of $\mathbf{G}_i(i=1,2,3,4)$, we are led to that
$ 
|\mathbf{\D}_{2,1}|\ls C_N\big( 2^{-k}2^{sj}\|\mP_k\mF_jg\|_{L^2_2}\|\mP_kh\|_{L^2}\\\times\|\mP_k\mF_jf\|_{L^2}+2^{-k}2^{sj}\|\U_{N_0}\mF_jg\|_{L^2_2}\|\U_{N_0}h\|_{L^2}\|\mP_k\mF_jf\|_{L^2} 
+2^{-jN}\|g\|_{L^2_2}\|h\|_{L^2}\|f\|_{L^2}\big).$
Similarly, we have
	$|\mathbf{\D}_{2,4}|\ls C_N\big( 2^{-k}2^{sj}\sum_{p>j}\|\mP_k\mF_pg\|_{L^2_2}\|\mP_k\mF_ph\|_{L^2}\|\mP_k\mF_jf\|_{L^2}+2^{-k}2^{sj}\sum_{p>j}\|\U_{N_0}\mF_pg\|_{L^2_2}\|\U_{N_0}\mF_ph\|_{L^2} 
   \|\mP_k\mF_jf\|_{L^2}+2^{-jN}\\\times\|g\|_{L^2_2}\|h\|_{L^2}\|f\|_{L^2}\big).$

\noindent$\bullet$\underline{\it Estimate of $\mathbf{\D}_{2,2}$ and $\mathbf{\D}_{2,3}$.}	
	For $\mathbf{\D}_{2,2}$, we also split it into two parts:
 $\mathbf{\D}_{2,2}:=\mathbf{\D}_{2,2}^{-1}+\mathbf{\D}_{2,2}^{\geq0}.$
Similar to the estimates of $\mathbf{\D}_{2,1}^{-1}$, we can derive that
$|\mathbf{\D}_{2,2}^{-1}|\ls C_N\big( 2^{-k}2^{sj}\|\mP_kg\|_{L^2_2}\|\mP_k\mF_jh\|_{L^2}\|\mP_k\mF_jf\|_{L^2}+2^{-k}2^{sj}\|\U_{N_0}g\|_{L^2_2}\\\times\|\U_{N_0}\mF_jh\|_{L^2} \|\mP_k\mF_jf\|_{L^2} +2^{-jN}\|g\|_{L^2_2}\|h\|_{L^2}\|f\|_{L^2}\big).$
For  $\mathbf{\D}_{2,2}^{\geq0}$, by the decomposition in phase space, we have 
\beno
&\mathbf{\D}_{2,2}^{\geq0}=\sum\limits_{a\geq-1}\Big(\sum\limits_{l\geq N_0}(Q_l(\U_{l-N_0}S_{a-N_0}g,\tP_l\tF_a h),\tP_l[\tF_a,\cP_k] \F^2_j\cP_kf) +\sum\limits_{q\geq l+N_0}(Q_l(\cP_qS_{a-N_0}g,\tP_q\tF_a h),\tP_q[\tF_a,\cP_k]  \\
&\F^2_j\cP_kf)+\sum_{|q-l|\leq N_0}(Q_l(\cP_qS_{a-N_0}g,\U_{l+N_0}\tF_a h),\U_{l+N_0}[\tF_a,\cP_k] \F^2_j\cP_kf)\Big)+\sum_{a\geq-1}\Big(\sum_{l\geq N_0}(Q_l(\U_{l-N_0}S_{a-N_0}g,\eeno\beno
&\tP_l[\F_a,\cP_k]h),\tP_l\tF_a\F^2_j\cP_kf) +\sum_{q\geq l+N_0}(Q_l(\cP_qS_{a-N_0}g,\tP_q[\F_a,\cP_k]h),\tP_q\tF_a\F^2_j\cP_kf) +\sum_{|q-l|\leq N_0}(Q_l(\cP_qS_{a-N_0}g,\\ &\U_{l+N_0}[\F_a,\cP_k]h),\U_{l+N_0}\tF_a\F^2_j\cP_kf)\Big)
\eeno
Following Lemma \ref{lemma1.7}(iv)(the case $k\geq0$) and the same argument used for $\mathbf{\D}_{2,1}^{-1}$, we can derive that
	\beno
	|\mathbf{\D}_{2,2}^{\geq0}|&\ls& C_N\big(2^{((\ga+2s)^+-1)k}  2^{sj}\|g\|_{L^2_2}\|\mP_k\mF_jh\|_{L^2}\|\mP_k\mF_jf\|_{L^2}+2^{sj}\sum_{a>k}\|\mP_ag\|_{L^2_2}\|\U_a\mF_jh\|_{L^2}\|\mP_k\mF_jf\|_{L^2}\\
	&&+2^{-jN}\|g\|_{L^2_2}\|h\|_{L^2}\|f\|_{L^2}\big),
	\eeno
which yields that
\beno
|\mathbf{\D}_{2,2}|&\ls& C_N\big(2^{((\ga+2s)^+-1)k}  2^{sj}\|g\|_{L^2_2}\|\mP_k\mF_jh\|_{L^2}\|\mP_k\mF_jf\|_{L^2}+2^{sj}\sum_{a>k}\|\mP_ag\|_{L^2_2}\|\U_a\mF_jh\|_{L^2}\|\mP_k\mF_jf\|_{L^2}\\
&&+2^{-k}2^{sj}\|\U_{N_0}g\|_{L^2_2}\|\U_{N_0}\mF_jh\|_{L^2}\|\mP_k\mF_jf\|_{L^2}+2^{-jN}\|g\|_{L^2_2}\|h\|_{L^2}\|f\|_{L^2}\big).
\eeno
	Similarly we may derive the same bound for $|\mathbf{\D}_{2,3}|$.
We finally conclude that
\beno
|\mathbf{\D}_2|&\ls& C_N\big(2^{-k}2^{sj}\|\mP_k\mF_jg\|_{L^2_2}\|\mP_kh\|_{L^2}\|\mP_k\mF_jf\|_{L^2}+2^{-k}2^{sj}\|\U_{N_0}\mF_jg\|_{L^2_2}\|\U_{N_0}h\|_{L^2}\|\mP_k\mF_jf\|_{L^2}\\
&&+2^{-k}2^{sj}\sum_{p>j}\|\mP_k\mF_pg\|_{L^2_2}\|\mP_k\mF_ph\|_{L^2}\|\mP_k\mF_jf\|_{L^2}+2^{-k}2^{sj}\sum_{p>j}\|\U_{N_0}\mF_pg\|_{L^2_2}\|\U_{N_0}\mF_ph\|_{L^2}\|\mP_k\mF_jf\|_{L^2}\\
&&+ 2^{((\ga+2s)^+-1)k}  2^{sj}\|g\|_{L^2_2}\|\mP_k\mF_jh\|_{L^2}\|\mP_k\mF_jf\|_{L^2}+2^{sj}\sum_{a>k}\|\mP_ag\|_{L^2_2}\|\U_a\mF_jh\|_{L^2}\|\mP_k\mF_jf\|_{L^2}\\
&&+2^{-k}2^{sj}\|\U_{N_0}g\|_{L^2_2}\|\U_{N_0}\mF_jh\|_{L^2}\|\mP_k\mF_jf\|_{L^2}+2^{-jN}\|g\|_{L^2_2}\|h\|_{L^2}\|f\|_{L^2}\big),
\eeno
from which together with the 
  estimate for $\mathbf{\D}_{1}$, we get \eqref{moes} and then complete the proof of this lemma.
\end{proof}

\smallskip
We now give the proof of finite smoothing effect in Sobolev spaces in Theorem \ref{RSWBol1}.
\begin{proof}[Proof of  Theorem \ref{RSWBol1}]  We only need to prove result $(ii)$ in the theorem by assuming that    $f\in L^\infty([0,T],L^\infty_xL^2_\ell)$ with $\ell\ge3$. We also recall that $\gamma+2s<0$ which is important to the main strategy.

By localizing \eqref{Boltzmann} in both phase and frequency spaces, we can derive that
\beno
\pa_t \F_j\cP_k\mathcal{F}_x(f)(t,m,v)+iv\cdot m\F_j\cP_k\mathcal{F}_x(f)(t,m,v)=i[\F_j,v]m\cP_k\mathcal{F}_x(f)(t,m,v)+\mathcal{F}_x(\F_j\cP_kQ(f,f))(t,m,v),
\eeno
where $m=(m_1,m_2,m_3)\in\Z^3$. 
Recall that $m\in \mathbf{M}_j:=\{m_1=\pm[2^{(1-\vep)j}],m_2=m_3=0\}$ where $j\in\N$ and $0<\vep\ll1$.
The standard energy method implies that
\ben\label{JJ2}
\notag&&\f d{dt}\sum_{m\in \mathbf{M}_j}\|\F_j\cP_k\mathcal{F}_x(f)(t,m,v)\|^2_{L^2_v}=2\sum_{m\in\mathbf{M}_j}\int_{\R^3_v}i[\F_j,v]m\cP_k\mathcal{F}_x(f)(t,m,v)\F_j\cP_k\mathcal{F}_x(f)(t,-m,v)dv\\
&&+2\sum_{m\in \mathbf{M}_j}\sum_{p\in\Z^3}(\F_j\cP_kQ(\mathcal{F}_x(f)(p),\mathcal{F}_x(f)(p-m)),\F_j\cP_k\mathcal{F}_x(f)(m))_{L^2_v}:=2(\mathbf{H}_1+\mathbf{H}_2).
\een
 We further split $\mathbf{H}_2$ into  three parts: $\mH_{2,1}$,  $\mH_{2,2}$ and $\mH_{2,3}$ where  $\mH_{2,1}:=(Q(\mathcal{F}_x(f)(p),\F_j\cP_k\mathcal{F}_x(f)(p-m)),\\ \F_j\cP_k\mathcal{F}_x(f)(m))_{L^2_v}$, $\mH_{2,2}:=(\F_jQ(\mathcal{F}_x(f)(p),\cP_k\mathcal{F}_x(f)(p-m))-Q(\mathcal{F}_x(f)(p),\F_j\cP_k\mathcal{F}_x(f)(p-m)),\F_j\cP_k\\ \mathcal{F}_x(f)(m))_{L^2_v}$ and $\mH_{2,3}:=(\cP_kQ(\mathcal{F}_x(f)(p),\mathcal{F}_x(f)(p-m))-Q(\mathcal{F}_x(f)(p),\cP_k\mathcal{F}_x(f)(p-m)),\F_j^2\cP_k\mathcal{F}_x(f)(m))_{L^2_v}$.

 Thanks to Lemma \ref{le1.2},  we have 
\ben
\notag|\mathbf{H}_1|&\leq& \sum_{m\in \mathbf{M}_j}|m|2^{-j}\big(\|\mF_j\cP_k\mathcal{F}_x(f)(t,m)\|^2_{L^2_v}+C_N\|r_N(\cdot,D)\cP_k\mathcal{F}_x(f)(t,m)\|^2_{L^2_v}\big)\\
&\ls& \sum_{m\in \mathbf{M}_j}2^{-\vep j}\|\mF_j\mP_k\mathcal{F}_x(f)(t,m)\|^2_{L^2_v}+C_N2^{-kN}2^{-jN}\|f\|^2_{L^2_{x,v}}.
\een

We observe that $\mH_{2,1}$ will induce the coercivity estimate for collision operator $Q$. While $\mH_{2,2}$ and $\mH_{2,3}$  focus on the commutators between localized operators   and  $Q$.
Thanks to the upper bound of collsion operator(Lemma \ref{upperQ}), $|\mH_{2,1}|$ can be bounded by 
$ 2^{(\ga+2s)k}2^{2sj}(\|\mathcal{F}_x(f)(p)\|_{L^1_2}+\|\mathcal{F}_x(f)(p)\|_{L^2_1})\|\F_j\cP_k\mathcal{F}_x(f)(p-m)\|_{L^2}\|\F_j\cP_k\mathcal{F}_x(f)(m)\|_{L^2}.$ 
For  $\mH_{2,2}$ and $\mH_{2,3}$, we can use Lemma \ref{FjQ}\eqref{rou} and Lemma \ref{PkQ2}\eqref{prees} respectively to have
$ |\mH_{2,2}| \leq  C_N2^{(\ga+2s)k}2^{2sj}\|\mathcal{F}_x(f)(p)\|_{L^2_3}\|\mP_k\mathcal{F}_x(f)(p-m)\|_{L^2}\|\mF_j\mP_k\mathcal{F}_x(f)(m)\|_{L^2} +C_N2^{-kN}2^{-jN}(\|\mathcal{F}_x(f)(p)\|_{L^1_2}\\+\|\mathcal{F}_x(f)(p)\|_{L^2_3})\|\mathcal{F}_x(f)(p-m)\|_{L^2}\|\mathcal{F}_x(f)(m)\|_{L^2},$ and
\beno &&|\mH_{2,3}|\leq C_N\Big( (2^{(\ga/2+s)k}2^{sj}+2^{-k})(\|\mathcal{F}_x(f)(p)\|_{L^1_2}+\|\mathcal{F}_x(f)(p)\|_{L^2_3})\|\mP_k\mathcal{F}_x(f)(p-m)\|_{L^2}\\
&&\times \|\mF_j\mP_k\mathcal{F}_x(f)(m)\|_{L^2}+2^{-3k}\sum_{a>k}\|\cP_a\mathcal{F}_x(f)(p)\|_{L^2_{3-2s}} \|\U_a\mathcal{F}_x(f)(p-m)\|_{L^2_2}\|\mP_k\mF_j\mathcal{F}_x(f)(m)\|_{L^2}\\
&&+2^{-kN}2^{-jN}(\|\mathcal{F}_x(f)(p)\|_{L^1_2}+\|\mathcal{F}_x(f)(p)\|_{L^2_3})\|\mathcal{F}_x(f)(p-m)\|_{L^2}\|\mathcal{F}_x(f)(m)\|_{L^2}\Big).
\eeno

 We conclude that
\beno
&&|\mH_{2}| \leq C_N(2^{(\ga+2s)k}2^{2sj}+2^{(\ga/2+s)k}2^{sj}+2^{-k})(\|f\|_{L^\infty_xL^1_2}+\|f\|_{L^\infty_xL^2_3})\|\mP_kf\|_{L^2_{x,v}}\|\mF_j\mP_kf\|_{L^2_{x,v}}\\
&&+C_N2^{-kN}2^{-jN}(\|f\|_{L^\infty_xL^1_2}+\|f\|_{L^\infty_xL^2_3})\|f\|^2_{L^2_{x,v}}+2^{-3k}\sum_{a>k}\|f\|_{L^\infty_xL^2_2}\|\cP_af\|_{L^2_xL^2_{3-2s}}\|\mF_j\mP_kf\|_{L^2_{x,v}}.
\eeno
Since $\sup\limits_t(\|f\|_{L^\infty_xL^1_2}+\|f\|_{L^\infty_xL^2_3}+\|f\|_{L^2_{x,v}})$ is bounded by the assumption that $f\in L^\infty([0,T],L^\infty_xL^2_\ell)$ with $\ell\ge3$,  the estimates of $\mH_{1},\mH_2$ and \eqref{JJ2} yield that
\beno
&&\f d{dt}\sum_{m\in\mathbf{M}_j}\|\F_j\cP_k\mathcal{F}_x(f)(t,m,v)\|^2_{L^2_v}\geq -(2^{(\ga+2s)k}2^{2sj}+(2^{(\ga/2+s)k}2^{sj}+2^{-k})\|\mP_kf\|_{L^2_{x,v}}\|\mF_j\mP_kf\|_{L^2_{x,v}}\\
&&-2^{-3k}\sum_{a>k}\|\cP_af\|_{L^2_xL^2_{3-2s}}\|\mF_j\mP_kf\|_{L^2_{x,v}}-\sum_{m\in \mathbf{M}_j}2^{-\vep j}\|\mF_j\mP_k\mathcal{F}_x(f)(t,m)\|^2_{L^2_v}-2^{-kN}2^{-jN},
\eeno
which implies that
\beno
&&\f d{dt}\sum_{l=k}^K \sum_{m\in \mathbf{M}_j}\|\F_j\cP_l\mathcal{F}_x(f)(t,m,v)\|^2_{L^2_v}\gs -\sum_{l=k}^\infty (2^{(\ga+2s)l}2^{2sj}+(2^{(\ga/2+s)l}2^{sj}+2^{-l})\|\mP_lf\|_{L^2_{x,v}}\|\mF_j\mP_lf\|_{L^2_{x,v}}\\
&&-\sum_{l=k}^\infty 2^{-3l}\sum_{a>l}\|\cP_af\|_{L^2_xL^2_{3-2s}}\|\mF_j\mP_lf\|_{L^2_{x,v}}
-\sum_{l=k}^\infty \sum_{m\in \mathbf{M}_j}2^{-\vep j}\|\mF_j\mP_l\mathcal{F}_x(f)(t,m)\|^2_{L^2_v}-2^{-jN}.
\eeno

For $0<\vep\ll1$, we set  $\a:=2s/|\ga+2s|$, $\de:=\f{4\vep}{|\ga+2s|}$  and $\b:=2\vep$  which implies that $\b=|\ga+2s|\de/2$. If $(\a+\de) j$ is an integer, then $k=k_j:=(\a+\de) j$. Otherwise, $k=k_j:=[(\a+\de) j]+1$. Now we define 
\ben\label{exf}
E_{f,M_2,K}(t):=\sum_{j=1}^{M_2} \sum_{l=k_j}^K 2^{(2\ell(\a+\de)+\vep) j}\sum_{m\in\mathbf{M}_j}\|\F_j\cP_l\mathcal{F}_x(f)(t,m,\cdot)\|^2_{L^2_v}.
\een
Here $M_2\gg 1$. From the above inequality, we may derive that 
\beno
&&\f d{dt}E_{f,M_2,K}(t)\gs -\sum_{j=1}^\infty\sum_{l=k_j}^\infty 2^{(2\ell (\a+\de) -\b/2)j}\|\mP_lf\|_{L^2_{x,v}}\|\mF_j\mP_lf\|_{L^2_{x,v}}\\
&&-\sum_{j=1}^\infty\sum_{l=k_j}^\infty2^{-3l }2^{(2\ell (\a+\de)+\vep) j}\sum_{a>l}\|\cP_af\|_{L^2_xL^2_{3-2s}}\|\mF_j\mP_lf\|_{L^2_{x,v}}
-\sum_{j=1}^\infty\sum_{l=k_j}^\infty 2^{2\ell(\a+\de) j} \|\mF_j\mP_lf\|^2_{L^2_{x,v}}-1,
\eeno
where we use the facts that
$(\ga+2s)l+2sj\leq (\ga+2s)(\a+\de) j+2sj\leq -|\ga+2s|\de j=-2\b j$ if   $l\geq k_j$, and for sufficiently large $N$, $\sum\limits_{j=1}^\infty 2^{(2\ell\a+\vep-N)j}\ls 1$.
 
 To bound the right hand side, since $\ell\geq3$, by Lemma \ref{lemma1.4},  we first notice that 
\beno
&&\sum_{j=1}^\infty\sum_{l= k_j}^\infty2^{-3l }2^{2\ell (\a+\de)j}\sum_{a>l}\|\cP_af\|_{L^2_xL^2_{3-2s}}\|\mF_j\mP_lf\|_{L^2_{x,v}}\leq \sum_{j=1}^\infty\sum_{l=k_j}^\infty \sum_{a>l}\|\cP_af\|_{L^2_xL^2_{\ell}}\\
&&\times2^{\ell l}\|\mF_j\mP_lf\|_{L^2_{x,v}}2^{(-\ell+3-2s)(a-l)}2^{(-2\ell-2s)l}2^{2\ell (\a+\de)j}
\ls \|f\|^2_{L^2_xL^2_{\ell}}.
\eeno  Similarly, we have  
\beno
&&\sum_{j=1}^\infty\sum_{l=k_j}^\infty 2^{(2\ell (\a+\de)-\b/2) j}\|\mP_lf\|_{L^2_{x,v}}\|\mF_j\mP_lf\|_{L^2_{x,v}}+\sum_{j=1}^\infty\sum_{l=k_j}^\infty 2^{2\ell(\a+\de) j} \|\mF_j\mP_lf\|^2_{L^2_{x,v}}
\ls \|f\|^2_{L^2_xL^2_{\ell}},
\eeno
which implies that $  \f d{dt}E_{f,M_2,K}(t) \gs -\|f\|^2_{L^2_xL^2_\ell}-1$. Recall that   $f\in L^\infty([0,T],L^\infty_xL^2_\ell)$, then we deduce that for $t\in (0,T]$, $ E_{f,M_2,K}(t)\ge E_{f,M_2,K}(0)-C(1+t).$
  It implies that
 if $f_0\in \mathcal{R}_{sd}(\ell;\vep,\a+\de)$, then
\beno
 E_{f}(t):=\sum_{j=1}^{\infty} \sum_{l>(\a+\de) j}2^{(2\ell(\a+\de)+\vep) j}\sum_{m\in\mathbf{M}_j}\|\F_j\cP_l\mathcal{F}_x(f)(t,m,\cdot)\|^2_{L^2_v}=\lim_{M_2,K\rightarrow \infty}E_{f,M_2,K}(t)\geq  E_{f}(0)-C(1+t)=+\infty.
\eeno

To complete the proof, we only need to prove \eqref{FSE}. Due to the fact that  $m\in \mathbf{M}_j:=\{m_1=\pm[2^{(1-\vep)j}],m_2=m_3=0\}$ and $\|f(t)\|_{H^{\mathbf{n}}_xH^{\mathbf{m}}_v}^2\sim \sum_{m\in\Z^3}\sum_{j\ge-1}|m|^{2\mathbf{n}}2^{2j\mathbf{m}}\|\mathfrak{F}_j\mathcal{F}_x(f)(t,m,\cdot)\|_{L^2}^2$, we have
\[ \|f(t)\|_{H^{\mathbf{n}}_xH^{\mathbf{m}}_v}^2\gs \sum_{j=1}^{\infty} \sum_{l>a j} 2^{2(\mathbf{n}(1-\vep)+\mathbf{m})j}\sum_{m\in\mathbf{M}_j}\|\F_j\cP_l\mathcal{F}_x(f)(m,\cdot)\|^2_{L^2_v}=+\infty, \]
since $\mathbf{n}(1-\vep)+\mathbf{m}> \ell \a +4\ell\vep/|\ga+2s|+\vep/2=\ell (\a+\de)+\vep/2$.
\end{proof}

We are   ready to give the proof of Corollary \ref{globaldecay}.
\begin{proof}[Proof of  Corollary \ref{globaldecay}]
	If the global solution $f$ satisfies that $f_0\in \mathcal{H}^2_{sd}(\ell;\vep,\a+\de)$, then 
	by the same argument in the above, we can derive that
	\beno
	\sum_{j=1}^{\infty} \sum_{l=k_j}^\infty 2^{(2\ell(\a+\de)+4(1-\vep)+\vep) j}\sum_{m\in\mathbf{M}_j}\|\F_j\cP_l\mathcal{F}_x(f)(t,m,v)\|^2_{L^2_v}=+\infty,\quad \forall t\in\R^+.
	\eeno
Recall that $\mathbf{M}_j:=\{m_1=\pm[2^{(1-\vep)j}],m_2=m_3=0\}$, then it yields that
\beno
\|f(t)\|_{H^\mathbf{n}_xH^{\mathbf{m}}_v}=+\infty~~\mbox{if}~~2\mathbf{n}(1-\vep)+2\mathbf{m}>2\ell (\a+\de)+4(1-\vep)+\vep.
\eeno
In particular, if $\mathbf{n}=0$, then $\mathbf{m}>\ell(\a+\de)+2(1-\vep)+\vep/2$; and if $\mathbf{m}=0$, then $\mathbf{n}>(\ell(\a+\de)+\vep/2)/(1-\vep)+2$. As a byproduct, we see that $\|\na_xf(t)\|_{L^2}\neq0$ and $f(t)\neq \mu$ for any $t\in\R^+$. 
\end{proof}

\section{Proof of Theorem \ref{RSWBol2} and Proposition \ref{coro1}: local properties}
In this section, we shall prove some local properties for the solution itself or for the collision operator. The first part of Theorem \ref{RSWBol2} is exact the results of the following lemma:
\begin{lem}\label{x5.1}
	Let $f(t,x,v)$ be a weak solution to \eqref{Boltzmann} with $f_0\in \mathcal{R}_{sd}(\ell;\vep,\a+\de)$ satisfying $f\in L^\infty([0,T],L^\infty_xL^2_\ell)$, where $\a$ and $\de$ are defined in Theorem \ref{RSWBol1}. Then for any $\bar{t}\in(0,T)$, there exists $\mathfrak{N}_1(\bar{t}),\mathfrak{N}_2(\bar{t}) \in \Z_+^3$ such that
	\beno
	&&|\mathfrak{N}_1(\bar{t})|\in [1,2+\big(\a(\ell-(\ga+2s))+\f{4\vep}{|\ga+2s|}(\ell-(\ga+2s))+\f\vep 2(1-(\ga+2s)/\ell)+2\big)/(1-\vep)];\\
	&&|\mathfrak{N}_2(\bar{t})|\in [1,4+\a (\ell-(\ga+2s))+(\f{4\vep}{|\ga+2s|}+\f\vep{2\ell})(\ell-(\ga+2s))-2\vep].
	\eeno
Moreover, it holds that 
	\ben
&&\label{GL}\|\pa^{\mathfrak{N}_1}_xf(\bar{t})\|_{L^1_xL^1_{\ga+2s}}=+\infty,\quad \|\pa^{\mathfrak{N}}_xf(\bar{t})\|_{L^1_xL^1_{\ga+2s}}<\infty,\quad\forall~ |\mathfrak{N}|<|\mathfrak{N}_1|;\\
&&\label{5.2}\|\pa^{\mathfrak{N}_2}_vf(\bar{t})\|_{L^1_xL^1_{\ga+2s}}=+\infty,\quad \|\pa^{\mathfrak{N}}_vf(\bar{t})\|_{L^1_xL^1_{\ga+2s}}<\infty,\quad\forall~ |\mathfrak{N}|<|\mathfrak{N}_2|.
	\een
\end{lem}
\begin{proof}
	We begin with the proof of \eqref{GL}. Let $\mathbf{n}\in\N$ verify $\mathbf{n}>2+\big(\a(\ell-(\ga+2s))+\f{4\vep}{|\ga+2s|}(\ell-(\ga+2s))+\f\vep 2(1-(\ga+2s)/\ell)+2\big)/(1-\vep)$. It is easy to check that this condition is equivalent to  $\f{(\mathbf{n}-2)\ell}{\ell-(\ga+2s)}(1-\vep)-\f{2\ell}{\ell-(\ga+2s)}>\ell\a+\f{4\ell \vep}{|\ga+2s|}+\f\vep2$. By Theorem \ref{RSWBol1}, we have  $\|f(\bar{t})\|_{H^{(\mathbf{n}-2)\ell/(\ell-(\ga+2s))}_xH^{-2\ell/(\ell-(\ga+2s))}_v}=+\infty.$ From  this together with the facts that \[\|f\|_{H^{(\mathbf{n}-2)\ell/(\ell-(\ga+2s))}_xH^{-2\ell/(\ell-(\ga+2s))}_v}\ls \|f\|^{-(\ga+2s)/(\ell-(\ga+2s))}_{L^2_xL^2_{\ell}}\|f\|^{\ell/(\ell-(\ga+2s))}_{H^{\mathbf{n}-2}_xH^{-2}_{\ga+2s}}\] and
	 $\|\<D_x\>^{\mathbf{n}} f \|_{L^1_xL^1_{\ga+2s}}\gs C_{\mathbf{n}} \|f\|_{H^{\mathbf{n}-2}_xH^{-2}_{\ga+2s}}$,   we derive that
 $	\|\<D_x\>^\mathbf{n} f(\bar{t}) \|_{L^1_xL^1_{\ga+2s}}\gs\|f(\bar{t})\|_{H^{\mathbf{n}-2}_xH^{-2}_{\ga+2s}}=+\infty$. In particular, it implies that there exists $\mathfrak{N}_1(\bar{t})\in\Z_+^3$ such that $\|\pa^{\mathfrak{N}_1}_xf(\bar{t})\|_{L^1_xL^1_{\ga+2s}}=+\infty$. The desired result follows the inductive argument and the fact that  $\|f(\bar{t})\|_{L^1_xL^1_{\ga+2s}}<+\infty$. 
 
 For the proof of  \eqref{5.2}, notice that $\mathbf{m}>4+\a (\ell-(\ga+2s))+(\f{4\vep}{|\ga+2s|}+\f\vep{2\ell})(\ell-(\ga+2s))-2\vep$ is equivalent to $\f{-2\ell}{\ell-(\ga+2s)}(1-\vep)+\f{(\mathbf{m}-2)\ell}{\ell-(\ga+2s)}>\ell \a +\f{4\ell \vep}{|\ga+2s|}+\f\vep2$. Again by Theorem \ref{RSWBol1}, we have
 \beno
 \|f\|_{H_x^{-2\ell/(\ell-(\ga+2s))}H^{\ell(\mathbf{m}-2)/(\ell-(\ga+2s))}_v}=+\infty.
 \eeno
 Using the   inequalities $\|f\|_{H_x^{-2\ell/(\ell-(\ga+2s))}H^{\ell(\mathbf{m}-2)/(\ell-(\ga+2s))}_v}\ls \|f\|_{L^2_xL^2_\ell}^{-(\ga+2s)/(\ell-(\ga+2s))}\|f\|^{\ell/(\ell-(\ga+2s))}_{H^{-2}_xH^{\mathbf{m}-2}_v}$ and\\ $\|\<D_v\>^{\mathbf{m}}f\|_{L^1_xL^1_{\ga+2s}}\gs \|f\|_{H^{-2}_xH^{\mathbf{m}-2}_v}$, we obtain that $\|\<D_v\>^{\mathbf{m}}f(\bar{t})\|_{L^1_xL^1_{\ga+2s}}\gs \|f(\bar{t})\|_{H^{-2}_xH^{\mathbf{m}-2}_v}=+\infty$. Then \eqref{5.2} follows by the same argument as \eqref{GL}. This ends the proof of this lemma.
\end{proof}
%Thanks to Lemma \ref{x5.1} and the function space of generalized Lebesgue integral(see subsection 1.2.2). We have that the following integrals have and only have one convergences:
%\beno
%\int_{\R^6}(\pa^\al f(t,x,v))^+dvdx,\quad\int_{\R^6}(\pa^\al f(t,x,v))^-dvdx.
%\eeno

Now we are ready to prove the second part of Theorem \ref{RSWBol2}.
\begin{proof}[Proof of Theorem \ref{RSWBol2}(Failure of Leibniz rule for collision operator)] %We recall that the condition \\$f\in L^\infty([0,T],L^\infty_xL^2_\ell)$   implies that   Lemma \ref{x5.1} holds. To complete the proof, we divide it into several steps.
	
%\noindent\underline{\it Step 1: Proof of $(i)$.}  
 By Lemma \ref{x5.1}, we assume that weak derivative $\pa^{\mathfrak{N}_1}_xf(\bar{t})$ with $|\mathfrak{N}_1|\ge 5$ exists and satisfies $\eqref{GL}$. Then it immediately implies that for any $R>0$,  
\ben\label{prswcon1} \int_{\T^3}\int_{|v_*|\geq5R}\<v_*\>^{\ga+2s}|\pa^{\mathfrak{N}_1}_xf(\bar{t},x,v_*)|dv_*dx=+\infty.\een 
Moreover, because of \eqref{bartlowbrho}, for sufficiently large $R$, we have \ben\label{prswcon2} \inf\limits_{x\in\T^3}\int_{|v|<R}f(\bar{t},x,v)|\chi(v)|dv>c/2,\een  
where $\chi\in C^\infty_c(\R^3)$ satisfies $0\le \chi\le 1$, $\mathrm{supp}\chi\subset\{v||v|<R\}$ and $\chi(v)=1$ if $|v|<R/2$.

In what follows, we shall make full use of \eqref{prswcon1} and \eqref{prswcon2} to prove \eqref{FLebnizx}. Let $\varrho\in C^\infty(\T^3)$ satisfy $\varrho>\bar{c}>0$. Using \eqref{GL}, we may have 
  \ben
 ( Q(f,f)(\bar{t},\cdot,\cdot),(-1)^{|\mathfrak{N}_1|}\pa^{\mathfrak{N}_1}_x\varrho\chi)_{L^2_{x,v}}
    =\lim_{n\rightarrow \infty} \int_{\T^3\times\R^6\times\S^2} \Big(\sum_{\al+\be=\mathfrak{N}_1} \mathcal{Q}^{(n),*}_{\mathfrak{N}_1,\al,\be}(\bar{t},x,v,v_*,\sigma)\Big) d\si dv_*dvdx \label{Qcsta},\een
   where \ben \mathcal{Q}^{(n),*}_{\mathfrak{N}_1,\al,\be}(\bar{t}):= 
\psi\big((|v|^2+|v_*|^2)/n\big) B(|v-v_*|,\si)\pa^\al_xf(\bar{t},x,v_*)\pa^{\be}_xf(\bar{t},x,v)(\chi(v')-\chi(v))\varrho(x). \label{QnN1f}\een

Now the result that \eqref{FLebnizx} does not hold is reduced to show that \[\sum_{\al+\be=\mathfrak{N}_1} \mathcal{Q}^{(\infty),*}_{\mathfrak{N}_1,\al,\be}(\bar{t},x,v,v_*,\sigma):=\sum_{\al+\be=\mathfrak{N}_1} B(|v-v_*|,\si)\pa^\al_xf(\bar{t},x,v_*)\pa^{\be}_xf(\bar{t},x,v)(\chi(v')-\chi(v))\varrho(x)\] is not in $L^1(\T^3\times\R^6\times\S^2)$. It suffices to prove $\sum_{\al+\be=\mathfrak{N}_1} \mathcal{Q}^{(\infty),*}_{\mathfrak{N}_1,\al,\be}(\bar{t},x,v,v_*,\sigma)$ is not in $L^1(\T^3\times \mathfrak{S}^1_R)$ where
\ben\label{SR1}
\mathfrak{S}^1_R:=\Big\{(\sigma,v,v_*)\in \S^2\times \R^6|\sin(\th/2)>2(5/4)^{1/2s}R\Big(\f{|v-v_*|^\ga}{\<v_*\>^{\ga+2s}}\Big)^{1/2s},|v|<R,|v_*|>5 R\Big\}.
\een 
We first observe that in the domain $\mathfrak{S}^1_R$, it holds that
$|v-v'|=|v-v_*|\sin(\th/2)\ge (|v-v_*|^{\ga+2s}\<v_*\>^{|\ga+2s|})^{1/2s}\\ \times 2(\f54)^{1/2s} R\geq2R$ which implies that $|v'|>R$ and then  $\chi(v')=0$. From these, we deduce that 
\beno
&&\int_{\T^3\times \mathfrak{S}^1_R} \big|\sum_{\al+\be=\mathfrak{N}_1} \mathcal{Q}^{(\infty),*}_{\mathfrak{N}_1,\al,\be}(\bar{t},x,v,v_*,\sigma)\big|d\si dv_*dvdx\\&&=\int_{\T^3}\int_{\mathfrak{S}^1_R}\big|\sum_{\al+\be=\mathfrak{N}_1} B(|v-v_*|,\si)\pa^\al_xf(\bar{t},x,v_*)\pa^{\be}_xf(\bar{t},x,v)\chi(v)\varrho(x)\big|d\si dv_*dvdx\ge \mathfrak{Q}_1-\mathfrak{Q}_2,
 \eeno 
where $\mathfrak{Q}_1:=\int_{\T^3}\int_{\mathfrak{S}^1_R} b(\cos\th)|v-v_*|^\ga|\pa^{\mathfrak{N}_1}_xf(\bar{t},x,v_*)| f(\bar{t},x,v)
\chi(v)\varrho(x)d\si dv_*dvdx$ and \\ $\mathfrak{Q}_2:=\int_{\T^3}\int_{\mathfrak{S}^1_R}\big|\sum_{\al+\be=\mathfrak{N}_1,\al\neq \mathfrak{N}_1} B(|v-v_*|,\si)\pa^\al_xf(\bar{t},x,v_*)\pa^{\be}_xf(\bar{t},x,v)\chi(v)\varrho(x)\big|d\si dv_*dvdx$.

\underline{Estimate of $\mathfrak{Q}_1$.} By \eqref{prswcon1} and \eqref{prswcon2}, it is not difficult to check that
\beno && \mathfrak{Q}_1\gs c\bar{c} \int_{\T^3}\int_{|v_*|\geq5R}\<v_*\>^{\ga+2s}|\pa^{\mathfrak{N}_1}_xf(\bar{t},x,v_*)|dv_*dx=+\infty.
\eeno
 
 \underline{Estimate of $\mathfrak{Q}_2$.} By Sobolev embedding theorem, we have $\|fg\|_{L^1}\ls \|f\|_{W^{m,1}}\|g\|_{W^{n,1}}$ if $m+n=3$ and $m,n\ge0$. This in particular implies that 
\beno &&\int_{\T^3}\int_{\mathfrak{S}^1_R} B(|v-v_*|,\si)|\pa^{\al}_xf(\bar{t},x,v_*)\pa^{\be}_xf(\bar{t},x,v)|\chi(v)\varrho(x)d\si dv_*dvdx\ls \int_{\T^3}\int_{|v_*|>5R}\int_{|v|<R} \<v_*\>^{\ga+2s}
\\&&\times|\pa^{\al}_xf|(\bar{t},x,v_*)|\pa^{\be}_xf|(\bar{t},x,v)\chi(v)\varrho(x)dv_*dvdx\ls \|\pa^{\al}_xf\|_{W^{m,1}_xL^1_{\gamma+2s}} \|\pa^{\be}_xf\|_{W^{n,1}_xL^1_v(\T^3\times\{|\cdot|\le R\})}.\eeno
From \eqref{GL} and the fact $|\mathfrak{N}_1|\ge5$, it yields that
\[\mathfrak{Q}_2\ls \| f\|_{W^{|\mathfrak{N}_1|-1,1}_xL^1_{\gamma+2s}}\|f\|_{W^{|\mathfrak{N}_1|-1,1}_xL^1_v(\T^3\times\{|\cdot|\le R\})}+\|f\|_{W^{3,1}_xL^1_{\gamma+2s}}\|\pa^{\mathfrak{N}_1}_xf\|_{L^1_xL^1_v(\T^3\times\{|\cdot|\le R\})}<\infty.\]
We arrive at that $\sum_{\al+\be=\mathfrak{N}_1} \mathcal{Q}^{(\infty),*}_{\mathfrak{N}_1,\al,\be}(\bar{t},x,v,v_*,\sigma)$ is not in $L^1(\T^3\times \mathfrak{S}^1_R)$ and thus is not in $L^1(\T^3\times\R^6\times\S^2)$.
 As a by-product, we deduce that Leibniz rule for the derivative $\pa^{\mathfrak{N}_1}_x$  on $Q$ does not hold, even in weak sense.

Next we turn to prove that \eqref{FLebnizv} does not hold for the spatially homogeneous solution.  Assume that \eqref{5.2} holds and the weak derivative  $\pa^{\mathfrak{N}_2}_vf(\bar{t})$ with $|\mathfrak{N}_2|\ge5$ exists. Let $\tilde{\chi}\in C^\infty_c(\R^3)$ verify   that $0\le\tilde{\chi}\le1$, $\supp \tilde{\chi}\subset \{v\in\R^3|R<|v|<7R\}$ and $\tilde{\chi}(v)\equiv1$ if $2R\leq|v|\le 5R$. Using \eqref{5.2}, we may get that
\beno 
(Q(f,f)(\bar{t},\cdot),(-1)^{|\mathfrak{N}_2|}\pa^{\mathfrak{N}_2}_v\tilde{\chi})_{L^2_{v}}=\lim_{n\rightarrow \infty} \int_{\R^6\times\S^2} \Big(\sum_{\al+\be=\mathfrak{N}_2} \mathcal{P}^{(n),*}_{\mathfrak{N}_2,\al,\be}(\bar{t},v,v_*,\sigma)\Big) d\si dv_*dv, 
\eeno
where 
\beno
 \mathcal{P}^{(n),*}_{\mathfrak{N}_2,\al,\be}(\bar{t},v,v_*,\sigma):=\psi((|v|^2+|v_*|^2)/n)  B(|v-v_*|,\si)\pa^\al_vf(\bar{t},v_*)\pa^{\be}_vf(\bar{t},v)(\tilde{\chi}(v')-\tilde{\chi}(v)). 
\eeno

Again  \eqref{FLebnizv} is reduced to prove that  \[ \sum_{\al+\be=\mathfrak{N}_2} \mathcal{P}^{(\infty),*}_{\mathfrak{N}_2,\al,\be}(\bar{t},v,v_*,\sigma) :=  \sum_{\al+\be=\mathfrak{N}_2} B(|v-v_*|,\si)\pa^\al_vf(\bar{t},v_*)\pa^{\be}_vf(\bar{t},v)(\tilde{\chi}(v')-\tilde{\chi}(v))\] is not in $L^1(\mathfrak{S}^2_R)$, where
\[\mathfrak{S}^2_R=\Big\{(\si,v,v_*)\in \S^2\times \R^3\times\R^3|\f{18}5R\Big(\f{|v-v_*|^\ga}{\<v_*\>^{\ga+2s}}\Big)^{1/2s}<\sin(\th/2)<\f{19}5R\Big(\f{|v-v_*|^\ga}{\<v_*\>^{\ga+2s}}\Big)^{1/2s},|v|<R,|v_*|\geq5 R\Big\}.\]  
We notice that in the domain $\mathfrak{S}^2_R$, it holds that $\tilde{\chi}(v)=0$. Moreover, since $|v-v'|=|v-v_*|\sin(\th/2)\in [3R,4R]$, then $ |v'|\in[2R,5R]$ which yields that $\tilde{\chi}(v')=1$. This implies that
\[\int_{\mathfrak{S}^2_R}|\mathcal{P}^{(\infty),*}_{\mathfrak{N}_2,\al,\be}(\bar{t},v,v_*,\sigma)|d\si dv_*dv=
\int_{\mathfrak{S}^2_R}   B(|v-v_*|,\si)|\pa^\al_vf(\bar{t},v_*)\pa^{\be}_vf(\bar{t},v)|  d\si dv_*dv. \]
Now the similar argument can be applied to get that
\beno &&
\int_{\mathfrak{S}^2_R}\Big| \sum_{\al+\be=\mathfrak{N}_2} \mathcal{P}^{(\infty),*}_{\mathfrak{N}_2,\al,\be}(\bar{t},v,v_*,\sigma)\Big|d\si dv_*dv\gs \Big(\int_{|v_*|\geq5R}\<v_*\>^{\ga+2s}|\pa^{\mathfrak{N}_2}_vf(\bar{t},v_*)|dv_*\Big)\Big(\int_{|v|<R}f(\bar{t},v)|dv\Big)\\&&-
\| f\lr{\cdot}^{\gamma+2s}\|_{W^{|\mathfrak{N}_2|-1,1}}\|f\|_{W^{|\mathfrak{N}_2|-1,1}(\{|\cdot|\le R\})}+\|f\lr{\cdot}^{\gamma+2s}\|_{W^{3,1}}\|\pa^{\mathfrak{N}_2}_vf\|_{L^1_v(\{|\cdot|\le R\})}=+\infty,
 \eeno 
where we use \eqref{5.2}. It ends the proof.
\end{proof}

\smallskip
Next we give the proof of Proposition \ref{coro1}.
\begin{proof}[Proof of Proposition \ref{coro1}]
%\noindent\underline{\it Step 2: Proof of $(ii)$.}
Logically, we only need to prove \eqref{finfty} under the conditions that 
the assertions in {\it Case 1} and {\it Case 2} do not hold. Suppose that $(\pa^{\mathfrak{N}_1}_xf)(\bar{t}),\big(\pa^{\mathfrak{N}_1}_x(\pa_t +v\cdot \na_x)\big) f(\bar{t})\in C(\T^3\times\R^3)$  and $\int_{\R^3}|\pa^{\mathfrak{N}_1}_xf(\bar{t},x,v)\<v\>^{\ga+2s}|dv$ is generalized continuous. We shall prove that there exists $\bar{x}\in \T^3$ such that
\[\int_{\R^3} |(\pa^{\mathfrak{N}_1}_xf)(\bar{t},\bar{x},v)|\<v\>^{\ga+2s}dv=+\infty.\] By the definition of generalized continuity, if $\int_{\R^3}|\pa^{\mathfrak{N}_1}_xf(\bar{t},x,v)\<v\>^{\ga+2s}|dv\equiv+\infty$, there is nothing need to prove. Otherwise, by \eqref{GL}, suppose that for any $x\in\T^3$, $\int_{\R^3} |(\pa^{\mathfrak{N}_1}_xf)(\bar{t},x,v)|\<v\>^{\ga+2s}dv=C_x<+\infty$. By the definition of generalized continuity, there exists an open neighborhood $U_{x}$ such that for all $y\in U_{x}$, it holds that $\int_{\R^3} |(\pa^{\mathfrak{N}_1}_xf)(\bar{t},y,v)|\<v\>^{\ga+2s}dv<2C_x$ which contradicts with   \eqref{GL} since $\T^3$ is compact.

In what follows, we only focus on the assertions in {\it Case 3}. Logically we only need to prove $(3.3)$(in Theorem \ref{RSWBol2}) holds under the assumptions that $(3.1)$ and $(3.2)$(in Theorem \ref{RSWBol2}) do not hold.
Without loss of generality, let us assume that
\[\int_{\R^3}(\pa^{\mathfrak{N}_1}_xf)^+(\bar{t},\bar{x},\cdot)\<v\>^{\ga+2s}dv=+\infty.\] 
Define $\mathfrak{A}:=\{v\in\R^3|\pa^{\mathfrak{N}_1}_xf(\bar{t},\bar{x},v)>0\}$. By continuity, $\mathfrak{A}$  is an open subset of $\R^3$. We introduce
$$
I_\mathfrak{A}(f)(x):=\int_{\mathfrak{A}}\pa^{\mathfrak{N}_1}_xf(\bar{t},x,v)\<v\>^{\ga+2s}dv.
$$
Then  $I_\mathfrak{A}(f)(\bar{x})=+\infty$. Since $(3.1)$(in Theorem \ref{RSWBol2}) does not hold, we can assume that $I_\mathfrak{A}(f)$ is generalized continuous at $\bar{x}$, that is, for any $N>0$, there exists an open set $U_{\bar{x},N}\subset \T^3$ such that for any $x\in U_{\bar{x},N}$, $I_\mathfrak{A}(f)(x)>N$. In addition, since $(3.2)$(in Theorem \ref{RSWBol2}) does not hold, we can also assume that for any $|\al|< |\mathfrak{N}_1|$, the integral $\int_{\mathfrak{A}}\pa^{\al}_xf(\bar{t},x,v_*)\<v_*\>^{\ga+2s}dv_*$ is continuous at point $\bar{x}$.

Let	$\mathcal{A}_n:=\mathfrak{S}^1_R\cap(\S^2\times\{v_*\in\mathfrak{A}||v_*|<n\}\times\R^3 )$
	where $n\in\N$ and $\mathfrak{S}^1_R$ is defined in \eqref{SR1}. Let $\mathcal{A}:=\lim\limits_{n\rightarrow \infty} \mathcal{A}_n$. We shall prove that for the open set $\mathcal{A}$,
 \beno
\mathbf{G}(x):=\pa^{\mathfrak{N}_1}_x\int_{\mathcal{A}} b(\cos \th)|v-v_*|^\ga f_*(\bar{t},x)f(\bar{t},x)(\chi(v')-\chi(v))d\si dv_* dv
\eeno
 is not continuous at $\bar{x}$. We prove it again by contradiction argument. Suppose that $\mathbf{G}(x)$  is  continuous at $\bar{x}$. Let $\varrho_m \in C^\infty(U_{\bar{x},N})$ verify $\lim\limits_{m\rightarrow \infty}\varrho_m =\de(\bar{x})$(the Dirac measure over point $\bar{x}$). Then we define 
\[
\mathfrak{B}_{n,m}:=\int_{\T^3}\int_{\mathcal{A}_n}b(\cos \th)|v-v_*|^\ga f_*(\bar{t},x)f(\bar{t},x)(\chi(v')-\chi(v))d\si dv_*dv\pa^{\mathfrak{N}_1}_x\varrho_m(x)dx.
\] 
Since $f\in L^\infty([0,T],L^\infty_xL^2_\ell)$, by Lebesgue Dominated Convergence theorem,  we have
\beno
\lim_{m\rightarrow \infty}\lim_{n\rightarrow \infty}\mathfrak{B}_{n,m}=(-1)^{\mathfrak{N}_1}G(\bar{x}).
\eeno 
On the other hand, by integration by parts,  one may have
\beno
\mathfrak{B}_{n,m}&=&\int_{\T^3}\int_{\mathcal{A}_n}b(\cos \th)|v-v_*|^\ga \pa^{\mathfrak{N}_1}_xf_*(\bar{t},x)f(\bar{t},x)(\chi(v')-\chi(v))d\si dv_*dv\varrho_m(x)dx+\sum_{\al+\be=\mathfrak{N}_1,|\al|<|\mathfrak{N}_1|}\\
&&\int_{\T^3}\int_{\mathcal{A}_n}b(\cos \th)|v-v_*|^\ga \pa^{\al}_xf_*(\bar{t},x)\pa^{\be}_xf(\bar{t},x)(\chi(v')-\chi(v))d\si dv_*dv\varrho_m(x)dx:=\mathfrak{B}_{n,m}^1+\mathfrak{B}_{n,m}^2.
\eeno

For $\mathfrak{B}_{n,m}^2$, we observe  that in the region $\mathcal{A}_n$, it holds that $|v-v'|=|v-v_*|\sin (\th/2)>2R$ which implies that $|v'|>R$ and then $\chi(v')=0$. These imply that
\beno
\mathfrak{B}_{n,m}^2=\sum_{\al+\be=\mathfrak{N}_1,|\al|<|\mathfrak{N}_1|}\int_{\T^3}\int_{n>|v_*|>5R,v_*\in \mathfrak{A}}\pa^{\al}_xf(\bar{t},x,v_*)\<v_*\>^{\ga+2s}dv_*\int_{|v|<R}\pa^{\be}_xf(\bar{t},x,v)\chi(v)dv\varrho_m(x)dx.
\eeno
Since $\pa^{\mathfrak{N}_1}_xf(\bar{t},x,v)$ and $\int_\mathfrak{A}\pa^{\al}_xf(\bar{t},x,v_*)\<v_*\>^{\ga+2s}dv_*$ with $|\al|< |\mathfrak{N}_1|$ are continuous at $\bar{x}$, we can deduce that
\beno
\lim_{m\rightarrow \infty}\lim_{n\rightarrow \infty}\mathfrak{B}^2_{n,m}=\sum_{\al+\be=\mathfrak{N}_1,|\al|<|\mathfrak{N}_1|}\int_{|v_*|>5R,v_*\in \mathfrak{A}}\pa^{\al}_xf(\bar{t},\bar{x},v_*)\<v_*\>^{\ga+2s}dv_*\int_{|v|<R}\pa^{\be}_xf(\bar{t},\bar{x},v)\chi(v)dv.
\eeno
	
For $\mathfrak{B}^1_{n,m}$, on one hand, $\lim\limits_{m\rightarrow \infty}\lim\limits_{n\rightarrow \infty}\mathfrak{B}^1_{n,m}=\lim\limits_{m\rightarrow \infty}\lim\limits_{n\rightarrow \infty}(\mathfrak{B}_{n,m}-\mathfrak{B}^2_{n,m})<\infty$. On the other hand, we have
\beno
&&\lim_{m\rightarrow \infty}\lim_{n\rightarrow \infty}\mathfrak{B}^1_{n,m}=\lim_{m\rightarrow \infty}\lim_{n\rightarrow \infty}\int_{\T^3}\int_{n>|v_*|>5R,v_*\in \mathfrak{A}}\pa^{\mathfrak{N}_1}_xf(\bar{t},x,v_*)\<v_*\>^{\ga+2s}dv_*\int_{|v|<R}f(\bar{t},x,v)\chi(v)dv\varrho_m(x)dx\\
&&=\lim_{m\rightarrow \infty}\int_{\T^3}\int_{|v_*|>5R,v_*\in\mathfrak{A}}\pa^{\mathfrak{N}_1}_xf(\bar{t},x,v_*)\<v_*\>^{\ga+2s}dv_*\int_{|v|<R}f(\bar{t},x,v)\chi(v)dv\varrho_m(x)dx.
\eeno

Due to the assumptions that for any $x\in\T^3, f(\bar{t},x,v)\not\equiv0$ and $f(\bar{t})\in C(\T^3\times\R^3)$,   we first deduce that  $\int_{|v|<R}f(\bar{t},x,v)\chi(v)dv$  has a positive lower bound around $\bar{x}$  for some large $R$. Secondly, thanks to the generalized continuity property of $I_{\mathfrak{A}}(f)$ at $\bar{x}$, we derive that
\beno
\lim_{m\rightarrow \infty}\lim_{n\rightarrow \infty}\mathfrak{B}^1_{n,m}\gs\lim_{m\rightarrow \infty}\int_{\T^3}\int_{|v_*|>5R,v_*\in \mathfrak{A}}\pa^{\mathfrak{N}_1}_xf(\bar{t},x,v_*)\<v_*\>^{\ga+2s}\varrho_m(x)dv_*dx=+\infty,
\eeno
which leads to the contradiction and thus $\mathbf{G}(x)$ is not continuous at $\bar{x}$. 
\end{proof}

\section{Proof of Theorem \ref{CtvLBE}: infinity smoothing effect}
In this section, we  shall focus on the $C^\infty$-regularity for the linear Boltzmann equation \eqref{LHB} with the slowly decaying initial data. The proof will be separated into two parts.

%\begin{lem}
%	For the linerized Boltzmann equation \eqref{LB}, if intial data $h_0(x,v)\in C^\infty_xL^2_l,$ where $l\geq22$ fixed. Then for any $t\in[0,\infty)$, $h(t)\in C^\infty_xL^2_l$.
%\end{lem}
%\begin{proof}
%	The existence of solution of \eqref{LB} is obvious, we now focues on the propagation of regularity in variable $x$. Noticing that in $\T^3$, $C^\infty_x\sim H^\infty_x$, we only need to prove that $H^\infty_x$ property can be propagated.
	
%	Apply $\<D_x\>^n,n\in\N$ to equation \eqref{LB}, we obtian that
%	\beno
%	\pa_t\<D_x\>^nh(t,x,v)+v\cdot \nabla_x \<D_x\>^nh(t,x,v)=Q(\mu,\<D_x\>^nh)+Q(\<D_x\>^nh,\mu),
%	\eeno
%	Multiply $\<v\>^{2l}\<D_x\>^nf$ on both sides and integrate w.r.t $x$ and $v$, we can derive that 
%	\beno
%	&&\f d{dt}\|h\|^2_{H^n_xL^2_l}= (\<v\>^{l}Q(\mu,\<D_x\>^nh),\<v\>^{l}\<D_x\>^nh)_{L^2_{x,v}}+(Q(\<D_x\>^nh,\mu),\<v\>^{2l}\<D_x\>^nh)_{L^2_{x,v}}\\
%	&=&(Q(\mu,\<v\>^{l}\<D_x\>^nh),\<v\>^{l}\<D_x\>^nh)_{L^2_{x,v}}+(Q(\<D_x\>^nh,\mu),\<v\>^{2l}\<D_x\>^nh)_{L^2_{x,v}}\\
%	&&+(\<v\>^{l}Q(\mu,\<D_x\>^nh)-Q(\mu,\<v\>^l\<D_x\>^nh),\<v\>^{l}\<D_x\>^nh)_{L^2_{x,v}}.
%	\eeno	
%	Futhermore, thanks to Lemma \ref{coer}, Lemma \ref{upperQ}, Lemma \ref{muQ} and Lemma \ref{com}, we can obatin the following energy inequality,
%	\beno
%	\f d{dt}\|h\|^2_{H^n_xL^2_l}+\|h\|^2_{H^n_xH^s_{l+\ga/2}}\leq C\|h\|^2_{H^n_xL^2_{l+\ga/2}}.
%	\eeno
%	which implies that $\|h(t)\|_{H^n_xL^2_l}\leq e^{Ct}\|f_0\|_{H^n_xL^2_l},t\in[0,\infty)$. We complete the proof of this lemma.
%\end{proof}

\begin{proof}[Proof of Theorem \ref{CtvLBE}(Part I)] We will give the proof to the result $(i)$. We first have  
	\beno
&&	\f d{dt}2^{2nj}2^{-2\ell k}\|\F_j\cP_kh\|^2_{L^2_{v}}=2^{2nj}2^{-2\ell k}(\F_j\cP_kQ(\mu,h),\F_j\cP_kh)_{L^2_{x,v}}+2^{2nj}2^{-2\ell k}(\F_j\cP_kQ(h,\mu),\F_j\cP_kh)_{L^2_{x,v}}\\
	&&\qquad+2^{2nj}2^{-2\ell k}([v,\F_j]\cdot\na_x \cP_k h,\F_j\cP_kh)_{L^2_{x,v}}
	:=\mathcal{R}^{j,k}_1+\mathcal{R}^{j,k}_2+\mathcal{R}^{j,k}_3.
	\eeno
We shall give the estimates term by term. For $\mathcal{R}^{j,k}_1$,
we split it into three parts:
\beno
\mathcal{R}^{j,k}_1&=&2^{2nj}2^{-2\ell k}\Big((Q(\mu,\F_j\cP_kh),\F_j\cP_kh)_{L^2_{x,v}}+(\F_jQ(\mu,\cP_kh)-Q(\mu,\F_j\cP_kh),\F_j\cP_kh)_{L^2_{x,v}}\\
&&+(\cP_kQ(\mu,h)-Q(\mu,\cP_kh),\F_j^2\cP_k h)_{L^2_{x,v}}\Big):=\mathcal{R}^{j,k}_{1,1}+\mathcal{R}^{j,k}_{1,2}+\mathcal{R}^{j,k}_{1,3}.
\eeno
 Thanks to Lemma \ref{coer} and Lemma \ref{lemma1.4}, we have  
 $\sum_{j,k=-1}^\infty \mathcal{R}^{j,k}_{1,1}\ls -\|h\|^2_{L^2_xH^{n+s}_{-\ell+\ga/2}}+\|h\|^2_{L^2_xH^{n}_{-\ell+\ga/2}}.$ 
Using Lemma \ref{FjQ}\eqref{pre} and choosing  $N>2n$, we can derive that for any $\vep>0$ and $\de\ll1$,
\[ \sum_{j,k=-1}^\infty |\mathcal{R}^{j,k}_{1,2}|\leq \vep \|h\|^2_{L^2_xH^{n+s}_{\ga/2-\ell}}+C_{\vep,N,n,\ell}(\|h\|^2_{L^2_xH^{n+(s-1/2)^++\de}_{2-\ell}}+\|h\|^2_{L^2_{x,v}}).\]
By Lemma \ref{PkQ2}\eqref{moes}, we get that
\beno
|\mathcal{R}^{j,k}_{1,3}|&\leq& 2^{2nj}2^{-2\ell k}\Big(\|\mP_k\mF_j h\|_{L^2_xL^2_{3}}\|\mP_k\mF_jh\|_{L^2_xH^s_{\ga/2}}+\|\mP_k\mF_j\mu\|_{L^2_2}\| h\|_{L^2_xL^2_2}\|\mP_k\mF_jh\|_{L^2_xH^s_{\ga/2}}+\sum_{p>j}\|\mP_k\mF_p\mu\|_{H^s}\\
&&\times\|h\|_{L^2_xL^2_2}\|\mP_k\mF_jh\|_{L^2_{x,v}}+\|\mF_j\mu\|_{H^s}\|h\|_{L^2_{x,v}}\|\mP_k\mF_jf\|_{L^2_{x,v}}+\sum_{p>j}\|\mF_p\mu\|_{H^s}\|h\|_{L^2_xL^2_2}\|\mP_k\mF_ph\|_{L^2_{x,v}}\eeno\beno
&&+\sum_{a>k}\|\cP_a\mu\|_{L^2_{5}}\|\U_a\mF_jh\|_{L^2_{x,v}}\|\mP_k \mF_jh\|_{L^2_xH^s_{\ga/2}}+\|\U_{N_0}\mF_jh\|_{L^2_xH^s_{\ga/2}}\|\mP_k\mF_jh\|_{L^2_{x,v}}+2^{-jN}\|h\|^2_{L^2_{x,v}}\Big).
\eeno
Summing up w.r.t $j$ and $k$ and using Cauchy inequality, we are led to that
$ 
\sum_{j,k=-1}^\infty|\mathcal{R}^{j,k}_{1,3}|\leq \vep\|h\|^2_{L^2_xH^{n+s}_{\ga/2-\ell}}+C_{\vep,N,n,\ell}(\|h\|^2_{L^2_xH^{n}_{3-\ell}}+\|h\|^2_{L^2_xL^2_3}).$
 We conclude that
\[ 
\sum_{j,k=-1}^\infty \mathcal{R}^{j,k}_1\ls -\|h\|^2_{L^2_xH^{n+s}_{-\ell+\ga/2}}+C_{N,n,\ell}(\|h\|^2_{L^2_xH^{n+(s-1/2)^++\de}_{3-\ell}}+\|h\|^2_{L^2_xL^2_3}).
\]

 Similar to the estimate of $\mathcal{R}^{j,k}_1$, we have the following decomposition:
\beno
\mathcal{R}^{j,k}_2&=&2^{2nj}2^{-2\ell k}\Big((Q(h,\F_j\cP_k\mu),\F_j\cP_kh)_{L^2_{x,v}}+(\F_jQ(h,\cP_k\mu)-Q(h,\F_j\cP_k\mu),\F_j\cP_kh)_{L^2_{x,v}}\\
&&+(\cP_kQ(h,\mu)-Q(h,\cP_k\mu),\F_j^2\cP_k h)_{L^2_{x,v}}\Big):=\mathcal{R}^{j,k}_{2,1}+\mathcal{R}^{j,k}_{2,2}+\mathcal{R}^{j,k}_{2,3}.
\eeno
We claim that  
$\sum_{j,k=-1}^\infty |\mathcal{R}^{j,k}_2|\leq \vep \|h\|^2_{L^2_xH^{n+s}_{\ga/2-\ell}}+C_{\vep,N,n,\ell}(\|h\|^2_{L^2_xH^{n+(s-1/2)^++\de}_{3-\ell}}+\|h\|^2_{L^2_xL^1_2}+\|h\|^2_{L^2_xL^2_3}).$
 We only provide the estimates for $\mathcal{R}^{j,k}_{2,1}$ since the other two terms can be bounded   as the same as we did for $\mathcal{R}^{j,k}_{1,2}$ and $\mathcal{R}^{j,k}_{1,3}$. By Lemma \ref{upperQ}, it holds that
$|\mathcal{R}^{j,k}_{2,1}|\leq 2^{2nj}2^{-2\ell k}(\|h\|_{L^2_xL^1_2}+\|h\|_{L^2_xL^2_{1}})\|\F_j\cP_k\mu\|_{H^{2s}_{\ga+2s}}\|\F_j\cP_kh\|_{L^2_{v}},$
which implies that
$\sum_{j,k=-1}^\infty |\mathcal{R}^{j,k}_{2,1}|\leq C_{n,\ell}(\|h\|_{L^2_xL^1_2}+\|h\|_{L^2_xL^2_{3}}).$ This ends the proof of the claim.

 For $\mathcal{R}^{j,k}_3$, due to Lemma \ref{le1.2}, it is easy to see that
\beno
|\mathcal{R}^{j,k}_3|&\leq& C_N2^{2nj}2^{-2\ell k}(2^{-j}\|\mF_j\mP_kh\|_{H^1_xL^2_v}+2^{-jN}\|\mP_kh\|_{H^1_xL^2_v})\|\F_j\cP_kh\|_{L^2_{x,v}}\\
&\leq& C_N2^{2(n-1)j}2^{-2\ell k}\|\mF_j\mP_kh\|^2_{H^1_xL^2_v}+C_N2^{(2n-N)j}2^{-2\ell k}\|\mP_kh\|_{H^1_xL^2_v}+2^{2nj}2^{-2\ell k}\|\F_j\cP_kh\|^2_{L^2_{x,v}}.
\eeno
By interpolation, we have
 $2^{2(n-1)j}\|\mF_j\mP_kh\|^2_{H^1_xL^2_v}\leq \|\mF_j\mP_kh\|^2_{H^n_xL^2_v}+2^{2nj}\|\mF_j\mP_kh\|^2_{L^2_{x,v}},$ which yields that
\beno
|\mathcal{R}^{j,k}_3|&\leq& C_N \big(2^{2nj}2^{-2\ell k}\|\mF_j\mP_kh\|^2_{L^2_{x,v}}+2^{-2\ell k}\|\mF_j\mP_kh\|^2_{H^n_xL^2_v}+2^{(2n-N)j}2^{-2\ell k}\|\mP_kh\|_{H^1_xL^2_v}\big).
\eeno
Choosing that $N>2n$, we arrive at that 
$\sum_{j,k=-1}^\infty |\mathcal{R}^{j,k}_3|\ls C_N \|h\|^2_{L^2_xH^n_{-\ell}}+C_N\|h\|^2_{H^n_xL^2_{-\ell}}.$

Plugging all the above estimates into the energy inequality and taking $\vep$  small, we can deduce that
\beno
\f d{dt}\|h\|^2_{L^2_xH^n_{-\ell}}+\|h\|^2_{L^2_xH^{n+s}_{\ga/2-\ell}}\leq C_{N,n,\ell} \|h\|^2_{L^2_xH^n_{-\ell}}+C_{N,n,\ell}(\|h\|^2_{L^2_xH^{n+(s-1/2)^++\de}_{3-\ell}}+\|h\|^2_{L^2_xL^1_2}+\|h\|^2_{L^2_xL^2_3}+\|h\|^2_{H^n_xL^2_{-\ell}}).
\eeno
Recall that  $\|h\|^2_{L^2_xL^1_2}+ \|h\|^2_{L^2_xL^2_3}+\|h\|^2_{H^n_xL^2_{-\ell}}$ is bounded because of the assumption. If $s\geq1/2$, we get  that
\beno
\|h\|_{H^{n+s-1/2+\de}_{3-\ell}}&\ls& C_{n,\ell}\|h\|^{\f{n+s-1/2+\de}{n+s}}_{H^{n+s}_{\ga/2-\ell}}\|h\|^{\f{1/2-\de}{n+s}}_{L^2_3} 
\leq\vep\|h\|^{2}_{H^{n+s}_{\ga/2-\ell}}+C_{n,\ell}\|h\|^2_{L^2_3} ,
\eeno
where $\ell:=(3-\ga/2)(2n+2s-1+2\de)/(1-2\de)$. 
Plug it into energy inequality and due to Lemma \ref{le1.6}, we can obtain that  for any $n>0$ and $0<\tau< T$,
$\|h(t)\|_{L^2_xH^n_{-\ell}}<C_\tau,\quad \forall t\in(\tau,T].$
It implies   $h(t)\in C^\infty_{x,v}$ for any positive time. Similarly when $s<1/2$, we may set $\ell:=(3-\ga/2)(n+\de)/(s-\de)$ to get the desired result.   Since $\de>0$ can be arbitrarily small, we complete the proof of Theorem \ref{CtvLBE}(i).	
\end{proof}

Next, we will give the proof of Theorem \ref{CtvLBE} (ii). 

\begin{proof}[Proof of Theorem \ref{CtvLBE}(Part II)]  	
	We first recall the fact: if $f(t,x,v)$ satisfies \eqref{LHB}, then $\omega(t,v):=\int_{\T^3}f(t,x,v)dx$ is a homogeneous solution to \eqref{LHB}. Thus, we only need to consider the linear homogeneous Boltzmann equation:
	$\pa_t h=Q(\mu,h)+Q(h,\mu).$

To prove the desired result, we only need to show the infinity regularity for $t$ variable thanks to the result $(i)$. We prove it by inductive method. For $n=0$, we already have that $h\in L^\infty([0,T],L^2_3)$. Assume that it holds that for some $m\geq0$,  $\pa^m_th\in L^\infty((\tau,T],H^{-2ms}_3)$ with $\tau>0$. In what follows,  we shall prove that $\pa^{m+1}_th\in L^\infty((\tau,T],H^{-2(m+1)s}_3)$. Using $h$-equation, we derive that
\beno
\|\pa^{m+1}_th\|_{H^{-2(m+1)s}_3}=\sup_{\|g\|_{L^2}=1}(Q(\mu,\pa^m_th)+Q(\pa^m_th,\mu),\<v\>^3\<D\>^{-2(m+1)s}g)_{L^2_v}.
\eeno
By Lemma \ref{forCtv}, we get that
 \[|(Q(\mu,\pa^m_th)+Q(\pa^m_th,\mu),\<v\>^3\<D\>^{-2(m+1)s}g)_{L^2}|\ls
  C_m\|\pa^m_th\|_{H^{-2ms}_{3+\ga+2s}}\|g\|_{L^2}.\]
Since $\ga+2s\leq0$, we get that $\|\pa^{m+1}_th\|_{H^{-2(m+1)s}_3}\ls C_m\|\pa^{m}_th\|_{H^{-2ms}_3},\forall t\in(\tau,T]$. Thus by induction, for any $m\geq0$, $\pa^m_th\in L^\infty((\tau,T],H^{-2ms}_3)$ with $\tau>0$. It yields that 
$$\<D_t\>^m\<D_v\>^{-2ms}\<\cdot\>^3h\in L^2((\tau,T)\times\R^3),~~\forall m\in \N.$$

Recall that   $\<D_v\>^{3ms}\<\cdot\>^{-\ell(m)}h\in L^2((\tau,T)\times\R^3),~~\forall m\in \N$ with $\ell(m)>\mathbf{1}_{2s\ge1}(3-\ga/2)(6ms+2s-1)+\mathbf{1}_{2s<1}3m(3-\ga/2)$. By interpolation, we can derive that
\beno
\<D_t\>^{m/2}\<D_v\>^{ms/2} \<\cdot\>^{(3-\ell(m))/2}h\in L^2((\tau,T)\times\R^3),\quad\forall m\in \N,
\eeno
which implies that $h\in C^\infty_{t,v}$ by Sobolev embedding. It ends the proof of Theorem \ref{CtvLBE}(Part II).
\end{proof}
%\begin{rmk}
%	The above proof is still valid when the solution of \eqref{LB} is smooth on variable $x$, that is $h\in L^\infty([0,T], H^\infty_xL^2_\ell),\ell\geq 3$ implies that $h(t)\in C^\infty_{x,v},\forall t\in(0,T]$.
%\end{rmk}

\section{Appendix}
In the Appendix, we give some useful lemmas including sharp coercivity estimates and upper bounds for the Boltzmann collision operator, as well as some commutators.

\subsection{Pseudo-Differential operator and basic commutators}
\begin{lem}[see \cite{HE}]\label{le1.1}
	Let $s, r\in\R$ and $a(v),b(\xi)\in C^\infty$ satisfy for any $\alpha\in\Z^3_+$,
	\ben\label{abconstants}
	|D_v^\al a(v)|\leq C_{1,\al}\<v\>^{r-|\al|},~|D_\xi^\al b(\xi)|\leq C_{2,\al}\<\xi\>^{s-|\al|}
	\een
	for constants $C_{1,\al},C_{2,\al}$. Then there exists a constant $C$ depending only on $s,r$ and finite numbers of $C_{1,\al},C_{2,\al}$ such that for any $f\in \mathscr{S}(\R^3)$,
	\beno
	\|a(v)b(D)f\|_{L^2}\leq C\|\<D\>^s\<v\>^rf\|_{L^2},~\|b(D)a(v)f\|_{L^2}\leq C\|\<v\>^r\<D\>^sf\|_{L^2}.
	\eeno
	As a direct consequence, we get that $\|\<D\>^m\<v\>^lf\|_{L^2}\sim\|\<v\>^l\<D\>^mf\|_{L^2}\sim\|f\|_{H^m_l}$.
\end{lem}

\begin{lem}[see \cite{HJ2}]\label{le1.2}
	Let $l,s,r\in\R,M(\xi)\in S_{1,0}^r$ and $\Phi(v)\in S_{1,0}^l$. Then there exists a constant $C$ such that $\|[M(D_v),\Phi(v)]f\|_{H^s}\leq C\|f\|_{H_{l-1}^{r+s-1}}$. Moreover, for any $N\in\N,$
	\ben\label{MPHICOMMU}
	M(D_v)\Phi=\Phi M(D_v)+\sum_{1\leq|\al|\leq N}\frac{1}{\al!}\Phi_\al M^\al(D_v)+r_N(v,D_v),
	\een
	where $\Phi_\al(v)=\pa_v^\al\Phi,~M^\al(\xi)=\pa_\xi^\al M(\xi)$ and $\<v\>^{N-l}r_N(v,\xi)\in S^{r-N}_{1,0}$. Moreover, for any $\beta,\beta'\in \in\Z^3_+$, we have
	\ben\label{rN}
	|\pa^\beta_v\pa^{\beta'}_\xi r_N(v,\xi)|\leq C_{\beta,\beta'}\<\xi\>^{r-N-|\beta|}\<v\>^{l-N-|\beta'|},\quad \|r_{2N+1}(v,D_v)\<D\>^{N}\<v\>^{N}f\|_{L^2}\leq C\|f\|_{L^2}.
	\een
	Furthermore, use (\ref{MPHICOMMU}) repeatedly, we can also obtain that
	\ben\label{MPHICOMMU2}
	M(D_v)\Phi=\Phi M(D_v)+\sum_{1\leq|\al|\leq N}C_{\al} M^\al(D_v)\Phi_\al+C_Nr_N(v,D_v).
	\een
\end{lem} 

\begin{proof} We only need to prove the second inequality in \eqref{rN} since the others have been established in \cite{HJ2}. We first have  \ben\label{rNN}
  \pa_v^\be\pa_\xi^{\be'}r_{2N+1}(v,\xi)\leq C_{N,\be,\be'}\<\xi\>^{ -N-1-|\beta'|}\<v\>^{-N-1-|\be|}.
  \een By the fundamental theorem for the algebra of pseudo-differential operators(see \cite{HKgo}),  if $r(v,\xi)$ is the symbol of operator $r_{2N+1}(v,D_v)\<D\>^{N }\<v\>^{N}$, then  
 $r(v,\xi):=\mathrm{Os}-\f1{(2\pi)^3}\int_{\R^6}e^{-iu\cdot\eta}\bar{r}(v,\xi+\eta)\<v+u\>^Ndud\eta,$
    where $\bar{r}(v,\xi)=r_{2N+1}(v,\xi)\<\xi\>^{N}$ and $``\mathrm{Os-}"$ means the oscillatory integral. It suffices to prove $r(v,\xi)\in S^0_{1,0}$. Using the identities
  $e^{-iu\cdot\eta}=\<\eta\>^{-2l}(1-\Delta_u)^le^{-iu\cdot\eta},  e^{-iu\cdot\eta}=\<u\>^{-2k}(1-\Delta_\eta)^ke^{-iu\cdot\eta},$
  for $l>|\be|/2+3/2$ and $k>N+3/2$, we have 
 \[
  \pa_v^\al\pa^\be_\xi r(v,\xi)=\f1{(2\pi)^3}\sum_{\al_1+\al_2=\al}C^{\al_1}_{\al}\int\Big(\int e^{-iu\cdot \eta}\<u\>^{-2k}(1-\Delta_\eta)^k\{\<\eta\>^{-2l}(1-\Delta_u)^l\pa^{\al_1}_v\pa^\be_\xi \bar{r}(v,\xi+\eta)\pa^{\al_2}_v(\<\cdot\>^N)(v+u)\}d\eta\Big)du.\]
  To get the desired result, it is easy to see that 
  \beno
 \int\{(1-\Delta_u)^l\pa^{\al_2}_v(\<\cdot\>^N)(v+u)\}\Big(\int e^{-iu\cdot \eta}(1-\Delta_\eta)^k\{\<\eta\>^{-2l}\pa^{\al_1}_v\pa^\be_\xi \bar{r}(v,\xi+\eta)\}d\eta\Big)\f{du}{\<u\>^{2k}}=\int\{(1-\Delta_u)^l\\ \pa^{\al_2}_v(\<\cdot\>^N)(v+u)\}\Big(\int_{|\eta|\leq\f{|\xi|}2}+\int_{|\eta|>\f{|\xi|}2}\Big)\f{du}{\<u\>^{2k}}:=
  \int\{(1-\Delta_u)^l\pa^{\al_2}_v(\<\cdot\>^N)(v+u)\}\Big(I_1(v;u)+I_2(v;u)\Big)\f{du}{\<u\>^{2k}}.
  \eeno
  Since $\<\xi\>$ and $\<\xi+\eta\>$ are equivalent in $I_1$, it follows from (\ref{rNN}) that
  $|I_1|\leq C \<v\>^{-N},$
  and moreover the same bound for $|I_2|$ holds because $2l>|\be|+3$. Furthermore, since $2k>N+3$, we can also obtain that
  %\beno
  %\int e^{-iu\cdot \eta}(1-\Delta_\eta)^k\{\<\eta\>^{-2l}\pa^{\al_1}_v\pa^\be_\xi r(v,\xi+\eta)\}d\eta\leq C_{k}\<v\>^{-N}.
  %\eeno
  %Then we have that
  \beno
  \int\{(1-\Delta_u)^l\pa^{\al_2}_v(\<\cdot\>^N)(v+u)\<v\>^{-N}\<u\>^{-2k}du\leq C,
  \eeno
  which implies that $|\pa_v^\al\pa^\be_\xi r(v,\xi)|\leq C_{\al,\be}\<\xi\>^{-|\be|}$. This ends the proof.
  \end{proof}

\begin{rmk}\label{CONSTS} We emphasize that in the statement of Lemma \ref{le1.2}, the constant $C$ appearing in the inequality depends only on $C_{1,\al},C_{2,\al}$ in \eqref{abconstants} with $a=\Phi$ and $b=M$ and also the constants $ C_{\beta,\beta'}$ for $r_N(v,\xi)$. This fact is crucial for the estimates of commutators and the profiles of weighted Sobolev spaces. For instance,  if $M(D_v)$ and $\Phi(v)$ are chosen to be the localized operators $\F_j$ and $\cP_k$, the constant $C$ in Lemma \ref{le1.2} does not depend on $j$ and $k$. Indeed,
	for any $k\geq0,N\in\N$, $2^{Nk}\vphi(2^{-k}v)$ satisfies that for any $\alpha\in\Z^3_+$,
	\ben\label{Ncon}
	|D_v^\al 2^{Nk}\vphi(2^{-k}v) |\leq C_{N,\al}\<v\>^{N-|\al|}|\vphi_\al(2^{-k}v)|\leq C_{N,\al}\<v\>^{N-|\al|}.
	\een
\end{rmk}

\begin{lem}\label{7.8}
	(Bernstein inequality). There exists a constant $C$ independent of $j$ and $f$ such that
	
	(1) For any $s\in\R$ and $j\geq 0$,
	\beno
	C^{-1}2^{js}\|\F_jf\|_{L^2(\R^3)}\leq\|\F_jf\|_{H^s(\R^3)}\leq C 2^{js}\|\F_jf\|_{L^2(\R^3)}.
	\eeno
	
	(2) For integers $j,k\geq0$ and $p,q\in[1,\infty],q\geq p$, the Bernstein inequality are shown as
	\beno
	&&\sup_{|\al|=k}\|\pa^\al\F_jf\|_{L^q(\R^3)}\ls2^{jk}2^{3j(\frac{1}{p}-\frac{1}{q})}\|\F_jf\|_{L^p(\R^3)},
	\sup_{|\al|=k}\|\pa^\al S_jf\|_{L^q(\R^3)}\ls2^{jk}2^{3j(\frac{1}{p}-\frac{1}{q})}\|S_jf\|_{L^p(\R^3)},\\
	&&2^{jk}\|\F_jf\|_{L^p(\R^3)}\ls\sup_{|\al|=k}\|\pa^\al\F_jf\|_{L^p(\R^3)}\ls2^{jk}\|\F_jf\|_{L^p(\R^3)}.
	\eeno
\end{lem}

\begin{lem}\label{lemma1.3}
	If $\cP_k,\U_k$ and $\F_j$ are defined in in Definition \ref{de2.1} and $n\in\R^+$, then
	
	(i) For any $N\in \N$, there exists a constant $C_{N}$ such that
	\beno
	\|[\cP_k,\F_j ]f\|_{L^2}&=&\|(\cP_k\F_j -\F_j\cP_k)f\|_{L^2}\leq C_{N} \big( 2^{ -j}2^{-k}\sum_{|\al|=1}^{2N}\|\cP_{k,\al}\F_{j,\al}f\|_{L^2}+2^{-jN}2^{-kN}\|f\|_{H_{-N}^{-N}}\big),\\
	\|[\U_k,\F_j]f\|_{L^2}&=&\|(\U_k\F_j-\F_j \U_k)f\|_{L^2}\leq C_{N } \big( 2^{ -j}\sum_{|\al|=1}^{2N}\|\U_{k,\al}\F_{j,\al}f\|_{L^2}+2^{-jN}\|f\|_{H_{-N}^{ -N}}\big),
	\eeno
	where $\cP_{k,\al},\F_{j,\al}$ and $\U_{k,\al}$ are defined in Definition \ref{de2.1}. Moreover, replace $\cP_{k,\al}$ and $\F_{j,\al}$ by $\tP_{k,\al}$ and $\tF_{j,\al}$ respectively, the above results still hold.
	
	(ii) For $|m-p|>N_0$ and $\forall N\in \N$, there exists a constant $C_N$ such that
	\beno
	\|\F_m\cP_k\F_pg\|_{L^2}\leq C_N2^{-(p+m+k)N}\|\F_pg\|_{L^2_{-N}},\quad \|\F_m\U_{k}\F_pg\|_{L^2}\leq C_N2^{-(p+m)N}\|\F_pg\|_{L^2_{-N}}.
	\eeno
	If $m>p+N_0$, we have
	$\|\F_m\cP_{k}S_pg\|_{L^2}\leq C_N2^{-(m+k)N}\|S_pg\|_{L^2_{-N}},
    \|\F_m\U_{k}S_pg\|_{L^2}\leq C_N2^{-mN}\|S_pg\|_{L^2_{-N}}.$

	(iii) For any $a,w\in \R$, we have
	$\|\U_{k+N_0}h\|_{H^a}\leq C_{a,w}2^{k(-w)^+}\|h\|_{H^a_w}$ and
	$\|S_{p+N_0}h\|_{L^2_l}\leq C_{l}\|h\|_{L^2_l}.$

	(iv) For any $j\geq -1$ and $l\in\R$, we have
	$ \|\F_jf\|_{L^1_l}+\|S_jf\|_{L^1_l}\leq C_l\|f\|_{L^1_l}.$
\end{lem}
\begin{proof} It follows the proof of Lemma A.4 in \cite{HJ2} and the estimate \eqref{rN}. We omit the details here.
\end{proof}

 \begin{lem}[see \cite{HJ2}]\label{lemma1.4}
	(i) Let $m, l\in \R.$ Then for $f\in H_l^m$,
	\ben\label{Ber}
	\sum_{k=-1}^\infty2^{2kl}\|\tP_k f\|^2_{H^m}\sim\sum_{k=-1}^\infty2^{2kl}\|\cP_k f\|^2_{H^m}\sim\|f\|^2_{H^m_l}\sim\sum_{j=-1}^\infty2^{2 j m}\|\F_jf\|^2_{L^2_l}\sim\sum_{j=-1}^\infty2^{2 j m}\|\tF_jf\|^2_{L^2_l}.
	\een
	Moreover, we   have
	\ben\label{7.70}
	\sum_{k=-1}^\infty2^{2kl}\|\mP_{k}f\|^2_{H^m}\le C_{m,l}\sum_{k=-1}^\infty2^{2kl}\|\cP_kf\|^2_{H^m},\quad
	\sum_{j=-1}^\infty2^{2jm}\|\mF_{j}f\|^2_{L^2_l}\le C_{m,l}\sum_{j=-1}^\infty2^{2jm}\|\F_jf\|^2_{L^2_l}.
	\een
	
	(ii) If $m, n, l\in\R$ and $\de>0$, then we have
	\ben\label{7.77}
	\sum_{j=-1}^\infty2^{2 j n}  \|\F_jf\|^2_{H^m_l}  \lesssim C_{m, n, l}\|f\|^2_{H^{m+n}_l},\quad \|f\|_{H^{-\frac{3}{2}-\delta}_l}   \lesssim C_l\|f\|_{L^1_l}.
	\een
\end{lem}

\begin{lem}\label{refc}
Let $\F_j,\cP_k,\mF_j,\mP_k$ be defined in subsection \ref{DDP}. Then we have  
\beno
&&2^{2js}2^{\ga k}\|\F_j\cP_kf\|^2_{L^2}-C_N(2^{(2s-1)j}2^{(\ga-1)k}\|\mF_j\mP_kf\|^2_{L^2}+2^{-jN}2^{-kN}\|f\|_{H^{-N}})\ls
(\F_j\cP_k(-\Delta_v)^s\<v\>^{\ga}f,\\&&\qquad\F_j\cP_kf)_{L^2_v}\ls 2^{2js}2^{\ga k}\|\F_j\cP_kf\|^2_{L^2}|+C_N(2^{(2s-1)j}2^{(\ga-1)k}\|\mF_j\mP_kf\|^2_{L^2}+2^{-jN}2^{-kN}\|f\|_{H^{-N}}).
\eeno
\end{lem}
\begin{proof}
We have the following decomposition:
\beno
&(\F_j\cP_k(-\Delta_v)^s\<v\>^{\ga}f,\F_j\cP_kf)=((-\Delta_v)^s\<v\>^{\ga}f,\F_j^2\cP_k^2f)+((-\Delta_v)^s\<v\>^{\ga}f,[\F_j^2,\cP_k]\cP_kf)=(f,(-\Delta_v)^s\F_j^2\<v\>^{\ga}\cP_k^2f)\\
&+(f,[\<v\>^{\ga},(-\Delta_v)^s\F_j^2]\cP_k^2f)+((-\Delta_v)^s\<v\>^{\ga}f,[\F_j^2,\cP_k]\cP_kf)\|(-\Delta_v)^{s/2}\F_j\<v\>^{\ga/2}\cP_kf\|^2_{L^2}+(f,[(-\Delta_v)^s\F_j^2,\\
&\<v\>^{\ga/2}\cP_k]\<v\>^{\ga/2}\cP_kf)+(f,[\<v\>^{\ga},(-\Delta_v)^s\F_j^2]\cP_k^2f)+((-\Delta_v)^s\<v\>^{\ga}f,[\F_j^2,\cP_k]\cP_kf):=M+L_1+L_2+L_3.
\eeno

 For the term $L_1$, due to Lemma \ref{le1.2}\eqref{MPHICOMMU2}, we get that
\beno
[(-\Delta_v)^s\F_j^2,\<v\>^{\ga/2}\cP_k]=\sum_{1\leq|\al|\leq 3N}C_\al D_\al(|\cdot|^{2s}\vphi(2^{-j}\cdot))(D_v)D^\al(\<\cdot\>^{\ga/2}\vphi(2^{-k}\cdot))(v)+C_Nr_N(v,D_v),
\eeno
where the symbol $r_N(v,\xi)\in S^{-2N}_{1,0}$  satisfies that
$\|r_N(\cdot,D_v)f\|_{L^2}\leq C_N\|f\|_{H^{-2N}}.$
Thus we obtain that
\beno
|L_1|
&\ls& C_N(2^{(2s-1)j}2^{(\ga-1)k}\|\mF_j\mP_kf\|^2_{L^2}+2^{-jN}2^{-kN}\|f\|^2_{H^{-N}}).
\eeno
By the same argument in the above, we may derive that
\ben\label{L123}
|L_1|+|L_2|+|L_3|\leq C_N(2^{(2s-1)j}2^{(\ga-1)k}\|\mF_j\mP_kf\|^2_{L^2}+2^{-jN}2^{-kN}\|f\|^2_{H^{-N}}).
\een

For the main term $M$, notice that
\beno
&&M\sim 2^{2js}\|\F_j \<v\>^{\ga/2}\cP_kf\|^2_{L^2}=2^{2js}\|\<v\>^{\ga/2}\cP_k\F_jf\|^2_{L^2}+2^{2js}\|[\F_j,\<v\>^{\ga/2}\cP_k]f\|^2_{L^2}\\
&\sim& 2^{2js}2^{\ga k}\|\F_j\cP_kf\|^2_{L^2}+2^{2js}2^{\ga k}\|[\cP_k,\F_j]f\|^2_{L^2}+2^{2js}\|[\F_j,\<v\>^{\ga/2}\cP_k]f\|^2_{L^2}.
\eeno
By Lemma \ref{lemma1.3}, we  obtain that
\ben\label{M123}
&&2^{2js}2^{\ga k}\|\F_j\cP_kf\|^2_{L^2}- C_N(2^{(2s-1)j}2^{(\ga-1)k}\|\mF_j\mP_kf\|^2_{L^2}+2^{-jN}2^{-kN}\|f\|_{H^{-N}})\\ \notag&
\ls& M \ls 2^{2js}2^{\ga k}\|\F_j\cP_kf\|^2_{L^2}+ C_N(2^{(2s-1)j}2^{(\ga-1)k}\|\mF_j\mP_kf\|^2_{L^2}+2^{-jN}2^{-kN}\|f\|_{H^{-N}}).
\een
We conclude the desired results by combining \eqref{L123} and \eqref{M123}.
\end{proof}

\subsection{Coercivity estimates and upper bound of collision operator}
\begin{lem}(\cite{HE})\label{coer}
	Suppose $g$ is a non-negative and smooth function verifying that
	\beno
	\|g\|_{L^1}>\de\quad\mbox{and}\quad \|g\|_{L^1_2}+\|g\|_{L\log L}<\lam,
	\eeno
	and let $\mathbf{A}=0,1$. Then for sufficiently small $\eta>0$, there exist constants $\mathbf{C}_1(\de,\lambda,\eta^{-1}),\mathbf{C}_2(\lam,\de)$, $\mathbf{C}_3(\lam,\de,\eta^{-1}),\mathbf{C}_4(\lam,\de)$ and $\mathbf{C}_5(\lam,\de)$ such that
	\begin{itemize}
		\item if $\ga+2s\geq0$,
		\beno
		(-Q(g,f),f)_{L^2_v}&\gs& \mathbf{A}\Big[\mathbf{C}_1\big((\de,\lambda,\eta^{-1})\|(-\Delta_{\S^2})^{\f s2}f\|^2_{L^2_{\ga/2}}+\|f\|^2_{H^s_{\ga/2}}\big)-\eta \mathbf{C}_2(\lam,\de)\|f\|^2_{L^2_{\ga/2+s}}-\mathbf{C}_3(\lam,\de,\eta^{-1})\|f\|^2_{L^2_{\ga/2}}\Big]\\
		&&+\mathbf{C}_4(\lam,\de)\|f\|^2_{L^2_{\ga/2+s}}-\mathbf{C}_5(\lam,\de)\|f\|^2_{L^2_{\ga/2}}.
		\eeno
		\item if $-1-2s<\ga<-2s$ and  $p>\f3{\ga+2s+3}$, then
		\beno
		&&(-Q(g,f),f)_{L^2_v}\gs \mathbf{A}\Big[\mathbf{C}_1\big((\de,\lambda,\eta^{-1})\|(-\Delta_{\S^2})^{\f s2}f\|^2_{L^2_{\ga/2}}+\|f\|^2_{H^s_{\ga/2}}\big)-\eta \mathbf{C}_2(\lam,\de)\|f\|^2_{L^2_{\ga/2+s}}\\
		&&-\mathbf{C}_3(\lam,\de,\eta^{-1})(1+\|g\|^{\f{(\ga+2s+3)p}{(\ga+2s+3)p-3}}_{L^p_{|\ga|}})\|f\|^2_{L^2_{\ga/2}}\Big] +\mathbf{C}_4(\lam,\de)\|f\|^2_{L^2_{\ga/2+s}}-\mathbf{C}_5(\lam,\de)(1+\|g\|^{\f{(\ga+2s+3)p}{(\ga+2s+3)p-3}}_{L^p_{|\ga|}})\|f\|^2_{L^2_{\ga/2}}.
		\eeno
		
	\end{itemize}
\end{lem}

\begin{lem}(\cite{HE})\label{upperQ}
	Let $a,b\in[0,2s],a_1,b_1,\om_1,\om_2,\om_3,\om_4\in\R$ with $a+b=2s,a_1+b_1=s$ and $\om_1+\om_2=\ga+s,\om_3+\om_4=\ga+2s$. Then for smooth functions $g,h$ and $f$, if $\ga<0$, it holds that
	\beno
	|(Q(g,h),f)|&\ls&\big(\|g\|_{L^1_{-\ga+2s}}+\|g\|_{L^1_{\ga+s+(-\om_1)^++(-\om_2)^+}}+\|g\|_{L^2_{-\ga}}\big)
	\Big(\big(\|(-\Delta_{\S^2})^{\f a2}h\|_{L^2_{\ga/2}}+\|h\|_{H^a_{\ga/2}}\big)\\
	&&\times\big(\|(-\Delta_{\S^2})^{\f b2}f\|_{L^2_{\ga/2}}+\|f\|_{H^b_{\ga/2}}\big)+\|h\|_{H^{a_1}_{\om_1}}\|f\|_{H^{b_1}_{\om_2}}\Big)
	\ls \|g\|_{L^2_5}\|h\|_{H^a_{\omega_3}}\|f\|_{H^b_{\omega_4}}.
	\eeno
	
	If $-1<\ga+2s<0$, we have
   $|(Q(g,h),f)|\ls (\|g\|_{L^1_{\om}}+\|g\|_{L^2_{-(\ga+2s)}})\|h\|_{H^a_{\om_3}}\|f\|_{H^b_{\om_4}}$,
	where $\om=\max\{-(\ga+2s),\ga+2s+(-\om_3)^++(-\om_4)^+\}$.
\end{lem}

%\begin{lem}(\cite{CHJ})\label{muQ}
%	Suppose that $\ga\in (-3,1], \gamma+2s>-1$ ,$k \ge 22$ and  $\mu(v)= (2\pi)^{-3/2} e^{-|v|^2/2}$. Then for smooth functions $g$ and $h$, we have
%	\ben\label{poly1}
%	&&|(Q( h , \mu), g \langle v \rangle^{2k})|\le   C_k\| h \|_{L^2_{k+\gamma/2}}\| g\|_{L^2_{k+\gamma/2}} + C_{k} \Vert h \Vert_{L^2_{k+\gamma/2-1/2}}\Vert g \Vert_{L^2_{k+\gamma/2-1/2}}
%	\een
%	where $C_k>0$ is a constant depending on $k$.
%\end{lem}

%\begin{lem}(\cite{CH1})\label{com}
%	Let $N_1=|N_2|+|N_3|+\max\{|l-2|,|l-1|\}$ and $\tilde{N}_1=N_2+N_3$ with $N_2,N_3,l\in\R$. Then if $\tilde{N}_1\geq l+\ga$ and $s<\f12$, one has
%	\beno
%	|(\<v\>^lQ(g,h)-Q(g,\<v\>^lh),f)_v|\ls\|g\|_{L^1_{N_1}}\|h\|_{H^{\varrho}_{\mathfrak{N}_2}}\|f\|_{H^\varrho_{N_3}},
%	\eeno
%	where $\varrho<s$. When $\tilde{N}_1\geq l-1/2+\ga+2s$ and $s\geq 1/2$, one has that in the case of $\ga+2s>0$, there holds
%	\beno
%	|(\<v\>^lQ(g,h)-Q(g,\<v\>^lh),f)_v|\ls\|g\|_{L^1_{N_1}}\|h\|_{H^{s}_{N_2}}\|f\|_{L^2_{N_3}}.
%	\eeno
%	While in the case of $\ga+2s\leq 0$, there holds
%	\beno
%	|(\<v\>^lQ(g,h)-Q(g,\<v\>^lh),f)_v|\ls\Big(\|g\|_{L^1_{N_1}}+\|g\|_{L^{\f32}_{N_1}}\Big)\|h\|_{H^{s}_{N_2}}\|f\|_{L^2_{N_3}}.
%	\eeno
%\end{lem}

\subsection{Other useful lemmas}
\begin{lem}\label{chv} (\cite{ADVW}) For any smooth function $f,g$, we have \\
	(1) (Regular change of variables)
	\[
	\int_{\R^3} \int_{\mathbb{S}^2} b(\cos \theta) |v-v_*|^\gamma f(v') d \sigma dv= \int_{\R^3} \int_{\mathbb{S}^2} b(\cos \theta)\frac 1 {\cos^{3+\gamma} (\theta/2)} |v-v_*|^\gamma f(v) d\sigma dv.
	\]
	(2) (Singular change of variables)
	\[
	\int_{\R^3} \int_{\mathbb{S}^2} b(\cos \theta) |v-v_*|^\gamma f(v') d \sigma dv_* = \int_{\R^3} \int_{\mathbb{S}^2} b(\cos \theta)\frac 1 {\sin^{3+\gamma} (\theta/2)} |v-v_*|^\gamma f(v_*) d\sigma dv_*.
	\]
\end{lem}

\begin{lem}\label{L116} If $\gamma \in(-3, 0), s \in (0, 1), \gamma+2s>-1$, then for  any $g$ and $f$,
	\beno
	\mathcal{R}  := \int_{\R^3} \int_{\R^3} |v-v_*|^\gamma g_* f^2 dv_* dv \lesssim \Vert g \Vert_{L^2_{|\gamma|+2 }} \Vert f \Vert_{H^\varrho_{ \gamma/2}}^2\quad \mbox{with}\quad \varrho<s.
	\eeno
\end{lem}
\begin{proof}
	For  $-3< \gamma  <0$, using the fact that
	$\langle v \rangle^\gamma \langle v-v_* \rangle^\gamma \lesssim \langle v \rangle^{\gamma}$,
	we derive that
	\beno
	\mathcal{R}  := \int_{\R^3} \int_{\R^3} \frac { \langle v-v_*\rangle^{|\gamma|} } { | v-v_*  |^{|\gamma|  }}  ( g_* \langle v \rangle^{|\gamma|}) (f \langle v \rangle^{\gamma/2} )^2 dv_* dv\lesssim \int_{\R^3} \int_{\R^3} (1+ | v-v_*  |^{\gamma } )  ( g_* \langle v \rangle^{|\gamma|}) (f \langle v \rangle^{\gamma/2} )^2 dv_* dv :=\mathcal{R}_1 +\mathcal{R}_2.
	\eeno
	
	We first   have
	$\mathcal{R}_1 \lesssim \Vert g\Vert_{L^1_{|\gamma|}} \Vert f \Vert_{L^2_{ \gamma/2}}^2.$ For $\mathcal{R}_2$,   by Hardy-Littlewood-Sobolev inequality, we have
	$\mathcal{R}_2 \lesssim \Vert g\Vert_{L^p_{|\gamma|}} \Vert f \Vert_{L^{2q}_{ \gamma/2}}^2$,  where $p, q \in (1, +\infty)$ and $1/p+1/q =2+\gamma/3$.
	By Sobolev embedding theorem,  it holds that
	$\Vert f \Vert_{L^{2q}}  \le \Vert f \Vert_{H^{\f32-\f3{2q}}}$. If $\ga+2s\geq0$, choosing $1/p=1/q=1+\ga/6$, then $p\in(1,2)$ and $3/2-3/(2q)=-\ga/4<s$. If $-1<\ga+2s<0$, taking $1/p = 1 +(\gamma+2s)/2$ and $1/q = 1-s-\ga/6$, we get that $p \in (1, 2)$ and $3/2-3/(2q)=\f32 s+\ga/4<s$. Then we conclude that
	$\mathcal{R}_2\ls \|g\|_{L^2_{|\ga|+2}}\|f\|^2_{H^\varrho_{\ga/2}},~~\varrho<s.$
	It ends the proof of this lemma by putting together the estimates of 	$\mathcal{R}_1$ and $\mathcal{R}_2$.
\end{proof}

\begin{lem}\label{L110}For  smooth functions $f, g, h$,  if $\gamma +2s>-1$, then for some $\varrho<s$,
	\[ 
	\int_{\R^6\times\mathbb{S}^2}b(\cos \theta) |v-v_*|^\gamma f_* g h' dv dv_* d\sigma\le \int_{\mathbb{S}^2} b(\cos \theta) \sin^{-3/ 2-\gamma/2} \frac \theta 2   d\sigma \|g\|_{L^2_{|\gamma|+2}}\|f\|_{H^\varrho_{\gamma/2}}\|h\|_{H^\varrho_{\gamma/2}},
	\]
	\[
	\int_{\R^6\times\mathbb{S}^2}b(\cos \theta) |v-v_*|^\gamma f_* g h' dv dv_* d\sigma
	\le \int_{\mathbb{S}^2} b( \cos \theta) \cos^{-3/ 2-\gamma/2} \frac \theta 2   d\sigma
	\|f\|_{L^2_{|\gamma|+2}}\|g\|_{H^\varrho_{\gamma/2}}\|h\|_{H^\varrho_{\gamma/2}}.
	\]
\end{lem}
\begin{proof} We only give the detailed proof for  the first inequality since the second one can be derived in the same manner.
	By singular change of variable in Lemma \ref{chv} and Lemma \ref{L116}, we have
	\beno
	&& \int_{\R^6\times\mathbb{S}^2}b(\cos \theta) |v-v_*|^\gamma f_* g h' dv dv_* d\sigma  \le\big(\int_{\R^6\times\mathbb{S}^2} b(\cos \theta) \sin^{-\f32-\f\gamma2} \frac \theta 2  |v-v_*|^\gamma |f_*|^2 |g| dv dv_*  d\sigma\big)^\f12\big(
	\int_{\R^6\times\mathbb{S}^2}   |g|
	\\
	&&\times b(\cos \theta)\sin^{ \f32+\f\gamma2} \frac \theta 2 |v-v_*|^\gamma ||h'|^2 dv dv_*  d\sigma\big)^{\f12}
	\le \big(\int_{\mathbb{S}^2} b(\cos \theta) \sin^{-\f32-\f\gamma2} \frac \theta 2 d\sigma\big)\|g\|_{L^2_{|\gamma|+2}}\|f\|_{H^\varrho_{\gamma/2}}\|h\|_{H^\varrho_{\gamma/2}}.
	\eeno This ends the proof of the lemma.
\end{proof}

\begin{lem}[see \cite{HJ2}]\label{lemma1.5}
	Recall  that $\<Q(g,h),f\>=\sum_{k=-1}^\infty\<Q_k(g, h),f\>$, where  $Q_k(g, h)=\iint_{\si\in\mathbb{S}^2,v_*\in\R^3}\Phi_k^\ga(|v-v_*|)  b(\cos\th)(g_*'h'-g_*h)d\si dv_*$ with $\ga\in(-3,2]$ and
	\ben
	\label{1.2}\Phi^\ga_k(v):=\begin{cases}
		|v|^\ga\varphi(2^{-k}|v|),~\mbox{if}~~k\ge  0; \\
		|v|^\ga\psi(|v|),~\mbox{if}~~k=-1.
	\end{cases}
	\een
	\begin{enumerate}
		\item For $k\ge 0$, $i\in\N$ and $N\gg1$,
		$\int_{\R^3} |\cF(\Phi_k^\ga)(y)||y|dy\ls 2^{k(\ga-1)},|\nabla^i\cF(\Phi_k^\ga)(\eta)|\ls C_{N,i}2^{k(\ga+3+i)}\<2^k\eta\>^{-N}$;
	 
		\item For $k=-1,i=0,1,2$,
		 $|\nabla^i\cF(\Phi_{-1}^\ga)(\eta)|\ls \<\eta\>^{-(\ga+3+i)}.$
		  \end{enumerate}
\end{lem} 
\begin{lem}\cite{HJ2}\label{le1.6}
	If $X(t)$ is a non-negative  continuous function satisfying that
	\beno
	\frac{d}{dt}X(t)+C_1X(t)^{1+\varepsilon}\leq C_2,
	\eeno
	where $C_1,C_2$ are positive constants, then for any $t>0$, $X(t)\leq \max\{(4C_2C_1^{-1})^{1/(1+\varepsilon)}, (4C_1^{-1})^{1/\varepsilon}t^{-1/\varepsilon}\}$.
\end{lem}

\begin{lem}\cite{HJZ}\label{pq}
	Suppose $0<\de\ll1$ and $a,b\geq0$ with $a+b=3/2+\de$. There holds
	\beno
	\left|\sum_{p,q\in\Z^3}A_pB_{q-p}C_q\right|\ls \left(\sum_{p\in\Z^3}|p|^{2a}A_p^2\right)^{1/2}\left(\sum_{p\in\Z^3}|p|^{2b}B_p^2\right)^{1/2}\left(\sum_{p\in\Z^3}C_p^2\right)^{1/2}.
	\eeno
\end{lem}

\subsection{Two important commutators} To make our paper self-contained, we give a detailed proof for the 
 commutators between collision operator  and weight function $\<v\>^k$ and between collision operator and the dyadic decomposition operator. We begin with a lemma:
  \begin{lem}(see \cite{AMSY,HLP})\label{L18}
    We have the following two types of decomposition about $\<v'\>^2$:
    \begin{itemize}
      \item  Let $\mathbf{h}:=\frac{v+v_*}{|v+v_*|},\mn:=\frac{v-v_*}{|v-v_*|}$, $\mj:=\frac{\mh-(\mh\cdot\mn)\mn}{\sqrt{1-(\mh\cdot\mn)^2}}$, $E(\theta):=\<v\>^2\cos^2(\theta/2)+\<v_*\>^2\sin^2(\theta/2)$ and  $\si=\cos\th\mn+\sin\th\hat{\omega}$ with $\hat{\omega}\in\S^1(\mn)\eqdefa\{\hat{\om}\in\S^2|\hat{\om}\perp \mn\}$. Then\,  $\mh\cdot\si=(\mh\cdot\mn)\cos\th+\sqrt{1-(\mh\cdot \mn)^2}\sin\th(\mj\cdot\hat{\omega})$ and
      \beno
      \<v'\>^2=E(\th)+\sin\theta(\mj \cdot\hat{\omega})\tilde{h}
      \eeno
      with
      $\tilde{h}:=\f12\sqrt{|v+v_*|^2|v-v_*|^2-(v+v_*)^2\cdot(v-v_*)^2}=\sqrt{|v|^2|v_*|^2-(v\cdot v_*)^2}.$
      \item If $\omega = \frac {\sigma - (\sigma \cdot \mn)\mn } {|\sigma - (\sigma \cdot \mn)\mn |}$(which implies $\omega \perp (v-v_*)$), then
      \ben\label{Eth}
      \langle v' \rangle^2 = E(\theta)+  \sin  \theta |v-v_*| v \cdot \omega, \quad\mbox{which implies}\quad \sin  \theta |v-v_*| v \cdot \omega = \sin\theta(\mj \cdot\hat{\omega})\tilde{h}.
      \een
      We also have $\omega = \tilde{\omega} \cos \frac \theta 2  + \frac {v'-v_*} {|v'-v_*|} \sin \frac \theta 2$, where $\tilde{\omega}=(v'-v)/|v'-v|$.
    \end{itemize}
  \end{lem}

\begin{lem}\label{VQ}
	Suppose $\ga<0,\ga+2s>-1$ and $k\geq14$. Then for any $s\leq \mathbf{a}\leq1$ and some $\varrho<s$,
	\beno
	&&|(\<v\>^kQ(g,f)-Q(g,\<v\>^kf),h\<v\>^k)|\leq C_k\Big(\|g\|_{L^2_{14}}\|f\|_{L^2_{k+\ga/2}}\|h\|_{L^2_{k+\ga/2}}+ \|f\|_{L^2_{14}}\|g\|_{L^2_{k+\ga/2}}\|h\|_{L^2_{k+\ga/2}}\\
	&&+\|f\|_{L^2_{14}}\|g\|_{H^\varrho_{k-1+\ga/2}}\|h\|_{H^\varrho_{k+\ga/2}}+\|g\|_{L^2_{14}}\|f\|_{H^s_{k+s-\mathbf{a}+\ga/2}}\|h\|_{H^\varrho_{k+\mathbf{a}-1+\ga/2}}+\|g\|_{L^2_{14}}\|f\|_{H^s_{k-1+\ga/2}}\|h\|_{H^\varrho_{k+\ga/2}}\Big).
	\eeno
\end{lem}
\begin{proof}[Proof of Lemma \ref{VQ}]
With the above decomposition in hand, it is not difficult to see that
	\beno
	&&(\<v\>^kQ(g,f)-Q(g,\<v\>^kf),h\<v\>^k)
	=\int_{\R^6\times\S^2}b(\cos\th)|v-v_*|^\ga g(v_*)f(v) h(v')\<v'\>^k(E(\th)^{k/2}-\<v\>^k)dvdv_*d\si\\
	&&+\int_{\R^6\times\S^2}b(\cos\th)|v-v_*|^\ga g(v_*)f(v) h(v')\<v'\>^k(\<v'\>^k-E(\th)^{k/2})dvdv_*d\si:=\mathscr{D}+\mathscr{E}.
	\eeno
	
	\underline{\it Step ~1: Estimate of $\mathscr{D}$.} By the definition of $E(\th)$, we have that
	\ben\label{Eexpan}
	&&|E(\th)^{k/2}-\<v\>^k|\le C_k\big((1-(\cos^2(\th/2))^{k/2})\<v\>^k+(\sin^2(\th/2))^{k/2}\<v_*\>^k+\<v\>^{k-2}\<v_*\>^2\\
	&&\notag\times\sin^2(\th/2)+\<v\>^{k-4}\<v_*\>^4\sin^2(\th/2)+\<v\>^2\<v_*\>^{k-2}\sin^{k-4}(\th/2)\big).
	\een
	
	\noindent$\bullet$ For the term containing $(1-(\cos^2(\th/2))^{k/2})\<v\>^k$, we consider it by two cases: $|v-v_*|\leq1$ and $|v-v_*|>1$. If $|v-v_*|\leq1$, then $\<v\>\sim \<v_*\>$. By Lemma \ref{L110}, we get that
	$\int_{|v-v_*|\leq1}b(\cos\th)|v-v_*|^\ga |g(v_*)f(v) h(v')|\<v'\>^k(1-(\cos^2(\th/2))^{k/2})\<v\>^kdvdv_*d\si\ls C_k\|g\|_{L^2_{|\ga|+4}}\|f\|_{H^\varrho_{k-2+\ga/2}}\|h\|_{H^\varrho_{k+\ga/2}}.$ 
	While if $|v-v_*|>1$, since $\ga<0$, then $|v-v_*|^\ga\sim\<v-v_*\>^\ga\ls\<v_*\>^{|\ga|}\<v\>^\ga$ and $\<v'\>^{|\ga|/2}\leq\<v\>^{|\ga|/2}\<v_*\>^{|\ga|/2}$, we get that
$ \int_{|v-v_*|>1}b(\cos\th)|v-v_*|^\ga |g(v_*)f(v) h(v')|\<v'\>^k(1-(\cos^2(\th/2))^{k/2})\<v_*\>^kdvdv_*d\si 
 \ls C_k\|g\|_{L^2_{|\ga|+4}}\|f\|_{L^2_{k+\ga/2}}\|h\|_{L^2_{k+\ga/2}}.$
 
	 \noindent$\bullet$ For the term containing $(\sin^2(\theta/2))^{k/2}\<v_*\>^k$, also by Lemma \ref{L110} and similar argument in the above, we can obtain that
	\beno
&&\int_{\R^6\times\S^2}b(\cos\th)|v-v_*|^\ga |g(v_*)f(v) h(v')|\<v'\>^k(\sin^2(\theta/2))^{k/2}\<v_*\>^kdvdv_*d\si\\
&\ls& C_k\|f\|_{L^2_{|\ga|+4}}\|g\|_{H^\varrho_{k-2+\ga/2}}\|h\|_{H^\varrho_{k+\ga/2}}+C_k\|f\|_{L^2_{|\ga|+4}}\|g\|_{L^2_{k+\ga/2}}\|h\|_{L^2_{k+\ga/2}}.
	\eeno
The remaining terms can be handle similarly, we omit the details and conclude that
	\beno &&\mathscr{D}\ls C_k\big( \|g\|_{L^2_{|\ga|+6}}\|f\|_{H^\varrho_{k-2+\ga/2}}\|h\|_{H^\varrho_{k+\ga/2}}+\|f\|_{L^2_{|\ga|+6}}\|g\|_{H^\varrho_{k-2+\ga/2}}\|h\|_{H^\varrho_{k+\ga/2}}\\
	&&+\|g\|_{L^2_{|\ga|+6}}\|f\|_{L^2_{k+\ga/2}}\|h\|_{L^2_{k+\ga/2}}+\|f\|_{L^2_{|\ga|+6}}\|g\|_{L^2_{k+\ga/2}}\|h\|_{L^2_{k+\ga/2}}\Big).
	\eeno
	
	\underline{\it Step ~2: Estimate of $\mathscr{E}$.}  Thanks to Lemma \ref{L18} and Taylor expansion, we have
	\ben\label{v'}
	\notag\<v'\>^k-(E(\th))^{k/2}&=&\frac{k}{2}\big[(\<v\>^2\cos^2(\th/2))^{k/2-1}+((E(\th))^{k/2-1}-(\<v\>^2\cos^2(\th/2))^{k/2-1})\big]|v-v_*|\sin\th(v\cdot\om)\\
	&&+\frac{k(k-2)}{4}\int_0^1(1-t)(E(\th)+t\tilde{h}\sin\th(\mj\cdot\hat{\om}))^{k/2-2}dt \tilde{h}^2\sin^2\th(\mj\cdot\hat{\om})^2,
	\een
which yileds that
	\beno
	\mathscr{E}&\leq &\left|k\int_{\R^6\times\S^2}b(\cos\th)\cos^{k-1}(\th/2)\sin(\th/2)(v\cdot\omega)|v-v_*|^{1+\ga}g(v_*)f(v)\<v\>^{k-2}h(v')\<v'\>^kdvdv_*d\si\right|\\
	&&+k\int_{\R^6\times\S^2}b(\cos\th)\sin\th|v-v_*|^{1+\ga}|v\cdot\om||(E(\th))^{k/2-1}-(\<v\>^2\cos^2(\th/2))^{k/2-1}||g(v_*)|f(v)|\\
	&&\times|h(v')\<v'\>^k|dvdv_*d\si+k^2\int_{\R^6\times\S^2}\int_0^1b(\cos\th)|v-v_*|^{\ga}(1-t)(E(\th)+t\tilde{h}\sin\th(\mj\cdot\tilde{\omega}))^{\f{k}2-2}\tilde{h}^2\sin^2\th(\mj\cdot\hat{\om})^2\\
	&&\times|g(v_*)|f(v)|h(v')\<v'\>^k|dvdv_*d\si dt:=\mathscr{E}_1+\mathscr{E}_2+\mathscr{E}_3.
	\eeno

	\noindent$\bullet\,$\underline{\it Estimate of $\mathscr{E}_2$ and $\mathscr{E}_3$.} Observe that $(E(\th))^{\f k2-1}-(\<v\>^2\cos^2(\f{\th}2))^{\f k2-1}$ enjoys the similar structure as \eqref{Eexpan}, i.e.
	\ben\label{Ek}
	&&\notag|(E(\th))^{k/2-1}-(\<v\>^2\cos^2(\th/2))^{k/2-1}|\leq (\sin^2(\th/2))^{k/2-1}\<v_*\>^{k-2}+\<v\>^{k-4}\<v_*\>^2\sin^2(\th/2)\\
	&& +\<v\>^2\<v_*\>^{k-4}\sin^{k-4}(\th/2).
	\een
Since $|v-v_*||v\cdot\om|=|v-v_*||v_*\cdot\om|\ls \<v\>\<v_*\>$(see Lemma \ref{L18}), by the same argument used for $\mathscr{D}$, we can derive that
	\beno
	&&|\mathscr{E}_2|\ls C_k\big( \|g\|_{L^2_{|\ga|+6}}\|f\|_{H^\varrho_{k-2+\ga/2}}\|h\|_{H^\varrho_{k+\ga/2}}+\|f\|_{L^2_{|\ga|+6}}\|g\|_{H^\varrho_{k-2+\ga/2}}\|h\|_{H^\varrho_{k+\ga/2}}\\
	&&+\|g\|_{L^2_{|\ga|+6}}\|f\|_{L^2_{k+\ga/2}}\|h\|_{L^2_{k+\ga/2}}+\|f\|_{L^2_{|\ga|+6}}\|g\|_{L^2_{k+\ga/2}}\|h\|_{L^2_{k+\ga/2}}\Big).
	\eeno
 $\mathscr{E}_3$ can be bounded similarly as $|\mathscr{E}_2|$. This is mainly due to the fact that $|\tilde{h}\sin\th(\mj\cdot\hat{\om})|=|v-v_*|\sin\th|(v\cdot\om)|=|v-v_*|\sin\th|(v_*\cdot\om)|\ls E(\th)\le \max\{\<v\>^2\cos^2(\th/2),\<v_*\>^2\sin^2(\th/2)\}$.
	
	\noindent$\bullet$\underline{\it Estimate of $\mathscr{E}_1$.} Finally, we estimate the term $\mathscr{E}_1$. Still by Lemma \ref{L18}, we have
	\beno
	|\mathscr{E}_1|&\leq& k\left|\int_{\R^6\times\mathbb{S}^2} b(\cos \theta) |v-v_*|^{1+\ga} ( v_*\cdot \tilde{\omega}) \cos^k(\th/2) \sin(\th/2) g(v_*) f(v)\<v\>^{k-2}  h(v')\<v'\>^k dvdv_* d\sigma\right|
	\\
	&& +k\left|\int_{\R^6\times\mathbb{S}^2} b(\cos \theta) |v-v_*|^{1+\ga}  (v_*\cdot \frac {v'-v_*} {|v'-v_*|}) \cos^{k-1}(\th/2)\sin^2(\th/2) g(v_*)f(v)\<v\>^{k-2} h(v')\<v'\>^k dvdv_* d\sigma\right|
	\\
	&:=& \mathscr{E}_{1, 1} +\mathscr{E}_{1, 2}.
	\eeno
	
	For $\mathscr{E}_{1, 2}$,  we copy  the argument used in the estimate of $\mathscr{D}$ to have
	\beno
	\mathscr{E}_{1, 2} &\lesssim&   k\int_{\R^6\times\mathbb{S}^2} b(\cos \theta) \sin^{2} \frac \theta 2 \cos^k\f\th 2|v-v_*|^{1+\gamma} \langle v_* \rangle  |g_*| |f| \langle v \rangle^{k-2}   |h'|  \langle v' \rangle^k dvdv_* d\sigma\\
	&\ls&C_k\big(\|g\|_{L^2_{|\ga|+6}}\|f\|_{H^\varrho_{k+\ga/2-2}}\|h\|_{H^\varrho_{k+\ga/2}}+\|g\|_{L^2_{|\ga|+6}}\|f\|_{L^2_{k+\ga/2}}\|h\|_{L^2_{k+\ga/2}}\big).
	\eeno
	For $\mathscr{E}_{1, 1}$,  since $s\leq\mathbf{a}\leq1$, we have
	\beno
	&&\mathscr{E}_{1, 1} =k \left|\int_{\R^6\times\mathbb{S}^2} b(\cos \theta) |v-v_*|^{1+\gamma} (v_* \cdot \tilde{\omega}) \cos^k \frac \theta 2 \sin \frac \theta 2 g_* h' \langle v' \rangle^{k} \frac 1 {\langle v \rangle^{2-\mathbf{a}} }(f \langle v \rangle^{k-\mathbf{a}} -f' \langle v' \rangle^{k-\mathbf{a}})dv dv_* d\sigma\right|\\
	&&+k \left|\int_{\R^6\times\mathbb{S}^2} b(\cos \theta) |v-v_*|^{1+\gamma} (v_* \cdot \tilde{\omega}) \cos^k \frac \theta 2 \sin \frac \theta 2 g_*   h'f' \langle v' \rangle^{2k-\mathbf{a}} (\frac 1 {\langle v \rangle^{2-\mathbf{a}}}- \frac 1 {\langle v' \rangle^{2-\mathbf{a}}}) dv dv_* d\sigma\right|
	:=\mathscr{E}_{1,1,1} + \mathscr{E}_{1,1,2}.
	\eeno
	We first give the estimate to $\mathscr{E}_{1,1,2}$. By combining  the fact $|\langle v \rangle^{-(2-\mathbf{a})}-\langle v' \rangle^{-(2-\mathbf{a})}|\ls |v-v_*| \sin \frac \theta 2\langle v' \rangle^{-3+\mathbf{a}}\langle v_*\rangle^{2}$, the   change of variable and splitting of the integration domain into $|v-v_*|>1$ and $|v-v_*|\leq1$,
	we   get that
	\beno
	\mathscr{E}_{1,1,2}
	&\lesssim&C_k\big(\|g\|_{L^2_{|\ga|+6}}\|f\|_{H^\varrho_{k-4+\ga/2}}\|h\|_{H^\varrho_{k+\ga/2}}+k^s\|g\|_{L^2_{|\ga|+6}}\|f\|_{L^2_{k-1+\ga/2}}\|h\|_{L^2_{k+\ga/2}}\big).
	\eeno
	For  $\mathscr{E}_{1, 1, 1}$, by Cauchy-Schwarz inequality, we have
	\beno
	\mathscr{E}_{1,1,1} &\lesssim & k \left( \int_{\R^6\times\mathbb{S}^2} b(\cos \theta) |v-v_*|^{\gamma}   |g_*| (f \langle v \rangle^{k-\mathbf{a}} -f' \langle v' \rangle^{k-\mathbf{a}})^2  dv dv_* d\sigma \right)^{1/2}
	\\
	&&\times\left(\int_{\R^6\times\mathbb{S}^2} b(\cos \theta) \frac{ |v-v_*|^{\gamma+2} \langle v_*\rangle^2} {\langle v \rangle^{4-2\mathbf{a}}} \cos^{2k} \frac \theta 2\sin^2 \frac \theta 2 |g_*| |h'|^2\langle v' \rangle^{2k}   dv dv_* d\sigma \right)^{1/2}:=k(\mathcal{I}_1)^{\f12}(\mathcal{I}_2)^{\f12}.
	\eeno
	Since $|v-v_*|^2 \langle v_*\rangle^2\langle v \rangle^{-(4-2\mathbf{a})}\<v'\>^{2-2\mathbf{a}}\lesssim \langle v_* \rangle^8$,
	by regular change of variable, we   derive that
	\beno
\mathcal{I}_2\ls
	C_k\big(\Vert g \Vert_{L^2_{|\gamma| +8}} \Vert h \Vert_{H^\varrho_{k-2+\mathbf{a}+\gamma/2}}^2+\Vert g \Vert_{L^2_{|\gamma| +8}} \Vert h \Vert_{L^2_{k-1+\mathbf{a}+\gamma/2}}^2\big).
	\eeno
	Observing that $(a-b)^2 =-2a(b-a)+ (b^2 -a^2)$, we have
	\begin{equation*}
		\begin{aligned}
			&\mathcal{I}_1=-2(Q(|g|, f \langle \cdot \rangle^{k-\mathbf{a}}) , f \langle \cdot \rangle^{k-\mathbf{a}} )+\int_{\R^3} \int_{\R^3} \int_{\mathbb{S}^2} b(\cos \theta) |v-v_*|^{\gamma}  |g_*| (|f'|^2 \langle v' \rangle^{2k-2\mathbf{a}}  - |f|^2 \langle v \rangle^{2k-2\mathbf{a}} )  dv dv_* d\sigma.
		\end{aligned}
	\end{equation*}
	By cancellation lemma  and Lemma \ref{upperQ}, we have $\mathcal{I}_1  \lesssim \Vert g \Vert_{L^2_{|\gamma| +7 }} \Vert f \Vert_{H^s_{k+s-\mathbf{a}+\gamma/2}}^2$, which implies that
	\beno
	\mathscr{E}_{1,1,1}\ls C_k\big(\|g\|_{L^2_{|\ga|+8}}\|f\|_{H^s_{k+s-\mathbf{a}+\ga/2}}\|h\|_{H^\varrho_{k-2+\mathbf{a}+\ga/2}}+\|g\|_{L^2_{|\ga|+8}}\|f\|_{H^s_{k+s-\mathbf{a}+\ga/2}}\|h\|_{L^2_{k-1+\mathbf{a}+\ga/2}}\big).
	\eeno
	From this together with the estimate of $\mathscr{E}_{1,1,2}$ and $\mathscr{E}_{1, 2}$, we  arrive at that
	\beno
	|\mathscr{E}_1|&\ls& C_k\big(\|g\|_{L^2_{|\ga|+8}}\|f\|_{H^s_{k+s-\mathbf{a}+\ga/2}}\|h\|_{H^\varrho_{k-2+\mathbf{a}+\ga/2}}+\|g\|_{L^2_{|\ga|+8}}\|f\|_{H^s_{k+s-\mathbf{a}+\ga/2}}\|h\|_{L^2_{k-1+\mathbf{a}+\ga/2}}\\
	&&+\|g\|_{L^2_{|\ga|+6}}\|f\|_{H^\varrho_{k+\ga/2-2}}\|h\|_{H^\varrho_{k+\ga/2}}+\|g\|_{L^2_{|\ga|+6}}\|f\|_{L^2_{k+\ga/2}}\|h\|_{L^2_{k+\ga/2}}\big).
	\eeno
	
	By gathering together the estimates of $\mathscr{D}$ and $\mathscr{E}$, we conclude the desired result.
\end{proof}

Next, we will give the estimates of the commutators between $\F_j(j\geq-1)$ and collision operator $Q$. We begin with a technical lemma

\begin{lem}\label{31}
	For smooth functions $g,h,f$, it holds that
	\begin{itemize}
		\item[(i).] If $p>j+3N_0$,
		\ben\label{2.30}\qquad
		&&\|\<\F_j Q_{-1}(\tF_pg,\F_ph),\F_j f\>\|
		\ls2^{-(3-2s)(p-j)}2^{ (2s-1/2)^+j}\|g\|_{L^2}\|\|\F_ph\|_{L^2}\|\F_jf\|_{L^2}
		\een
		\item[(ii).]
		$|\<\F_jQ_{-1}(\tF_jg,S_{j-3N_0}h),\F_j f\>|
        \ls 2^{(2s-1/2)^+j}\|h\|_{L^2}\|\tF_jg\|_{L^2}\|\F_jf\|_{L^2}.$
		 
		\item[(iii).] If $|p-j|\leq3N_0$,
		\ben\label{2.32}
		&&|\<\F_j Q_{-1}(S_{p+4N_0}g,\F_ph)-Q_{-1}(S_{p+4N_0}g,\F_j \F_ph),\F_j f\>|\\
		&\ls& \notag2^{ (2s-1/2)^+j}\|g\|_{L^2}\|\F_ph\|_{L^2}\|\F_jf\|_{L^2}1_{s\neq1/2}+2^{ (1/2+\de)j}\|g\|_{L^2}\|\F_ph\|_{L^2}\|\F_jf\|_{L^2}1_{s=1/2}.
		\een
	\end{itemize}
\end{lem}
\begin{proof}  We first give a detailed proof for the case $2s\geq1$ and $j\geq0$. The case $2s<1$ and $j=-1$ can be handled  in a similar way.
	
	\underline{\it Step 1: Proof of $(i)$.} Denote $T:=|\<\F_j Q_{-1}(\tF_pg,\F_ph),\F_j f\>|$. Then by  (\ref{bobylev}), we have
	\beno
	&&T
	=\Big|\int_{\si\in \mathbb{S}^2,\eta,\xi\in \R^3}b(\frac{\xi}{|\xi|}\cdot\si)[\cF(\Phi_{-1}^\ga)(\eta-\xi^-)-\cF(\Phi_{-1}^\ga)(\eta)](\cF \tF_pg)(\eta)(\cF\F_p h)(\xi-\eta) \vphi(2^{-j}\xi)\overline{(\cF f)}(\xi)\\
	&&\times\vphi(2^{-j}\xi)d\si d\eta d\xi\Big|
	\ls\Big|\int_{\si\in \mathbb{S}^2,\eta,\xi\in \R^3}b(\frac{\xi}{|\xi|}\cdot\si)\xi^-\cdot\nabla\cF(\Phi^\ga_{-1})(\eta)(\cF \tF_pg)(\eta)(\cF\F_p h)(\xi-\eta) \vphi(2^{-j}\xi)\overline{(\cF f)}(\xi)\\
	&&\times\vphi(2^{-j}\xi)d\si d\eta d\xi\Big|+\Big|\int_0^1\int_{\si\in \mathbb{S}^2,\eta,\xi\in \R^3}b(\frac{\xi}{|\xi|}\cdot\si)\nabla^2\cF(\Phi^{\ga}_{-1})(\eta-t\xi^-):\xi^-\otimes\xi^-|(\cF \tF_pg)(\eta)(\cF\F_p h)(\xi-\eta)\\
	&&\times\vphi(2^{-j}\xi)\overline{(\cF f)}(\xi) \vphi(2^{-j}\xi)d\si d\eta d\xi dt\Big|.
	\eeno
	Since  $p>j+3N_0$, we have $\<\eta\>\sim\<\xi-\eta\>\sim\<\eta-t\xi^-\>\sim 2^p$ for $t\in[0,1]$. By  Lemma \ref{lemma1.5}, we derive that
	\beno
	T&\ls&\int (\<\eta\>^{-(\ga+4)}|\xi|+\<\eta\>^{-(\ga+5)}|\xi|^2)|(\cF \tF_pg)(\eta)(\cF\F_p h)(\xi-\eta)\vphi(2^{-j}\xi)(\cF f)(\xi)|d\xi d\eta\\
	&\ls&2^{-(\ga+4)p}\int_{|\xi|\sim 2^j}|\xi||\vphi(2^{-j}\xi)(\cF f)(\xi)|\int_{\eta\in\R^3}|(\cF \tF_pg)(\eta)(\cF\F_p h)(\xi-\eta)|d\eta d\xi.
	\eeno
	By Cauchy-Schwartz inequality, we observe that
	\beno
	T&\ls&2^{-(\ga+4)p}2^{\frac{5}{2} j}\|\tF_pg\|_{L^2}\|\F_ph\|_{L^2}\|\F_jf\|_{L^2}\ls2^{-(\ga+2s+1)p}2^{-(3-2s)(p-j)}2^{(2s-1/2)j}\|\tF_pg\|_{L^2}\|\F_ph\|_{L^2}\|\F_jf\|_{L^2}.
	\eeno

	\underline{\it Step 2: Proof of (iii).} Set $R:=|\<\F_j Q_{-1}(S_{p+4N_0}g,\F_ph)-Q_{-1}(S_{p+4N_0}g,\F_j\F_ph),\F_jf\>|$. By \eqref{bobylev}, one has
	\beno
	R&=&|\int_{\si\in \mathbb{S}^2,\eta,\xi\in \R^3}b(\frac{\xi}{|\xi|}\cdot\si)[\cF(\Phi_{-1}^\ga)(\eta-\xi^-)-\cF(\Phi_{-1}^\ga)(\eta)](\cF S_{p+4N_0} g)(\eta)(\cF \F_ph)(\xi-\eta)\\
	\notag &&\times\<\xi\>^\ell \vphi(2^{-j}\xi)\overline{(\cF f)}(\xi)( \vphi(2^{-j}\xi)- \vphi(2^{-j}(\xi-\eta))d\si d\eta d\xi|.
	\eeno
	We split   $R$  into two parts: $R_{1}$ and $R_{2}$, which correspond to the integration domain of $R$: $2|\xi^-|\leq \<\eta\>$ and $2|\xi^-|>\<\eta\>$ respectively. The proof falls in several steps.
	\smallskip
	
	\noindent\underline{\it Step 2.1: Estimate of $R_1$.} In the region $2|\xi^-|\leq \<\eta\>$, we have $\sin(\th/2)\leq\<\eta\>/|\xi|$ and $\<\eta-t\xi^-\>\sim\<\eta\>$ for $t\in[0,1]$. By Taylor expansion, Lemma \ref{lemma1.5} and the fact that $| \vphi(2^{-j}\xi)-\vphi(2^{-j}(\xi-\eta))|\ls 2^{-j}|\eta|$, we have
	\beno
	|R_{1}|&\leq&\bigg|\int_{2|\xi^-|\leq \<\eta\>}b(\frac{\xi}{|\xi|}\cdot\si)(\na \cF(\Phi^\ga_{-1}))(\eta)\cdot\xi^-(\cF S_{p+4N_0}g)(\eta)(\cF\F_ph)(\xi-\eta) \vphi(2^{-j}\xi)\overline{(\cF f)}(\xi)( \vphi(2^{-j}\xi)\\
	&&- \vphi(2^{-j}(\xi-\eta))d\si d\eta d\xi\bigg|+\bigg|\int_0^1\int_{{2|\xi^-|\leq \<\eta\>}}b(\frac{\xi}{|\xi|}\cdot\si)(\na^2\cF(\Phi^\ga_{-1}))(\eta-t\xi^-):\xi^-\otimes\xi^-)\\
	&&\times(\cF S_{p+4N_0}g)(\eta)(\cF\F_ph)(\xi-\eta) \vphi(2^{-j}\xi)\overline{(\cF f)}(\xi)( \vphi(2^{-j}\xi)- \vphi(2^{-j}(\xi-\eta))d\si d\eta d\xi dt\bigg|\\
	&\ls& R_{1,1}+R_{1,2},\quad\mbox{where}\\
	R_{1,1}&=& 2^{ -j}\int_{}|\eta|\<\eta\>^{-(\ga+4)}|\xi|\min\{1,(\<\eta\>/|\xi|)^{2-2s}\}|(\cF S_{p+4N_0}g)(\eta)(\cF\F_ph)(\xi-\eta)\vphi(2^{-j}\xi)\overline{(\cF f)}(\xi)|d\eta d\xi,\\
	R_{1,2}&=& 2^{ -j}\int_{}|\eta|\<\eta\>^{-(\ga+5)}|\xi|^2(\<\eta\>/|\xi|)^{2-2s}|(\cF S_{p+4N_0}g)(\eta)(\cF\F_ph)(\xi-\eta)\vphi(2^{-j}\xi)\overline{(\cF f)}(\xi)|d\eta d\xi.
	\eeno
	
	\noindent$\bullet$ \underline{Estimate of $R_{1,1}$.}  We split the integration domain of $R_{1,1}$ into two parts: $\<\eta\>\geq |\xi|$ and $\<\eta\>< |\xi|$, which are due to $\min\{1,(\<\eta\>/|\xi|)^{2-2s}\}$. Recalling that $\<\eta\>\ls 2^j,\<\xi\>\sim 2^j$ and $p\sim j$, we have
	\beno
	R_{1,1}
	&\ls&2^{-j}\int_{\<\eta\>\geq|\xi|}\<\eta\>^{-(\ga+3)}|\xi||(\cF S_{p+4N_0}g)(\eta)(\cF\F_ph)(\xi-\eta)\vphi(2^{-j}\xi)\overline{(\cF f)}(\xi)|d\eta d\xi\\
	&&+2^{-j}\int_{{\<\eta\><|\xi|}}\<\eta\>^{-(\ga+2s+1)}|\xi|^{2s-1}|(\cF S_{p+4N_0}g)(\eta)(\cF\F_ph)(\xi-\eta)\vphi(2^{-j}\xi)\overline{(\cF f)}(\xi)|d\eta d\xi\\&\ls&2^{-(\ga+2s+1)j}2^{(2s-1/2)j}\|S_{p+4N_0}g\|_{L^2}\|\F_ph\|_{L^2}\|\F_jf\|_{L^2}.
	\eeno
	
	\noindent$\bullet$ \underline{Estimate of $R_{1,2}$.} We have
	$R_{1,2}
    \ls 2^{-(\ga+2s+1)j}2^{(2s-1/2)j}\|S_{p+4N_0}g\|_{L^2}\|\F_ph\|_{L^2}\|\F_jf\|_{L^2}.$
	Patching together all the estimates, we have
$ 
  R_{1}\ls2^{-(\ga+2s+1)j}2^{(2s-1/2)j}\|S_{p+4N_0}g\|_{L^2}\|\F_ph\|_{L^2}\|\F_jf\|_{L^2}.
 $

	\noindent\underline{\it Step 2.2: Estimate of $R_2$.} We first note that
	\beno
	|R_{2}|
	&\ls&2^{ -j}\int_{2|\xi^-|\geq\<\eta\>}|\eta|b(\frac{\xi}{|\xi|}\cdot\si)|\cF(\Phi^\ga_{-1}(\eta)||(\cF S_{p+4N_0}g)(\eta)(\cF\F_ph)(\xi-\eta)\vphi(2^{-j}\xi)\overline{(\cF f)}(\xi)|d\si d\eta d\xi\\
	&+&2^{ -j}\int_{2|\xi^-|\geq\<\eta\>}|\eta-\xi^-|b(\frac{\xi}{|\xi|}\cdot\si)|\cF(\Phi^\ga_{-1})(\eta-\xi^-)|(\cF S_{p+4N_0}g)(\eta)(\cF\F_ph)(\xi-\eta)\vphi(2^{-j}\xi)\overline{(\cF f)}(\xi)|d\si d\eta d\xi\\
	&+&2^{ -j}\int_{2|\xi^-|\geq\<\eta\>}|\xi^-|b(\frac{\xi}{|\xi|}\cdot\si)|\cF(\Phi^\ga_{-1})(\eta-\xi^-)|(\cF S_{p+4N_0}g)(\eta)(\cF\F_ph)(\xi-\eta)\vphi(2^{-j}\xi)\overline{(\cF f)}(\xi)|d\si d\eta d\xi\\
	&:=& R_{2,1}+R_{2,2}+R_{2,3}.
	\eeno
	Since  $2|\xi^-|>\<\eta\>$, in what follows, we will frequently use the facts that \ben\label{thetaR} \sin(\th/2)\gs\<\eta-\xi^-\>/(3|\xi|), \quad \sin(\th/2)\geq\<\eta\>/(2|\xi|). \een
	\noindent$\bullet$ \underline{Estimate of $R_{2,1}$.} We have
	$R_{2,1}
    \ls2^{-j}\int\,\<\eta\>^{-(\ga+2s+2)}|\xi|^{2s}|(\cF S_{p+4N_0}g)(\eta)(\cF\F_ph)(\xi-\eta)\vphi(2^{-j}\xi)\overline{(\cF f)}(\xi)|d\eta d\xi.$
	Similar to $R_{1,2}$, we easily have
	$R_{2,1}
	\ls 2^{-(\ga+2s+1)j}2^{ (2s-1/2)j}\|S_{p+4N_0}g\|_{L^2}\|\F_ph\|_{L^2}\|\F_jf\|_{L^2}.$

	\noindent$\bullet$ \underline{Estimate of $R_{2,2}$.}
	We may split the integration domain of $R_{2,2}$ into two parts:  $|\eta-\xi^-|\geq\<\eta\>$ and $|\eta-\xi^-|<\<\eta\>$. In the region $|\eta-\xi^-|\geq\<\eta\>$, thanks to Lemma \ref{lemma1.5}, one may get $|\eta-\xi^-||\cF(\Phi^\ga_{-1})(\eta-\xi^-)|\ls \<\eta\>^{-(\ga+2)}$. While in the region $|\eta-\xi^-|\leq\<\eta\>$, we use the change of variables: $\eta-\xi^-\rightarrow \tilde{\eta}$ and \eqref{thetaR}. Then  we have
	\beno
	R_{2,2}
	&\ls&2^{-j}\int_{\eta,\xi}\<\eta\>^{-(\ga+2s+2)}|\xi|^{2s}|(\cF S_{p+4N_0}g)(\eta)(\cF\F_ph)(\xi-\eta)\vphi(2^{-j}\xi)\overline{(\cF f)}(\xi)|d\eta d\xi\\&&+2^{ -j}\left(\int_{\sin(\theta/2)\ge\<\eta\>/(2|\xi|)}b(\frac{\xi}{|\xi|}\cdot\si)|\cF S_{p+4N_0}g(\eta)|^2|\cF\F_jf(\xi)|^2\<\xi\>^{2s}\<\eta\>^{2s}d\si d\eta d\xi\right)^{1/2}\\
	&&\times\left(\int_{\sin(\theta/2)\gs\<\tilde{\eta}\>/(3|\xi|)}b(\frac{\xi^+}{|\xi^+|}\cdot\si)|\tilde{\eta}|^2|\cF\Phi^\ga_{-1}(\tilde{\eta})|^2\<\xi^+\>^{-2s}\<\tilde{\eta}\>^{-2s}|\F_ph(\xi^+-\tilde{\eta})|^2d\si d\tilde{\eta} d\xi^+\right)^{1/2}.\eeno
	From this, we get that
	$R_{2,2}\ls2^{(2s-1/2)j}\|S_{p+4N_0}g\|_{L^2}\|\F_ph\|_{L^2}\|\F_jf\|_{L^2}.$

	\noindent$\bullet$ \underline{Estimate of $R_{2,3}$.}
	Similar to $R_{1,2}$, we shall split the integration domain of $R_{2,3}$ into two parts: $|\eta-\xi^-|\geq4\<\eta\>$ and $|\eta-\xi^-|<4\<\eta\>$.  In the region $|\eta-\xi^-|\ge 4\<\eta\>$,  one get that $|\xi^-||\cF(\Phi^\ga_{-1})(\eta-\xi^-)|\ls \<\eta\>^{-(\ga+2)}$.   \\
	$\bullet$  If $2s>1$, by Cauchy-Schwartz inequality, we derive that
	\beno
	&&R_{2,3}
	\ls2^{-j}   \int\<\eta\>^{-(\ga+2s+2)}|\xi|^{2s}|(\cF S_{p+4N_0}g)(\eta)(\cF\F_ph)(\xi-\eta)\vphi(2^{-j}\xi)\overline{(\cF f)}(\xi)|d\eta d\xi+2^{-j}\bigg( \int\mathrm{1}_{\sin(\theta/2)\ge\f{\<\eta\>}{2|\xi|}}
	\\
	&&\times |\xi|^2 \sin(\theta/2)b(\frac{\xi}{|\xi|}\cdot\si)|\cF S_{p+4N_0}g(\eta)|^2|\cF\F_jf(\xi)|^2\<\xi\>^{2s-1}\<\eta\>^{2s-1}d\si d\eta d\xi\bigg)^{1/2}\bigg(\int_{\sin(\theta/2)\gs\<\tilde{\eta}\>/(3|\xi|)}\sin(\theta /2)\\
  &&\times b(\frac{\xi^+}{|\xi^+|}\cdot\si)|\cF\Phi^\ga_{-1}(\tilde{\eta})|^2\<\xi^+\>^{1-2s}\<\tilde{\eta}\>^{1-2s}|\F_ph(\xi^+-\tilde{\eta})|^2d\si d\tilde{\eta} d\xi^+\bigg)^{1/2}\ls2^{(2s-1/2)j}\|S_{p+4N_0}g\|_{L^2}\|\F_ph\|_{L^2}\|\F_jf\|_{L^2}.\eeno
	$\bullet$ If $2s=1$, for any $0<\delta\ll 1$,   we have
	\beno
	&&R_{2,3}\ls2^{-j}\int\<\eta\>^{-(\ga+2s+2)}|\xi|^{2s}|(\cF S_{p+4N_0}g)(\eta)(\cF\F_ph)(\xi-\eta)\vphi(2^{-j}\xi)\overline{(\cF f)}(\xi)|d\eta d\xi+\bigg(\int\mathrm{1}_{\sin(\theta/2)\ge\<\eta\>/(2|\xi|)}\\
  &&\times\sin^{1-2\de}(\th/2)b(\frac{\xi}{|\xi|}\cdot\si)|\cF S_{p+4N_0}g(\eta)|^2|\cF\F_jf(\xi)|^2d\si d\eta d\xi\bigg)^{1/2}\bigg(\int_{\sin(\theta/2)\gs\<\tilde{\eta}\>/(3|\xi|)}\sin^{1+2\de}(\th/2)b(\frac{\xi^+}{|\xi^+|}\cdot\si)\\
  &&\times|\cF\Phi^\ga_{-1}(\tilde{\eta})|^2|\F_ph(\xi^+-\tilde{\eta})|^2d\si d\tilde{\eta} d\xi^+\bigg)^{1/2}\ls2^{(1/2+\de)j}\|S_{p+4N_0}g\|_{L^2}\|\F_ph\|_{L^2}\|\F_jf\|_{L^2}.
	\eeno
	We derive  that
 $R_{2,3}\ls2^{( 2s-1/2)j}\|g\|_{L^2}\|\F_ph\|_{L^2}\|\F_jf\|_{L^2}1_{s>1/2}+2^{(1/2+\de)j}\|g\|_{L^2}\|\F_ph\|_{L^2}\|\F_jf\|_{L^2}1_{s=1/2}\},$
	from which together with the estimates of $R_{2,1}$ and  $R_{2,2}$, we get that
	\beno
	R_{2}\ls2^{(2s-1/2)j}\|g\|_{L^2}\|\F_ph\|_{L^2}\|\F_jf\|_{L^2}1_{s>1/2}+2^{( 1/2+\de)j}\|g\|_{L^2}\|\F_ph\|_{L^2}\|\F_jf\|_{L^2}1_{s=1/2}\}.
	\eeno
	
	Finally patching together all the estimates, we end the proof of (\ref{2.32}).
	\smallskip
	
	\underline{\it Step 3: Proof of $(ii)$.} We denote $P:=|\<\F_j Q_{-1}(\tF_jg,S_{j-3N_0}h),\F_j f\>|$. We have
	\beno
	P&=& \Big|\int_{\si\in \mathbb{S}^2,\eta,\xi\in \R^3}b(\frac{\xi}{|\xi|}\cdot\si)[\cF(\Phi_{-1}^\ga)(\eta-\xi^-)-\cF(\Phi_{-1}^\ga)(\eta)](\cF \tF_jg)(\eta)(\cF S_{j-3N_0} h)(\xi-\eta)\\
	\notag &&\times\<\xi\>^{2\ell} \vphi(2^{-j}\xi)\overline{(\cF f)}(\xi)\vphi(2^{-j}\xi)d\si d\eta d\xi\Big|.
	\eeno
	Notice that $|\xi|\sim 2^j,|\eta|\sim 2^j$ and $|\eta-\xi|\ls 2^{j-3N_0}$, then $|\eta-\xi^-|=|\eta-\xi+\xi^+|\sim 2^j$. We  split the integration domain of $P$ into two parts: $2|\xi^-|\leq \<\eta\>$ and $2|\xi^-|>\<\eta\>$ and denote them by $P_{1}$ and $P_{2}$ respectively.
	We may copy the argument for $R_{1}$ to $P_{1}$ to get that
	$|P_{1}|\ls2^{-(\ga+2s+1)j}2^{(2s-1/2)j}\|\tF_jg\|_{L^2}\|S_{j-3N_0}h\|_{L^2}\|\F_jf\|_{L^2}.
   $
	As fpr $P_{2}$, we have
	\beno
	|P_{2}|
	&\ls&\int_{2|\xi^-|\geq\<\eta\>}b(\frac{\xi}{|\xi|}\cdot\si)|\cF(\Phi^\ga_{-1}(\eta)||(\cF \tF_jg)(\eta)(\cF\F_ph)(\xi-\eta)\vphi(2^{-j}\xi)\overline{(\cF f)}(\xi)|d\si d\eta d\xi\\
	&&+\int_{2|\xi^-|\geq\<\eta\>}b(\frac{\xi}{|\xi|}\cdot\si)|\cF(\Phi^\ga_{-1})(\eta-\xi^-)|(\cF \tF_jg)(\eta)(\cF S_{j-3N_0}h)(\xi-\eta)\vphi(2^{-j}\xi)\overline{(\cF f)}(\xi)|d\si d\eta d\xi.
	\eeno
	Since $\<\eta-\xi^-\>,\<\eta\> \sim2^j$, we deduce that
	$|P_{2}|\ls
  2^{-(\ga+2s+1)j}2^{(2s-1/2)j}\|\tF_jg\|_{L^2}\|S_{j-3N_0}h\|_{L^2}\|\F_jf\|_{L^2}.$
	 
	Patching together the estimate of $P_1$ and $P_2$,  we conclude the desired result.
	We emphasize that it is even easier to get the estimate for the case $2s<1$ since we only need the first order Taylor expansion. Thus we omit the details and end the proof of this lemma.
\end{proof}

\begin{lem}\label{32}
	For smooth functions $g,h,f$ and $N>0$, we have that
	\begin{itemize}
		\item[(i)] If $p>j+3N_0$,
		\ben\label{reg1}
		&&\<\F_j Q_{-1}(\tF_pg,\F_ph),\F_j f\>| \\
		&\ls&\notag C_{N}2^{-(3-2s)(p-j)}\|g\|_{L^2_{(-\omega_3)^++(-\omega_4)^+}}(\|\mF_p h\|_{H^{c}_{\omega_3}}+2^{-pN}\|h\|_{H_{-N}^{-N}})(\|\mF_j f\|_{H^{d}_{\omega_4}}+2^{-jN}\|f\|_{H_{-N}^{-N}}) \een \end{itemize}
		  \begin{itemize}\item[(ii)] \ben\label{reg2}
		&&|\<\F_j Q_{-1}(\tF_jg,S_{j-3N_0}h),\F_j f\>|\ls C_{N}\|h\|_{L^2_{(-\omega_3)^++(-\omega_4)^+}}(\|\mF_jg\|_{H^{c}_{\omega_3}}+2^{-jN}\|g\|_{H_{-N}^{-N}})\\
		&&\notag\times (\|\mF_jf\|_{H^{d}_{\omega_4}}+2^{-jN}\|f\|_{H_{-N}^{-N}}) +C_N2^{-2jN}\|g\|_{L^1}\|h\|_{H_{-N}^{-N}}\|f\|_{H_{-N}^{-N}}.
		\een
		\item[(iii)] If $|p-j|\leq3N_0$,
		\ben\label{reg3}
		\notag&&|\<\F_j Q_{-1}(S_{p+4N_0}g,\F_ph)-Q_{-1}(S_{p+4N_0}g,\F_j \F_ph),\F_jf\>|
		\\&&\ls C_{N}\|g\|_{L^2_{2+(-\omega_3)^++(-\omega_4)^+}}(\|\mF_ph\|_{H^{c}_{\omega_3}}+2^{-pN}\|h\|_{H_{-N}^{-N}})(\|\mF_jf\|_{H^{d}_{\omega_4}}+2^{-jN}\|f\|_{H_{-N}^{-N}}),
		\een
	\end{itemize}
	where $\om_3,\om_4\in\R$
	and $c,d\geq0$ satisfying $c+d=(2s-1/2)^+_{2s\neq1}+( 1/2+\de)1_{2s=1}$ with $0<\de\ll1$, $N\in\N$ can be large enough  and $\mF_j$ is defined in Definition \ref{Fj}. We remark that $\om_3,\om_4$ and $c,d$ can be different in different lines.
\end{lem}
\begin{proof} We only provide the proof for the case $2s>1$ since then case $2s\leq1$ can be handled similarly.
	
	\noindent\underline{\it Step 1: Proof of (i).}  Notice that we have $\<\F_j Q_{-1}(\tF_pg,\F_ph),\F_j f\>
	:= G_1+G_2$,
	where
 $G_1 =\sum\limits_{l\geq N_0}\<\F_j Q_{-1}(\cP_l\tF_pg,\\\F_p\tP_lh),\F_j \tP_lf\>+\sum\limits_{l<N_0}\<\F_j Q_{-1}(\cP_l\tF_pg,\F_p\U_{N_0}h),\F_j\ \U_{N_0}f\>$ and $G_2 =\sum\limits_{l\geq N_0}(\<Q_{-1}(\cP_l\tF_pg,\tP_l\F_ph),(\F_j^2 \tP_l-\tP_l\F_j^2 )f\>\\+\<Q_{-1}(\cP_l\tF_pg,(\tP_l\F_p-\F_p\tP_l)h),\F_j^2\tP_lf\>)+\sum\limits_{l<N_0}(\<Q_{-1}(\cP_l\tF_pg,\F_p\U_{N_0}h),(\F_j^2 \U_{N_0}-\U_{N_0}\F_j^2 )f\>+\<Q_{-1}(\cP_l\tF_pg,\\(\F_p\U_{N_0}-\U_{N_0}\F_p)h)\F_j^2 \U_{N_0}f\>):=G_{2,1}+G_{2,2}+G_{2,3}+G_{2,4}$.

	\noindent \underline{\it Step 1.1 Estimate of $G_1$.} Since $\xi\sim 2^j,\<\eta-\xi\>\sim 2^p$ and $p>j+3N_0$, we have that $\<\eta\>\sim 2^p$, which implies
	\beno
	G_1=\sum_{l\geq N_0}\<\F_j Q_{-1}(\tF_p\cP_l\tF_pg,\F_p\tP_lh),\F_j \tP_lf\>+\sum_{l< N_0}\<\F_j Q_{-1}(\tF_p\cP_l\tF_pg,\F_p\U_{N_0}h),\F_j\U_{N_0}f\>.
	\eeno
	From  Lemma \ref{31}(\ref{2.30}), we have
$  |G_{1}|\ls\sum_{l\geq N_0}2^{-(3-2s)(p-j)}2^{( 2s-1/2)j}\|\cP_l\tF_pg\|_{L^2}\|\F_p\tP_lh\|_{L^2}\|\F_j\tP_lf\|_{L^2}+\sum_{l<N_0}2^{-(3-2s)(p-j)}2^{(2s-1/2)j} \|\cP_l\tF_pg\|_{L^2}\|\F_p\U_{N_0}h\|_{L^2}\|\F_j\U_{N_0}f\|_{L^2}.$
 Thanks to Lemma \ref{lemma1.3}(i), we have 
	\ben\label{FPPLC} \|\F_p\tP_lh\|_{L^2}\ls C_{N}(\|\tP_l\F_ph\|_{L^2}+\sum_{|\al|=1}^{2N }\|\tP_{l,\al}\tF_{p,\al}h\|_{L^2}+2^{-lN}2^{-pN}\|h\|_{H_{-N}^{-N}}), \een
	from this together with  Lemma \ref{lemma1.4}(\ref{7.70}) imply that
	\[ \sum_{l\geq N_0}2^{(2s-1/2)j}\|\cP_l\tF_pg\|_{L^2}\|\F_p\tP_lh\|_{L^2}\|\F_j\tP_lf\|_{L^2}\ls C_{N}\sum_{l\geq N_0}2^{( 2s-1/2)j}\|\cP_l\tF_pg\|_{L^2_{(-\omega_3)^++(-\omega_2)^+}}2^{(\omega_3+\omega_4)l}\big(\|\tP_l\F_ph\|_{L^2}\]
\[+\sum_{|\al|=1}^{2N }\|\tP_{l,\al}\tF_{p,\al}h\|_{L^2}+2^{-lN}2^{-pN}\|h\|_{H_{-N}^{-N}}\big) \big(\|\tP_l\F_jf\|_{L^2}+\sum_{|\al|=1}^{2N }\|\tP_{l,\al}\tF_{j,\al}f\|_{L^2}+2^{-lN}2^{-jN}\|f\|_{H_{-N}^{-N}}\big)\]\[
	 \ls C_{N} \|\tF_pg\|_{L^2_{(-\omega_3)^++(-\omega_2)^+}}(\|\mF_{p}h\|_{H^{c}_{\omega_3}}+2^{-pN}\|h\|_{H_{-N}^{-N}})(\|\mF_{j}f\|_{H^{d}_{\omega_4}}+2^{-jN}\|f\|_{H_{-N}^{-N}}). \]
	We can also copy the above argument to $\|\cP_l\tF_pg\|_{L^2}\|\F_p\U_{N_0}h\|_{L^2}\|\F_j\U_{N_0}f\|_{L^2}$. Thanks to
	facts $\|\tF_pg\|_{L^2_{l}}\ls\|g\|_{L^2_{l}}$(Lemma \ref{lemma1.4}), we conclude that
	\beno
	|G_{1}|&\ls&C_{N}2^{-(3-2s)(p-j)}\|\tF_pg\|_{L^2_{(-\omega_3)^++(-\omega_2)^+}}(\|\mF_{p}h\|_{H^{c}_{\omega_3}}+2^{-pN}\|h\|_{H_{-N}^{-N}})(\|\mF_{j}f\|_{H^{d}_{\omega_4}}+2^{-jN}\|f\|_{H_{-N}^{-N}}).
	\eeno

	\noindent  \underline{\it Step 1.2 Estimate of $G_2$.} We shall give the estimates term by term.
	
	$\bullet$ \underline{\it Estimate of $G_{2,1}$.} We introduce the following decomposition: $G_{2,1}=\sum_{i=1}^5G^{(i)}_{2,1}$ where\\
	$G^{(1)}_{2,1}=\sum\limits_{l\geq N_0}\sum\limits_{|a-p|>N_0}\<Q_{-1}(\cP_l\tF_pg,\F_a\tP_l\F_ph),(\F_j \tP_l-\tP_l\F_j^2)f\>$, $G^{(2)}_{2,1}=\sum\limits_{l\geq N_0}\sum\limits_{|a-p|\leq N_0}\sum\limits_{m<j-N_0}\<Q_{-1}(\cP_l\tF_pg,\\
	\F_a\tP_l\F_ph),\F_m(\F_j^2 \tP_l-\tP_l\F_j^2 )f\>$,  $G^{(3)}_{2,1}=\sum\limits_{l\geq N_0}\sum\limits_{|a-p|\leq N_0}\sum\limits_{j-N_0\leq m<j+N_0}\<Q_{-1}(\cP_l\tF_pg,\F_a\tP_l\F_ph),\F_m(\F_j^2 \tP_l\\-\tP_l\F_j^2 )f\>$,  $G^{(4)}_{2,1}=\sum\limits_{l\geq N_0}\sum\limits_{|a-p|\leq N_0}\sum\limits_{j+N_0\leq m<p-2N_0}\<Q_{-1}(\cP_l\tF_pg,\F_a\tP_l\F_ph),\F_m(\F_j^2 \tP_l-\tP_l\F_j^2 )f\>$ and\\ $G^{(5)}_{2,1}=\sum\limits_{l\geq N_0}\sum\limits_{|a-p|\leq N_0}\sum\limits_{m\geq p-2N_0}\<Q_{-1}(\cP_l\tF_pg,\F_a\tP_l\F_ph),\F_m(\F_j^2 \tP_l-\tP_l\F_j^2 )f\>$.
	
	\underline{\it Estimate of $G^{(1)}_{2,1}$ and $G^{(5)}_{2,1}$.}  From Lemma \ref{lemma1.7}(iv), we have
$|G^{(1)}_{2,1}| 
  \ls\sum_{l\geq N_0}\sum_{|a-p|>N_0}(\|\cP_l\tF_pg\|_{L^1}+\|\cP_l\tF_pg\|_{L^2})\\\|\F_a\tP_l\F_ph\|_{H^{2s}}\|(\F_j^2 \tP_l-\tP_l\F_j^2 )f\|_{L^2}.$
 By Bernstein inequality(see Lemma \ref{7.8}) that $\|\F_pf\|_{L^2}\ls2^{\frac{3}{2}p}\|\F_pf\|_{L^1}$ and Lemma \ref{lemma1.3}(ii), we derive that
	$|G^{(1)}_{2,1}|
    \ls C_{N}2^{-jN-pN}2^{-(p-j)N}\|g\|_{L^1}\|h\|_{H_{-N}^{-N}}\|f\|_{H_{-N}^{-N}}.$
 Similarly, since $p>j+3N_0$, we have $m\geq p-2N_0>j+N_0$. Then Lemma \ref{lemma1.3}(ii) implies that
	\beno
	|G^{(5)}_{2,1}|
	&\ls&\sum_{l\geq N_0} \sum_{m\geq p-2N_0}(\|\cP_l\tF_pg\|_{L^1}+\|\cP_l\tF_pg\|_{L^2})2^{2sm}\|\tF_p\tP_l\F_ph\|_{L^2}\|\F_m\tP_l\F^2_j f\|_{L^2}\\
	&\ls& C_{N}2^{-jN-pN}2^{-(p-j)N}\|g\|_{L^1}\|h\|_{H_{-N}^{-N}}\|f|
	_{H_{-N}^{-N}}.\eeno
	
	\underline{\it Estimate of $G^{(2)}_{2,1}$ and $G^{(4)}_{2,1}$.} It is not difficult to see that
	$G^{(2)}_{2,1}
	=\sum\limits_{l\geq N_0}\sum\limits_{|a-p|\leq N_0}\sum\limits_{m<j-N_0}\<Q_{-1}(\cP_l\tF_pg,\F_a\tP_l\F_ph),\\\F_m\tP_l\F_j^2 f\>.$
	Notice that $m<j-N_0<p-3N_0<a-2N_0$.   Then by Lemma \ref{31} and Lemma \ref{lemma1.3}(ii), we have
	\beno
	|G^{(2)}_{2,1}|
	&\ls&\sum_{l\geq N_0}\sum_{m<j-N_0} 2^{-(3-2s)(p-m)}2^{(2s-1/2)m}(\|\cP_l\tF_pg\|_{L^2}\|\tF_p\tP_l\F_ph\|_{L^2}\|\F_m\tP_l\F^2_jf\|_{L^2}\\
	&\ls& C_{N}2^{-(3-2s)p}2^{-jN}\|\tF_pg\|_{L^2_{(-\omega_3)^++(-\omega_4)^+}}\|\F_ph\|_{L^2_{\omega_3}}\|\mF_jf\|_{H^{-N}_{\om_4}}.
	\eeno
	
	For $G^{(4)}_{2,1}$, we first have $G^{(4)}_{2,1}=-\sum\limits_{l\geq N_0}\sum\limits_{|a-p|\leq N_0}\sum\limits_{j+N_0\leq m<p-2N_0}\<Q_{-1}(\cP_l\tF_pg,\F_a\tP_l\F_ph),\F_m\tP_l\F_j^2 )f\>.$
	Since    $m<p-2N_0\leq a-N_0$, we may copy the argument for $G^{(2)}_{2,1}$ to get that
	\beno
	|G^{(4)}_{2,1}|
	&\ls&C_{N}2^{-(3-2s)p}2^{-jN}\|\tF_pg\|_{L^2_{(-\omega_3)^++(-\omega_4)^+}}\|\F_ph\|_{L^2_{\omega_3}}\|\mF_jf\|_{H^{-N}_{\omega_4}}.
	\eeno

	\underline{\it Estimate of $G^{(3)}_{2,1}$.} We first note that $m<j+N_0<p-2N_0\leq a-N_0$.  Lemma \ref{31}, Lemma \eqref{lemma1.3} and Lemma \ref{lemma1.4}(\ref{7.70}) imply that
	$|G^{(3)}_{2,1}|
    \ls C_{N}2^{-(3-2s)(p-j)}\|\tF_pg\|_{L^2_{(-\omega_3)^++(-\omega_4)^+}}\|\F_ph\|_{H^{c}_{\omega_3}}(\|\mF_jf\|_{H^{d}_{\omega_4}}+2^{-jN}\|f\|_{H_{-N}^{-N}}).$
 	Now putting together all these estimates, we obtain that
	\beno
	|G_{2,1}|
	&\ls&C_{N}2^{-(3-2s)(p-j)}\|\tF_pg\|_{L^2_{(-\omega_3)^++(-\omega_4)^+}}\|\F_ph\|_{H^{c}_{\omega_3}}(\|\mF_jf\|_{H^{d}_{\omega_4}}+2^{-jN}\|f\|_{H_{-N}^{-N}}).
	\eeno
	
	$\bullet$\underline{\it Estimate of $G_{2,4}$.} We set $G_{2,4}:=G_{2,4}^{(1)}+G_{2,4}^{(2)}$, where
	$G_{2,4}^{(1)}=\sum\limits_{l< N_0}\sum\limits_{|a-p|>N_0}\<Q_{-1}(\cP_l\tF_pg, \F_a(\F_p\U_{N_0}-\U_{N_0}\F_p)h),\F_j^2 \U_{N_0}f\>$ and $G_{2,4}^{(2)}=\sum\limits_{l< N_0}\sum\limits_{|a-p|\leq N_0,|m-j|<N_0}\<Q_{-1}(\cP_l\tF_pg,\F_a(\F_p\U_{N_0}-\U_{N_0}\F_p)h),\F_m\F_j^2 \U_{N_0}f\>$.

	\underline{\it Estimate of $G_{2,4}^{(1)}$.} We first observe that $G_{2,4}^{(1)}=\sum\limits_{l< N_0}\sum\limits_{|a-p|>N_0}\<Q_{-1}(\cP_l\tF_pg,\F_a\U_{N_0}\F_ph),\F_j^2\<D\>^{2\ell} \U_{N_0}f\>$.
	Then by Lemma \ref{lemma1.7}(iv), we have
	\beno
	|G_{2,4}^{(1)}|&\ls&\sum_{l< N_0}\sum_{|a-p|>N_0}\|\cP_l\tF_pg\|_{L^1}\|\F_a\U_{N_0}\F_ph\|_{H^{2s}}\|\F_j^2 \U_{N_0}f\|_{L^2}\ls C_N\sum_{l< N_0}\sum_{|a-p|>N_0}2^{-a(4N+1)}2^{-p(4N+1)}2^{2sa}\\&&\times\|\cP_l\tF_pg\|_{L^1}\|\F_ph\|_{L^2}\|\F_j^2 \U_{N_0}f\|_{L^2}\ls C_N2^{-pN-jN}2^{-2(p-j)}\|\tF_pg\|_{L^2}\|h\|_{H_{-N}^{-N}}\|f\|_{H_{-N}^{-N}}.
	\eeno
	
	\underline{\it Estimate of $G_{2,4}^{(2)}$.}  Since
	$G_{2,4}^{(2)}=\sum\limits_{l< N_0}\sum\limits_{|a-p|\leq N_0,|m-j|<N_0}\<Q_{-1}(\cP_l\tF_pg,\F_a(\F_p\U_{N_0}-\U_{N_0}\F_p)h),\F_m\F_j^2 \U_{N_0}f\>$,  Lemma \ref{31}(\ref{2.30}) and Lemma \ref{lemma1.3} imply that
	\beno
	|G_{2,4}^{(2)}|
	&\ls&C_{N}2^{-(3-2s)(p-j)}\|\tF_pg\|_{L^2_{(-\omega_3)^++(-\omega_4)^+}}(\|\mF_p h\|_{H^{c}_{\omega_3}}+2^{-pN}\|h\|_{H_{-N}^{-N}})(\|\mF_j f\|_{H^{d}_{\omega_4}}+2^{-jN}\|f\|_{H_{-N}^{-N}}).
	\eeno
	
	We conclude that
	$|G_{2,4}|\ls C_{N}2^{-(3-2s)(p-j)}\|\tF_pg\|_{L^2_{(-\omega_3)^++(-\omega_4)^+}}(\|\mF_p h\|_{H^{c}_{\omega_3}}+2^{-pN}\|h\|_{H_{-N}^{-N}})(\|\mF_j f\|_{H^{d}_{\omega_4}}+2^{-jN}\|f\|_{H_{-N}^{-N}}).$
	It is easy to check that structures of  $G_{2,2}$ and $G_{2,3}$ are similar to $G_{2,1}$ and $G_{2,4}$.
	We have
	\beno
	|G_2|
	&\ls &C_{N}2^{-(3-2s)(p-j)}\|\tF_pg\|_{L^2_{(-\omega_3)^++(-\omega_4)^+}}(\|\mF_p h\|_{H^{c}_{\omega_3}}+2^{-pN}\|h\|_{H_{-N}^{-N}})(\|\mF_j f\|_{H^{d}_{\omega_4}}+2^{-jN}\|f\|_{H_{-N}^{-N}}).
	\eeno
 
  Finally  due to Lemma \ref{lemma1.3}(iii)(iv), we conclude  (\ref{reg1}).

	\noindent\underline{\it Step 2: Proof of (ii).} Observe that $\<\F_j Q_{-1}(\tF_jg,S_{j-3N_0}h),\F_j f\>:=X_1+X_2$,
	where $X_1=\sum\limits_{l\geq N_0}\<\F_j Q_{-1}(\tF_j\cP_lg,\\ \tP_lS_{j-3N_0}h),\F_j \tP_lf\>+\sum\limits_{l<N_0}\<\F_j Q_{-1}(\tF_j\tP_lg,\U_{N_0}S_{j-3N_0}h),\F_j \U_{N_0}f\>$, $X_2=\sum_{l\geq N_0}(\<Q_{-1}(\cP_l\tF_jg,\tP_lS_{j-3N_0}h),(\tP_l\F_j^2\\-\F_j^2 \tP_l)f\>+\<Q_{-1}((\cP_l\tF_j-\tF_j\cP_l)g,\tP_lS_{j-3N_0}h),\F_j^2 \tP_lf\>)+\sum\limits_{l< N_0}(\<Q_{-1}(\cP_l\tF_jg,\U_{N_0}S_{j-3N_0}h),(\U_{N_0}\F_j^2 -\F_j^2 \U_{N_0})f\>+\<Q_{-1}((\cP_l\tF_j-\tF_j\cP_l)g,\U_{N_0}S_{j-3N_0}h),\F_j^2 \U_{N_0}f\>):=X_{2,1}+X_{2,2}+X_{2,3}+X_{2,4}$.

	$\bullet$ \noindent\underline{\it Estimate of $X_1$.} We split $X_1$ into two parts:  
$X_{1,1}=\sum\limits_{l\geq N_0}\<\F_j Q_{-1}(\tF_j\cP_lg,S_{j-2N_0}\tP_lS_{j-3N_0}h), 
\F_j \tP_lf\>+\sum\limits_{l< N_0}\<\F_j Q_{-1}(\tF_j\cP_lg,S_{j-2N_0}\U_{N_0}S_{j-3N_0}h),\F_j \U_{N_0}f\>\quad\mbox{and}\quad
X_{1,2}=\sum\limits_{l\geq N_0}\sum\limits_{a\geq j-2N_0}\<\F_j Q_{-1}(\tF_j\cP_lg,\F_a\tP_lS_{j-3N_0}h),\\
\F_j\tP_lf\>+\sum\limits_{l<N_0}\sum\limits_{a\geq j-2N_0}\<\F_j Q_{-1}(\tF_j\cP_lg,\F_a\U_{N_0}S_{j-3N_0}h),\F_j \U_{N_0}f\>.$
	
	\smallskip
	\noindent\underline{\it Estimate of $X_{1,1}$.}
	From Lemma \ref{31}, \eqref{FPPLC} and Lemma \ref{lemma1.4}(\ref{7.70}), we have
	\beno
	|X_{1,1}|&\ls& \sum_{l\geq N_0}2^{( 2s-1/2)j}\|\tP_lS_{j-3N_0}h\|_{L^2}\|\F_j\cP_lg\|_{L^2}\|\F_j\tP_lf\|_{L^2}+\sum_{l<N_0}2^{(2s-1/2)j}\|\U_{N_0}S_{j-3N_0}h\|_{L^2}\|\F_j\cP_lg\|_{L^2}\\
	&&\times\|\F_j\U_{N_0}f\|_{L^2}
	\ls C_{N}\|h\|_{L^2_{(-\omega_3)^++(-\omega_4)^+}}(\|\mF_jg\|_{H^{c}_{\omega_3}}+2^{-jN}\|g\|_{H_{-N}^{-N}})(\|\mF_jf\|_{H^{d}_{\omega_4}}+2^{-jN}\|f\|_{H_{-N}^{-N}}).
	\eeno

	\noindent\underline{\it Estimate of $X_{1,2}$.} Due to Lemma \ref{lemma1.7}$(iv)$ and Bernstein's inequality(see Lemma \ref{7.8}), one has
	\beno
	|X_{1,2}|&\ls&\sum_{l\geq N_0}\sum_{a\geq j-2N_0}2^{\frac{3}{2}j}\|\tF_j\cP_lg\|_{L^1}\|\F_a\tP_lS_{j-3N_0}h\|_{L^2}2^{2sj}\|\F_j\tP_lf\|_{L^2}\\
	&&+\sum_{l< N_0}\sum_{a\geq j-2N_0}2^{\frac{3}{2}j}\|\tF_j\cP_lg\|_{L^1}\|\F_a\U_{N_0}S_{j-3N_0}h\|_{L^2}2^{2sj}\|\F_j\U_{N_0}f\|_{L^2}).\eeno
	Applying Lemma \ref{lemma1.3} to $\F_a\tP_lS_{j-3N_0}h$ and $\F_a\U_{N_0}S_{j-3N_0}h$, we get that
  $|X_{1,2}|
       \ls C_{N}2^{-2jN}\|g\|_{L^1}\|h\|_{H_{-N}^{-N}}\|f\|_{H_{-N}^{-N}}. $
	
	We conclude that
	\beno
	|X_{1}|&\ls& C_{N}\|h\|_{L^2_{(-\omega_3)^++(-\omega_4)^+}}(\|\mF_jg\|_{H^{c}_{\omega_3}}+2^{-jN}\|g\|_{H_{-N}^{-N}})(\|\mF_jf\|_{H^{d}_{\omega_4}}+2^{-jN}\|f\|_{H_{-N}^{-N}})\\
	&&+C_{N}2^{-2jN}\|g\|_{L^1}\|h\|_{H_{-N}^{-N}}\|f\|_{H_{-N}^{-N}}.
	\eeno

	$\bullet$ \noindent\underline{\it Estimate of $X_2$.} We  give the estimates term by term.
	
	\noindent\underline{\it Estimate of $X_{2,1}$.} We have  $X_{2,1}=X_{2,1}^{(1)}+X_{2,1}^{(2)}+X_{2,1}^{(3)}$, where  $X_{2,1}^{(1)}=\sum\limits_{l\geq N_0}\sum\limits_{|a-j|>N_0}\<Q_{-1}(\cP_l\tF_jg,\\\tP_lS_{j-3N_0}h),\F_a(\tP_l\F_j^2 -\F_j^2 \tP_l)f\>$, $X_{2,1}^{(2)}=\sum\limits_{l\geq N_0}\sum\limits_{|a-j|\leq N_0}\sum\limits_{b>j-2N_0}\<Q_{-1}(\cP_l\tF_jg,\F_b\tP_lS_{j-3N_0}h),\F_a(\tP_l\F_j^2 \\-\F_j^2 \tP_l)f\>$ and  $X_{2,1}^{(3)}=\sum\limits_{l\geq N_0}\sum\limits_{|a-j|\leq N_0}\<Q_{-1}(\cP_l\tF_jg,S_{j-2N_0}\tP_lS_{j-3N_0}h),\F_a(\tP_l\F_j^2 -\F_j^2 \tP_l)f\>$.
	
	Similar to the estimate of $G_{2,1}^{(1)}$, we have $|X_{2,1}^{(1)}|+|X_{2,1}^{(2)}|\ls C_{N}2^{-2jN}\|g\|_{L^1}\|h\|_{H_{-N}^{-N}}\|f\|_{H_{-N}^{-N}}.$
	For $X_{2,1}^{(3)}$, we have   $|\xi|\sim 2^j$ and $|\xi-\eta|\ls 2^{j-2N_0}$, which implies $|\eta|\sim 2^j$. Similar to $X_{1,1}$, we get that
$ 
  |X_{2,1}^{(3)}|\ls C_{N}\|h\|_{L^2_{(-\omega_3)^++(-\omega_4)^+}}(\|\mF_jg\|_{H^{c}_{\omega_3}}+2^{-jN}\|g\|_{H_{-N}^{-N}})(\|\mF_jf\|_{H^{d}_{\omega_4}}+2^{-jN}\|f\|_{H_{-N}^{-N}}).$
	Then we conclude that
\[  |X_{2,1}|\ls C_{N}\|h\|_{L^2_{(-\omega_3)^++(-\omega_4)^+}}(\|\mF_jg\|_{H^{c}_{\omega_3}}+2^{-jN}\|g\|_{H_{-N}^{-N}})(\|\mF_jf\|_{H^{d}_{\omega_4}}+2^{-jN}\|f\|_{H_{-N}^{-N}})\]\[+C_{N}2^{-2jN}\|g\|_{L^1}\|h\|_{H_{-N}^{-N}}\|f\|_{H_{-N}^{-N}}.\]
 
	\noindent\underline{\it Estimate of $X_{2,4}$.} We introduce $X_{2,4}=X_{2,4}^{(1)}+X_{2,4}^{(2)}$, where $X_{2,4}^{(1)}=\sum\limits_{l<N_0}\sum\limits_{a>j-2N_0}\<Q_{-1}((\cP_l\tF_j-\tF_j\cP_l)g,\\\F_a\U_{N_0}S_{j-3N_0}h),\F_j^2 \U_{N_0}f\>$ and $X_{2,4}^{(2)}=\sum\limits_{l<N_0}\<Q_{-1}((\cP_l\tF_j-\tF_j\cP_l)g,S_{j-2N_0}\U_{N_0}S_{j-3N_0}h),\F_j^2 \U_{N_0}f\>$.
	
	Similar to $X_{1.2}$, we first have
 $|X_{2,4}^{(1)}| 
    \ls C_{N}2^{-2jN}\|g\|_{L^1}\|h\|_{H_{-N}^{-N}}\|f\|_{H_{-N}^{-N}}.$
	Copying the argument used for $X_{2,1}^{(3)}$ to $X_{2,4}^{(2)}$, then we have
	$|X_{2,4}^{(2)}| \ls  C_{N}\|h\|_{L^2_{(-\omega_3)^++(-\omega_4)^+}}(\|\mF_jg\|_{H^{c}_{\omega_3}}+2^{-jN}\|g\|_{H_{-N}^{-N}})(\|\mF_jf\|_{H^{d}_{\omega_4}}+2^{-jN}\|f\|_{H_{-N}^{-N}}).$
	Then we conclude that
	\beno
	|X_{2,4}|&\ls& C_{N}\|h\|_{L^2_{(-\omega_3)^++(-\omega_4)^+}}(\|\mF_jg\|_{H^{c}_{\omega_3}}+2^{-jN}\|g\|_{H_{-N}^{-N}})(\|\mF_jf\|_{H^{d}_{\omega_4}}+2^{-jN}\|f\|_{H_{-N}^{-N}})\\
  &&+C_{N}2^{-2jN}\|g\|_{L^1}\|h\|_{H_{-N}^{-N}}\|f\|_{H_{-N}^{-N}}.
	\eeno
	
	The estimates of $X_{2,2}$ and $X_{2,3}$ could be handled in a similar manner as  $X_{2,1}$ and $X_{2,4}$. We skip the details here and then conclude our desired result.
	\smallskip
	
	\noindent\underline{\it Step 3: Proof of (iii).} For $(iii)$, we introduce the following decomposition:  $\<\F_j Q_{-1}(S_{p+4N_0}g,\F_ph)\\-Q_{-1}(S_{p+4N_0}g,\F_j \F_ph),\F_j\<D\>^\ell f\>=\sum\limits_{i=1}^4Y_i$, where
	 $Y_1=\sum\limits_{l\geq N_0}\<\F_j Q_{-1}(\cP_lS_{p+4N_0}g,\F_p\tP_lh)-Q_{-1}(\cP_lS_{p+4N_0}g,\\ \F_j \F_p\tP_lh),\F_j \tP_lf\>$, $Y_2=\sum\limits_{l<N_0}\<\F_j Q_{-1}(\cP_lS_{p+4N_0}g,\F_p\U_{N_0}h)-Q_{-1}(\cP_lS_{p+4N_0}g,\F_j \F_p\U_{N_0}h),\F_j \U_{N_0}f\>$, $Y_3=\sum\limits_{l\geq N_0}\big(\<Q_{-1}(\cP_lS_{p+4N_0}g,\tP_l\F_ph),(\tP_l\F^2_j -\F^2_j \tP_l)f\>+\<Q_{-1}(\cP_lS_{p+4N_0}g,\tP_l\F_j \F_ph), (\tP_l\F_j -\F_j \tP_l)f\>
  +\<Q_{-1}(\cP_lS_{p+4N_0}g,\\\F_j (\F_p\tP_l-\tP_l\F_p)h),\F_j \tP_lf)\>+\<Q_{-1}(\cP_lS_{p+4N_0}g,(\F_j\tP_l-\tP_l\F_j )\F_jh),\F_j
  \tP_lf)\>\big), Y_4=\sum\limits_{l< N_0}\big(\<Q_{-1}(\cP_lS_{p+4N_0}g,\\\U_{N_0}\F_ph),(\U_{N_0}\F^2_j -\F^2_j \U_{N_0})f\>+\<Q_{-1}(\cP_lS_{p+4N_0}g,\U_{N_0}\F_j \F_ph)(\U_{N_0}\F_j
  -\F_j \U_{N_0})f\>+\<Q_{-1}(\cP_lS_{p+4N_0}g,\F_j (\F_p\U_{N_0}-\U_{N_0}\F_p)h),\F_j\U_{N_0}f\> +\<Q_{-1}(\cP_lS_{p+4N_0}g,(\F_j \U_{N_0}-\U_{N_0}\F_j)\F_jh),\F_j \U_{N_0}f\>\big)$

	Since  $Y_1$ and $Y_2$ enjoy almost the same structure, we only need to give the detailed proof for $Y_1$. By \eqref{bobylev}, we first note that
	$Y_1=\sum\limits_{l\geq N_0}\<\F_j
	Q_{-1}(S_{p+4N_0}\cP_lS_{p+4N_0}g,\F_p\tP_lh)-Q_{-1}(S_{p+4N_0}\cP_lS_{p+4N_0}g,\F_j \F_p\tP_lh),\\\F_j \tP_lf\>$. Then by (\ref{2.32}), \eqref{FPPLC} and Lemma \ref{lemma1.3}(iii), one has
	\beno |Y_1|
	&\ls& C_{N}\|g\|_{L^2_{(-\omega_3)^++(-\omega_4)^+}}(\|\mF_jh\|_{H^{c}_{\omega_3}}+2^{-pN}\|h\|_{H_{-N}^{-N}}) (\|\mF_jf\|_{H^{d}_{\omega_4}}+2^{-jN}\|f\|_{H_{-N}^{-N}}).
	\eeno
	As for $Y_3$ and $Y_4$, let us choose $Z:=\sum\limits_{l\geq N_0}\<Q_{-1}(\cP_lS_{p+4N_0}g,\tP_l\F_ph),(\tP_l\F^2_j -\F^2_j \tP_l)f\>$ as a typical term to give the estimate. It is easy to see that $Z:=Z_1+Z_2+Z_3$ where $Z_1=\sum\limits_{l\geq N_0}\sum\limits_{|a-p|>N_0}\<Q_{-1}(\cP_lS_{p+4N_0}g,\F_a\tP_l\F_ph),\\(\tP_l\F^2_j -\F^2_j \tP_l)f\>$, $Z_2=\sum\limits_{l\geq N_0}\sum\limits_{|a-p|\leq N_0}\sum\limits_{|b-j|>N_0}\<Q_{-1}(\cP_lS_{p+4N_0}g,\F_a\tP_l\F_ph), \F_b(\tP_l\F^2_j -\F^2_j \tP_l)f\>$ and\\
	$Z_3=\sum\limits_{l\geq N_0}\sum\limits_{|a-p|\leq N_0}\sum\limits_{|b-j|\leq N_0}\<Q_{-1}(S_{p+6N_0}\cP_lS_{p+4N_0}g,\F_a\tP_l\F_ph),\F_b(\tP_l\F^2_j -\F^2_j \tP_l)f\>$.

	Applying the argument used for $G^{(1)}_{2,1}$, one may obtain that
	$|Z_1+Z_2|\ls C_{N}2^{-2jN}\|g\|_{L^1}\|h\|_{H_{-N}^{-N}}\|f\|_{H_{-N}^{-N}}$.  Moreover, due to Lemma \ref{lemma1.7}(iv), Lemma \ref{lemma1.3}(iii) and Lemma \ref{lemma1.4},  we have
	\beno
	|Z_3|&\ls&C_{N}\|g\|_{L^2_{2+(-\omega_3)^++(-\omega_4)^+}}\|\F_ph\|_{H^{c}_{\omega_3}}(\|\mF_jf\|_{H^{d}_{\omega_4}}+2^{-jN}\|f\|_{H_{-N}^{-N}}).
	\eeno
	From these, we get that $|Z|\ls C_{N}\|g\|_{L^2_{2+(-\omega_3)^++(-\omega_4)^+}}(\|\mF_ph\|_{H^{c}_{\omega_3}}+2^{-pN}\|h\|_{H_{-N}^{-N}}) (\|\mF_jf\|_{H^{d}_{\omega_4}}+2^{-jN}\|f\|_{H_{-N}^{-N}}).$
	We complete the proof of $(iii)$ by patching together the estimates of $Y_i$.
	This ends the proof of the lemma.
\end{proof}

\begin{lem}\label{F_jQ}
 Let  $c_j,d_j\geq0,c_j+d_j=( 2s-1/2)^+\mathbf{1}_{2s\neq1}+(1/2+\de)\mathbf{1}_{2s=1},j=1,2,3$. $\om_i\in\R,i=1,\cdots,6$. Then for any $N\in \N$, we have
\ben\label{Q_1}
\notag  &&|(\F_j Q_{-1}(g,h)-Q_{-1}(g,\F_j h),\F_j f)|\\
\notag  &\ls& C_{N}(\|g\|_{L^2_{2-\omega_1-\omega_2}}(\|\mF_jh\|_{H^{c_1}_{\omega_1}}+2^{-jN}\|h\|_{H_{-N}^{-N}})(\|\mF_jf\|_{H^{d_1}_{\omega_2}}+2^{-jN}\|f\|_{H_{-N}^{-N}})\\
\notag  &&+\sum_{p>j+3N_0}C_{N}2^{-(3-2s)(p-j)}\|g\|_{L^2_{-\omega_3-\omega_4}}(\|\mF_p h\|_{H^{c_2}_{\omega_3}}+2^{-pN}\|h\|_{H_{-N}^{-N}})(\|\mF_j f\|_{H^{d_2}_{\omega_4}}+2^{-jN}\|f\|_{H_{-N}^{-N}})\\
\notag  &&+\notag C_{N}\|h\|_{L^2_{-\omega_5-\omega_6}}(\|\mF_jg\|_{H^{c_3}_{\omega_5}}+2^{-jN}\|g\|_{H_{-N}^{-N}})(\|\mF_jf\|_{H^{d_3}_{\omega_6}}+2^{-jN}\|f\|_{H_{-N}^{-N}})+C_N2^{-jN}\|g\|_{L^1}\|h\|_{H_{-N}^{-N}}\|f\|_{H_{-N}^{-N}}.\\
&&
\een
  
Let $a,b\geq0, a+b=(2s-1)\mathbf{1}_{2s>1}+(2s-1+\de)\mathbf{1}_{2s=1}+0\cdot \mathbf{1}_{2s<1},\om_1+\om_2=(\ga+2s-1)\mathbf{1}_{2s\neq1}+(\ga+2s-1+\de)\mathbf{1}_{2s=1},\de\ll1$. Then for any $N\in\N$,
  \ben\label{Q_2}
 \sum_{l=0}^\infty|(\F_j Q_{l}(g,h)-Q_{l}(g,\F_j h),\F_j f)| 
\ls \|g\|_{L^1_2}(\|\mF_jh\|_{H^a_{\om_1}}+2^{-jN}\|h\|_{H^{-N}_{-N}})(\|\mF_jf|\|_{H^b_{\om_2}}+2^{-jN}\|f\|_{H^{-N}_{-N}}).
\een
\end{lem}
\begin{proof}[Proof of Lemma \ref{F_jQ}]
\eqref{Q_1} is easily derived from \eqref{reg1},\eqref{reg2} and \eqref{reg3}
in Lemma \ref{32} since
\[(\F_j Q_{-1}(g,h)-Q_{-1}(g,\F_j h),\F_j f)=\sum_{|p-j|\leq3N_0}\<\F_j Q_{-1}(S_{p+4N_0}g,\F_ph)-Q_{-1}(S_{p+4N_0}g,\F_j \F_ph),f\>\]\[ +\sum_{p>j+3N_0}\<\F_j Q_{-1}(\tF_{p}g,\F_ph),\F_j f\>
+\<\F_j Q_{-1}(\tF_jg,S_{j-3N_0}h),\F_j f\>.\] As for \eqref{Q_2}, we refer readers to follow the proof of Lemma 2.1-Lemma 2.7  in \cite{HJ2}. We omit details here.
\end{proof}

 {\bf Acknowledgments.} Ling-Bing He and Jie Ji are supported by NSF of China under  Grant  12141102.

\end{document}